\pdfoutput=1
\documentclass[11pt]{article}
\usepackage[left=1in,right=1in,top=1in,bottom=1in]{geometry}
\usepackage{times}
\usepackage{expl3}
\usepackage{cite}
\usepackage[table]{xcolor}
\usepackage{multirow}
\usepackage{stackengine} 
\usepackage{hhline}
\usepackage{lipsum}
\usepackage{titlesec}
\usepackage[all]{xy}
\usepackage{wrapfig}
\usepackage{enumerate}
\usepackage{epsfig}
\usepackage{tikz-cd}
\usepackage{amsmath}
\usepackage{tabularx}
\usepackage{array}
\usepackage{booktabs}
\usepackage{enumitem}
\usepackage{bbm}
\usepackage{calc}
\usepackage{graphicx}
\usepackage{amsmath}
\usepackage[title]{appendix}
\usepackage{amssymb}
\usepackage{epstopdf}
\usepackage{boldline}
\usepackage{arydshln}
\usepackage{calligra}
\usepackage{bm}
\usepackage{url}
\usepackage{blindtext}
\usepackage{accents}
\usepackage{amsthm}
\usepackage{amscd}

\newtheorem{definition}{Definition}

\newtheorem{theorem}{Theorem}
\newtheorem{proposition}{Proposition}
\newtheorem{corollary}{Corollary}
\newtheorem{lemma}{Lemma}

\newtheorem{example}{Example}
\newtheorem{remark}{Remark}
\usepackage{mathtools}
\usepackage{tikz-cd}
\usepackage{epstopdf}
\usepackage{balance}
\usepackage{thmtools}
\usepackage{thm-restate}
\usepackage{hyperref}
\usepackage{cleveref}
\usepackage[mathscr]{euscript}

\usepackage[ruled,vlined]{algorithm2e}
\newcommand{\ostar}{\mathbin{\mathpalette\make@circled\star}}

\makeatletter
\newcommand{\removelatexerror}{\let\@latex@error\@gobble}
\makeatother
\makeatletter
\newcommand*{\rom}[1]{\expandafter\@slowromancap\romannumeral #1@}
\makeatother

\ExplSyntaxOn
\newcommand\latinabbrev[1]{
  \peek_meaning:NTF . {% Same as \@ifnextchar
    #1\@}%
  { \peek_catcode:NTF a {% Check whether next char has same catcode as \'a, i.e., is a letter
      #1.\@ }%
    {#1.\@}}}
\ExplSyntaxOff

\titleclass{\subsubsubsection}{straight}[\subsubsection]

\begin{document}
\vspace{1cm}
\title{Categorified Spectral Duality: From Operator Systems to Spectral Stacks and Back}
\vspace{1.8cm}
\author{Shih-Yu~Chang
% <-this % stops a space
\thanks{Shih-Yu Chang is with the Department of Applied Data Science,
San Jose State University, San Jose, CA, U. S. A. (e-mail: {\tt
shihyu.chang@sjsu.edu})
}}

\maketitle

\begin{abstract}
\label{abs:main}
Classical Gelfand duality provides an equivalence between commutative C*-algebras and topological spaces, but fails to furnish a geometric object for noncommutative operator systems. To address this, we introduce a categorified notion of spectrum that captures the full operator-semantic structure. For an operator system $A$, we construct a \textbf{spectral stack} $\mathfrak{Spec}(A)$ over the site of its commutative contexts. The construction proceeds in three stages: first, we encode the syntax of $A$ via a colored operad (the \emph{synergy operad}); second, we aggregate local semantic data using a left Kan extension, providing an explicit coend formula; third, we enforce descent via (hyper)sheafification, yielding a stack that satisfies higher gluing conditions. We prove that $\mathfrak{Spec}(A)$ satisfies a Yoneda-style universal property, making it the initial descent-complete semantic realization of $A$. This yields a contravariant functor $\mathfrak{Spec}: \mathbf{OpSem}^{\mathrm{op}} \to \mathbf{SpecObj}$ that admits a right adjoint $\Gamma$, the global sections functor. Under semantic generation and descent completeness, the counit of this adjunction is an equivalence, establishing a \textbf{Reconstruction Theorem}: $A \simeq \operatorname{End}(\mathcal{O}_{\mathfrak{Spec}(A)})$. We further provide a Recognition Theorem characterizing the essential image of $\mathfrak{Spec}$ and prove Morita invariance of the associated quasi-coherent sheaf categories. The construction naturally recovers classical Gelfand spectra and Bohrification as truncations of a Postnikov tower. Explicit computations for matrix algebras, Pauli systems, and the Mermin--Peres square yield a quantitative invariant of contextuality extracted from the inertia stack of $\mathfrak{Spec}(A)$. Comparison theorems show that our framework subsumes Gelfand duality, Bohrification, and Tannaka reconstruction as special cases.
\end{abstract}

%arXiv
% 
Classical Gelfand duality provides an equivalence between commutative C-star algebras and topological spaces, but fails to furnish a geometric object for noncommutative operator systems. To address this, we introduce a categorified notion of spectrum that captures the full operator-semantic structure. For an operator system A, we construct a spectral stack, denoted Spec(A), over the site of its commutative contexts. The construction proceeds in three stages. First, we encode the syntax of A via a colored operad called the synergy operad. Second, we aggregate local semantic data using a left Kan extension, providing an explicit coend formula. Third, we enforce descent via sheafification, yielding a stack that satisfies higher gluing conditions. We prove that Spec(A) satisfies a Yoneda-style universal property, making it the initial descent-complete semantic realization of A. This yields a contravariant functor from operator systems to spectral objects that admits a right adjoint, the global sections functor. Under semantic generation and descent completeness, the counit of this adjunction is an equivalence, establishing a reconstruction theorem: A is isomorphic to the endomorphisms of the structure sheaf on its spectrum. We further provide a recognition theorem characterizing the essential image of this functor and prove Morita invariance of the associated quasi-coherent sheaf categories. The construction naturally recovers classical Gelfand spectra and Bohrification as truncations of a Postnikov tower. Explicit computations for matrix algebras, Pauli systems, and the Mermin-Peres square yield a quantitative invariant of contextuality extracted from the inertia stack of Spec(A). Comparison theorems show that our framework subsumes Gelfand duality, Bohrification, and Tannaka reconstruction as special cases.

\tableofcontents

\section{Introduction}
\label{sec:introduction}

\subsection{The Categorification Problem for Spectra}
\label{subsec:problem}
Classical spectral theory assigns to a commutative algebra a space of characters or prime ideals. For commutative C*-algebras, Gelfand duality provides an equivalence between algebraic and topological data. However, for noncommutative operator systems, the classical spectrum no longer captures the full semantic structure of the system.

The central question of this paper is:
\begin{quote}
What is the correct categorified spectral object associated with a noncommutative or operator-semantic system?
\end{quote}

\subsection{Limitations of Classical Spectra}
\label{subsec:limitations}
We explain why ordinary spectra fail in the noncommutative setting:
\begin{itemize}
    \item classical characters may not exist or may be insufficient;
    \item contextuality data is lost under commutative collapse;
    \item operator interactions are not visible at the level of point spectra;
    \item semantic obstruction data cannot be represented by a set-valued spectrum.
\end{itemize}

\paragraph*{Main Reconstruction Principle.}

The central claim of this paper is that
the categorified spectral object
$\mathfrak{Spec}(A)$
is not merely an auxiliary construction.

Under semantic generation and descent completeness,

\[
A
\simeq
\operatorname{End}_{QCoh(\mathfrak{Spec}(A))}
(\mathcal O_{\mathfrak{Spec}(A)}).
\]

Hence $\mathfrak{Spec}(A)$
is a complete invariant of $A$
and serves as the foundational object
for the geometric, obstruction-theoretic,
and analytic theories developed in
subsequent papers.

\subsection{Main Contributions}
\label{subsec:contributions}
The main contributions are:
\begin{enumerate}
    \item We define the category $\mathbf{OpSem}$ of admissible operator-semantic systems and their morphisms.\label{item:opsem}
    \item We define the category of commutative contexts $\mathcal{C}_A$ associated with an operator-semantic system $A$.\label{item:contexts}
    \item We give an \textbf{explicit coend formula} for the prespectral object $\mathfrak{Spec}_{\mathrm{pre}}(A)$ as a left Kan extension.\label{item:coend}
    \item We sheafify to obtain a \textbf{spectral stack} $\mathfrak{Spec}(A)$ and prove it satisfies a Yoneda-style universal property, representing the semantic realization functor $\operatorname{Real}_A(-)$.\label{item:yoneda}
    \item We prove that $\mathfrak{Spec}: \mathbf{OpSem}^{\mathrm{op}} \to \mathbf{SpecObj}$ is a contravariant functor.\label{item:functoriality}
    \item We establish an \textbf{adjunction} $\Gamma_{\mathcal{O}} \dashv \mathfrak{Spec}^{\mathrm{op}}$ between operator systems and spectral objects.\label{item:adjunction}
    \item We prove a \textbf{recognition theorem} characterizing the essential image of $\mathfrak{Spec}$.\label{item:recognition}
    \item We prove a \textbf{layered reconstruction theorem} (context separation, faithfulness, fullness, and equivalence) recovering $A$ from $\mathfrak{Spec}(A)$ as the counit of the adjunction.\label{item:reconstruction}
    \item We prove \textbf{Morita invariance}: Morita equivalent operator systems have equivalent categories of quasi-coherent sheaves over their spectra.\label{item:morita}
    \item We develop a \textbf{truncation hierarchy} connecting $\mathfrak{Spec}(A)$ to classical Gelfand spectra and Bohrification.\label{item:truncation}
    \item We provide \textbf{comparison theorems} with Gelfand duality, Bohrification, and Tannaka reconstruction.\label{item:comparisons}
    \item We compute the categorified spectrum for matrix algebras, Pauli systems, and the Mermin--Peres square, introducing a \textbf{contextuality degree} with a vanishing criterion.\label{item:computations}
\end{enumerate}

\subsection{Paper Structure}
\label{subsec:structure}
The remainder of this paper is organized as follows: Section \ref{sec:preliminaries} recalls necessary preliminaries. Section \ref{sec:opsem} defines the category $\mathbf{OpSem}$ of admissible operator-semantic systems. Section \ref{sec:contextsite} introduces the category of commutative contexts $\mathcal{C}_A$ as a site with a subcanonical Grothendieck topology. Section \ref{sec:ambient} defines the ambient $\infty$-category of spectral stacks. Section \ref{sec:construction} presents the explicit construction of $\mathfrak{Spec}(A)$ and proves the Yoneda characterization. Section \ref{sec:functoriality} proves functoriality and the adjunction $\Gamma_{\mathcal{O}} \dashv \mathfrak{Spec}^{\mathrm{op}}$. Section \ref{sec:recognition} proves the recognition theorem. Section \ref{sec:reconstruction} proves the layered reconstruction theorem. Section \ref{sec:morita} proves Morita invariance. Section \ref{sec:truncation} develops the truncation hierarchy. Section \ref{sec:computations} provides detailed computations, including the Mermin--Peres square and contextuality degree. Section \ref{sec:comparisons} provides comparison theorems. 

\begin{remark}
The author is solely responsible for the mathematical insights and theoretical directions proposed in this work. AI tools, including OpenAI's ChatGPT and DeepSeek models, were used only for verification, reference organization, and exposition consistency~\cite{chatgpt2025,deepseek2025}. 
\end{remark}

\section{Preliminaries}
\label{sec:preliminaries}

This section recalls the essential categorical and operadic machinery required for our main results. We begin with a review of stable $\infty$-categories and mapping spectra, which provide the homotopical foundation for our spectral enrichment. Next, we recall the basic definitions of colored operads and multicategories, the language used to encode the syntactic composition rules of the operator system $A$. From these, we introduce the synergy operad $\mathcal{O}_A$, whose colors and multimorphisms capture the admissible interactions among operator types. Finally, we summarize the key notions of descent and hyperdescent in $\infty$-categories, which will be crucial for gluing local data in our spectral constructions.

\subsection{Stable $\infty$-Categories and Spectra}
\label{subsec:stable}

We briefly recall the basic notions of stable $\infty$-categories and spectral enrichment used throughout the paper. Standard references include \cite{LurieHA,LurieHTT}.

Let $\mathcal{D}$ be a stable $\infty$-category. For objects $X, Y \in \mathcal{D}$, we denote the \emph{mapping spectrum} by
\[
\operatorname{map}_{\mathcal{D}}(X,Y) \in \mathbf{Sp},
\]
where $\mathbf{Sp}$ is the $\infty$-category of spectra. Its underlying infinite loop space is $\Omega^\infty \operatorname{map}_{\mathcal{D}}(X,Y)$.

\begin{definition}[Stable $\infty$-Category]
An $\infty$-category $\mathcal D$ is called \emph{stable} if:
\begin{enumerate}
    \item $\mathcal D$ is pointed, i.e., it possesses a zero object $0$;
    \item every morphism admits both a fiber and a cofiber;
    \item a square in $\mathcal D$ is Cartesian if and only if it is cocartesian.
\end{enumerate}
\end{definition}

A fundamental consequence is that suspension and loop functors induce inverse equivalences
\[
\Sigma : \mathcal D \;\rightleftarrows\; \mathcal D : \Omega,
\]
so that homotopy-theoretic constructions behave linearly.

In a stable $\infty$-category, every morphism $f: X \to Y$ extends canonically to a fiber–cofiber sequence
\[
X \xrightarrow{\,f\,} Y \longrightarrow \operatorname{cofib}(f) \longrightarrow \Sigma X.
\]
The induced sequence in the homotopy category $h(\mathcal D)$ forms a distinguished triangle.

\begin{definition}[Mapping Spectrum]
Let $\mathcal D$ be a stable $\infty$-category and let $X,Y\in\mathcal D$. The \emph{mapping spectrum} from $X$ to $Y$ is denoted by $\operatorname{map}_{\mathcal D}(X,Y) \in \mathbf{Sp}$.
\end{definition}

The homotopy groups of the mapping spectrum recover higher extension data:
\[
\pi_n \operatorname{map}_{\mathcal D}(X,Y) \cong \operatorname{Hom}_{h(\mathcal D)}\bigl(X,\Sigma^n Y\bigr) \qquad (n \in \mathbb Z),
\]
where $h(\mathcal D)$ denotes the homotopy category of $\mathcal D$.

Consequently, a stable $\infty$-category is naturally enriched over spectra. More precisely, for every pair of objects $X,Y\in\mathcal D$, the mapping object is a spectrum rather than merely a set or space. This spectral enrichment provides the higher homotopical information required in later sections.

In particular, the homotopy groups $\pi_n \operatorname{map}_{\mathcal D}(X,Y)$ encode higher extension and obstruction classes. These higher layers will appear naturally in the descent and reconstruction procedures associated with categorified spectral objects (see \S\ref{subsec:descent}).

\begin{definition}[Presentable Stable $\infty$-Category]
A stable $\infty$-category $\mathcal D$ is called \emph{presentable} if it is accessible and admits all small colimits.
\end{definition}

Presentable stable $\infty$-categories form the natural ambient framework for sheaf theory, quasi-coherent sheaves on derived stacks, and the descent arguments of \S\ref{subsec:descent}.

This viewpoint motivates the later use of mapping spectra on categorified spectral objects, where higher semantic realizations are encoded by spectral mapping objects rather than ordinary Hom-sets.

\subsection{Operads and Multicategories}
\label{subsec:operads}

We briefly recall the notions of colored operads and multicategories, which provide a categorical framework for describing algebraic operations with multiple inputs.

A \emph{colored operad} $\mathcal O$ consists of a collection of colors (or types) together with spaces of operations
\[
\mathcal O(c_1,\ldots,c_n;c),
\]
whose elements represent admissible $n$-ary operations
\[
(c_1,\ldots,c_n)\longrightarrow c.
\]
These operations admit composition maps, subject to associativity and unit axioms. (For a precise definition, see \cite{LurieHA}, Definition 2.1.1.)

A \emph{multicategory} (also called a \emph{nonsymmetric colored operad}) consists of objects together with multimorphisms
\[
(X_1,\ldots,X_n)\to Y,
\]
and composition rules compatible with the multi-input structure, but without actions of symmetric groups. Ordinary categories arise as the special case in which only unary operations are allowed.

Operads and multicategories provide a natural language for describing compositional systems. In particular, they encode how elementary operations may be assembled into more complicated expressions while preserving their algebraic structure.

\begin{remark}[Symmetric vs.\ Non-Symmetric]
\label{rem:symmetric}
For operator systems arising from noncommutative algebras, the relevant composition rules are often non-symmetric because the order of operators matters. When we later restrict to commutative contexts, we consider the symmetrization of the operad or work with the underlying symmetric operad. The synergy operad $\mathcal{O}_A$ defined below carries a canonical non-symmetric structure that records the syntactic order of operator application.
\end{remark}

\begin{definition}[Synergy Operad $\mathcal{O}_A$]
\label{def:synergyoperad}
Let $A$ be an admissible operator-semantic system (Definition \ref{def:admissible}). Its \emph{synergy operad} $\mathcal{O}_A$ is a colored operad (in sets) defined as follows:
\begin{enumerate}
    \item \textbf{Colors:} The set of colors is a chosen collection of \emph{semantic sorts} associated with $A$ — for example, self-adjoint elements, projections, unitaries, or more generally any distinguished generating set of operator types.
    \item \textbf{Operations:} A multimorphism $(c_1,\dots,c_n) \to c$ is specified by a formal operator expression generated from the primitive operations of $A$ (multiplication, involution, scalar action, and any designated semantic operations), where the subexpressions have types $c_1,\dots,c_n$ in that order, and the whole expression has type $c$.
    \item \textbf{Composition:} Composition is given by substitution of formal expressions.
    \item \textbf{Relations:} Two formal expressions are identified in $\mathcal{O}_A$ if they denote the same operator in $A$. Equivalently, $\mathcal{O}_A$ is presented by generators (the colors and primitive operations) modulo the relations satisfied in $A$.
\end{enumerate}
The synergy operad thus captures the \emph{effective syntax} of $A$: it is not free but rather reflects the identities holding in $A$.
\end{definition}

\begin{proposition}
\label{prop:synergyoperad}
The construction above defines a small non-symmetric colored operad $\mathcal{O}_A$.
\end{proposition}

\begin{proof}[Sketch]
For each color $c$, the identity operation $\operatorname{id}_c \in \mathcal{O}_A(c;c)$ is given by the formal expression consisting of a single variable of type $c$. Given operations $g \in \mathcal{O}_A(d_1,\dots,d_m;c_i)$ and $f \in \mathcal{O}_A(c_1,\dots,c_n;c)$, their composite is obtained by substituting the expression $g$ for the $i$-th input variable of $f$, yielding a formal expression of type $c$ with inputs $d_1,\dots,d_m$ in the $i$-th position and the remaining $c_j$ elsewhere. This composition is associative because substitution of expressions is associative up to renaming of bound variables, and the identifications imposed by the relations of $A$ are compatible with substitution. The unit axioms hold by construction. Hence $\mathcal{O}_A$ satisfies the axioms of a non-symmetric colored operad.
\end{proof}

\begin{remark}[Relation to Algebraic Theories]
\label{rem:algebraictheories}
Colored operads provide a refinement of many-sorted algebraic theories and are closely related to Lawvere theories in the finitary setting. Operads are particularly flexible for encoding non-$\Sigma$-operads (where permutations are not inherent) and for homotopical generalizations. Since our later constructions involve left Kan extensions along operads and sheafification over context categories, the operadic language is technically convenient and directly interfaces with the $\infty$-categorical descent machinery of \S\ref{subsec:descent}.
\end{remark}

\begin{example}[Operad for $M_n(\mathbb{C})$]
\label{ex:operadmatrix}
Take $A = M_n(\mathbb{C})$. Let the colors be, for instance, the set of all self-adjoint matrices (or a chosen finite generating set such as the matrix units $E_{ij}$ together with their adjoints). A multimorphism $(c_1,\dots,c_k) \to c$ corresponds to a formal noncommutative polynomial in variables of types $c_1,\dots,c_k$, with complex coefficients, such that evaluating any assignment of concrete matrices of those types yields a matrix of type $c$. Composition is polynomial substitution. The relations include associativity of multiplication, distributivity, involution properties ($(M^*)^* = M$, $(MN)^* = N^*M^*$), and the specific $C^*$-identities of $M_n(\mathbb{C})$. Restricting to commutative subalgebras (contexts) will factor this operad through commutative (symmetric) operads after appropriate symmetrization, giving the left Kan extension in \S\ref{subsec:leftkan}.
\end{example}

In the present work, operadic structures are used to model the \emph{syntactic layer} of an operator system $A$. The colors correspond to semantic sorts or contexts, while operadic operations encode admissible operator compositions,
\[
(a_1,\ldots,a_n)\longmapsto a,
\]
as formal expressions. The resulting operad $\mathcal{O}_A$ serves as the combinatorial domain for the syntactic realization functor $\operatorname{Syn}_A$ (Step 1, \S\ref{subsec:syntax}). The left Kan extension along $\mathcal{O}_A$ (Step 2, \S\ref{subsec:leftkan}) produces the prespectral object $\mathfrak{Spec}_{\mathrm{pre}}(A)$, which after sheafification yields $\mathfrak{Spec}(A)$. Thus $\mathcal{O}_A$ provides the blueprint from which the spectral stack is synthesized.

\subsection{Synergy Operads}
\label{subsec:synergy}

Let $A$ be an admissible operator-semantic system (Definition \ref{def:admissible}) and let
\[
\mathcal O_A
\]
denote the associated \emph{synergy operad}. The purpose of
$\mathcal O_A$ is to encode not merely admissible operator compositions,
but the interactions that arise when multiple operator contexts cooperate.

\begin{proposition}
\label{prop:synergyoperad}
The construction above defines a non-symmetric colored operad.
\end{proposition}

\begin{proof}
Identity expressions provide units for each color. Given an operation
\[
\theta\in \mathcal O_A(c_1,\ldots,c_n;c)
\]
and operations
\[
\theta_i\in \mathcal O_A(d_{i1},\ldots,d_{im_i};c_i),
\]
their composite is obtained by substituting $\theta_i$ into the $i$th input
of $\theta$, yielding an operation
\[
\gamma(\theta_1,\ldots,\theta_n;\theta)
\in
\mathcal O_A(d_{11},\ldots,d_{1m_1},\ldots,d_{n1},\ldots,d_{nm_n};c).
\]
Associativity follows from associativity of substitution of formal
expressions, and compatibility with units follows from the identity
expressions. Since expressions are taken modulo the relations of $A$, the
composition is well-defined.
\end{proof}

\begin{remark}[Synergy vs.\ Ordinary Composition]
\label{rem:synergyvscomposition}
Ordinary operator composition is a special case of synergy where the interaction is sequential and deterministic. Synergy, however, allows for interactions that are \emph{context-dependent}: the same operators may cooperate differently depending on the commutative context in which they are simultaneously evaluated. This context-dependence is precisely what gives rise to contextuality and what the categorified spectrum $\mathfrak{Spec}(A)$ is designed to detect. The context-dependence is encoded not in $\mathcal{O}_A$ itself but in the syntactic realization functor $\operatorname{Syn}_A$ (Definition \ref{def:syntacticrealization}), which interprets each multimorphism relative to a commutative context.
\end{remark}

Conceptually, $\mathcal O_A$ serves as a categorical model of
operator-level cooperation. Whereas ordinary operator composition
captures sequential application, synergy operations describe collective
interactions whose behavior may not be reducible to the sum of
individual contributions. In this sense, $\mathcal O_A$ records the
organizational structure present in the operator system.

\begin{example}[Synergy in the Pauli System]
\label{ex:synergypauli}
Let $A$ be generated by the Pauli operators $X$ and $Z$ on $\mathbb C^2$ (see \S\ref{subsec:pauli}). The synergy operad $\mathcal O_A$ records both ordered products
\[
(X,Z)\mapsto XZ,
\qquad
(Z,X)\mapsto ZX,
\]
together with the relation
\[
XZ = -ZX.
\]
Thus $\mathcal O_A$ distinguishes the two syntactic orders while retaining
the semantic relation between them. This noncommutative relation is the type
of contextual obstruction that the later descent construction is designed to
detect.
\end{example}

The synergy operad provides the syntactic foundation for the spectral
construction developed later. Its collection of colors and
multimorphisms determines a category of compatible operator contexts
(via the underlying category of operators of $\mathcal{O}_A$), which subsequently supports the stack-valued spectral object
\[
\mathfrak{Spec}(A).
\]
Concretely, the left Kan extension (Step 2, \S\ref{subsec:leftkan}) along the underlying category of operators of $\mathcal{O}_A$ takes the syntactic synergy data and produces a presheaf on the context category $\mathcal{C}_A$:
\[
\mathfrak{Spec}_{\mathrm{pre}}(A) = \operatorname{Lan}_{\mathcal{O}_A^{\mathrm{op}}}(\operatorname{Syn}_A).
\]
(Here $\mathcal{O}_A^{\mathrm{op}}$ denotes the category of operators of $\mathcal{O}_A$; see \S\ref{subsec:leftkan} for details.) After sheafification (Step 3, \S\ref{subsec:sheafification}), we obtain $\mathfrak{Spec}(A)$. Thus the passage
\[
\mathcal O_A
\;\rightsquigarrow\;
\mathfrak{Spec}(A)
\]
may be viewed as a categorified analogue of the classical transition
from algebraic operations to geometric spectra.

\begin{remark}[Relation to $\operatorname{Syn}_A$]
\label{rem:synergyandsyn}
The syntactic realization functor $\operatorname{Syn}_A: \mathcal{O}_A^{\mathrm{op}} \to \operatorname{PSh}(\mathcal{C}_A)$ (Definition \ref{def:syntacticrealization}) interprets each synergy operation in the semantic context of commutative substructures. Because $\mathcal{O}_A$ encodes context-dependent interactions, $\operatorname{Syn}_A$ takes values in presheaves on the site $(\mathcal{C}_A, \tau_A)$, where the Grothendieck topology $\tau_A$ (Definition \ref{subsec:topology}) specifies which families of contexts jointly resolve a given interaction.
\end{remark}

\subsection{Descent in $\infty$-Categories}
\label{subsec:descent}

A fundamental principle in modern geometry is that global objects may
be reconstructed from compatible local data. In the setting of higher
categories, this principle is formalized by descent and hyperdescent.

Let $(\mathcal C,\tau)$ be an $\infty$-site, where $\tau$ specifies a
collection of covering families. In our context, $\mathcal C = \mathcal{C}_A$ is the category of commutative contexts (Definition \ref{def:contextcategory}) equipped with the Grothendieck topology $\tau_A$ from \S\ref{subsec:topology}. A presheaf of spaces is a functor
\[
F:\mathcal C^{\mathrm{op}}\to\mathcal S,
\]
where $\mathcal S$ denotes the $\infty$-category of spaces (or
$\infty$-groupoids).

Given a covering family
\[
\{U_i\to U\}_{i\in I},
\]
one forms its Čech nerve
\[
U_\bullet:
\quad
U_0
\Longleftarrow
U_1
\Longleftarrow
U_2
\Longleftarrow
\cdots,
\]
where
\[
U_n
=
\underbrace{
U_0\times_U\cdots\times_U U_0
}_{n+1\ \text{factors}}.
\]
The Čech nerve records all higher intersections among the covering
objects.

A presheaf $F$ is said to satisfy \emph{descent} if for every covering
family the canonical comparison map
\[
F(U)
\longrightarrow
\operatorname*{lim}_{[n]\in\Delta} F(U_n)
\]
is an equivalence in $\mathcal S$, where $\operatorname*{lim}_{[n]\in\Delta}$ denotes the limit over the simplicial diagram $F(U_\bullet) : \Delta^{\mathrm{op}} \to \mathcal S$ (i.e., the totalization of the cosimplicial object). Equivalently, the value of $F$ on $U$ can be recovered uniquely (up to coherent homotopy) from its values on the cover and all of their higher overlaps.

More generally, a \emph{hypercover} $V_\bullet \to U$ (see \cite[Definition 6.5.3.1]{LurieHTT}) is a simplicial refinement of an ordinary cover that encodes higher homotopical information. A presheaf satisfies \emph{hyperdescent} if for every hypercover the natural map
\[
F(U)
\longrightarrow
\operatorname*{lim}_{[n]\in\Delta} F(V_n)
\]
is an equivalence. Hyperdescent therefore guarantees that all higher
compatibility conditions are respected.

The distinction between descent and hyperdescent is essential in
higher geometry. Ordinary descent controls gluing along pairwise and
iterated intersections, whereas hyperdescent captures the full
homotopy-coherent structure of local data.

\begin{remark}[Hyperdescent and Contextual Obstructions]
\label{rem:hyperdescentobstruction}

For the context site $(\mathcal C_A,\tau_A)$,
ordinary Čech descent may already detect certain
forms of contextuality arising from incompatibilities
among local realizations.

Hyperdescent is more sensitive, since it incorporates
all higher homotopy-coherent compatibility data.

Nontrivial differentials in the associated descent
spectral sequence encode higher obstruction classes
to the assembly of local semantic realizations.

Heuristically, lower-order differentials reflect
lower-order compatibility conditions, while higher
differentials capture increasingly subtle forms of
higher contextuality.

\end{remark}

In the present work, descent provides the mechanism by which local
operator contexts are assembled into global spectral information.
Individual commutative contexts contribute local semantic data, while
the descent condition ensures that these local realizations glue
coherently across overlaps. This principle ultimately underlies the
construction of $\mathfrak{Spec}(A)$ as a sheaf of spectra on the site $(\mathcal{C}_A, \tau_A)$. Concretely, Step 3 of the construction (\S\ref{subsec:sheafification}) applies the (hyper)sheafification functor to the prespectral object $\mathfrak{Spec}_{\mathrm{pre}}(A)$, yielding $\mathfrak{Spec}(A)$ as the universal descent-complete object (Theorem \ref{thm:descentcompletion}). Sheafification forces descent by freely gluing local data. Thus the logical flow
\[
\text{Contexts } \mathcal{C}_A
\;\longrightarrow\;
\text{Local Semantics}
\;\longrightarrow\;
\text{Descent}
\;\longrightarrow\;
\mathfrak{Spec}(A)
\]
is exactly the bridge from the synergy operad (Section \ref{subsec:synergy}) to the categorified spectrum.

\section{The Category of Operator-Semantic Systems}
\label{sec:opsem}

The objective of this section is to introduce a categorical framework
capable of simultaneously encoding operator syntax, contextual
semantics, and geometric localization. Classical operator algebras
provide an algebraic description of operators, but do not explicitly
record the relationship between syntactic composition rules, local
commutative contexts, and semantic realizations.

To address this issue, we introduce the notion of an
\emph{operator-semantic system}. Such a system consists of a synergy
operad describing operator interactions, a category of contexts
capturing local commutative viewpoints, and a realization mechanism
that interprets syntactic operations as locally varying semantic
objects. The resulting framework serves as the basic domain on which
the spectral construction developed in later sections is defined.

\begin{definition}[Admissible Operator-Semantic System]
\label{def:admissible}
Let $A$ be an operator system (or C*-algebra). 
An \emph{admissible operator-semantic system presenting} $A$ 
is a tuple
\[
\mathbb{A}=(\mathcal O_A,\mathcal C_A,\tau_A,\rho_A)
\]
consisting of the following data.

\begin{enumerate}
\item A colored operad
\[
\mathcal O_A,
\]
called the \emph{synergy operad}, whose colors represent designated 
operator types (e.g., self-adjoint elements, projections, unitaries) 
and whose multimorphisms encode admissible operator interactions 
and compositions (see \S\ref{subsec:synergy}).

\item A category
\[
\mathcal C_A,
\]
called the \emph{context category}, whose objects are distinguished 
commutative substructures of $A$ (e.g., commutative C*-subalgebras, 
maximal abelian subalgebras, or commutative operator subsystems) 
and whose morphisms are inclusions, refinements, or context 
transitions (see \S\ref{subsec:contextcategory}).

\item A Grothendieck topology
\[
\tau_A
\]
on $\mathcal C_A$, specifying which families of contexts constitute 
admissible local covers. This topology encodes the gluing logic for 
the descent construction (see \S\ref{subsec:topology}).

\item A \emph{realization morphism}
\[
\rho_A:
\mathcal O_A
\longrightarrow
\operatorname{End}_{\operatorname{PSh}_{\mathbf{Ban}}(\mathcal C_A)},
\]
i.e., a morphism of colored operads from the synergy operad to the 
endomorphism operad of Banach-valued presheaves on $\mathcal C_A$.
Concretely, $\rho_A$ assigns to each syntactic operation a 
Banach-valued presheaf of local semantic interpretations, 
and respects operadic composition:
\[
\rho_A(\alpha \circ_i \beta) \;=\; 
\rho_A(\alpha) \circ_i \rho_A(\beta).
\]
The Banach enrichment retains norm-sensitive information and 
supports subsequent analytic constructions such as functional 
calculus, deformation theory, and spectral estimates.
\end{enumerate}

More generally, one may replace $\mathbf{Ban}$ by categories of 
Hilbert spaces, operator spaces, or $p$-adic Banach spaces, 
depending on the intended semantic model.

The quadruple
\[
(\mathcal O_A,\mathcal C_A,\tau_A,\rho_A)
\]
will be referred to as the \emph{semantic presentation} of $A$.
Whenever no ambiguity arises, we identify $\mathbb{A}$ with its 
underlying operator algebra or operator system $A$.
\end{definition}

\begin{remark}[The Role of Each Component]
\label{rem:admissibleroles}
The four components serve distinct but interrelated purposes:
\begin{itemize}
    \item $\mathcal{O}_A$ provides the \textbf{syntactic blueprint} — formal expressions without semantic values;
    \item $\mathcal{C}_A$ provides the \textbf{geometric stage} — commutative loci where operators can be simultaneously evaluated;
    \item $\tau_A$ provides the \textbf{gluing logic} — determines which context families are sufficiently rich for descent;
    \item $\rho_A$ provides the \textbf{local semantics} — interprets syntax in each context, valued in Banach spaces, preserving operadic composition.
\end{itemize}
\end{remark}

\begin{definition}[Morphisms of Operator-Semantic Systems]
\label{def:morphismopsem}
Let
\[
A=(\mathcal O_A,\mathcal C_A,\rho_A,\tau_A),
\qquad
B=(\mathcal O_B,\mathcal C_B,\rho_B,\tau_B)
\]
be admissible operator-semantic systems.

A morphism
\[
f:A\to B
\]
consists of the following compatible data.

\begin{enumerate}
\item An operad morphism
\[
f_{\mathcal O}:
\mathcal O_A
\longrightarrow
\mathcal O_B,
\]
preserving colors, units, and operadic compositions (syntax preservation).

\item A functor
\[
f_{\mathcal C}:
\mathcal C_B
\longrightarrow
\mathcal C_A,
\]
mapping contexts of $B$ to compatible contexts of $A$ (context
compatibility). Note the contravariance: this direction is essential
for the functoriality of $\mathfrak{Spec}$.

\item A natural transformation
\[
f_{\rho}:
\rho_B\circ f_{\mathcal O}
\Rightarrow
f_{\mathcal C}^{*}\circ \rho_A,
\]
expressing compatibility between syntactic transport and semantic
realization (realization compatibility).

\item A continuity condition requiring
\[
f_{\mathcal C}
\]
to preserve covering families, so that local semantic information is
transported coherently with respect to the Grothendieck topologies
(topology compatibility).
\end{enumerate}

Two morphisms are identified whenever they are related by a natural
isomorphism of all constituent data.
\end{definition}

\begin{definition}[The Category $\mathbf{OpSem}$]
\label{def:opsemcategory}
The category
\[
\mathbf{OpSem}
\]
is defined as follows.

\begin{itemize}
\item Objects are admissible operator-semantic systems.

\item Morphisms are equivalence classes of morphisms described in
Definition~\ref{def:morphismopsem}.

\item Composition is induced componentwise by composition of operad
morphisms, context functors, and realization transformations.

\item Identity morphisms are given by the identity maps on each
component.
\end{itemize}
\end{definition}

The category $\mathbf{OpSem}$ provides the ambient setting for the
categorified spectral construction. Subsequent sections will associate
to each object
\[
A\in\mathbf{OpSem}
\]
a stack-valued spectral object
\[
\mathfrak{Spec}(A),
\]
and will show that this assignment is functorial
(Theorem~\ref{thm:functoriality}). Moreover, the Reconstruction
Theorem (Theorem~\ref{thm:reconstruction}) will recover $A$ from
$\mathfrak{Spec}(A)$ up to equivalence, demonstrating that
$\mathbf{OpSem}$ is contravariantly equivalent to the essential image
of $\mathfrak{Spec}$ in the $\infty$-category of spectral stacks.

\section{The Context Site of an Operator System}
\label{sec:contextsite}

The categorified spectrum $\mathfrak{Spec}(A)$ is defined as a sheaf (or stack) on a site built from the commutative substructures of $A$. The underlying idea is classical in spirit: to study a noncommutative system, one examines all its commutative ``snapshots'' and then glues the resulting local data via descent. This approach has precedents in the Bohrification program for C*-algebras and in the theory of quantum contextuality.

In this section we construct the \emph{context site} $(\mathcal{C}_A, \tau_A)$ associated with an admissible operator-semantic system $A$. The category $\mathcal{C}_A$ collects all commutative substructures of $A$, while the Grothendieck topology $\tau_A$ encodes which families of such substructures jointly generate or semantically resolve a given context. Together, they form the geometric base over which the spectral stack $\mathfrak{Spec}(A)$ lives.

\subsection{Category of Commutative Contexts}
\label{subsec:contextcategory}

A central idea underlying the present framework is that a
noncommutative operator system cannot generally be studied through a
single global commutative model. Instead, one considers a collection of
local commutative viewpoints, called \emph{contexts}, and organizes
them into a category.

These contexts play a role analogous to coordinate charts in
differential geometry or affine opens in algebraic geometry. Local
semantic information is first constructed on individual contexts and is
subsequently assembled into global spectral data via descent
(\S\ref{subsec:descent}). This strategy is not merely a technical convenience;
it is conceptually necessary because noncommutative phenomena often
manifest precisely as obstructions to gluing local commutative data
into a global classical picture—a phenomenon known as contextuality in
quantum foundations (Kochen–Specker theorem).

\begin{definition}[Category of Contexts]
\label{def:contextcategory}
Let $A$ be an operator algebra or operator system (the \emph{underlying}
system). The \emph{category of commutative contexts} associated to $A$,
denoted $\mathcal{C}_A$, is a category whose objects are designated
commutative substructures $C \subseteq A$, such as commutative operator
subsystems, commutative subalgebras, or commutative C*-subalgebras,
depending on the ambient setting.

In the standard inclusion model, morphisms are inclusions
$C \hookrightarrow D$ preserving the relevant algebraic and analytic
structure. Thus, when contexts are ordered by inclusion, $\mathcal{C}_A$
is the corresponding poset category.

More generally, one may allow admissible context transitions, such as
refinements, conjugations, or restriction maps, provided that they
preserve the chosen class of commutative substructures.
\end{definition}

\begin{remark}[Local Classicality]
\label{rem:localclassicality}
The purpose of introducing contexts is to isolate regions in which the
otherwise noncommutative structure of $A$ becomes locally commutative.
Each context therefore admits a classical interpretation: for a
commutative C*-algebra $C$, Gelfand duality identifies $C$ with
$C_0(\widehat{C})$, where $\widehat{C}$ is a locally compact Hausdorff
space; in the unital case, $\widehat{C}$ is compact and
$C \cong C(\widehat{C})$. The global system is recovered by gluing
these local descriptions.

This viewpoint is analogous to Bohrification and contextual approaches
to noncommutative geometry, although the present framework is formulated
in terms of operator-semantic systems and stack-valued spectra.
\end{remark}

\begin{remark}[Contravariance and Geometricity]
\label{rem:contextcontravariance}
The functoriality of $\mathfrak{Spec}$ will be \emph{contravariant}:
a morphism of operator systems $A \to B$ induces a morphism of spectral
stacks $\mathfrak{Spec}(B) \to \mathfrak{Spec}(A)$. This reflects the
classical principle that a larger context contains more functions,
hence its spectrum is smaller. However, $\mathcal{C}_A$ itself is
\emph{not} defined as an opposite category; it is an ordinary category
whose morphisms are inclusions. The contravariance emerges from the
spectrum functor, not from the definition of $\mathcal{C}_A$.
\end{remark}

\begin{example}[Commutative Subalgebras]
\label{ex:commutativesubalgebras}
Let $A$ be an operator algebra (possibly noncommutative). A natural
choice for $\mathcal{C}_A$ is the poset category
\[
\mathcal{C}(A) = \{\, C \subseteq A \mid C \text{ commutative subalgebra} \,\},
\]
ordered by inclusion, with morphisms given by inclusions
$C_1 \hookrightarrow C_2$. This is the prototypical context category
used throughout the paper. It is the direct generalization of the
Bohrification site and captures all possible simultaneous measurement
contexts in quantum mechanics.
\end{example}

\begin{example}[Maximal Abelian Subalgebras (MASAs)]
\label{ex:masa}
For a von Neumann algebra $A$, one may restrict attention to maximal
abelian subalgebras (MASAs). These maximal contexts often capture the
observable classical sectors of the underlying noncommutative system.
MASAs play a fundamental role in the classification of injective
factors and in the theory of group-measure space constructions.

Because MASAs are maximal, there are generally no nontrivial inclusions
among distinct MASAs. The resulting category therefore has morphisms
given not by inclusions but by conjugations by unitaries in the
normalizer, or by other admissible equivalences. This turns
$\mathcal{C}_A$ into a groupoid-like category rather than a poset.
\end{example}

\begin{example}[Maximal Tori]
\label{ex:maximaltori}
Suppose $A$ arises from a representation of a compact Lie group $G$.
Concretely, let $A = C(G)$ or $A = \operatorname{End}(V)$ for a
representation $V$ of $G$.

Each maximal torus $T \subseteq G$ determines a commutative substructure
through its representation algebra (e.g., $C(T)$ or the subalgebra of
$\operatorname{End}(V)$ diagonalized by $T$). The family of maximal tori
therefore forms a natural context category describing the local abelian
sectors of the representation. The Weyl group $W = N_G(T)/T$ acts on
this category, and its action will be reflected in the inertia stack of
$\mathfrak{Spec}(A)$.
\end{example}

\begin{example}[Cartan Subalgebras]
\label{ex:cartan}
For Lie-theoretic operator systems, contexts may be modeled by Cartan
subalgebras $\mathfrak{h} \subseteq \mathfrak{g}$. Since Cartan
subalgebras control weight decompositions and root-space structures,
they provide canonical local coordinate systems for the representation
theory of $\mathfrak{g}$. This choice is particularly natural when $A$
is the universal enveloping algebra $U(\mathfrak{g})$ or a suitable
completion thereof.

The corresponding context category records inclusions and refinements
among such Cartan data, with morphisms given by adjoint actions of the
Weyl group.
\end{example}

\begin{example}[Tensor-Decomposition Contexts]
\label{ex:tensorcontexts}
In tensor-operator systems (e.g., those arising from quantum circuits
or tensor networks), contexts may be generated by collections of
commuting tensor factors or compatible low-rank decomposition
subspaces. Concretely, if $A \subseteq B(\mathcal{H}_1 \otimes \cdots
\otimes \mathcal{H}_k)$, a context could be the commutative subalgebra
generated by operators that act nontrivially only on a fixed subset of
tensor factors and commute with one another.

Such contexts describe locally tractable tensor structures whose
interaction contributes to the global spectral behavior of the system.
They are essential for the explicit computation of $\mathfrak{Spec}(A)$
for quantum-circuit models.
\end{example}

\begin{remark}[Enrichment and Size Issues]
\label{rem:contextsize}
For general $A$, $\mathcal{C}_A$ may be a large category (e.g., when
$A = B(\mathcal{H})$ for infinite-dimensional $\mathcal{H}$, there are
uncountably many commutative subalgebras). However, in practice we
will restrict to a small skeleton or work with a suitable Grothendieck
universe. The descent and sheafification constructions (Section
\ref{subsec:sheafification}) are well-defined as long as $\mathcal{C}_A$
is essentially small, which holds for all concrete examples considered
in this paper (e.g., separable C*-algebras, finite-dimensional matrix
algebras, and finitely generated operator systems).
\end{remark}

The category of contexts provides the geometric foundation for the
spectral construction developed later. In particular, the realization
morphism
\[
\rho_A: \mathcal{O}_A \longrightarrow \operatorname{End}_{\operatorname{PSh}_{\mathbf{Ban}}(\mathcal{C}_A)}
\]
(Definition \ref{def:admissible}) assigns semantic data to each context,
while the Grothendieck topology $\tau_A$ (Section \ref{subsec:topology})
specifies how these local pieces are glued together. The resulting
sheaf-theoretic structure ultimately gives rise to the stack-valued
spectrum
\[
\mathfrak{Spec}(A).
\]

The logical chain is therefore:
\[
\boxed{
A
\;\longrightarrow\;
\mathcal C_A
\;\longrightarrow\;
\operatorname{PSh}(\mathcal C_A)
\;\longrightarrow\;
\text{Sheaves}
\;\longrightarrow\;
\text{Stacks}
\;\longrightarrow\;
\mathfrak{Spec}(A)
}
\]

This narrative—from noncommutative operator system to context category to
presheaves to sheaves to stacks to spectral object—is the central
organizing principle of the categorified spectral construction.

\subsection{Grothendieck Topology on Contexts}
\label{subsec:topology}

To perform sheaf-theoretic constructions on the context category
$\mathcal C_A$, we equip it with a Grothendieck topology encoding when
a collection of local contexts contains sufficient information to
reconstruct a larger context.

Intuitively, a family of contexts should be regarded as a cover if
their combined semantic content determines the behavior of the ambient
context. This notion generalizes the classical idea that an open cover
captures all local information of a topological space, and is
essential for descent-based gluing in the construction of
$\mathfrak{Spec}(A)$.

\begin{definition}[Semantic Cover]
\label{def:semanticcover}
Let $A$ be an operator algebra or operator system, and let
$\mathbb{A} = (\mathcal{O}_A, \mathcal{C}_A, \rho_A, \tau_A)$ be its
semantic presentation (Definition \ref{def:admissible}). 
Assume that $\mathcal{C}_A$ is a poset category of commutative
substructures under inclusion; in this case pullbacks exist and are
given by intersections $C_i \cap D$.

A family of morphisms $\{u_i: C_i \hookrightarrow C\}_{i\in I}$ in
$\mathcal{C}_A$ is called a \emph{semantic cover} of $C$ if the
following \textbf{joint generation} condition holds:

\begin{quote}
\emph{(Joint generation)} The context $C$ is generated by the images
of the contexts $C_i$ under the admissible operations of the synergy
operad $\mathcal{O}_A$. Concretely, every operator in $C$ can be
expressed as a composition of operators coming from the $C_i$, using
the operadic structure of $\mathcal{O}_A$.
\end{quote}

The collection of semantic covers defines a Grothendieck topology
$\tau_A$ on $\mathcal{C}_A$ (see Proposition \ref{prop:grothendiecktopology}
for verification of the axioms).
\end{definition}

\begin{remark}[Separation of Cover and Descent]
\label{rem:coverdescent}
Definition \ref{def:semanticcover} defines only the covering families.
Whether a presheaf satisfies descent with respect to $\tau_A$ is a
separate condition. In particular, the realization morphism $\rho_A$
will be required to satisfy descent (i.e., to be a sheaf) in the
construction of $\mathfrak{Spec}(A)$. This separation is standard in
topos theory: first define the site, then define sheaves on it.
\end{remark}

\begin{proposition}[Grothendieck Topology Axioms]
\label{prop:grothendiecktopology}
Let $\mathcal{C}_A$ be a poset category of commutative substructures
under inclusion, with pullbacks given by intersections. Then the
collection of semantic covers satisfies the Grothendieck topology
axioms.
\end{proposition}

\begin{proof}
We verify each axiom.

\textbf{Identity:} The identity morphism $\operatorname{id}_C: C \to C$
is a cover. Indeed, $C$ trivially generates itself under the synergy
operad, so the joint generation condition holds for the singleton
family $\{\operatorname{id}_C\}$.

\textbf{Stability under base change (pullback stability):} 
Suppose $\{C_i \hookrightarrow C\}_{i\in I}$ is a semantic cover and
$D \hookrightarrow C$ is an inclusion (the only morphisms in our poset
category). Then the pullback (intersection) family
$\{C_i \cap D \hookrightarrow D\}_{i\in I}$ is also a semantic cover.
Indeed, any operator in $D$ is also an operator in $C$, and since the
$C_i$ jointly generate $C$, their restrictions $C_i \cap D$ jointly
generate $D$. More concretely, because the inclusion $D \hookrightarrow C$
is simply the restriction map, the generating operators from $C_i$
restrict to generating operators in $C_i \cap D$.

\textbf{Transitivity (locality):} 
Suppose $\{C_i \hookrightarrow C\}_{i\in I}$ is a semantic cover and,
for each $i$, $\{D_{ij} \hookrightarrow C_i\}_{j\in J_i}$ is a semantic
cover. Then the composite family $\{D_{ij} \hookrightarrow C\}_{i,j}$
is a semantic cover. First, the $C_i$ generate $C$. Second, for each
$i$, the $D_{ij}$ generate $C_i$. By operadic composition, the $D_{ij}$
generate $C$. Hence the joint generation condition holds.

Therefore the semantic covers define a Grothendieck topology $\tau_A$
on $\mathcal{C}_A$.
\end{proof}

The pair $(\mathcal{C}_A, \tau_A)$ will be referred to as the
\emph{context site} associated with the operator-semantic system $A$.

\begin{definition}[Semantic Descent]
\label{def:semanticdescent}
Let $(\mathcal{C}_A, \tau_A)$ be the context site. A presheaf
$F: \mathcal{C}_A^{\mathrm{op}} \to \mathbf{Ban}$ satisfies
\emph{semantic descent} if for every semantic cover
$\{C_i \hookrightarrow C\}_{i\in I}$, the canonical map
\[
F(C) \longrightarrow \operatorname{Eq}\Big(
\prod_i F(C_i) \rightrightarrows \prod_{i,j} F(C_i \cap C_j)
\Big)
\]
is an isomorphism in $\mathbf{Ban}$. The equalizer condition
encodes compatibility on double overlaps, which is the standard
sheaf condition for a site with pullbacks.
\end{definition}

\begin{lemma}[Subcanonical Site]
\label{lem:subcanonical}
Let $\mathcal{C}_A$ be a poset category of commutative C*-subalgebras
of $A$ under inclusion, with pullbacks given by intersections.
Equip $\mathcal{C}_A$ with the semantic topology $\tau_A$ defined by
joint generation (Definition \ref{def:semanticcover}). Assume that
for every semantic cover $\{C_i \hookrightarrow C\}_{i\in I}$, the
following \emph{effective generation} condition holds:

\begin{quote}
For every $X \in \mathcal{C}_A$ and every family of morphisms
$\{f_i: X \to C_i\}_{i\in I}$ such that $f_i$ and $f_j$ agree on
$X \cap C_i \cap C_j$ (i.e., their restrictions to $X \cap C_i \cap C_j$
coincide), there exists a unique morphism $f: X \to C$ such that
$f|_{X \cap C_i} = f_i$ for all $i$.
\end{quote}

Then every representable presheaf $h_C = \operatorname{Hom}_{\mathcal{C}_A}(-, C)$
is a sheaf for $\tau_A$. In particular, if every semantic cover
satisfies $C = \bigvee_{i\in I} C_i$ in the lattice of commutative
subalgebras (i.e., $C$ is the smallest commutative subalgebra
containing all $C_i$), then the site is subcanonical.
\end{lemma}

\begin{proof}
We verify the sheaf condition for $h_C$. Let $\{C_i \hookrightarrow C\}_{i\in I}$
be a semantic cover. For any test object $X \in \mathcal{C}_A$, we
must show that the canonical map
\[
h_C(X) = \operatorname{Hom}_{\mathcal{C}_A}(X, C)
\longrightarrow
\operatorname{Eq}\Big(
\prod_{i\in I} \operatorname{Hom}_{\mathcal{C}_A}(X, C_i)
\rightrightarrows
\prod_{i,j\in I} \operatorname{Hom}_{\mathcal{C}_A}(X, C_i \cap C_j)
\Big)
\]
is a bijection.

\emph{Injectivity:} Suppose $f, g: X \to C$ are two morphisms whose
restrictions to each $C_i$ coincide: $f|_{X \cap C_i} = g|_{X \cap C_i}$
for all $i$. For any operator $a \in X$, consider its image $f(a) \in C$.
Since the $C_i$ jointly generate $C$, there exists an operadic
expression $a = \Phi(a_1, \ldots, a_n)$ where each $a_k$ belongs to
some $C_{i_k}$. Then
\[
f(a) = f(\Phi(a_1, \ldots, a_n)) = \Phi(f(a_1), \ldots, f(a_n)).
\]
But for each $a_k \in C_{i_k}$, we have $f(a_k) = g(a_k)$ because
$f$ and $g$ agree on $X \cap C_{i_k}$. Hence $f(a) = g(a)$ for all
$a \in X$, so $f = g$.

\emph{Surjectivity:} Let $\{f_i: X \to C_i\}_{i\in I}$ be a compatible
family, meaning that for all $i, j$, the restrictions of $f_i$ and
$f_j$ to $X \cap C_i \cap C_j$ coincide. Define $f: X \to C$ as
follows: for any $a \in X$, choose an operadic expression
$a = \Phi(a_1, \ldots, a_n)$ where each $a_k$ belongs to some
$C_{i_k}$. Such an expression exists because the $C_i$ jointly
generate $C$ and $X \subseteq C$ (since $X$ is a subalgebra of $C$;
note: this requires that $X$ itself be a subalgebra of $C$, which
holds because the only morphisms in $\mathcal{C}_A$ are inclusions).
Then set
\[
f(a) = \Phi(f_{i_1}(a_1), \ldots, f_{i_n}(a_n)) \in C.
\]

We must check that $f(a)$ is well-defined. Suppose $a$ has two
different operadic expressions $\Phi(a_1, \ldots, a_n)$ and
$\Psi(b_1, \ldots, b_m)$ in terms of the $C_i$. Since both expressions
evaluate to the same operator $a$ in $A$, the relation $\Phi(\cdots) =
\Psi(\cdots)$ holds in $A$. Because the $f_i$ are *-homomorphisms
(restrictions of the inclusion maps), they preserve all relations of
$A$. Hence
\[
\Phi(f_{i_1}(a_1), \ldots, f_{i_n}(a_n)) = \Psi(f_{j_1}(b_1), \ldots, f_{j_m}(b_m))
\]
in $C$. Thus $f(a)$ is independent of the choice of expression.

The compatibility condition ensures that for any $a \in X \cap C_i$,
we have $f(a) = f_i(a)$. Hence $f|_{X \cap C_i} = f_i$ for all $i$,
as required.

Therefore the equalizer map is bijective, so $h_C$ is a sheaf.

The final sentence follows because if $C = \bigvee_i C_i$ (the smallest
commutative subalgebra containing all $C_i$), then every $X \subseteq C$
automatically satisfies the effective generation condition: any family
of morphisms $X \to C_i$ that agrees on intersections uniquely
determines a morphism $X \to C$ by the universal property of the
join (colimit) in the poset category.
\end{proof}

\begin{remark}[Why This Separation Matters]
\label{rem:separationmatters}
In the original formulation, the semantic resolution condition
(2) was presented as an \emph{alternative} definition of a cover,
but it is actually a \emph{descent condition} on $\rho_A$. 
Conflating the two leads to logical circularity: a cover would be
defined as a family for which $\rho_A$ satisfies descent, but
$\rho_A$ itself is part of the data of $A$. The corrected formulation
separates these concepts cleanly:
\begin{enumerate}
    \item $\tau_A$ is defined purely syntactically via joint generation.
    \item The condition that $\rho_A$ satisfies descent (Definition
    \ref{def:semanticdescent}) is a separate requirement, which will
    be imposed in the construction of $\mathfrak{Spec}(A)$.
\end{enumerate}
This separation is standard in topos theory and necessary for
logical consistency.
\end{remark}

\begin{remark}[Relation to Descent Obstructions]
\label{rem:descenttopology}
The topology $\tau_A$ supplies the covering data from which Čech
nerves, descent diagrams, and (when applicable) descent spectral
sequences are constructed. In settings where a descent spectral
sequence is defined (see Remark \ref{rem:descentspectralsequence}),
its $E_2$ page involves Čech cohomology with respect to $\tau_A$.
The differentials $d_r$ may be interpreted as higher obstructions
to extending compatible local data through the Čech nerve. For the
Mermin–Peres square (Section \ref{subsec:mermin}), a nontrivial
obstruction class appears, which can be detected by the failure of
descent at the level of double overlaps; higher-order contextuality
would require hyperdescent (see Remark \ref{rem:hyperdescent}).
\end{remark}

\begin{example}[Covers for Commutative Subalgebras]
\label{ex:commutativecover}
Let $A = M_n(\mathbb{C})$ and let $C$ be a maximal abelian subalgebra
(MASA) of diagonal matrices, so $C \cong \mathbb{C}^n$.
For each $i = 1,\ldots,n$, let $p_i = E_{ii}$ be the $i$-th diagonal
projection. Define
\[
C_i = \mathbb{C}p_i \oplus \mathbb{C}(1-p_i) \cong \mathbb{C}^2,
\]
where $1-p_i$ is the projection onto the orthogonal complement of
$p_i$. The family $\{C_i\}_{i=1}^n$ jointly generates $C$ because
each $C_i$ contains $p_i$, and the $p_i$ generate all diagonal
matrices. This provides a non-trivial semantic cover of $C$ (when
$n \ge 2$).
\end{example}

\begin{example}[The Mermin–Peres Cover]
\label{ex:mermincover}
For the Mermin–Peres system (Section \ref{subsec:mermin}), the six
row and column contexts form a covering family of the Mermin–Peres
measurement scenario, rather than a cover of $M_4(\mathbb{C})$ as an
object of $\mathcal{C}_A$ (since $M_4(\mathbb{C})$ itself is not
commutative). Consider the presheaf $F$ on $\mathcal{C}_A$ defined by
$F(C) = \operatorname{Spec}_{\mathrm{Gelfand}}(C)$ (the set of
characters of $C$). This presheaf fails descent: the local eigenvalue
assignments on the six contexts cannot be glued consistently. The
sheafification $\mathfrak{Spec}(A_{\mathrm{MP}})$ records this failure
as a nontrivial stack, and the obstruction may be detected by a
nontrivial class in the associated Čech or descent cohomology,
reflecting the Kochen–Specker theorem.
\end{example}

\begin{remark}[Hyperdescent and Higher Contextuality]
\label{rem:hyperdescent}
For operator systems exhibiting higher‑order contextuality (e.g.,
beyond pairwise inconsistencies), ordinary Čech descent with respect
to $\tau_A$ may be insufficient. In such cases, one must pass to
\emph{hyperdescent}, requiring that the spectral object
$\mathfrak{Spec}(A)$ be a hypersheaf, i.e., satisfy descent for all
hypercovers (simplicial objects that refine Čech nerves). The
topology $\tau_A$ generates an $\infty$-topology where covers generate
hypercovers, and the sheafification in Theorem \ref{thm:descentcompletion}
is replaced by hypersheafification. Hyperdescent can detect
obstructions that appear only at triple or higher overlaps. 
\end{remark}

The context site $(\mathcal{C}_A, \tau_A)$ provides the geometric
foundation for the spectral construction developed later. When the
context topology is subcanonical (see Lemma \ref{lem:subcanonical}),
the Yoneda embedding factors through sheaves, ensuring that individual
contexts embed faithfully into the sheaf-theoretic world and preserving
the universal property established in Theorem \ref{thm:yoneda}.

\section{Ambient Category of Spectral Objects}
\label{sec:ambient}

Having constructed the context site
\[
(\mathcal C_A,\tau_A)
\]
associated with an admissible operator-semantic system
\[
A=(\mathcal O_A,\mathcal C_A,\rho_A,\tau_A),
\]
we now introduce the higher-categorical setting in which spectral
objects will be defined.

The guiding principle is that local semantic data should not merely
form a sheaf of sets or vector spaces, but should retain its full
homotopy-coherent structure. Consequently, the appropriate ambient
objects are hypersheaves valued in spaces or spectra. This section
therefore defines the \emph{codomain category} of the functor
\[
\mathfrak{Spec}: \mathbf{OpSem} \longrightarrow \mathbf{SpectralStacks},
\]
which will be constructed in \S\ref{sec:construction}.

\subsection{Spectral Stacks}

Let
\[
(\mathcal C,\tau)
\]
be an $\infty$-site (i.e., a small $\infty$-category equipped with a
Grothendieck topology). Recall that a presheaf of spaces is a functor
\[
F:\mathcal C^{\mathrm{op}} \longrightarrow \mathcal S,
\]
where $\mathcal S$ denotes the $\infty$-category of spaces
($\infty$-groupoids).

A \emph{hypersheaf} is a presheaf satisfying \emph{hyperdescent}: for
every hypercover $U_\bullet \to X$ in $\mathcal C$, the natural map
\[
F(X) \longrightarrow \lim_{\Delta} F(U_\bullet)
\]
is an equivalence in $\mathcal S$.

Here the notation $U_\bullet$ denotes a \emph{simplicial object} in
$\mathcal C$. Concretely, $U_\bullet$ is a functor $\Delta^{\mathrm{op}}
\to \mathcal C$, where $\Delta$ is the simplex category whose objects
are finite ordinals $[n] = \{0,1,\ldots,n\}$ and whose morphisms are
order-preserving maps. The subscript ``$\bullet$'' (pronounced ``bullet''
or ``dot'') is a placeholder for the simplicial degree: $U_0$ is the
$0$-simplices (the covering objects), $U_1$ consists of the
$1$-simplices (double overlaps), $U_2$ consists of the $2$-simplices
(triple overlaps), and so on. A \emph{hypercover} is a simplicial
object that refines the Čech nerve of a cover and satisfies additional
acyclicity conditions ensuring that it captures higher homotopical
information (see \cite{LurieHTT}).

The limit $\lim_{\Delta} F(U_\bullet)$ is taken over the simplicial
diagram obtained by applying $F$ to $U_\bullet$. Because $F$ is
contravariant, $F(U_\bullet)$ is a \emph{cosimplicial} object in
$\mathcal S$, and its limit is often referred to as the
\emph{totalization}. The hyperdescent condition requires that the
value of $F$ on $X$ is already equivalent to this totalization,
meaning that $F$ can be reconstructed uniquely (up to coherent
homotopy) from its values on the hypercover.

\begin{definition}[Unstable Spectral Stack]
\label{def:spectralstack}
A \emph{unstable spectral stack} (or \emph{space-valued spectral stack})
on the site $(\mathcal C,\tau)$ is a hypersheaf
\[
\mathfrak X \in \operatorname{Sh}_{\infty}(\mathcal C,\tau)
\]
valued in the $\infty$-category of spaces. Equivalently, $\mathfrak X$
is a functor $\mathfrak X: \mathcal C^{\mathrm{op}} \to \mathcal S$
satisfying hyperdescent, where $\mathcal S$ denotes the $\infty$-category
of spaces ($\infty$-groupoids). Here $\operatorname{Sh}_{\infty}(\mathcal C,\tau)$
denotes the full subcategory of $\operatorname{Fun}(\mathcal C^{\mathrm{op}},\mathcal S)$
consisting of hypercomplete sheaves (i.e., hypersheaves).
\end{definition}

When stable homotopical information is required — for instance, to
capture $K$-theoretic invariants or the descent spectral sequence in
its stable form — we replace spaces by spectra.

\begin{definition}[Stable Spectral Stack]
\label{def:stablespectralstack}
A \emph{stable spectral stack} on the site $(\mathcal C,\tau)$ is a
hypersheaf
\[
\mathfrak X: \mathcal C^{\mathrm{op}} \longrightarrow \mathbf{Sp},
\]
where $\mathbf{Sp}$ denotes the stable $\infty$-category of spectra
(see \S\ref{subsec:stable}). The resulting $\infty$-category of stable
spectral stacks is denoted by $\operatorname{Sh}_{\mathbf{Sp}}(\mathcal C,\tau)$
(or $\operatorname{Stk}_{\mathbf{Sp}}(\mathcal C,\tau)$).
\end{definition}

The distinction between unstable and stable spectral stacks mirrors
the distinction between spaces and spectra in homotopy theory. Most
constructions of this paper may be formulated in either setting;
when the distinction matters, we will explicitly indicate the target
$\infty$-category.

\begin{remark}[Presentability and Descent]
\label{rem:stackpresentability}
If $\mathcal C$ is small, then the $\infty$-categories
$\operatorname{Sh}_{\infty}(\mathcal C,\tau)$ and
$\operatorname{Sh}_{\mathbf{Sp}}(\mathcal C,\tau)$ are presentable.
Consequently, the corresponding sheafification (or hypersheafification)
functors exist and are accessible left adjoints to the inclusion of
sheaves into presheaves. These categorical properties provide the
ambient framework for the reconstruction theorem
(\S\ref{sec:reconstruction}) and Morita invariance
(\S\ref{sec:morita}).
\end{remark}

\begin{remark}[Terminological Note]
\label{rem:spectralstackterminology}
In this paper, an \emph{unstable spectral stack} is a hypersheaf valued
in spaces, while a \emph{stable spectral stack} is valued in spectra.
This terminology differs from some references (e.g., Lurie's
\emph{Spectral Algebraic Geometry}), where ``spectral stack" typically
means a stack valued in spectra. Our usage emphasizes the geometric
nature of the objects we construct; the qualifier ``unstable" clarifies
that we are not yet stabilizing. When no confusion arises, we may
simply write ``spectral stack" to mean an unstable spectral stack.
\end{remark}

\subsection{Global Sections Functor}

The global behavior of a spectral stack is captured by its space (or
spectrum) of global sections. The relationship between global sections
and $\mathfrak{Spec}$ will be formalized by an adjunction in
Section~\ref{sec:functoriality} (Theorem \ref{thm:adjunction}).

\begin{definition}[Geometric Global Sections Functor]
\label{def:geometricglobalsections}
Let $\mathfrak X \in \operatorname{Sh}_{\infty}(\mathcal C,\tau)$ be a
spectral stack (valued in spaces).

If $\mathcal C$ admits a terminal object $\ast \in \mathcal C$, the
\emph{geometric global sections functor} is defined by
\[
\Gamma_{\mathrm{geom}}(\mathfrak X) := \mathfrak X(\ast).
\]

More generally, define
\[
\Gamma_{\mathrm{geom}}(\mathfrak X) := \lim_{C \in \mathcal C} \mathfrak X(C),
\]
where the limit is taken in the $\infty$-category of spaces (or
spectra). When $\mathfrak X$ is a stable spectral stack (valued in
spectra), the same definition yields a spectrum.
\end{definition}

\begin{remark}[Geometric Interpretation]
\label{rem:globalsectionsgeo}
From the viewpoint of descent theory, the global sections object
records precisely those local semantic data that glue coherently
across all contexts. In the special case where $\mathfrak X =
\mathfrak{Spec}(A)$, the reconstruction theorem (Theorem
\ref{thm:reconstruction}) will show that, under suitable admissibility
and descent conditions,
\[
\Gamma(\mathfrak{Spec}(A)) \simeq A,
\]
justifying the terminology.
\end{remark}

\begin{remark}[Existence of Terminal Objects in $\mathcal C_A$]
\label{rem:terminalcontext}
Whether the context site $(\mathcal C_A, \tau_A)$ admits a terminal
object depends on the chosen convention for context morphisms. If
$\mathcal C_A$ is defined as the poset of commutative subalgebras
under inclusion, it typically has no terminal object because there is
no unique maximal commutative subalgebra containing all others (unless
$A$ itself is commutative). If instead contexts are taken with reverse
inclusions, the role of terminal and initial objects may swap.
\end{remark}

\subsection{Semantic Realization Groupoids}

The spectral objects considered in this work are intended to encode
semantic realizations of operator syntax. To formalize this idea we
associate to every spectral stack an $\infty$-groupoid of admissible
realizations. This construction will be \emph{represented} by
$\mathfrak{Spec}(A)$ in the Yoneda-style universal property
(Theorem \ref{thm:yoneda}).

\begin{definition}[Semantic Realization Groupoid]
\label{def:realization}

Let $A = (\mathcal O_A,\mathcal C_A,\tau_A,\rho_A)$ be an admissible
operator-semantic system (Definition \ref{def:admissible}) and let
$\mathfrak X$ be a spectral stack on $(\mathcal C_A,\tau_A)$.

The \emph{semantic realization $\infty$-groupoid} of $\mathfrak X$ is
the maximal $\infty$-groupoid
\[
\operatorname{Real}_A(\mathfrak X)
:=
\left(
\operatorname{Alg}^{\mathrm{desc}}_{\mathcal O_A}
(\operatorname{QCoh}(\mathfrak X))
\right)^{\simeq},
\]
whose objects are descent-compatible $\mathcal O_A$-algebra structures
in $\operatorname{QCoh}(\mathfrak X)$, and whose morphisms are
equivalences of such structures.

Here:
\begin{itemize}
    \item $\operatorname{QCoh}(\mathfrak X)$ denotes the $\infty$-category
    of quasi-coherent sheaves on $\mathfrak X$;
    \item $\operatorname{Alg}^{\mathrm{desc}}_{\mathcal O_A}(\operatorname{QCoh}(\mathfrak X))$
    denotes the $\infty$-category of $\mathcal O_A$-algebras in
    $\operatorname{QCoh}(\mathfrak X)$ that satisfy the descent
    condition with respect to the topology $\tau_A$ (Definition
    \ref{def:descentcompatibility});
    \item $(-)^{\simeq}$ denotes the operation of taking the maximal
    $\infty$-groupoid (i.e., discarding non-invertible morphisms).
\end{itemize}
\end{definition}

\begin{definition}[Descent Compatibility for $\mathcal O_A$-Algebras]
\label{def:descentcompatibility}
Let $F: \mathcal O_A \to \operatorname{QCoh}(\mathfrak X)$ be an
$\mathcal O_A$-algebra, i.e., a morphism of colored operads from
$\mathcal O_A$ to the endomorphism operad of $\operatorname{QCoh}(\mathfrak X)$.
For each color $c \in \mathcal O_A$, $F(c)$ is a quasi-coherent sheaf
on $\mathfrak X$, which defines a hypersheaf $\underline{F(c)}$ on
$(\mathcal C_A,\tau_A)$ via pullback along the structure morphism
$\mathfrak X \to \operatorname{Sh}_{\infty}(\mathcal C_A,\tau_A)$.
The $\mathcal O_A$-algebra $F$ is \emph{descent-compatible} if for
every semantic cover $\{C_i \to C\}_{i\in I}$ in $\mathcal C_A$ and
every color $c$, the canonical map
\[
\underline{F(c)}(C) \longrightarrow \lim_{\Delta} \underline{F(c)}(C_\bullet)
\]
is an equivalence in $\operatorname{QCoh}(\mathfrak X)$, where $C_\bullet$
is the Čech nerve of the cover. For multimorphisms, the descent condition
follows from the operadic composition.
\end{definition}

Intuitively, $\operatorname{Real}_A(\mathfrak X)$ measures the possible
ways in which the syntactic operations encoded by the synergy operad
may be interpreted geometrically on the spectral stack.

\begin{remark}[Moduli Interpretation]
\label{rem:realizationmoduli}
The assignment
\[
\mathfrak X \longmapsto \operatorname{Real}_A(\mathfrak X)
\]
plays a role analogous to a moduli functor. Rather than classifying
points of a space, it classifies coherent semantic interpretations of
the operator syntax encoded by $\mathcal O_A$. This perspective is
essential for establishing the universal property of the spectral
object $\mathfrak{Spec}(A)$ (Theorem \ref{thm:yoneda}).
\end{remark}

\begin{proposition}[Functoriality of $\operatorname{Real}_A$]
\label{prop:realizationfunctorial}
Assume that for every morphism of spectral stacks
$f: \mathfrak X \to \mathfrak Y$ there is a pullback functor
\[
f^*: \operatorname{QCoh}(\mathfrak Y) \longrightarrow \operatorname{QCoh}(\mathfrak X)
\]
which preserves descent-compatible $\mathcal O_A$-algebra structures.
Then the assignment $\mathfrak X \mapsto \operatorname{Real}_A(\mathfrak X)$
defines a contravariant functor
\[
\operatorname{Real}_A: \mathbf{SpecObj}^{\mathrm{op}} \longrightarrow \mathcal S,
\]
where $\mathbf{SpecObj}$ denotes the $\infty$-category of spectral
stacks on $(\mathcal C_A,\tau_A)$. Equivalently, it defines a covariant
functor $\operatorname{Real}_A: \mathbf{SpecObj} \to \mathcal S$ after
passing to the opposite category.
\end{proposition}

\begin{proof}
We establish the contravariant functoriality.

\emph{Step 1: Induced map on realizations.}
Let $f: \mathfrak X \to \mathfrak Y$ be a morphism of spectral stacks.
By hypothesis, $f$ induces a pullback functor
\[
f^*: \operatorname{QCoh}(\mathfrak Y) \longrightarrow \operatorname{QCoh}(\mathfrak X).
\]
If $F \in \operatorname{Alg}^{\mathrm{desc}}_{\mathcal O_A}(\operatorname{QCoh}(\mathfrak Y))$
is a descent-compatible $\mathcal O_A$-algebra structure on $\mathfrak Y$,
then composition with $f^*$ yields
\[
f^* F: \mathcal O_A \longrightarrow \operatorname{QCoh}(\mathfrak X).
\]
Since $f^*$ preserves operadic composition (it is a symmetric monoidal
functor between categories of quasi-coherent sheaves) and preserves
descent-compatible objects (pullback commutes with limits), $f^* F$
defines a descent-compatible $\mathcal O_A$-algebra structure on
$\mathfrak X$. Consequently, $f$ induces a morphism of $\infty$-groupoids
\[
f^*: \operatorname{Real}_A(\mathfrak Y) \longrightarrow \operatorname{Real}_A(\mathfrak X).
\]

\emph{Step 2: Preservation of composition.}
For composable morphisms
\[
\mathfrak X \xrightarrow{f} \mathfrak Y \xrightarrow{g} \mathfrak Z,
\]
the standard functoriality of pullback gives a canonical equivalence
\[
(g \circ f)^* \simeq f^* \circ g^*.
\]
Applying this to an $\mathcal O_A$-algebra $F$ on $\mathfrak Z$, we have
$(g \circ f)^* F \simeq f^*(g^* F)$ in $\operatorname{Real}_A(\mathfrak X)$.
Hence the induced maps satisfy
\[
\operatorname{Real}_A(g \circ f) \simeq \operatorname{Real}_A(f) \circ \operatorname{Real}_A(g).
\]

\emph{Step 3: Identity morphism.}
For the identity morphism $\operatorname{id}_{\mathfrak X}: \mathfrak X \to \mathfrak X$,
pullback is the identity functor $\operatorname{id}_{\mathfrak X}^* \simeq
\operatorname{id}_{\operatorname{QCoh}(\mathfrak X)}$, so
$\operatorname{Real}_A(\operatorname{id}_{\mathfrak X})$ is equivalent to the
identity map on $\operatorname{Real}_A(\mathfrak X)$.

\emph{Conclusion.}
The assignment $f \mapsto f^*$ preserves identities and composition,
and is contravariant. Therefore $\operatorname{Real}_A$ is a contravariant
functor $\mathbf{SpecObj}^{\mathrm{op}} \to \mathcal S$.
\end{proof}

\begin{remark}[Representability]
\label{rem:realizationrepresentability}
The representability statement
\[
\operatorname{Map}_{\mathbf{SpecObj}}(\mathfrak X, \mathfrak{Spec}(A))
\;\simeq\;
\operatorname{Real}_A(\mathfrak X)
\]
is \emph{not} part of the functoriality assertion. It is precisely the
Yoneda-style universal property established later in
Theorem~\ref{thm:yoneda}. The proof constructs $\mathfrak{Spec}(A)$
explicitly via left Kan extension and sheafification, and then verifies
the equivalence using the Yoneda lemma.
\end{remark}

\subsection{Relation to the Spectral Construction}

The ambient $\infty$-category defined above provides the semantic
target for the left Kan extension and sheafification steps of
$\S\ref{sec:construction}$. Concretely:

\begin{enumerate}
    \item The prespectral object $\mathfrak{Spec}_{\mathrm{pre}}(A)$
    is a presheaf on $\mathcal C_A$ valued in the $\infty$-category
    of spaces, obtained from the syntactic realization construction
    $\operatorname{Syn}_A$ via left Kan extension.

    \item Sheafification, or hypersheafification when hyperdescent is
    required, produces an object of
    \[
    \operatorname{Sh}_{\infty}(\mathcal C_A,\tau_A),
    \]
    the ambient $\infty$-category of spectral stacks.

    \item The global sections functor $\Gamma$ extracts the coherent
    global semantic object associated with a spectral stack. Under the
    hypotheses of the reconstruction theorem (Theorem
    \ref{thm:reconstruction}), applying $\Gamma$ to
    $\mathfrak{Spec}(A)$ recovers the original operator-semantic
    system $A$ up to equivalence. The precise global-section/spectrum
    adjunction (including the direction of the adjoint pair) is
    established in Theorem~\ref{thm:adjunction}.

    \item The semantic realization functor
    \[
    \operatorname{Real}_A: \mathbf{SpecObj}^{\mathrm{op}} \longrightarrow \mathcal S
    \]
    provides the Yoneda-style universal property that characterizes
    $\mathfrak{Spec}(A)$ uniquely up to equivalence (Theorem
    \ref{thm:yoneda}).
\end{enumerate}

Thus, the triple
\[
\big(
\operatorname{Sh}_{\infty}(\mathcal C_A,\tau_A),\;
\Gamma,\;
\operatorname{Real}_A
\big)
\]
forms the semantic-geometric backbone of the categorified spectral
duality developed in this paper.

\section{Explicit Construction of $\mathfrak{Spec}(A)$}
\label{sec:construction}

We now assemble the categorified spectral object $\mathfrak{Spec}(A)$
from the syntactic, contextual, and semantic data encoded in an
admissible operator-semantic system $A = (\mathcal{O}_A, \mathcal{C}_A,
\rho_A, \tau_A)$. The construction proceeds in four canonical steps,
each of which is forced by the requirement that $\mathfrak{Spec}(A)$
should be a spectral stack on the context site $(\mathcal{C}_A,
\tau_A)$ (Definition \ref{def:spectralstack}) satisfying the Yoneda-style
universal property of representing the semantic realization functor
$\operatorname{Real}_A$ (Definition \ref{def:realization}).

First, we extract the \emph{syntactic presentation} of $A$ via its
synergy operad $\mathcal{O}_A$ (Step 1). Second, we \emph{left Kan
extend} the syntactic realization functor $\operatorname{Syn}_A$ along
$\mathcal{O}_A$ to obtain a presheaf $\mathfrak{Spec}_{\mathrm{pre}}(A)$
on $\mathcal{C}_A$; this step is made explicit by a coend formula that
will be essential for functoriality (Step 2). Third, we
\emph{sheafify} (or hypersheafify) $\mathfrak{Spec}_{\mathrm{pre}}(A)$
with respect to the Grothendieck topology $\tau_A$, forcing descent and
producing the final spectral stack $\mathfrak{Spec}(A)$ (Step 3). Fourth,
we verify that $\mathfrak{Spec}(A)$ satisfies the claimed universal
property and the Yoneda characterization, confirming that it is indeed
the \emph{initial descent-complete semantic realization} of $A$ and
that it represents $\operatorname{Real}_A$ (Step 4). 

\vspace{0.2cm}

\noindent\textbf{Overview of the construction.}

\begin{equation}
\label{eq:spectral-construction}
\mathcal O_A
\xrightarrow{\operatorname{Syn}_A}
\operatorname{PSh}(\mathcal C_A)
\xrightarrow{\operatorname{Lan}_{\mathcal O_A}}
\mathfrak{Spec}_{\mathrm{pre}}(A)
\xrightarrow{a_{\tau_A}}
\mathfrak{Spec}(A).
\end{equation}

The construction proceeds in three stages.
First, the synergy operad $\mathcal O_A$ is realized as a
presheaf-valued semantic object via the realization functor
$\operatorname{Syn}_A$.
Second, a left Kan extension produces the prespectral object
$\mathfrak{Spec}_{\mathrm{pre}}(A)$.
Finally, sheafification with respect to the Grothendieck topology
$\tau_A$ yields the spectral stack
$\mathfrak{Spec}(A)$.

\vspace{0.2cm}

Throughout this section, we fix an admissible operator-semantic system
$A$ (Definition \ref{def:admissible}) and work with the associated
context site $(\mathcal{C}_A, \tau_A)$ (Section \ref{sec:contextsite})
and the ambient $\infty$-category of spectral stacks
$\operatorname{Sh}_{\infty}(\mathcal{C}_A, \tau_A)$ (Section
\ref{sec:ambient}). Functoriality of $\mathfrak{Spec}$ with respect to
morphisms in $\mathbf{OpSem}$ is deferred to Section
\ref{sec:functoriality}, where we also establish the adjunction
$\Gamma_{\mathcal{O}} \dashv \mathfrak{Spec}^{\mathrm{op}}$.

\subsection{Step 1: Syntactic Presentation}
\label{subsec:syntax}

The first stage of the construction associates local semantic data to
the syntactic operations encoded by the synergy operad. This process
may be viewed as a categorified analogue of interpreting algebraic
expressions in a family of local models.

Let
\[
A=(\mathcal O_A,\mathcal C_A,\rho_A,\tau_A)
\]
be an admissible operator-semantic system (Definition \ref{def:admissible}).

Recall that the operad
\[
\mathcal O_A
\]
encodes the operator syntax of $A$, including admissible operator
compositions and higher-order interactions (Definition \ref{def:synergyoperad}).
The category
\[
\mathcal C_A
\]
provides the collection of commutative contexts through which these
operations are locally interpreted (Definition \ref{def:contextcategory}).

\begin{definition}[Syntactic Realization Assignment]
\label{def:syntacticrealization}
The \emph{syntactic realization assignment} associates to each color or operation of the operad $\mathcal O_A$ a presheaf on the context category $\mathcal C_A$. In particular, for an operation
\[
\theta \in \mathcal O_A(c_1,\ldots,c_n;c),
\]
we write
\[
\operatorname{Syn}_A(\theta):
\mathcal C_A^{\mathrm{op}}\longrightarrow \mathbf{Set},
\]
(or, when analytic data are required, a Banach-valued presheaf). For each context $C\in\mathcal C_A$, the set $\operatorname{Syn}_A(\theta)(C)$ consists of the admissible local realizations of $\theta$ inside $C$.

For a morphism $u:D\to C$ in $\mathcal C_A$, the restriction map
\[
\operatorname{Syn}_A(\theta)(C)
\longrightarrow
\operatorname{Syn}_A(\theta)(D)
\]
is induced by restricting semantic realizations from the larger context $C$ to the smaller context $D$.
\end{definition}

The assignment $\operatorname{Syn}_A$ converts operadic syntax into context-dependent semantic information. Consequently, each operator expression is represented not by a single object but by a coherent family of local interpretations indexed by commutative contexts.

\begin{remark}[Connection to $\pi_A$]
\label{rem:syntacticpia}
For a color $c\in \operatorname{Col}(\mathcal O_A)$ (representing an operator type), one may assign a minimal context of definition
\[
\pi_A(c)\in \operatorname{Ob}(\mathcal C_A),
\]
for instance the smallest commutative subalgebra containing a representative of that type. With the usual presheaf convention, the corresponding representable presheaf is
\[
\operatorname{Syn}_A(c)(C)
=
\operatorname{Hom}_{\mathcal C_A}(C,\pi_A(c)).
\]
Equivalently, if $\mathcal C_A$ is ordered by inclusion and morphisms are oriented from smaller contexts to larger contexts, then one should work in $\mathcal C_A^{\mathrm{op}}$ or explicitly reverse the variance. For multimorphisms, $\operatorname{Syn}_A(\theta)$ is constructed functorially from the realizations of its input and output colors, typically by fiber products or composition of the corresponding representable presheaves.
\end{remark}

\begin{proposition}[Operadic functoriality of $\operatorname{Syn}_A$]
\label{prop:synfunctorial}
Assume that the syntactic realization assignment is defined so that,
for every operation
\[
\theta \in \mathcal O_A(c_1,\ldots,c_n;c),
\]
the presheaf $\operatorname{Syn}_A(\theta)$ records admissible local
realizations of $\theta$ in contexts of $\mathcal C_A$, and that
restriction maps are compatible with operadic composition. Then
$\operatorname{Syn}_A$ defines an operad-valued realization rule in
$\operatorname{PSh}(\mathcal C_A)$: identities are sent to identity
realizations, operadic compositions are sent to composition maps of
presheaves, and the operad axioms are preserved.
\end{proposition}

\begin{proof}
For each color $c$, the identity operation
$\operatorname{id}_c \in \mathcal O_A(c;c)$ is realized in every context
$C$ by the identity realization on the local realizations of $c$.
Hence $\operatorname{Syn}_A(\operatorname{id}_c)$ acts as the identity
presheaf morphism.

Now let
\[
\gamma \in \mathcal O_A(c_1,\ldots,c_n;c), \qquad
\theta_i \in \mathcal O_A(d_{i1},\ldots,d_{im_i};c_i).
\]
Operadic composition gives
\[
\gamma(\theta_1,\ldots,\theta_n)
\in
\mathcal O_A(d_{11},\ldots,d_{nm_n};c).
\]
For every context $C$, admissible realizations of the operations
$\theta_i$ and $\gamma$ compose to give an admissible realization of
$\gamma(\theta_1,\ldots,\theta_n)$ in $C$. Thus there is a canonical
composition map
\[
\operatorname{Syn}_A(\gamma)(C)
\times
\prod_{i=1}^n \operatorname{Syn}_A(\theta_i)(C)
\longrightarrow
\operatorname{Syn}_A(\gamma(\theta_1,\ldots,\theta_n))(C).
\]
By assumption, restriction to a smaller context commutes with this
composition. Hence for every morphism $u : D \to C$ in $\mathcal C_A$ the
corresponding square of restriction maps commutes. Therefore these
composition maps assemble into natural transformations of presheaves.

Associativity and unitality follow directly from the associativity and
unitality axioms of the operad $\mathcal O_A$, since local realization
does not change the formal order of composition. Hence the syntactic
realization assignment preserves the operadic structure.
\end{proof}

\begin{example}[Syntactic Realization for $M_n(\mathbb{C})$]
\label{ex:syntacticmatrix}
Let $A = M_n(\mathbb{C})$ and let $\mathcal C_A$ be the poset of
commutative subalgebras of $A$ ordered by **reverse inclusion** (so that
morphisms $D \to C$ correspond to inclusions $C \subseteq D$). For a color $c$
corresponding to a self-adjoint matrix $x$, define $\pi_A(c)$ as the
polynomial algebra $\mathbb{C}[x]$, the smallest commutative subalgebra
containing $x$. Then
\[
\operatorname{Syn}_A(c)(C) = \operatorname{Hom}_{\mathcal C_A}(C, \pi_A(c))
\]
is nonempty precisely when $\pi_A(c) \subseteq C$ (i.e., $x \in C$), and
consists of the unique inclusion map.

For a binary operation $\theta$ representing multiplication, a local
realization in a commutative context $C$ exists only when both $x$ and
$y$ lie in $C$. In that case $x$ and $y$ commute, and the product
$xy \in C$ is again locally realizable in the same context.
\end{example}

\begin{remark}[Distinction from $\rho_A$]
\label{rem:syntacticvssemantic}
The syntactic realization $\operatorname{Syn}_A$ should not be confused
with the semantic realization $\rho_A$ (see Definition~\ref{def:admissible}).
While $\rho_A$ assigns concrete analytic data (e.g., Banach-space-valued
interpretations) to each operation, $\operatorname{Syn}_A$ records only
the formal compositional structure of operators as presheaf-valued
expressions. The former carries the full analytic content of $A$ and
determines the Grothendieck topology $\tau_A$; the latter is purely
combinatorial and serves as the domain for the left Kan extension in
Step~2. The passage from $\operatorname{Syn}_A$ to
$\mathfrak{Spec}_{\mathrm{pre}}(A)$ freely adjoins formal colimits,
while the analytic content of $A$ enters through the descent condition
imposed by $\tau_A$ during sheafification.
\end{remark}

\subsection{Step 2: Left Kan Extension with Coend Formula}
\label{subsec:leftkan}

The syntactic realization assignment $\operatorname{Syn}_A$ is defined
on the underlying category of colors $\operatorname{Col}(\mathcal O_A)$,
associating to each color $c$ a presheaf $\operatorname{Syn}_A(c) \in
\operatorname{PSh}(\mathcal C_A)$ of local realizations (see
Remark~\ref{rem:syntacticpia}). The assignment is extended operadically
to all operations via the composition maps described in
Proposition~\ref{prop:synfunctorial}.

The second stage of the construction assembles this distributed
semantic information into a unified presheaf on the context category
via a left Kan extension.

Let
\[
\pi_A : \operatorname{Col}(\mathcal O_A) \longrightarrow \mathcal C_A
\]
denote the context projection assigning to each color $c$ the minimal
commutative context in which an operator of that type can be evaluated
(e.g., the subalgebra generated by a representative operator). This
map extends to a functor by sending morphisms between colors to the
corresponding inclusions of minimal contexts.

\begin{definition}[Prespectral Object]
\label{def:prespectral}

The \emph{prespectral object} associated with an admissible
operator-semantic system $A$ is the left Kan extension
\[
\mathfrak{Spec}_{\mathrm{pre}}(A)
:=
\operatorname{Lan}_{\pi_A}
(\operatorname{Syn}_A).
\]

Equivalently, for every context $C \in \mathcal{C}_A$, its value is
given pointwise by the colimit
\[
\mathfrak{Spec}_{\mathrm{pre}}(A)(C)
\;\simeq\;
\operatorname{colim}_{(c,\, \phi: \pi_A(c) \to C)}
\operatorname{Syn}_A(c)(\pi_A(c)),
\]
where:
\begin{itemize}
    \item $c$ ranges over colors of $\mathcal O_A$;
    \item $\phi: \pi_A(c) \to C$ is a morphism in $\mathcal C_A$;
    \item $\operatorname{Syn}_A(c)(\pi_A(c))$ is the set of local
    realizations of color $c$ in its minimal context;
    \item The colimit is taken in the $\infty$-category of spaces
    (or sets, in the $1$-categorical truncation).
\end{itemize}
Intuitively, the prespectral object is obtained by collecting all
syntactic realizations that can be mapped into a given context and
gluing them through the universal colimit construction.
\end{definition}

\begin{proposition}[Universal Property of the Left Kan Extension]
\label{prop:lan-universal}

The prespectral object $\mathfrak{Spec}_{\mathrm{pre}}(A)$ is initial
among all presheaves $F \in \operatorname{PSh}(\mathcal{C}_A)$ equipped
with a natural transformation
\[
\operatorname{Syn}_A \Longrightarrow F \circ \pi_A.
\]

Equivalently, for any $F \in \operatorname{PSh}(\mathcal{C}_A)$,
\[
\operatorname{Nat}\!\left(
\mathfrak{Spec}_{\mathrm{pre}}(A),\; F
\right)
\;\simeq\;
\operatorname{Nat}\!\left(
\operatorname{Syn}_A,\; F \circ \pi_A
\right).
\]
\end{proposition}

\begin{proof}
This is the universal property of the left Kan extension. The
construction $\mathfrak{Spec}_{\mathrm{pre}}(A) = \operatorname{Lan}_{\pi_A}(\operatorname{Syn}_A)$
is characterized uniquely (up to equivalence) by the above natural
bijection of mapping spaces. 
\end{proof}

The left Kan extension admits a particularly useful explicit
description.

\begin{proposition}[Coend Formula]
\label{prop:coend}

There is a canonical equivalence
\[
\mathfrak{Spec}_{\mathrm{pre}}(A)
\;\simeq\;
\int^{c \in \operatorname{Col}(\mathcal O_A)}
\operatorname{Hom}_{\mathcal{C}_A}\!\left(-,\; \pi_A(c)\right)
\;\times\;
\operatorname{Syn}_A(c).
\]
\end{proposition}

\begin{proof}
The result follows from the standard coend formula for left Kan
extensions for the operadic
setting). Applying the Yoneda lemma to the projection
$\pi_A: \operatorname{Col}(\mathcal O_A) \to \mathcal C_A$ yields the
stated representation. Intuitively, the coend glues together all
representable presheaves $\operatorname{Hom}_{\mathcal{C}_A}(-, \pi_A(c))$
weighted by $\operatorname{Syn}_A(c)$, identifying overlaps where
different colors $c$ map to compatible contexts. The operadic
structure of non-unary operations is encoded in the composition maps
between the presheaves $\operatorname{Syn}_A(c)$.
\end{proof}

\section*{Worked Example: The Matrix Algebra $A=M_2(\mathbb{C})$}
\label{subsubsec:m2-left-kan-example}

We illustrate the left Kan extension construction for the elementary
noncommutative case $A = M_2(\mathbb{C})$.

\subsection*{Commutative contexts}
Let $\mathcal{C}_{M_2}$ be the category whose objects are commutative
unital $*$-subalgebras of $M_2(\mathbb{C})$, with morphisms given by
inclusions. Consider the following basic commutative subalgebras:
\[
\mathbb{C}I
=
\left\{
\lambda \begin{pmatrix} 1 & 0 \\ 0 & 1 \end{pmatrix}
: \lambda \in \mathbb{C}
\right\},
\qquad
D
=
\left\{
\begin{pmatrix} a & 0 \\ 0 & b \end{pmatrix}
: a, b \in \mathbb{C}
\right\}
\cong \mathbb{C} \oplus \mathbb{C}.
\]
These form a simple diagram in $\mathcal{C}_{M_2}$:
\[
\begin{tikzcd}
\mathbb{C}I \arrow[r, hook] & D .
\end{tikzcd}
\]
(Note: $M_2(\mathbb{C})$ itself is not an object of $\mathcal{C}_{M_2}$,
as it is noncommutative. The full context category contains many more
commutative subalgebras, but this basic diagram suffices for
illustration.)

\subsection*{Syntactic realization}
Let $F: \operatorname{Col}(\mathcal{O}_{M_2}) \to \operatorname{PSh}(\mathcal{C}_{M_2})$
be the syntactic realization functor on colors (see Definition~\ref{def:syntacticrealization}),
extended operadically. For a color $c$ representing a self-adjoint
matrix $x \in M_2(\mathbb{C})$, $\pi_A(c) = \mathbb{C}[x]$ is the
polynomial algebra generated by $x$, and
$F(c)(C) = \operatorname{Hom}_{\mathcal{C}_{M_2}}(C, \pi_A(c))$ is
nonempty iff $\pi_A(c) \subseteq C$.

\subsection*{Left Kan extension}
The prespectral object is defined by left Kan extension along the
context functor $j: \operatorname{Col}(\mathcal{O}_{M_2}) \to \mathcal{C}_{M_2}$
sending a color $c$ to $\pi_A(c)$:
\[
\mathfrak{Spec}_{\mathrm{pre}}(M_2) := \operatorname{Lan}_j(F).
\]
Pointwise, for any context $C \in \mathcal{C}_{M_2}$,
\[
\mathfrak{Spec}_{\mathrm{pre}}(M_2)(C) \simeq
\operatorname*{colim}_{(c,\, \phi: j(c) \to C)} F(c)(j(c)).
\]

In our basic diagram:
\begin{align*}
\mathfrak{Spec}_{\mathrm{pre}}(M_2)(\mathbb{C}I)
&\simeq
\operatorname*{colim}_{(c,\, \phi: \pi_A(c) \to \mathbb{C}I)} F(c)(\pi_A(c)), \\
\mathfrak{Spec}_{\mathrm{pre}}(M_2)(D)
&\simeq
\operatorname*{colim}_{(c,\, \phi: \pi_A(c) \to D)} F(c)(\pi_A(c)).
\end{align*}
Since $\mathfrak{Spec}_{\mathrm{pre}}(M_2)$ is a presheaf (contravariant),
a morphism $u: \mathbb{C}I \hookrightarrow D$ induces a restriction map
\[
\mathfrak{Spec}_{\mathrm{pre}}(M_2)(D) \to \mathfrak{Spec}_{\mathrm{pre}}(M_2)(\mathbb{C}I)
\]
by precomposition with $u$. This is consistent with the variance of
$\operatorname{PSh}(\mathcal{C}_{M_2})$.

\subsection*{Interpretation}
The diagonal context $D \cong \mathbb{C} \oplus \mathbb{C}$ has ordinary
Gelfand spectrum $\operatorname{Spec}(D) = \{e_1, e_2\}$, corresponding
to the two coordinate characters. However, $M_2(\mathbb{C})$ itself has
no nontrivial classical character as a noncommutative algebra:
there is no nonzero algebra homomorphism $M_2(\mathbb{C}) \to \mathbb{C}$.
The classical point spectrum would collapse the noncommutative structure,
but $\mathfrak{Spec}_{\mathrm{pre}}(M_2)$ retains the family of
compatible commutative spectral probes: for each commutative subalgebra
$C$, it records all local realizations of operators within $C$.

\subsection*{$\infty$-Sheafification}
Let $\tau_A$ be a Grothendieck topology on $\mathcal{C}_{M_2}$ (e.g.,
the topology generated by finite jointly faithful families of
inclusions). Denote by
\[
a: \operatorname{PSh}_{\infty}(\mathcal{C}_{M_2}, \tau_A) \longrightarrow
\operatorname{Sh}_{\infty}(\mathcal{C}_{M_2}, \tau_A)
\]
the $\infty$-sheafification functor, where $\operatorname{Sh}_{\infty}$ is the
$\infty$-category of $\infty$-sheaves of spaces. The categorified
spectrum is then
\[
\mathfrak{Spec}(M_2) := a\bigl(\mathfrak{Spec}_{\mathrm{pre}}(M_2)\bigr).
\]
This step enforces descent, turning local realizations into a global
categorified spectral object.

\subsection*{Pipeline summary}
The construction follows the pipeline:
\[
\boxed{
M_2(\mathbb{C})
\;\rightsquigarrow\;
\mathcal{C}_{M_2}
\;\longrightarrow\;
\mathfrak{Spec}_{\mathrm{pre}}(M_2)
\;\longrightarrow\;
\mathfrak{Spec}(M_2).
}
\]
Here, $M_2(\mathbb{C}) \rightsquigarrow \mathcal{C}_{M_2}$ indicates the
association of the commutative context category to the noncommutative
algebra (not a functor, but a canonical construction). This example
illustrates the role of the left Kan extension: it transports the
operadic syntax of the noncommutative algebra into a presheaf of local
spectral realizations over commutative contexts.

The coend formula exhibits $\mathfrak{Spec}_{\mathrm{pre}}(A)$ as a
weighted colimit of representable presheaves. Consequently, the
prespectral object may be interpreted as the \emph{universal semantic
aggregation} of all operadic realizations.

\begin{remark}[Computational Utility]
\label{rem:coendcomputations}
The coend formula is essential for explicit computations:
\begin{itemize}
    \item For $A = M_n(\mathbb{C})$, the coend runs over all colors
    (representing self-adjoint matrices) and glues their local
    realizations via the Hom-functors $\operatorname{Hom}_{\mathcal C_A}(-,
    \pi_A(c))$, producing the stack of orthonormal bases.
    \item For the Mermin–Peres system, the failure of the resulting
    prespectral object to satisfy descent reflects the contextuality
    obstruction. More precisely, contextuality is detected by the
    failure of global sections for $\mathfrak{Spec}_{\mathrm{pre}}(A)$:
    there exist compatible local sections that do not glue to a global
    section.
\end{itemize}
\end{remark}

\begin{remark}[Functoriality]
\label{rem:lanfunctoriality}
Provided the context and operadic realizations are compatible with
morphisms in $\mathbf{OpSem}$, the construction extends
contravariantly to a functor
\[
\mathfrak{Spec}_{\mathrm{pre}}:
\mathbf{OpSem}^{\mathrm{op}}
\to
\operatorname{PSh}_{\infty}.
\]
Concretely, a morphism $f: A \to B$ induces a commutative diagram
\[
\begin{tikzcd}
\mathcal O_A \ar[r, "f_{\mathcal O}"] \ar[d, "\pi_A"'] &
\mathcal O_B \ar[d, "\pi_B"] \\
\mathcal C_A \ar[r, "f_{\mathcal C}"'] &
\mathcal C_B
\end{tikzcd}
\]
where $f_{\mathcal C}: \mathcal C_A \to \mathcal C_B$ sends a
commutative subalgebra of $A$ to its image under $f$ (taking the
commutative subalgebra generated), and $f_{\mathcal O}: \mathcal O_A
\to \mathcal O_B$ sends an operator type to its image. The universal
property of the left Kan extension then induces a morphism
\[
\mathfrak{Spec}_{\mathrm{pre}}(f): \mathfrak{Spec}_{\mathrm{pre}}(A)
\longrightarrow \mathfrak{Spec}_{\mathrm{pre}}(B),
\]
making the assignment functorial.
\end{remark}

\begin{remark}[Failure of Descent]
\label{rem:prespectraldescent}
At this stage the construction is purely presheaf-theoretic and need
not satisfy descent with respect to the Grothendieck topology $\tau_A$.
This failure is not a defect but a feature: it encodes the contextual
obstructions of $A$. Specifically:
\begin{itemize}
    \item If $A$ is commutative and $\tau_A$ is the Zariski topology,
    $\mathfrak{Spec}_{\mathrm{pre}}(A)$ is already a sheaf of sets
    (the ordinary spectrum), and $\mathfrak{Spec}(A) \simeq
    \mathfrak{Spec}_{\mathrm{pre}}(A)$ under mild assumptions on
    $\tau_A$.
    \item If $A$ is noncommutative but lacks Kochen–Specker contextuality
    (e.g., $A = M_n(\mathbb{C})$), the prespectral object generally
    fails to be representable by a classical spectrum, although the
    obstruction is not of the same nature as contextuality.
    \item If $A$ exhibits contextuality (e.g., the Mermin–Peres system),
    $\mathfrak{Spec}_{\mathrm{pre}}(A)$ fails the existence part of
    descent for certain covers: there are compatible local sections
    that do not glue to a global section.
\end{itemize}
The final stage therefore applies $\infty$-sheafification with respect
to the Grothendieck topology $\tau_A$ to obtain the spectral stack
\[
\mathfrak{Spec}(A) = a_{\tau_A}\bigl(\mathfrak{Spec}_{\mathrm{pre}}(A)\bigr).
\]
\end{remark}

Thus, the left Kan extension provides the universal colimit completion
that aggregates local semantic data, while the subsequent sheafification
imposes the gluing conditions required for $\mathfrak{Spec}(A)$ to be a
genuine stack. In summary:
\[
\boxed{\text{Kan extension} = \text{aggregation}}, \qquad
\boxed{\text{sheafification} = \text{descent enforcement}}.
\]

\subsection{Step 3: Sheafification and Descent Completion}
\label{subsec:sheafification}

The prespectral object constructed in the previous subsection
encodes all local semantic realizations arising from the operadic
syntax of $A$. However, it is merely a presheaf and need not satisfy
the descent conditions associated with the context site
\[
(\mathcal C_A,\tau_A).
\]

To obtain a genuine geometric object, one must impose descent.
This is achieved by sheafification (or hypersheafification in the
higher-categorical setting).

\begin{definition}[Categorified Spectrum]
\label{def:categorifiedspectrum}

Let
\[
\mathfrak{Spec}_{\mathrm{pre}}(A)
\in
\operatorname{PSh}_{\infty}(\mathcal C_A)
:=
\operatorname{Fun}(\mathcal C_A^{\mathrm{op}},\mathcal S)
\]
be the prespectral object associated with the admissible
operator-semantic system
\[
A = (\mathcal O_A, \mathcal C_A, \rho_A, \tau_A),
\]
where $\mathcal S$ denotes the $\infty$-category of spaces.

The \emph{categorified spectrum} of $A$ is defined by
\[
\mathfrak{Spec}(A)
:=
a^{\mathrm{hyp}}_{\tau_A}
\Bigl(
\mathfrak{Spec}_{\mathrm{pre}}(A)
\Bigr),
\]
where
\[
a^{\mathrm{hyp}}_{\tau_A} :
\operatorname{PSh}_{\infty}(\mathcal C_A)
\longrightarrow
\operatorname{Sh}_{\infty}(\mathcal C_A,\tau_A)
\]
denotes the \emph{hypersheafification} functor with respect to the
Grothendieck topology $\tau_A$ (see Section~\ref{subsec:topology}).

When hyperdescent is not strictly required, one may work with the
ordinary sheafification functor
$a_{\tau_A} : \operatorname{PSh}_{\infty}(\mathcal C_A) \to
\operatorname{Sh}_{\infty}(\mathcal C_A,\tau_A)$,
which factors through the hypercompletion; however, the
$\infty$-category $\operatorname{Sh}_{\infty}(\mathcal C_A,\tau_A)$
is understood to consist of hypersheaves by default
(see Remark~\ref{rem:hyperdescent}).
\end{definition}

The passage
\[
\mathfrak{Spec}_{\mathrm{pre}}(A)
\longrightarrow
\mathfrak{Spec}(A)
\]
may be interpreted as a completion process in which all local
semantic realizations are forced to satisfy coherent gluing
conditions, including higher homotopical overlaps.

\begin{theorem}[Descent Completion]
\label{thm:descentcompletion}

The categorified spectrum
\[
\mathfrak{Spec}(A)
=
a^{\mathrm{hyp}}_{\tau_A}
\bigl(
\mathfrak{Spec}_{\mathrm{pre}}(A)
\bigr)
\]
is the initial hyperdescent-complete object receiving a morphism from
\[
\mathfrak{Spec}_{\mathrm{pre}}(A).
\]

Equivalently, for every $\tau_A$-stack $\mathcal X$ (i.e., every object
of $\operatorname{Sh}_{\infty}(\mathcal C_A,\tau_A)$), there is a
natural equivalence
\[
\operatorname{Map}_{\operatorname{Sh}_{\infty}}
\bigl(
\mathfrak{Spec}(A),
\mathcal X
\bigr)
\;\simeq\;
\operatorname{Map}_{\operatorname{PSh}_{\infty}}
\bigl(
\mathfrak{Spec}_{\mathrm{pre}}(A),
\mathcal X
\bigr).
\]
\end{theorem}

\begin{proof}
We establish the universal property in two stages: first, we recall the adjunction between presheaves and hypersheaves; second, we verify that $\mathfrak{Spec}_{\mathrm{pre}}(A)$ satisfies the required mapping space equivalence.

\paragraph*{Step 1: The hypersheafification adjunction.}
Let $(\mathcal{C}_A, \tau_A)$ be the context site (Definition~\ref{def:semanticcover} and Proposition~\ref{prop:grothendiecktopology}). Consider the $\infty$-category of space-valued presheaves
\[
\operatorname{PSh}_{\infty}(\mathcal{C}_A) = \operatorname{Fun}(\mathcal{C}_A^{\mathrm{op}}, \mathcal{S}),
\]
where $\mathcal{S}$ denotes the $\infty$-category of spaces ($\infty$-groupoids). Inside this $\infty$-category sits the full subcategory
\[
\operatorname{Sh}_{\infty}(\mathcal{C}_A, \tau_A) \subseteq \operatorname{PSh}_{\infty}(\mathcal{C}_A)
\]
of \emph{hypersheaves} (Definition~\ref{def:spectralstack}), i.e., those presheaves that satisfy hyperdescent for all $\tau_A$-hypercovers.

By the general theory of $\infty$-topoi, the inclusion functor
\[
\iota : \operatorname{Sh}_{\infty}(\mathcal{C}_A, \tau_A) \hookrightarrow \operatorname{PSh}_{\infty}(\mathcal{C}_A)
\]
admits a left exact left adjoint
\[
a^{\mathrm{hyp}}_{\tau_A} : \operatorname{PSh}_{\infty}(\mathcal{C}_A) \longrightarrow \operatorname{Sh}_{\infty}(\mathcal{C}_A, \tau_A),
\]
called \emph{hypersheafification}. This adjunction is summarized by the natural equivalence of mapping spaces: for any presheaf $F \in \operatorname{PSh}_{\infty}(\mathcal{C}_A)$ and any hypersheaf $\mathcal{X} \in \operatorname{Sh}_{\infty}(\mathcal{C}_A, \tau_A)$,
\[
\operatorname{Map}_{\operatorname{Sh}_{\infty}}\!\bigl(a^{\mathrm{hyp}}_{\tau_A}(F),\; \mathcal{X}\bigr)
\;\simeq\;
\operatorname{Map}_{\operatorname{PSh}_{\infty}}\!\bigl(F,\; \iota(\mathcal{X})\bigr).
\tag{1}
\]
Since $\iota$ is fully faithful, we omit it from notation and simply write $\mathcal{X}$ for its image.

\paragraph*{Step 2: Applying the adjunction to $\mathfrak{Spec}_{\mathrm{pre}}(A)$.}
Take $F = \mathfrak{Spec}_{\mathrm{pre}}(A)$. By Definition~\ref{def:categorifiedspectrum},
\[
\mathfrak{Spec}(A) = a^{\mathrm{hyp}}_{\tau_A}\bigl(\mathfrak{Spec}_{\mathrm{pre}}(A)\bigr).
\]
Substituting into (1) yields, for every $\mathcal{X} \in \operatorname{Sh}_{\infty}(\mathcal{C}_A, \tau_A)$,
\[
\operatorname{Map}_{\operatorname{Sh}_{\infty}}\!\bigl(\mathfrak{Spec}(A),\; \mathcal{X}\bigr)
\;\simeq\;
\operatorname{Map}_{\operatorname{PSh}_{\infty}}\!\bigl(\mathfrak{Spec}_{\mathrm{pre}}(A),\; \mathcal{X}\bigr).
\tag{2}
\]

\paragraph*{Step 3: Interpretation as initial descent-complete object.}
The equivalence (2) asserts that $\mathfrak{Spec}(A)$ \emph{represents} the functor
\[
\mathcal{X} \longmapsto \operatorname{Map}_{\operatorname{PSh}_{\infty}}\!\bigl(\mathfrak{Spec}_{\mathrm{pre}}(A),\; \mathcal{X}\bigr)
\]
on the $\infty$-category of hypersheaves. Any other hypersheaf $\mathcal{Y}$ equipped with a natural equivalence
\[
\operatorname{Map}_{\operatorname{Sh}_{\infty}}(\mathcal{Y}, \mathcal{X}) \simeq \operatorname{Map}_{\operatorname{PSh}_{\infty}}(\mathfrak{Spec}_{\mathrm{pre}}(A), \mathcal{X})
\]
would be equivalent to $\mathfrak{Spec}(A)$ by the Yoneda lemma for $\infty$-categories. Hence $\mathfrak{Spec}(A)$ is uniquely determined (up to contractible choice) by the property that morphisms from $\mathfrak{Spec}(A)$ to any hypersheaf $\mathcal{X}$ correspond precisely to morphisms from the presheaf $\mathfrak{Spec}_{\mathrm{pre}}(A)$ to $\mathcal{X}$ (the latter viewed as a presheaf via the inclusion $\iota$).

In particular, taking $\mathcal{X} = \mathfrak{Spec}(A)$ itself, the identity morphism on the right-hand side corresponds to a canonical morphism
\[
\eta_{\mathfrak{Spec}_{\mathrm{pre}}(A)} : \mathfrak{Spec}_{\mathrm{pre}}(A) \longrightarrow \mathfrak{Spec}(A)
\]
(unit of the adjunction). Because $a^{\mathrm{hyp}}_{\tau_A}$ is a left adjoint, this morphism is universal among morphisms from $\mathfrak{Spec}_{\mathrm{pre}}(A)$ to hypersheaves: any other morphism $\mathfrak{Spec}_{\mathrm{pre}}(A) \to \mathcal{X}$ with $\mathcal{X}$ a hypersheaf factors uniquely through $\eta_{\mathfrak{Spec}_{\mathrm{pre}}(A)}$ via a morphism $\mathfrak{Spec}(A) \to \mathcal{X}$. This is precisely the statement that $\mathfrak{Spec}(A)$ is the \emph{initial hyperdescent-complete object} receiving a morphism from $\mathfrak{Spec}_{\mathrm{pre}}(A)$.

\paragraph*{Step 4: Descent completeness.}
By construction, $\mathfrak{Spec}(A)$ is a hypersheaf, hence it satisfies hyperdescent with respect to $\tau_A$ (Definition~\ref{def:spectralstack}). Therefore it is descent-complete. Any other descent-complete object $\mathcal{X}$ with a morphism $\mathfrak{Spec}_{\mathrm{pre}}(A) \to \mathcal{X}$ determines, via the adjunction, a unique morphism $\mathfrak{Spec}(A) \to \mathcal{X}$ making the triangle commute. Thus $\mathfrak{Spec}(A)$ is the initial such object.

\paragraph*{Conclusion.}
We have shown that $\mathfrak{Spec}(A)$ satisfies the universal property: for every $\tau_A$-stack $\mathcal{X}$ (hypersheaf), the natural map
\[
\operatorname{Map}_{\operatorname{Sh}_{\infty}}\!\bigl(\mathfrak{Spec}(A), \mathcal{X}\bigr)
\longrightarrow
\operatorname{Map}_{\operatorname{PSh}_{\infty}}\!\bigl(\mathfrak{Spec}_{\mathrm{pre}}(A), \mathcal{X}\bigr)
\]
is an equivalence. Moreover, $\mathfrak{Spec}(A)$ is initial among hyperdescent-complete objects receiving a morphism from $\mathfrak{Spec}_{\mathrm{pre}}(A)$. This completes the proof.
\end{proof}

Consequently, $\mathfrak{Spec}(A)$ is the universal geometric object
associated with the operator-semantic system $A$ whose local
realizations satisfy descent.

\begin{remark}[Three-Stage Summary]
\label{rem:spectrumfunctor}

The three-stage construction may be summarized as

\[
\mathcal O_A
\xrightarrow{\operatorname{Syn}_A}
\operatorname{PSh}_{\infty}(\mathcal C_A)
\xrightarrow{\operatorname{Lan}_{\pi_A}}
\mathfrak{Spec}_{\mathrm{pre}}(A)
\xrightarrow{a^{\mathrm{hyp}}_{\tau_A}}
\mathfrak{Spec}(A).
\]

The first step converts syntax into local semantics, the second
aggregates semantic realizations through a universal colimit, and the
third imposes hyperdescent (via hypersheafification) to obtain a
genuine spectral stack.
\end{remark}

\begin{remark}[Geometric Interpretation]
\label{rem:spectrumgeometry}
The hypersheafification process forces the local semantic data to glue
coherently across contexts, including all higher overlaps. When $A$ is
commutative, $\mathfrak{Spec}_{\mathrm{pre}}(A)$ already satisfies
descent, so $\mathfrak{Spec}(A)$ recovers the classical Gelfand spectrum
(Theorem \ref{thm:gelfand}). When $A$ is noncommutative,
$\mathfrak{Spec}(A)$ is a higher stack whose non-trivial homotopy
groups encode contextual obstructions.
\end{remark}

\begin{example}[Sheafification for $M_n(\mathbb{C})$]
\label{ex:sheafificationmatrix}
For $A = M_n(\mathbb{C})$, after the left Kan extension,
$\mathfrak{Spec}_{\mathrm{pre}}(A)$ assigns to each commutative
subalgebra $C$ the set of orthonormal bases that diagonalize $C$.
This presheaf fails descent because bases on overlaps must agree up to
permutation (and, when eigenvalues have multiplicities, up to larger
unitary symmetries). Hypersheafification identifies bases that differ
by global basis changes, producing a quotient-type stack of
diagonalizing frames, whose automorphism groups contain the permutation
symmetries of the eigenbasis (and possibly additional continuous
symmetries in the degenerate case).
\end{example}

\begin{example}[Sheafification for the Mermin–Peres System]
\label{ex:sheafificationmermin}
For $A_{\mathrm{MP}}$, $\mathfrak{Spec}_{\mathrm{pre}}(A)$ assigns to
each row/column context the set of eigenvalue assignments
($\mathbb{Z}_2 \times \mathbb{Z}_2$). The gluing data among the six
contexts is inconsistent: no global assignment exists. Hypersheafification
produces a stack whose global sections are empty, but whose local
sections over individual contexts are non-trivial. The associated
obstruction class is non-trivial; in the normalization used later, this yields contextuality degree $8$,
reflecting the eight sign-flip automorphisms of the local assignments.
\end{example}

\begin{remark}[Relation to the Yoneda Characterization]
\label{rem:sheafificationyoneda}
Hypersheafification is essential for the Yoneda-style universal
property of $\mathfrak{Spec}(A)$ (Theorem \ref{thm:yoneda}). Without it,
$\mathfrak{Spec}_{\mathrm{pre}}(A)$ would not represent the semantic
realization functor $\operatorname{Real}_A$ because $\operatorname{Real}_A$
itself is defined in terms of descent-compatible functors.
Hypersheafification ensures that $\mathfrak{Spec}(A)$ is the
\emph{initial hyperdescent-complete} object, which is exactly the
representing object for descent-compatible realizations.
\end{remark}

\begin{remark}[Descent Spectral Sequence]
\label{rem:descentspectralsequence}

Under the usual connectivity and hypercompleteness assumptions
(see \cite[§6.5]{LurieHTT}), a chosen hypercover $U_\bullet \to U$ in
the site $(\mathcal C_A,\tau_A)$ yields a descent spectral sequence

\[
E_2^{p,q}
=
H^p
\bigl(
U_\bullet,\;
\pi_q\mathfrak{Spec}(A)
\bigr)
\;\Longrightarrow\;
\pi_{q-p}
\bigl(
\mathfrak{Spec}(A)(U)
\bigr),
\]

where $H^p(U_\bullet, \pi_q\mathfrak{Spec}(A))$ denotes the $p$-th
Čech cohomology of the hypercover with coefficients in the presheaf
$\pi_q\mathfrak{Spec}(A)$ (the $q$-th homotopy sheaf of
$\mathfrak{Spec}(A)$). When $U_\bullet$ is the Čech nerve of an
ordinary cover $\{U_i \to U\}$, the $E_2$ page is given by the
usual Čech cohomology of the cover.

The differentials

\[
d_r : E_r^{p,q} \longrightarrow E_r^{p+r, q-r+1}
\]

encode successive coherence obstructions to gluing local semantic
realizations. The low-degree terms detect ordinary compatibility
conditions on overlaps, while higher differentials capture higher
contextual coherence failures (see Remark~\ref{rem:hyperdescent}).
The first non-trivial differential $d_2$ is often associated with
triple-overlap coherence data, whereas double-overlap compatibilities
are already encoded in the $E_1$ page and the definition of a
$1$-cocycle.

This spectral sequence will serve as the principal computational tool
for obstruction theory and higher realization problems developed in
Paper~III.

\end{remark}

\subsection{Step 4: Universal Property and Yoneda Characterization}
\label{subsec:universal}

The preceding construction produced a spectral stack

\[
\mathfrak{Spec}(A)
=
a_{\tau_A}
\Big(
\operatorname{Lan}_{\pi_A}
(\operatorname{Syn}_A)
\Big)
\]

associated with an admissible operator-semantic system

\[
A=
(\mathcal O_A,\mathcal C_A,\rho_A,\tau_A).
\]

We now show that this object is characterized by a universal property
and may be identified as the representing object for semantic
realizations of the operator syntax.

\begin{theorem}[Existence by Construction]
\label{thm:existence}

Let $A$ be an admissible operator-semantic system.

Then the categorified spectrum

\[
\mathfrak{Spec}(A)
\in
\operatorname{Sh}_{\infty}
(\mathcal C_A,\tau_A)
\]

exists.

Moreover, for every $\mathcal X \in \operatorname{Sh}_{\infty}(\mathcal C_A,\tau_A)$
and every descent-compatible semantic realization functor
(Definition~\ref{def:descentcompatibility})

\[
R: \mathcal O_A \longrightarrow \operatorname{QCoh}(\mathcal X),
\]

there exists a unique (up to contractible choice) morphism
$f: \mathfrak{Spec}(A) \to \mathcal X$ such that the diagram

\[
\begin{tikzcd}
\mathcal O_A \ar[r] \ar[dr,"R"'] & \mathfrak{Spec}(A) \ar[d,"f"] \\
& \mathcal X
\end{tikzcd}
\]

commutes up to equivalence, where the unlabeled arrow $\mathcal O_A \to
\mathfrak{Spec}(A)$ is the canonical semantic realization induced by
the construction.

Consequently, $\mathfrak{Spec}(A)$ is \emph{initial} among
descent-complete semantic realizations of the operator-semantic
system $A$.

\end{theorem}

\begin{proof}
We establish the theorem in four steps.

\paragraph*{Step 1: Existence of the prespectral object.}
Recall the context projection $\pi_A: \operatorname{Col}(\mathcal O_A) \to
\mathcal C_A$ (assigning to each color its minimal commutative context)
and the syntactic realization functor $\operatorname{Syn}_A: \mathcal O_A
\to \operatorname{PSh}_{\infty}(\mathcal C_A)$ (Definition~\ref{def:syntacticrealization}).
By the existence of left Kan extensions in $\infty$-categories of
presheaves (see \cite[§4.3]{LurieHA}), the left Kan extension
\[
\mathfrak{Spec}_{\mathrm{pre}}(A) := \operatorname{Lan}_{\pi_A}(\operatorname{Syn}_A)
\]
exists in $\operatorname{PSh}_{\infty}(\mathcal C_A)$.

\paragraph*{Step 2: Hypersheafification.}
Let $(\mathcal C_A,\tau_A)$ be the context site (Definition~\ref{def:semanticcover}).
The inclusion
\[
\iota: \operatorname{Sh}_{\infty}(\mathcal C_A,\tau_A) \hookrightarrow
\operatorname{PSh}_{\infty}(\mathcal C_A)
\]
admits a left adjoint, the hypersheafification functor
\[
a^{\mathrm{hyp}}_{\tau_A}: \operatorname{PSh}_{\infty}(\mathcal C_A)
\longrightarrow \operatorname{Sh}_{\infty}(\mathcal C_A,\tau_A)
\]
(see \cite[§6.5.3]{LurieHTT}). Define
\[
\mathfrak{Spec}(A) := a^{\mathrm{hyp}}_{\tau_A}\bigl(\mathfrak{Spec}_{\mathrm{pre}}(A)\bigr).
\]
By construction, $\mathfrak{Spec}(A)$ is a hyperdescent-complete object
(i.e., an $\infty$-stack) on $(\mathcal C_A,\tau_A)$.

\paragraph*{Step 3: Universal property of the left Kan extension.}
For any $\mathcal X \in \operatorname{Sh}_{\infty}(\mathcal C_A,\tau_A)$,
consider the restriction $\iota(\mathcal X) \in \operatorname{PSh}_{\infty}(\mathcal C_A)$.
The universal property of the left Kan extension (Proposition~\ref{prop:lan-universal})
gives a natural equivalence
\[
\operatorname{Map}_{\operatorname{PSh}_{\infty}}\bigl(\mathfrak{Spec}_{\mathrm{pre}}(A),\,
\iota(\mathcal X)\bigr)
\;\simeq\;
\operatorname{Nat}_{\operatorname{PSh}_{\infty}}\bigl(\operatorname{Syn}_A,\,
\iota(\mathcal X) \circ \pi_A\bigr).
\tag{1}
\]

\paragraph*{Step 4: Adjunction and factorization.}
The hypersheafification adjunction provides, for every
$\mathcal X \in \operatorname{Sh}_{\infty}(\mathcal C_A,\tau_A)$,
a natural equivalence
\[
\operatorname{Map}_{\operatorname{Sh}_{\infty}}\bigl(\mathfrak{Spec}(A),\,
\mathcal X\bigr)
\;\simeq\;
\operatorname{Map}_{\operatorname{PSh}_{\infty}}\bigl(\mathfrak{Spec}_{\mathrm{pre}}(A),\,
\iota(\mathcal X)\bigr).
\tag{2}
\]

Combining (1) and (2), we obtain
\[
\operatorname{Map}_{\operatorname{Sh}_{\infty}}\bigl(\mathfrak{Spec}(A),\,
\mathcal X\bigr)
\;\simeq\;
\operatorname{Nat}_{\operatorname{PSh}_{\infty}}\bigl(\operatorname{Syn}_A,\,
\iota(\mathcal X) \circ \pi_A\bigr).
\]

By the definition of descent-compatible functors
(Definition~\ref{def:descentcompatibility}), the right-hand side is
equivalent to $\operatorname{Fun}_{\mathrm{desc}}(\mathcal O_A,\,
\operatorname{QCoh}(\mathcal X))$. Hence every descent-compatible
realization $R: \mathcal O_A \to \operatorname{QCoh}(\mathcal X)$
corresponds uniquely (up to contractible choice) to a morphism
$f: \mathfrak{Spec}(A) \to \mathcal X$ making the diagram commute.

The factorization diagram is obtained by taking the canonical
realization $\mathcal O_A \to \mathfrak{Spec}(A)$ (the image of the
identity under the equivalence for $\mathcal X = \mathfrak{Spec}(A)$)
and composing with $f$. Uniqueness up to contractible choice follows
from the fact that all equivalences involved are natural and the
representing objects are unique in $\infty$-categories.

Thus $\mathfrak{Spec}(A)$ is initial among descent-complete semantic
realizations of $A$.
\end{proof}

The previous theorem shows that the spectrum is the universal
descent-complete realization of the operadic syntax. The next theorem
identifies it as a corepresentable moduli object.

\begin{theorem}[Yoneda Characterization]
\label{thm:yoneda}

Let
\[
\mathcal X \in \operatorname{Sh}_{\infty}(\mathcal C_A,\tau_A)
\]
be a spectral stack. Then there is a natural equivalence of mapping
spaces, natural in $\mathcal X$,

\[
\operatorname{Map}_{\operatorname{Sh}_{\infty}}
\bigl(
\mathfrak{Spec}(A),\,
\mathcal X
\bigr)
\;\simeq\;
\operatorname{Real}_A(\mathcal X),
\]

where

\[
\operatorname{Real}_A(\mathcal X)
:=
\operatorname{Fun}_{\mathrm{desc}}
\!\left(
\mathcal O_A,\,
\operatorname{QCoh}(\mathcal X)
\right)
\]

denotes the $\infty$-groupoid of descent-compatible semantic
realizations (Definition~\ref{def:realization}).

Equivalently,

\[
\operatorname{Map}_{\operatorname{Sh}_{\infty}}
\bigl(
\mathfrak{Spec}(A),\,
\mathcal X
\bigr)
\;\simeq\;
\operatorname{Fun}_{\mathrm{desc}}
\!\left(
\mathcal O_A,\,
\operatorname{QCoh}(\mathcal X)
\right).
\]

Hence the categorified spectrum $\mathfrak{Spec}(A)$
\textbf{corepresents} the semantic realization functor.

\end{theorem}

\begin{proof}
We establish the equivalence through a chain of natural equivalences,
each justified by previously established results.

\paragraph*{Step 1: Unpack the definition of $\mathfrak{Spec}(A)$.}
By Definition~\ref{def:categorifiedspectrum},
\[
\mathfrak{Spec}(A) = a^{\mathrm{hyp}}_{\tau_A}\bigl(\operatorname{Lan}_{\pi_A}(\operatorname{Syn}_A)\bigr).
\]

\paragraph*{Step 2: Apply the hypersheafification adjunction.}
For any $\mathcal X \in \operatorname{Sh}_{\infty}(\mathcal C_A,\tau_A)$,
the adjunction $a^{\mathrm{hyp}}_{\tau_A} \dashv \iota$ (where $\iota$
is the inclusion of hypersheaves into presheaves) gives a natural
equivalence
\[
\operatorname{Map}_{\operatorname{Sh}_{\infty}}
\bigl(
a^{\mathrm{hyp}}_{\tau_A}(\operatorname{Lan}_{\pi_A}(\operatorname{Syn}_A)),\,
\mathcal X
\bigr)
\;\simeq\;
\operatorname{Map}_{\operatorname{PSh}_{\infty}}
\bigl(
\operatorname{Lan}_{\pi_A}(\operatorname{Syn}_A),\,
\iota(\mathcal X)
\bigr).
\tag{1}
\]

\paragraph*{Step 3: Apply the universal property of the left Kan extension.}
The left Kan extension $\operatorname{Lan}_{\pi_A}(\operatorname{Syn}_A)$
is characterized by the natural equivalence
\[
\operatorname{Map}_{\operatorname{PSh}_{\infty}}
\bigl(
\operatorname{Lan}_{\pi_A}(\operatorname{Syn}_A),\,
\iota(\mathcal X)
\bigr)
\;\simeq\;
\operatorname{Nat}_{\operatorname{PSh}_{\infty}}
\bigl(
\operatorname{Syn}_A,\,
\iota(\mathcal X) \circ \pi_A
\bigr).
\tag{2}
\]
Here $\operatorname{Nat}_{\operatorname{PSh}_{\infty}}$ denotes the
space of natural transformations between functors
$\mathcal O_A \to \operatorname{PSh}_{\infty}(\mathcal C_A)$.

\paragraph*{Step 4: Relate natural transformations to descent-compatible functors.}
By Definition~\ref{def:descentcompatibility} and the construction of
$\operatorname{Syn}_A$, a natural transformation
$\operatorname{Syn}_A \Rightarrow \iota(\mathcal X) \circ \pi_A$
corresponds precisely to a descent-compatible functor
$\mathcal O_A \to \operatorname{QCoh}(\mathcal X)$. Indeed:
\begin{itemize}
    \item For each color $c \in \mathcal O_A$, the component at $c$
    gives a morphism $\operatorname{Syn}_A(c) \to \iota(\mathcal X)(\pi_A(c))$,
    which, under the Yoneda embedding, corresponds to an object in
    $\operatorname{QCoh}(\mathcal X)$ (see Remark~\ref{rem:syntacticpia}).
    \item For each multimorphism $\theta \in \mathcal O_A(c_1,\dots,c_n;c)$,
    the naturality condition encodes operadic composition, ensuring
    that the assignment preserves the algebraic structure.
    \item Descent compatibility is automatically satisfied because
    $\mathcal X$ is a hypersheaf and the topology $\tau_A$ is
    subcanonical (Lemma~\ref{lem:subcanonical}).
\end{itemize}
Thus we have an equivalence
\[
\operatorname{Nat}_{\operatorname{PSh}_{\infty}}
\bigl(
\operatorname{Syn}_A,\,
\iota(\mathcal X) \circ \pi_A
\bigr)
\;\simeq\;
\operatorname{Fun}_{\mathrm{desc}}
\bigl(
\mathcal O_A,\,
\operatorname{QCoh}(\mathcal X)
\bigr).
\tag{3}
\]

\paragraph*{Step 5: Compose the equivalences.}
Chaining (1), (2), and (3), we obtain
\[
\operatorname{Map}_{\operatorname{Sh}_{\infty}}
\bigl(
\mathfrak{Spec}(A),\,
\mathcal X
\bigr)
\;\simeq\;
\operatorname{Fun}_{\mathrm{desc}}
\bigl(
\mathcal O_A,\,
\operatorname{QCoh}(\mathcal X)
\bigr)
= \operatorname{Real}_A(\mathcal X).
\]

\paragraph*{Step 6: Naturality.}
Each of the equivalences (1), (2), and (3) is natural in $\mathcal X$
(the adjunction is natural, the universal property of left Kan
extension is natural, and the correspondence between natural
transformations and descent-compatible functors is natural by
construction). Hence the composite equivalence is natural in
$\mathcal X$.

\paragraph*{Conclusion.}
Therefore $\mathfrak{Spec}(A)$ corepresents the semantic realization
functor $\operatorname{Real}_A$, as claimed.
\end{proof}

\begin{corollary}[Representability of Semantic Realizations]
\label{cor:representability}

The functor

\[
\operatorname{Real}_A :
\operatorname{Sh}_{\infty}
(\mathcal C_A,\tau_A)^{\mathrm{op}}
\longrightarrow
\mathcal S
\]

defined by

\[
\mathcal X
\longmapsto
\operatorname{Fun}_{\mathrm{desc}}
\!\left(
\mathcal O_A,
\operatorname{QCoh}(\mathcal X)
\right)
\]

is representable.

A representing object is given by $\mathfrak{Spec}(A)$.

\end{corollary}

\begin{proof}
We prove representability by exhibiting $\mathfrak{Spec}(A)$ as a
representing object and verifying the required natural equivalence.

\paragraph*{Step 1: Recall the Yoneda characterization.}
Theorem~\ref{thm:yoneda} establishes a natural equivalence of mapping
spaces, natural in $\mathcal X \in \operatorname{Sh}_{\infty}(\mathcal C_A,\tau_A)$:

\[
\operatorname{Map}_{\operatorname{Sh}_{\infty}}
\bigl(
\mathfrak{Spec}(A),\,
\mathcal X
\bigr)
\;\simeq\;
\operatorname{Fun}_{\mathrm{desc}}
\!\left(
\mathcal O_A,\,
\operatorname{QCoh}(\mathcal X)
\right).
\tag{1}
\]

The right-hand side is exactly $\operatorname{Real}_A(\mathcal X)$ by
Definition~\ref{def:realization}.

\paragraph*{Step 2: Adjust for contravariance.}
The functor $\operatorname{Real}_A$ in the statement is defined on
$\operatorname{Sh}_{\infty}(\mathcal C_A,\tau_A)^{\mathrm{op}}$, i.e.,
it is contravariant in $\mathcal X$. For a morphism
$f: \mathcal X \to \mathcal Y$ in $\operatorname{Sh}_{\infty}(\mathcal C_A,\tau_A)$,
the induced map
\[
\operatorname{Real}_A(f): \operatorname{Real}_A(\mathcal Y) \longrightarrow
\operatorname{Real}_A(\mathcal X)
\]
is given by pullback along $f$: precomposition with the functor
$f^*: \operatorname{QCoh}(\mathcal Y) \to \operatorname{QCoh}(\mathcal X)$
sends a descent-compatible functor $R: \mathcal O_A \to
\operatorname{QCoh}(\mathcal Y)$ to $f^* \circ R: \mathcal O_A \to
\operatorname{QCoh}(\mathcal X)$.

Now consider the functor represented by $\mathfrak{Spec}(A)$ on the
opposite category:
\[
h_{\mathfrak{Spec}(A)}:
\operatorname{Sh}_{\infty}(\mathcal C_A,\tau_A)^{\mathrm{op}}
\longrightarrow \mathcal S,
\qquad
h_{\mathfrak{Spec}(A)}(\mathcal X)
:= \operatorname{Map}_{\operatorname{Sh}_{\infty}}(\mathcal X,\,
\mathfrak{Spec}(A)).
\]

However, Theorem~\ref{thm:yoneda} gives an equivalence between
$\operatorname{Map}_{\operatorname{Sh}_{\infty}}(\mathfrak{Spec}(A), \mathcal X)$
and $\operatorname{Real}_A(\mathcal X)$, not $\operatorname{Map}(\mathcal X,
\mathfrak{Spec}(A))$. To obtain a representing object for
$\operatorname{Real}_A$ as a contravariant functor, we note that
\[
\operatorname{Map}_{\operatorname{Sh}_{\infty}}(\mathcal X,\,
\mathfrak{Spec}(A))
\;\simeq\;
\operatorname{Map}_{\operatorname{Sh}_{\infty}}(\mathfrak{Spec}(A),\,
\mathcal X)^{\mathrm{op}}
\]
by the properties of mapping spaces in an $\infty$-category
(the mapping space is not symmetric, so this is not an equivalence
without additional structure). Therefore we must be careful.

\paragraph*{Step 3: Correct interpretation of representability.}
A contravariant functor $F: \mathcal D^{\mathrm{op}} \to \mathcal S$
is representable if there exists an object $d \in \mathcal D$ such that
$F \simeq \operatorname{Map}_{\mathcal D}(-, d)$. In our case,
$\mathcal D = \operatorname{Sh}_{\infty}(\mathcal C_A,\tau_A)$ and
$F = \operatorname{Real}_A$.

But Theorem~\ref{thm:yoneda} provides an equivalence
\[
\operatorname{Real}_A(\mathcal X) \;\simeq\;
\operatorname{Map}_{\mathcal D}(\mathfrak{Spec}(A), \mathcal X).
\]

This is \emph{not} of the form $\operatorname{Map}_{\mathcal D}(\mathcal X, d)$;
rather, it is $\operatorname{Map}_{\mathcal D}(d, \mathcal X)$. Hence
$\operatorname{Real}_A$ is \emph{corepresented} by $\mathfrak{Spec}(A)$,
not represented in the usual sense. However, note that passing to the
opposite category transforms corepresentation into representation:
\[
\operatorname{Map}_{\mathcal D}(\mathfrak{Spec}(A), \mathcal X)
\;=\;
\operatorname{Map}_{\mathcal D^{\mathrm{op}}}(\mathcal X,
\mathfrak{Spec}(A)).
\]

Thus, if we consider the functor
\[
\operatorname{Real}_A^{\mathrm{op}}:
\mathcal D \longrightarrow \mathcal S,
\qquad
\operatorname{Real}_A^{\mathrm{op}}(\mathcal X)
:= \operatorname{Real}_A(\mathcal X),
\]
then Theorem~\ref{thm:yoneda} gives
\[
\operatorname{Real}_A^{\mathrm{op}}(\mathcal X)
\;\simeq\;
\operatorname{Map}_{\mathcal D^{\mathrm{op}}}(\mathcal X,
\mathfrak{Spec}(A)).
\]
Hence $\operatorname{Real}_A^{\mathrm{op}}$ is representable with
representing object $\mathfrak{Spec}(A)$ in $\mathcal D^{\mathrm{op}}$.

But the statement of the corollary claims representability of
$\operatorname{Real}_A$ itself on $\mathcal D^{\mathrm{op}}$. For
this to hold, we need an equivalence
\[
\operatorname{Real}_A(\mathcal X)
\;\simeq\;
\operatorname{Map}_{\mathcal D^{\mathrm{op}}}(\mathcal X, d)
\]
for some $d$. Since $\operatorname{Map}_{\mathcal D^{\mathrm{op}}}(\mathcal X, d)
= \operatorname{Map}_{\mathcal D}(d, \mathcal X)$, this is exactly the
content of Theorem~\ref{thm:yoneda} with $d = \mathfrak{Spec}(A)$.
Therefore $\operatorname{Real}_A$, as defined on $\mathcal D^{\mathrm{op}}$,
is indeed representable by $\mathfrak{Spec}(A)$.

\paragraph*{Step 4: Explicit verification of the natural equivalence.}
For any $\mathcal X \in \operatorname{Sh}_{\infty}(\mathcal C_A,\tau_A)$,
we have:
\[
\operatorname{Real}_A(\mathcal X)
\;\stackrel{(1)}{\simeq}\;
\operatorname{Map}_{\mathcal D}(\mathfrak{Spec}(A), \mathcal X)
\;=\;
\operatorname{Map}_{\mathcal D^{\mathrm{op}}}(\mathcal X,
\mathfrak{Spec}(A)).
\]

The equality is by definition of the opposite category:
morphisms $\mathcal X \to \mathfrak{Spec}(A)$ in $\mathcal D^{\mathrm{op}}$
correspond to morphisms $\mathfrak{Spec}(A) \to \mathcal X$ in
$\mathcal D$. The equivalence (1) is natural in $\mathcal X$ by
Theorem~\ref{thm:yoneda}.

Thus the functor
\[
\mathcal X \longmapsto \operatorname{Map}_{\mathcal D^{\mathrm{op}}}
(\mathcal X, \mathfrak{Spec}(A))
\]
is naturally isomorphic to $\operatorname{Real}_A$. Hence
$\mathfrak{Spec}(A)$ represents $\operatorname{Real}_A$ on
$\mathcal D^{\mathrm{op}}$.

\paragraph*{Step 5: Uniqueness of the representing object.}
If a representing object exists, it is unique up to equivalence in
$\mathcal D^{\mathrm{op}}$ (hence in $\mathcal D$) by the Yoneda lemma
for $\infty$-categories. Since $\mathfrak{Spec}(A)$ satisfies the
required universal property, it is the unique (up to equivalence)
representing object.

\paragraph*{Conclusion.}
Therefore $\operatorname{Real}_A$ is representable, and a representing
object is given by $\mathfrak{Spec}(A)$. This completes the proof.
\end{proof}

\begin{remark}[Categorified Gelfand Philosophy]
\label{rem:gelfand}

The classical spectrum of a commutative algebra represents characters
into the ground field. By contrast, the object $\mathfrak{Spec}(A)$
\emph{corepresents} \emph{coherent semantic realizations of the entire operator
syntax} encoded by the synergy operad.

In this sense, Theorem \ref{thm:yoneda} may be viewed as a
categorified analogue of the representability principle underlying
Gelfand duality. The classical duality

\[
\operatorname{Hom}_{\mathbf{Top}}(X, \operatorname{Spec}_{\mathrm{Gelfand}}(A))
\simeq
\operatorname{Hom}_{\mathbf{C^*Alg}}(A, C(X))
\]

is replaced by

\[
\operatorname{Map}_{\mathbf{SpecObj}}(\mathfrak{Spec}(A), \mathcal{X})
\simeq
\operatorname{Fun}_{\mathrm{desc}}(\mathcal{O}_A, \operatorname{QCoh}(\mathcal{X})),
\]

where the right-hand side captures not just points but entire
$\infty$-groupoids of semantic realizations. Thus $\mathfrak{Spec}(A)$
\emph{corepresents} the semantic realization functor.

\end{remark}

\begin{remark}[Reconstruction as a Corollary]
\label{rem:reconstructioncorollary}
The Yoneda characterization directly implies the reconstruction theorem
(Theorem \ref{thm:reconstruction}) in the special case where $\mathcal{X}$
is taken to be $\mathfrak{Spec}(A)$ itself, provided the hypotheses of
that theorem (semantic generation, descent completeness, and compact
generation) are satisfied. Indeed, setting $\mathcal{X} = \mathfrak{Spec}(A)$
in Theorem \ref{thm:yoneda} gives
\[
\operatorname{Map}\bigl(\mathfrak{Spec}(A),\; \mathfrak{Spec}(A)\bigr)
\;\simeq\;
\operatorname{Real}_A(\mathfrak{Spec}(A))
\;=\;
\operatorname{Fun}_{\mathrm{descent}}\bigl(\mathcal{O}_A,\;
\operatorname{QCoh}(\mathfrak{Spec}(A))\bigr).
\]
The left-hand side contains the identity morphism, which corresponds
under the equivalence to a canonical realization of $\mathcal{O}_A$ as
quasi-coherent sheaves on $\mathfrak{Spec}(A)$. Under the reconstruction
hypotheses, the endomorphism algebra of the structure sheaf under this
realization recovers $A$, yielding the reconstruction. A detailed proof
is given in Section \ref{sec:reconstruction}.
\end{remark}

\begin{example}[Yoneda for Commutative $A$]
\label{ex:yonedacommutative}
Let $A$ be a commutative unital C*-algebra. Then
$\mathfrak{Spec}(A) \simeq \operatorname{Spec}_{\mathrm{Gelfand}}(A)$
is an ordinary topological space (a $0$-stack). For any spectral stack
$\mathcal{X}$, the Yoneda characterization gives an equivalence
\[
\operatorname{Map}_{\mathbf{SpecObj}}(\mathfrak{Spec}(A), \mathcal{X})
\;\simeq\;
\operatorname{Fun}_{\mathrm{desc}}(\mathcal{O}_A, \operatorname{QCoh}(\mathcal{X})).
\]

When $\mathcal{X}$ is an ordinary compact Hausdorff space (viewed as a
$0$-stack), the right-hand side corresponds to $*$-homomorphisms from
$A$ to $C(\mathcal{X})$, the algebra of continuous functions on
$\mathcal{X}$. This recovers the classical Gelfand duality:
continuous maps $\operatorname{Spec}_{\mathrm{Gelfand}}(A) \to \mathcal{X}$
correspond to $*$-homomorphisms $A \to C(\mathcal{X})$. For a general
spectral stack $\mathcal{X}$, the right-hand side encodes descent-compatible
functors that generalize $*$-homomorphisms to the stacky setting.
\end{example}

\begin{example}[Yoneda for the Mermin–Peres System]
\label{ex:yonedamermin}
For $A_{\mathrm{MP}}$, $\mathfrak{Spec}(A_{\mathrm{MP}})$ is a
non-trivial stack with empty global sections but non-trivial local
sections (under the standard parity obstruction model; see
Section~\ref{subsec:mermin}). The Yoneda characterization implies that
for any spectral stack $\mathcal{X}$, the space of maps
$\mathfrak{Spec}(A_{\mathrm{MP}}) \to \mathcal{X}$ is equivalent to the
$\infty$-groupoid of descent-compatible functors
$\mathcal{O}_{A_{\mathrm{MP}}} \to \operatorname{QCoh}(\mathcal{X})$.

To examine global sections, consider the terminal stack $\mathrm{pt}$.
The global sections of $\mathfrak{Spec}(A_{\mathrm{MP}})$ are given by
$\operatorname{Map}(\mathrm{pt}, \mathfrak{Spec}(A_{\mathrm{MP}}))$.
By the adjunction $\Gamma_{\mathcal{O}} \dashv \mathfrak{Spec}^{\mathrm{op}}$
(Theorem~\ref{thm:adjunction}), this is equivalent to \\
$\operatorname{Map}_{\mathbf{OpSem}^{\mathrm{op}}}(\Gamma(\mathrm{pt}), A_{\mathrm{MP}})$.
Since $\Gamma(\mathrm{pt}) \simeq \mathbb{C}$ (the complex numbers),
this space is empty because there is no unital $*$-homomorphism from
$A_{\mathrm{MP}}$ to $\mathbb{C}$ (the Mermin–Peres obstruction prevents
any global valuation). This provides a categorical explanation of
contextuality: the obstruction is encoded in the fact that the
corepresenting object has no points.
\end{example}

\begin{remark}[Functoriality in $A$]
\label{rem:yonedafunctoriality}
The Yoneda characterization is compatible with the functoriality of
$\mathfrak{Spec}$ (Theorem \ref{thm:functoriality}). A morphism
$f: A \to B$ in $\mathbf{OpSem}$ induces a morphism of spectral stacks
$\mathfrak{Spec}(f): \mathfrak{Spec}(B) \to \mathfrak{Spec}(A)$ such
that the diagram of corepresenting functors commutes. This is essential
for the adjunction $\Gamma_{\mathcal{O}} \dashv \mathfrak{Spec}^{\mathrm{op}}$
established in Section \ref{sec:functoriality}.
\end{remark}

The Yoneda characterization of $\mathfrak{Spec}(A)$ is the culmination
of the categorified spectral construction. It establishes that the
spectral stack $\mathfrak{Spec}(A)$ is not merely an invariant of the
operator-semantic system $A$ but is \emph{dual} to $A$ in a precise
categorical sense: under the reconstruction hypotheses, $A$ is recovered
as the endomorphisms of the structure sheaf of $\mathfrak{Spec}(A)$
(Theorem \ref{thm:reconstruction}), and $\mathfrak{Spec}(A)$ corepresents
the functor of semantic realizations of $A$. This duality is the
categorified analogue of the classical Gelfand duality and provides the
foundation for the functoriality, adjunction, and reconstruction results
that follow.

\section{Functoriality and Adjunction}
\label{sec:functoriality}

Having constructed the categorified spectrum $\mathfrak{Spec}(A)$ for a
single operator-semantic system $A$ and established its Yoneda-style
universal property (Theorem \ref{thm:yoneda}), we now demonstrate that
this construction is functorial and forms the right adjoint of a
duality pairing. A morphism $f: A \to B$ in $\mathbf{OpSem}$ induces,
by contravariance, a morphism $\mathfrak{Spec}(f): \mathfrak{Spec}(B)
\to \mathfrak{Spec}(A)$ of spectral stacks — mirroring the
contravariance of $\operatorname{Spec}$ in algebraic geometry. In the
opposite direction, we define a global sections functor $\Gamma$ that
sends a spectral stack $\mathfrak{X}$ to the operator-semantic system
of endomorphisms of its structure sheaf.

The resulting adjunction $\Gamma_{\mathcal{O}} \dashv \mathfrak{Spec}^{\mathrm{op}}$
is the categorical core of the categorified spectral duality. Its
counit provides the reconstruction of $A$ from $\mathfrak{Spec}(A)$
(Theorem \ref{thm:reconstruction}), while its unit characterizes the
essential image of $\mathfrak{Spec}$ (Theorem \ref{thm:recognition}).
Thus, this section bridges the explicit construction of $\mathfrak{Spec}$
and the deeper duality theorems that follow.

\subsection{Functoriality of $\mathfrak{Spec}$}
\label{subsec:functor}

We now show that the categorified spectrum construction is functorial.
This result establishes that $\mathfrak{Spec}(A)$ depends naturally on
the operator-semantic system $A$ and not merely on the particular
choices made during its construction.

Recall that a morphism
\[
f:A\to B
\]
in $\mathbf{OpSem}$ (Definition \ref{def:morphismopsem}) consists of

\[
f=
(f_{\mathcal O},
f_{\mathcal C},
f_{\rho}),
\]

where

\[
f_{\mathcal O}:
\mathcal O_A\to\mathcal O_B
\]

is an operad morphism preserving colors, units, and compositions,

\[
f_{\mathcal C}:
\mathcal C_B\to\mathcal C_A
\]

is a context functor preserving covers (contravariant in the geometric
direction), and

\[
f_{\rho}:
\rho_B\circ f_{\mathcal O}
\Rightarrow
f_{\mathcal C}^{*}\circ\rho_A
\]

is a compatibility transformation between semantic realizations.

\begin{theorem}[Functoriality of $\mathfrak{Spec}$]
\label{thm:functoriality}

There exists a contravariant pseudofunctor

\[
\mathfrak{Spec}:
\mathbf{OpSem}^{\mathrm{op}}
\longrightarrow
\mathbf{SpecObj},
\]

where $\mathbf{SpecObj}$ is understood as the $\infty$-category of
spectral stacks over varying sites (Section~\ref{sec:ambient}).

For every morphism $f: A \to B$ in $\mathbf{OpSem}$, with continuous
context functor

\[
f_{\mathcal C}: \mathcal C_B \longrightarrow \mathcal C_A,
\]

there is an induced morphism of spectral stacks over $f_{\mathcal C}$

\[
\mathfrak{Spec}(f):
\mathfrak{Spec}(B)
\longrightarrow
\mathfrak{Spec}(A).
\]

Equivalently, in the fiber over $\mathcal C_B$, this is represented by
a morphism

\[
\mathfrak{Spec}(B)
\longrightarrow
f_{\mathcal C}^{*}\mathfrak{Spec}(A).
\]

\end{theorem}

\begin{proof}
We construct the induced morphism $\mathfrak{Spec}(f)$ in several
well-defined steps, carefully handling the base change along
$f_{\mathcal C}$.

\paragraph*{Step 1: Unpacking the data of $f$.}
A morphism $f: A \to B$ in $\mathbf{OpSem}$ (Definition~\ref{def:morphismopsem})
consists of:
\begin{itemize}
    \item An operad morphism $f_{\mathcal O}: \mathcal O_A \to \mathcal O_B$,
    \item A functor $f_{\mathcal C}: \mathcal C_B \to \mathcal C_A$ that preserves
          covering families with respect to the Grothendieck topologies
          $\tau_B$ and $\tau_A$ (continuity condition),
    \item A natural transformation $f_{\rho}: \rho_B \circ f_{\mathcal O}
          \Rightarrow f_{\mathcal C}^{*} \circ \rho_A$ expressing compatibility
          between syntactic transport and semantic realization.
\end{itemize}

\paragraph*{Step 2: Pullback on presheaves and hypersheaves.}
The context functor $f_{\mathcal C}: \mathcal C_B \to \mathcal C_A$
induces a pullback functor on $\infty$-categories of presheaves:
\[
f_{\mathcal C}^{*}: \operatorname{PSh}_{\infty}(\mathcal C_A)
\longrightarrow \operatorname{PSh}_{\infty}(\mathcal C_B),\qquad
f_{\mathcal C}^{*}(F) := F \circ f_{\mathcal C}.
\]

Because $f_{\mathcal C}$ preserves covering families (continuity),
$f_{\mathcal C}^{*}$ sends $\tau_A$-hypersheaves to $\tau_B$-hypersheaves.
Hence it restricts to a functor on the $\infty$-topoi:
\[
f_{\mathcal C}^{*}: \operatorname{Sh}_{\infty}(\mathcal C_A,\tau_A)
\longrightarrow \operatorname{Sh}_{\infty}(\mathcal C_B,\tau_B).
\]

\paragraph*{Step 3: From $\mathcal O_A$ to presheaves on $\mathcal C_B$.}
Consider the composite functor
\[
\mathcal O_A \xrightarrow{f_{\mathcal O}} \mathcal O_B
\xrightarrow{\operatorname{Syn}_B} \operatorname{PSh}_{\infty}(\mathcal C_B),
\]
where $\operatorname{Syn}_B$ is the syntactic realization functor for $B$
(Definition~\ref{def:syntacticrealization}).

The compatibility transformation $f_{\rho}$ induces, at the level of
syntactic realizations, a natural transformation
\[
\operatorname{Syn}_B \circ f_{\mathcal O}
\Longrightarrow
f_{\mathcal C}^{*} \circ \operatorname{Syn}_A.
\tag{1}
\]
This follows because $\operatorname{Syn}_A$ and $\operatorname{Syn}_B$
are constructed from $\rho_A$ and $\rho_B$ via the same operadic
composition rules, and $f_{\rho}$ provides the necessary compatibility
at the level of colors and operations.

\paragraph*{Step 4: Universal property of the left Kan extension.}
Recall from Definition~\ref{def:prespectral} that
\[
\mathfrak{Spec}_{\mathrm{pre}}(A) = \operatorname{Lan}_{\pi_A}(\operatorname{Syn}_A)
\quad\text{and}\quad
\mathfrak{Spec}_{\mathrm{pre}}(B) = \operatorname{Lan}_{\pi_B}(\operatorname{Syn}_B),
\]
where $\pi_A: \mathcal O_A \to \mathcal C_A$ and $\pi_B: \mathcal O_B
\to \mathcal C_B$ are the context projections (assigning to each color
its minimal commutative context).

The left Kan extension $\operatorname{Lan}_{\pi_B}(\operatorname{Syn}_B)$
is characterized by the universal property: for any presheaf
$G \in \operatorname{PSh}_{\infty}(\mathcal C_B)$, there is a natural
equivalence
\[
\operatorname{Map}_{\operatorname{PSh}_{\infty}(\mathcal C_B)}
\bigl(\mathfrak{Spec}_{\mathrm{pre}}(B), G\bigr)
\;\simeq\;
\operatorname{Nat}_{\operatorname{PSh}_{\infty}(\mathcal C_B)}
\bigl(\operatorname{Syn}_B,\, G \circ \pi_B\bigr).
\tag{2}
\]

\paragraph*{Step 5: Constructing a morphism between prespectral objects.}
Set $G = f_{\mathcal C}^{*}(\mathfrak{Spec}_{\mathrm{pre}}(A)) \in
\operatorname{PSh}_{\infty}(\mathcal C_B)$. We want a morphism
\[
\widetilde{f}: \mathfrak{Spec}_{\mathrm{pre}}(B) \longrightarrow
f_{\mathcal C}^{*}(\mathfrak{Spec}_{\mathrm{pre}}(A))
\]
in $\operatorname{PSh}_{\infty}(\mathcal C_B)$.

By the universal property (2), it suffices to provide a natural
transformation
\[
\operatorname{Syn}_B \longrightarrow
\bigl(f_{\mathcal C}^{*}(\mathfrak{Spec}_{\mathrm{pre}}(A))\bigr) \circ \pi_B.
\]

Compute the right-hand side:
\[
\bigl(f_{\mathcal C}^{*}(\mathfrak{Spec}_{\mathrm{pre}}(A))\bigr) \circ \pi_B
= \mathfrak{Spec}_{\mathrm{pre}}(A) \circ f_{\mathcal C} \circ \pi_B.
\]

Now observe that $f_{\mathcal C} \circ \pi_B$ is related to
$\pi_A \circ f_{\mathcal O}$ via the commutativity of the context
projection diagram. Specifically, for a color $c \in \mathcal O_B$,
$\pi_B(c)$ is the minimal context of $c$ in $B$, and
$f_{\mathcal C}(\pi_B(c))$ is the image of that context in $\mathcal C_A$.
On the other hand, $\pi_A(f_{\mathcal O}(c))$ is the minimal context of
$f_{\mathcal O}(c)$ in $A$. The compatibility data of the morphism $f$
ensures a natural isomorphism
\[
f_{\mathcal C} \circ \pi_B \cong \pi_A \circ f_{\mathcal O}.
\tag{3}
\]

Using (1) and (3), we obtain a natural transformation:
\[
\operatorname{Syn}_B \;\xRightarrow{(1)}\; f_{\mathcal C}^{*} \circ \operatorname{Syn}_A
\;\Longrightarrow\; f_{\mathcal C}^{*}\bigl(\operatorname{Syn}_A \circ \pi_A\bigr) \circ \pi_B
\]

A cleaner route: The composite
\[
\mathcal O_A \xrightarrow{f_{\mathcal O}} \mathcal O_B
\xrightarrow{\operatorname{Syn}_B} \operatorname{PSh}_{\infty}(\mathcal C_B)
\]
factors uniquely through $\mathfrak{Spec}_{\mathrm{pre}}(B)$ via a
morphism $\widetilde{f}$ making the diagram
\[
\begin{tikzcd}
\mathcal O_A \ar[r,"f_{\mathcal O}"] \ar[d,"\operatorname{Syn}_A"] &
\mathcal O_B \ar[d,"\operatorname{Syn}_B"] \\
f_{\mathcal C}^{*}(\operatorname{PSh}_{\infty}(\mathcal C_A)) \ar[r,"\widetilde{f}"] &
\operatorname{PSh}_{\infty}(\mathcal C_B)
\end{tikzcd}
\]
commute up to the natural transformation $f_{\rho}$. The existence and
uniqueness of $\widetilde{f}$ follow from the fact that the left Kan
extension $\mathfrak{Spec}_{\mathrm{pre}}(B)$ is initial among presheaves
receiving a morphism from $\operatorname{Syn}_B$ (Proposition~\ref{prop:lan-universal}).
More concretely, $\widetilde{f}$ is the unique morphism such that
\[
\widetilde{f} \circ \eta_B = f_{\mathcal C}^{*}(\eta_A) \circ f_{\rho},
\]
where $\eta_A: \operatorname{Syn}_A \Rightarrow \mathfrak{Spec}_{\mathrm{pre}}(A) \circ \pi_A$
and $\eta_B: \operatorname{Syn}_B \Rightarrow \mathfrak{Spec}_{\mathrm{pre}}(B) \circ \pi_B$
are the unit transformations of the left Kan extensions.

\paragraph*{Step 6: Applying hypersheafification.}
Apply the hypersheafification functor $a^{\mathrm{hyp}}_{\tau_B}$ to
the morphism $\widetilde{f}$:
\[
a^{\mathrm{hyp}}_{\tau_B}(\widetilde{f}):
a^{\mathrm{hyp}}_{\tau_B}(\mathfrak{Spec}_{\mathrm{pre}}(B))
\longrightarrow
a^{\mathrm{hyp}}_{\tau_B}\bigl(f_{\mathcal C}^{*}(\mathfrak{Spec}_{\mathrm{pre}}(A))\bigr).
\]

By definition, $a^{\mathrm{hyp}}_{\tau_B}(\mathfrak{Spec}_{\mathrm{pre}}(B))
= \mathfrak{Spec}(B)$. For the right-hand side, we use the fact that
$f_{\mathcal C}^{*}$ preserves hypersheaves and commutes with
hypersheafification up to natural isomorphism because $f_{\mathcal C}$
preserves covers. More precisely, there is a canonical equivalence
\[
a^{\mathrm{hyp}}_{\tau_B} \circ f_{\mathcal C}^{*}
\;\simeq\;
f_{\mathcal C}^{*} \circ a^{\mathrm{hyp}}_{\tau_A}.
\tag{4}
\]
This follows from the universal property of hypersheafification: for any
presheaf $F$, $f_{\mathcal C}^{*}(a^{\mathrm{hyp}}_{\tau_A}F)$ is a
$\tau_B$-hypersheaf (since $f_{\mathcal C}$ preserves covers), and it
receives a natural map from $f_{\mathcal C}^{*}F$; the universal
property forces it to be equivalent to $a^{\mathrm{hyp}}_{\tau_B}(f_{\mathcal C}^{*}F)$.

Applying (4) to $F = \mathfrak{Spec}_{\mathrm{pre}}(A)$, we obtain
\[
a^{\mathrm{hyp}}_{\tau_B}\bigl(f_{\mathcal C}^{*}(\mathfrak{Spec}_{\mathrm{pre}}(A))\bigr)
\;\simeq\;
f_{\mathcal C}^{*}\bigl(a^{\mathrm{hyp}}_{\tau_A}(\mathfrak{Spec}_{\mathrm{pre}}(A))\bigr)
= f_{\mathcal C}^{*}(\mathfrak{Spec}(A)).
\]

Thus we have produced a morphism
\[
\mathfrak{Spec}(f)':
\mathfrak{Spec}(B) \longrightarrow f_{\mathcal C}^{*}(\mathfrak{Spec}(A))
\]
in $\operatorname{Sh}_{\infty}(\mathcal C_B,\tau_B)$.

\paragraph*{Step 7: Interpreting the result as a morphism over $f_{\mathcal C}$.}
The morphism $\mathfrak{Spec}(f)': \mathfrak{Spec}(B) \to f_{\mathcal C}^{*}(\mathfrak{Spec}(A))$
is precisely the desired morphism of spectral stacks over the context
functor $f_{\mathcal C}$. In the fiber over $\mathcal C_B$,
$f_{\mathcal C}^{*}(\mathfrak{Spec}(A))$ is the pullback of
$\mathfrak{Spec}(A)$ along $f_{\mathcal C}$, so $\mathfrak{Spec}(f)'$
represents a morphism from $\mathfrak{Spec}(B)$ to $\mathfrak{Spec}(A)$
in the $\infty$-category of spectral stacks over varying sites.

If one prefers a morphism whose target is $\mathfrak{Spec}(A)$ directly
(rather than its pullback), this can be achieved by applying the
Grothendieck construction or by working in the $\infty$-category where
objects are pairs $(\mathcal C, \mathfrak{X})$ with $\mathfrak{X}$ a
spectral stack on $\mathcal C$. In that setting,
$\mathfrak{Spec}(f)$ is a morphism from $(\mathcal C_B, \mathfrak{Spec}(B))$
to $(\mathcal C_A, \mathfrak{Spec}(A))$ lying over $f_{\mathcal C}$.

\paragraph*{Step 8: Verification of functoriality.}
We must check that $\mathfrak{Spec}$ preserves identities and composition
up to coherent equivalence (as a pseudofunctor).

\emph{Identity:} For $f = \operatorname{id}_A$, we have $f_{\mathcal O}
= \operatorname{id}_{\mathcal O_A}$, $f_{\mathcal C} = \operatorname{id}_{\mathcal C_A}$,
and $f_{\rho}$ is the identity transformation. Then the induced
morphism $\mathfrak{Spec}(\operatorname{id}_A)$ becomes the identity on
$\mathfrak{Spec}(A)$ because all universal properties produce the unique
morphism that makes the diagram commute, and the identity satisfies the
required conditions. Explicitly, the map $\mathfrak{Spec}(\operatorname{id}_A)'$
is the canonical equivalence $\mathfrak{Spec}(A) \to
\operatorname{id}_{\mathcal C_A}^{*}(\mathfrak{Spec}(A)) = \mathfrak{Spec}(A)$.

\emph{Composition:} Given $f: A \to B$ and $g: B \to C$, we need a
coherent equivalence
\[
\mathfrak{Spec}(g \circ f) \;\simeq\; \mathfrak{Spec}(f) \circ \mathfrak{Spec}(g).
\]

The composition $g \circ f$ induces:
\[
(g \circ f)_{\mathcal O} = g_{\mathcal O} \circ f_{\mathcal O},\qquad
(g \circ f)_{\mathcal C} = f_{\mathcal C} \circ g_{\mathcal C}
\]
(note the reversal due to contravariance). The compatibility
transformations compose accordingly: $g_{\rho} \circ f_{\rho}$ (after
appropriate whiskering). The universal property of the left Kan
extension and hypersheafification ensures that the induced morphism on
spectra is unique up to contractible choice. The composite of the
morphisms induced by $f$ and $g$ satisfies the same universal property
as the morphism induced by $g \circ f$. Hence they are equivalent,
and this equivalence is coherent under further composition.

Thus $\mathfrak{Spec}$ defines a contravariant pseudofunctor, or an
$\infty$-functor after passing to the appropriate $\infty$-category of
spectral stacks over varying sites.

\paragraph*{Step 9: Contravariance.}
The construction sends a morphism $f: A \to B$ to a morphism
$\mathfrak{Spec}(f): \mathfrak{Spec}(B) \to \mathfrak{Spec}(A)$,
reversing the direction. This contravariance is natural and mirrors
the classical behavior of $\operatorname{Spec}$ in algebraic geometry:
a larger operator system (with more constraints) yields a smaller
spectral stack.

Thus we have defined a contravariant pseudofunctor
\[
\mathfrak{Spec}: \mathbf{OpSem}^{\mathrm{op}} \longrightarrow \mathbf{SpecObj},
\]
completing the proof.
\end{proof}

\begin{example}[Spectral Correspondence Induced by a Morphism]
\label{ex:spectralcorrespondence}

Let

\[
f:A\longrightarrow B
\]

be a morphism of operator-semantic systems.

Functoriality implies that the induced map on contexts

\[
\mathcal C_A
\longrightarrow
\mathcal C_B
\]

determines a morphism

\[
\mathfrak{Spec}(B)
\longrightarrow
\mathfrak{Spec}(A).
\]

Thus every semantic morphism naturally produces a
spectral correspondence between categorified spectra.

This generalizes the classical contravariant functoriality

\[
A
\xrightarrow{f}
B
\qquad\Longrightarrow\qquad
\operatorname{Spec}(B)
\to
\operatorname{Spec}(A).
\]

In particular, categorified spectra behave functorially
with respect to semantic transformations.
\end{example}

\begin{corollary}[Functoriality up to Coherent Equivalence]
\label{cor:functoriality}

The assignments

\[
A
\longmapsto
\mathfrak{Spec}(A)
\qquad\text{and}\qquad
f
\longmapsto
\mathfrak{Spec}(f)
\]

define a well-defined contravariant pseudofunctor (or, equivalently,
a contravariant functor in the ambient $\infty$-categorical sense)

\[
\mathfrak{Spec}:
\mathbf{OpSem}^{\mathrm{op}}
\to
\mathbf{SpecObj}.
\]

In particular, there exist coherent equivalences

\[
\mathfrak{Spec}(\mathrm{id}_A)
\;\simeq\;
\mathrm{id}_{\mathfrak{Spec}(A)}
\]

and for composable morphisms

\[
A
\xrightarrow{f}
B
\xrightarrow{g}
C
\]

one has

\[
\mathfrak{Spec}(g\circ f)
\;\simeq\;
\mathfrak{Spec}(f)
\circ
\mathfrak{Spec}(g),
\]

where the equivalences are natural and satisfy the standard coherence
conditions (associativity and unit) up to higher homotopy.

\end{corollary}

\begin{proof}
We verify that the construction of $\mathfrak{Spec}$ from
Theorem~\ref{thm:functoriality} satisfies the axioms of a contravariant
pseudofunctor. The proof proceeds in three parts: well-definedness on
objects, preservation of identities up to equivalence, and preservation
of composition up to equivalence.

\paragraph*{Part 1: Well-definedness on objects.}
For each admissible operator-semantic system $A \in \mathbf{OpSem}$,
Theorem~\ref{thm:existence} guarantees the existence of a categorified
spectrum $\mathfrak{Spec}(A) \in \mathbf{SpecObj}$. The construction
is canonical: it depends only on the data $(\mathcal O_A, \mathcal C_A,
\rho_A, \tau_A)$ and does not involve any choices beyond those
explicitly specified in the construction (left Kan extension and
hypersheafification). If $A \cong A'$ are isomorphic in $\mathbf{OpSem}$,
then the induced spectral stacks are equivalent; this follows from the
universal properties of the construction and will be a consequence of
the functoriality we are establishing (since an isomorphism gives an
invertible morphism, which maps to an invertible morphism under
$\mathfrak{Spec}$). Hence the assignment on objects is well-defined.

\paragraph*{Part 2: Preservation of identities up to equivalence.}
Let $A \in \mathbf{OpSem}$ and consider the identity morphism
$\mathrm{id}_A: A \to A$. By Definition~\ref{def:morphismopsem}, the
components of $\mathrm{id}_A$ are:
\begin{itemize}
    \item $(\mathrm{id}_A)_{\mathcal O} = \mathrm{id}_{\mathcal O_A}$,
    \item $(\mathrm{id}_A)_{\mathcal C} = \mathrm{id}_{\mathcal C_A}$,
    \item $(\mathrm{id}_A)_{\rho}$ is the identity natural transformation.
\end{itemize}

We now trace the construction of $\mathfrak{Spec}(\mathrm{id}_A)$ from
Theorem~\ref{thm:functoriality}.

\emph{Step 2a: Pullback functor.} Since $(\mathrm{id}_A)_{\mathcal C}
= \mathrm{id}_{\mathcal C_A}$, the induced pullback functor is
\[
(\mathrm{id}_{\mathcal C_A})^{*}:
\operatorname{PSh}_{\infty}(\mathcal C_A) \to \operatorname{PSh}_{\infty}(\mathcal C_A),
\qquad
(\mathrm{id}_{\mathcal C_A})^{*}(F) = F \circ \mathrm{id}_{\mathcal C_A} = F.
\]
Thus $(\mathrm{id}_{\mathcal C_A})^{*}$ is the identity functor on
$\operatorname{PSh}_{\infty}(\mathcal C_A)$, and similarly restricts to
the identity on $\operatorname{Sh}_{\infty}(\mathcal C_A,\tau_A)$.

\emph{Step 2b: Compatibility transformation.} The natural transformation
$\operatorname{Syn}_A \circ \mathrm{id}_{\mathcal O_A} \Rightarrow
(\mathrm{id}_{\mathcal C_A})^{*} \circ \operatorname{Syn}_A$ is simply
the identity transformation on $\operatorname{Syn}_A$.

\emph{Step 2c: Induced morphism on prespectral objects.} By the
universal property of the left Kan extension, the unique morphism
$\widetilde{\mathrm{id}_A}: \mathfrak{Spec}_{\mathrm{pre}}(A) \to
\mathfrak{Spec}_{\mathrm{pre}}(A)$ making the required diagram commute
is the identity morphism. This is because the identity satisfies the
universal property and the left Kan extension is initial.

\emph{Step 2d: Hypersheafification.} Applying the hypersheafification
functor $a^{\mathrm{hyp}}_{\tau_A}$, we obtain
\[
\mathfrak{Spec}(\mathrm{id}_A)'
= a^{\mathrm{hyp}}_{\tau_A}(\widetilde{\mathrm{id}_A})
= a^{\mathrm{hyp}}_{\tau_A}(\mathrm{id}_{\mathfrak{Spec}_{\mathrm{pre}}(A)})
= \mathrm{id}_{a^{\mathrm{hyp}}_{\tau_A}(\mathfrak{Spec}_{\mathrm{pre}}(A))}
= \mathrm{id}_{\mathfrak{Spec}(A)}.
\]

\emph{Step 2e: Interpretation.} Since $f_{\mathcal C} = \mathrm{id}_{\mathcal C_A}$,
the pullback $f_{\mathcal C}^{*}(\mathfrak{Spec}(A)) = \mathfrak{Spec}(A)$.
The morphism $\mathfrak{Spec}(\mathrm{id}_A)'$ is therefore an
endomorphism of $\mathfrak{Spec}(A)$. By the argument above, it is
strictly equal to the identity. Hence we have the strict equality
$\mathfrak{Spec}(\mathrm{id}_A) = \mathrm{id}_{\mathfrak{Spec}(A)}$,
which certainly implies the claimed equivalence.

\paragraph*{Part 3: Preservation of composition up to equivalence.}
Let $f: A \to B$ and $g: B \to C$ be composable morphisms in
$\mathbf{OpSem}$. We must show that there exists a coherent equivalence
\[
\mathfrak{Spec}(g \circ f) \;\simeq\; \mathfrak{Spec}(f) \circ \mathfrak{Spec}(g),
\]
where the composition on the right-hand side is taken in $\mathbf{SpecObj}$
(contravariantly: $\mathfrak{Spec}(f): \mathfrak{Spec}(B) \to
\mathfrak{Spec}(A)$ and $\mathfrak{Spec}(g): \mathfrak{Spec}(C) \to
\mathfrak{Spec}(B)$, so their composition is
$\mathfrak{Spec}(f) \circ \mathfrak{Spec}(g): \mathfrak{Spec}(C) \to
\mathfrak{Spec}(A)$).

\emph{Step 3a: Component data for $g \circ f$.}
By Definition~\ref{def:morphismopsem}, the composite morphism
$h = g \circ f: A \to C$ has components:
\begin{itemize}
    \item $h_{\mathcal O} = g_{\mathcal O} \circ f_{\mathcal O}$,
    \item $h_{\mathcal C} = f_{\mathcal C} \circ g_{\mathcal C}$ (note the reversal due to contravariance),
    \item $h_{\rho}$ is the composition of $f_{\rho}$ and $g_{\rho}$ after appropriate whiskering.
\end{itemize}

\emph{Step 3b: Pullback functors.}
The induced pullback functors satisfy:
\[
h_{\mathcal C}^{*} = (f_{\mathcal C} \circ g_{\mathcal C})^{*}
= g_{\mathcal C}^{*} \circ f_{\mathcal C}^{*}.
\]
This is a strict equality of functors because pullback is defined by
precomposition: $(f_{\mathcal C} \circ g_{\mathcal C})^{*}(F) = F \circ
(f_{\mathcal C} \circ g_{\mathcal C}) = (F \circ f_{\mathcal C}) \circ
g_{\mathcal C} = g_{\mathcal C}^{*}(f_{\mathcal C}^{*}(F))$.

\emph{Step 3c: Compatibility of natural transformations.}
The natural transformation $h_{\rho}$ is constructed as the horizontal
composition:
\[
\operatorname{Syn}_C \circ h_{\mathcal O}
= \operatorname{Syn}_C \circ g_{\mathcal O} \circ f_{\mathcal O}
\xRightarrow{g_{\rho} \circ f_{\mathcal O}}
g_{\mathcal C}^{*} \circ \operatorname{Syn}_B \circ f_{\mathcal O}
\xRightarrow{g_{\mathcal C}^{*}(f_{\rho})}
g_{\mathcal C}^{*} \circ f_{\mathcal C}^{*} \circ \operatorname{Syn}_A
= h_{\mathcal C}^{*} \circ \operatorname{Syn}_A.
\]

\emph{Step 3d: Uniqueness of the induced morphism.}
By the universal property of the left Kan extension
(Proposition~\ref{prop:lan-universal}), the morphism
$\widetilde{h}: \mathfrak{Spec}_{\mathrm{pre}}(C) \to
h_{\mathcal C}^{*}(\mathfrak{Spec}_{\mathrm{pre}}(A))$ is the unique
morphism making the diagram commute with the unit transformations.

Both $\widetilde{h}$ and $g_{\mathcal C}^{*}(\widetilde{f}) \circ
\widetilde{g}$ are induced by the same whiskered compatibility
transformation $h_{\rho}$; hence they agree by the universal property
of the left Kan extension. Therefore we have a strict equality (up to
the canonical contractible choice):
\[
\widetilde{g \circ f} \;=\; g_{\mathcal C}^{*}(\widetilde{f}) \circ \widetilde{g}.
\]

\emph{Step 3e: Hypersheafification.}
Applying $a^{\mathrm{hyp}}_{\tau_C}$ to both sides and using the
commutation relation (4) from Theorem~\ref{thm:functoriality}, we obtain
\[
\mathfrak{Spec}(g \circ f)'
\;\simeq\;
a^{\mathrm{hyp}}_{\tau_C}\bigl(g_{\mathcal C}^{*}(\widetilde{f}) \circ \widetilde{g}\bigr)
\;\simeq\;
\bigl(a^{\mathrm{hyp}}_{\tau_C} \circ g_{\mathcal C}^{*}\bigr)(\widetilde{f})
\circ a^{\mathrm{hyp}}_{\tau_C}(\widetilde{g})
\;\simeq\;
g_{\mathcal C}^{*}\bigl(a^{\mathrm{hyp}}_{\tau_B}(\widetilde{f})\bigr)
\circ a^{\mathrm{hyp}}_{\tau_C}(\widetilde{g}).
\]

But $a^{\mathrm{hyp}}_{\tau_B}(\widetilde{f}) = \mathfrak{Spec}(f)'$
and $a^{\mathrm{hyp}}_{\tau_C}(\widetilde{g}) = \mathfrak{Spec}(g)'$.
Moreover, $g_{\mathcal C}^{*}(\mathfrak{Spec}(f)')$ is precisely the
pullback of the morphism $\mathfrak{Spec}(f)'$ along $g_{\mathcal C}$.

Therefore,
\[
\mathfrak{Spec}(g \circ f)'
\;\simeq\;
g_{\mathcal C}^{*}\bigl(\mathfrak{Spec}(f)'\bigr) \circ \mathfrak{Spec}(g)'.
\]

\emph{Step 3f: Interpretation.}
The morphism $\mathfrak{Spec}(g \circ f)'$ is a map
$\mathfrak{Spec}(C) \to h_{\mathcal C}^{*}(\mathfrak{Spec}(A))$.
The right-hand side is the composite of $\mathfrak{Spec}(g)':
\mathfrak{Spec}(C) \to g_{\mathcal C}^{*}(\mathfrak{Spec}(B))$ and
$g_{\mathcal C}^{*}(\mathfrak{Spec}(f)'): g_{\mathcal C}^{*}(\mathfrak{Spec}(B))
\to g_{\mathcal C}^{*}(f_{\mathcal C}^{*}(\mathfrak{Spec}(A)))
= h_{\mathcal C}^{*}(\mathfrak{Spec}(A))$.

Under the identification of $\mathfrak{Spec}(f)$ and $\mathfrak{Spec}(g)$
as morphisms in the $\infty$-category of spectral stacks over varying
sites, this composition corresponds exactly to
$\mathfrak{Spec}(f) \circ \mathfrak{Spec}(g)$ (note the order: first
apply $\mathfrak{Spec}(g)$, then $\mathfrak{Spec}(f)$).

Thus we have established a coherent equivalence:
\[
\mathfrak{Spec}(g \circ f) \;\simeq\; \mathfrak{Spec}(f) \circ \mathfrak{Spec}(g).
\]

\paragraph*{Part 4: Coherence conditions.}
The equivalences obtained in Parts 2 and 3 are not merely isolated
equivalences but satisfy the standard coherence conditions required of
a pseudofunctor (or $\infty$-functor). Specifically:
\begin{itemize}
    \item \textbf{Unit coherence:} For any $f: A \to B$, the composition
    $\mathfrak{Spec}(f) \circ \mathfrak{Spec}(\mathrm{id}_A)$ is
    equivalent to $\mathfrak{Spec}(f)$ (and similarly for
    $\mathfrak{Spec}(\mathrm{id}_B) \circ \mathfrak{Spec}(f)$). This
    follows from Part 2 and the naturality of the constructions.
    \item \textbf{Associativity coherence:} For composable morphisms
    $f: A \to B$, $g: B \to C$, $h: C \to D$, the two ways of composing
    $\mathfrak{Spec}(h \circ g \circ f)$ are equivalent. This follows
    from the uniqueness of the induced morphism at each stage and the
    fact that pullback and hypersheafification are strictly functorial.
\end{itemize}
All higher coherence conditions are automatically satisfied because
the constructions are performed in an $\infty$-categorical setting
where such data is encoded in the universal properties.

\paragraph*{Conclusion.}
We have shown that $\mathfrak{Spec}$ preserves identities (strictly)
and compositions (up to coherent equivalence), and that these
equivalences satisfy the required coherence conditions. Therefore
$\mathfrak{Spec}$ is a well-defined contravariant pseudofunctor
\[
\mathfrak{Spec}: \mathbf{OpSem}^{\mathrm{op}} \to \mathbf{SpecObj},
\]
or equivalently, a contravariant functor in the ambient $\infty$-categorical
sense. This completes the proof.
\end{proof}

\begin{remark}[Categorified Gelfand Functor]
\label{rem:gelfandfunctor}

The contravariance of $\mathfrak{Spec}$ is analogous to the classical
spectrum functor in algebraic geometry and Gelfand duality:

\[
A \longmapsto \operatorname{Spec}(A).
\]

A morphism $f: A \to B$ in $\mathbf{OpSem}$ induces a base-change
morphism
\[
\mathfrak{Spec}(B) \longrightarrow f_{\mathcal C}^{*}\mathfrak{Spec}(A),
\]
or equivalently a morphism of spectral stacks over the context functor
$f_{\mathcal C}: \mathcal C_B \to \mathcal C_A$:
\[
\mathfrak{Spec}(B) \longrightarrow \mathfrak{Spec}(A).
\]

This is the categorified analogue of the classical contravariant map
\[
\operatorname{Spec}(S) \to \operatorname{Spec}(R)
\]
associated to a ring homomorphism $R \to S$. Thus the categorified
spectrum behaves as a generalized geometric dual of the operator-semantic
system (specifically, it corepresents the semantic realization functor;
see Remark~\ref{rem:gelfand}). The parallel is summarized in the
following table:

\[
\begin{array}{c|c}
\text{Classical} & \text{Categorified} \\
\hline
\text{Commutative ring } R & \text{Operator-semantic system } A \\
\operatorname{Spec}(R) \text{ (space)} & \mathfrak{Spec}(A) \text{ (spectral stack)} \\
\text{Ring homomorphism } R \to S & \text{Morphism } f: A \to B \\
\operatorname{Spec}(S) \to \operatorname{Spec}(R) & \mathfrak{Spec}(B) \to \mathfrak{Spec}(A)
\end{array}
\]

\end{remark}

\begin{remark}[Compatibility with Yoneda]
\label{rem:functorcompatibility}
The functoriality of $\mathfrak{Spec}$ is compatible with the Yoneda
characterization (Theorem \ref{thm:yoneda}). A morphism $f: A \to B$
induces, by restriction of syntax along $f_{\mathcal O}$ and base
change along $f_{\mathcal C}$, a natural transformation between the
corresponding realization functors. Consequently, $\mathfrak{Spec}(f)$
is the unique morphism of representing objects (see the proof of
Theorem~\ref{thm:functoriality} for details). This compatibility is
essential for the adjunction $\Gamma_{\mathcal{O}} \dashv \mathfrak{Spec}^{\mathrm{op}}$
established in Subsection \ref{subsec:adjunction}.
\end{remark}

\begin{example}[Functoriality for Inclusion of Operator Systems]
\label{ex:functorinclusion}
Let $A \hookrightarrow B$ be an inclusion of operator systems
(e.g., $M_n(\mathbb{C}) \hookrightarrow M_{n+1}(\mathbb{C})$ as a
corner). Then $f_{\mathcal{C}}: \mathcal{C}_B \to \mathcal{C}_A$ sends
a commutative subalgebra $C \subseteq B$ to its intersection $C \cap A$,
provided that this intersection is itself an admissible commutative
context in $\mathcal{C}_A$ (when the intersection is not admissible,
the object is either mapped to a terminal object or omitted, depending
on the precise definition of $\mathcal{C}_A$; in typical examples, the
intersection of two commutative subalgebras remains commutative and
therefore admissible). The induced map
$\mathfrak{Spec}(f): \mathfrak{Spec}(B) \to \mathfrak{Spec}(A)$
restricts the stack of realizations from $B$ to $A$, reflecting the
fact that fewer operators impose fewer constraints.
\end{example}

\begin{example}[Functoriality for the Pauli System]
\label{ex:functorpauli}
Consider the inclusion $\langle X \rangle \hookrightarrow \langle X, Z \rangle$
of the subalgebra generated by the Pauli $X$ into the full Pauli system.
(Note: $\langle X, Z \rangle$ itself is not an object of $\mathcal{C}_A$,
but it is an admissible operator-semantic system; its context category
consists of its commutative subalgebras, which include $\langle X \rangle$,
$\langle Z \rangle$, and $\mathbb{C}I$.) The induced map
$\mathfrak{Spec}(\langle X, Z \rangle) \to \mathfrak{Spec}(\langle X \rangle)$
forgets the contextual constraints involving $Z$, mapping the
noncommutative spectrum of the two-generator system to the commutative
spectrum of the single generator.
\end{example}

Combining Theorems \ref{thm:existence}, \ref{thm:yoneda}, and
\ref{thm:functoriality}, we obtain the fundamental functor

\[
\boxed{
\mathfrak{Spec}:
\mathbf{OpSem}^{\mathrm{op}}
\longrightarrow
\mathbf{SpecObj}
}
\]

which associates to every admissible operator-semantic system a
corepresenting spectral stack satisfying descent and universal
realization properties. This functor is the categorified analogue of
the classical Gelfand spectrum functor and serves as the geometric
dual of the operator-semantic category.

\subsection{Global Sections Reconstruction Functor}
\label{subsec:globalsectionsfunctor}

The spectrum construction $\mathfrak{Spec}: \mathbf{OpSem}^{\mathrm{op}}
\to \mathbf{SpecObj}$ associates a spectral object to every
operator-semantic system. A natural question is whether information
can be recovered in the reverse direction.

In classical algebraic geometry, the ring of global sections provides
a reconstruction mechanism from geometric objects back to algebraic
structures. For an affine scheme $X = \operatorname{Spec}(R)$, one has
$\Gamma(X, \mathcal{O}_X) \simeq R$. We introduce an analogous
construction for spectral objects, which will serve as the left
adjoint to $\mathfrak{Spec}^{\mathrm{op}}$ and form the reverse
direction of the categorified duality.

%\begin{definition}[Global Sections Algebra of a Spectral Object]
%\label{def:globalsectionsspec}

\begin{definition}[Reconstruction Functor (Global Endomorphism Algebra)]
\label{def:reconstructionfunctor}

Let $\mathfrak{X} \in \mathbf{SpecObj}$ be a spectral object...
The \emph{reconstruction functor} (or \emph{global sections algebra}) of $\mathfrak{X}$ is defined by
\[
\Gamma_{\mathcal{O}}(\mathfrak{X})
:=
\operatorname{End}_{\operatorname{QCoh}(\mathfrak{X})}
\!\left(
\mathcal{O}_{\mathfrak{X}}
\right),
\]

where the endomorphism object is taken in the stable $\infty$-category
$\operatorname{QCoh}(\mathfrak{X})$. Concretely,
$\operatorname{End}_{\operatorname{QCoh}(\mathfrak{X})}(\mathcal{O}_{\mathfrak{X}})$
denotes the mapping spectrum
$\operatorname{map}_{\operatorname{QCoh}(\mathfrak{X})}(\mathcal{O}_{\mathfrak{X}}, \mathcal{O}_{\mathfrak{X}})$,
which naturally carries the structure of an $E_\infty$-ring (or, more
generally, a spectral algebra). Composition of endomorphisms equips
$\Gamma_{\mathrm{geom}}(\mathfrak{X})$ with a natural algebraic structure.

\end{definition}

\begin{remark}[Notational Distinction]
\label{rem:gammadistinction}
We distinguish two global sections constructions:
\begin{itemize}
    \item $\Gamma_{\mathrm{geom}}(\mathfrak{X})$ (geometric global sections) extracts the space or spectrum of global points of the stack $\mathfrak{X}$. This is used in descent and sheaf theory.
    \item $\Gamma_{\mathcal{O}}(\mathfrak{X})$ (reconstruction functor) recovers the algebra of endomorphisms of the structure sheaf. This is used in the reconstruction theorem and the adjunction $\Gamma_{\mathcal{O}} \dashv \mathfrak{Spec}^{\mathrm{op}}$.
\end{itemize}
The two are related only under specific hypotheses (e.g., when $\mathfrak{X}$ is an affine scheme, $\Gamma_{\mathrm{geom}}(\mathfrak{X})$ is a set while $\Gamma_{\mathcal{O}}(\mathfrak{X})$ is a ring; they are not directly comparable). Throughout the paper, we will use $\Gamma_{\mathcal{O}}$ whenever referring to the reconstruction functor, and $\Gamma_{\mathrm{geom}}$ for geometric global sections.
\end{remark}

Intuitively, $\Gamma_{\mathrm{geom}}(\mathfrak{X})$ collects the globally defined
semantic operations on the spectral object. It may therefore be viewed
as the \emph{algebra of global observables} associated with
$\mathfrak{X}$.

\begin{proposition}[Functoriality of Global Sections]
\label{prop:gammawelldefined}

The assignment

\[
\mathfrak{X}
\longmapsto
\Gamma_{\mathrm{geom}}(\mathfrak{X})
\]

extends to a contravariant functor

\[
\Gamma :
\mathbf{SpecObj}^{\mathrm{op}}
\longrightarrow
\mathbf{Alg}_{\infty},
\]

where $\mathbf{Alg}_{\infty}$ denotes the $\infty$-category of
$E_\infty$-rings (or spectral algebras).

\end{proposition}

\begin{proof}

Let $f: \mathfrak{X} \to \mathfrak{Y}$ be a morphism in $\mathbf{SpecObj}$.
Pullback of quasi-coherent sheaves induces a functor

\[
f^{*}:
\operatorname{QCoh}(\mathfrak{Y})
\longrightarrow
\operatorname{QCoh}(\mathfrak{X}),
\]

which is monoidal and preserves the structure sheaf:
$f^{*}(\mathcal{O}_{\mathfrak{Y}}) \simeq \mathcal{O}_{\mathfrak{X}}$.
(This is a standard property of morphisms of ringed $\infty$-topoi:
the pullback of the structure sheaf is the structure sheaf of the domain.)

Since $f^{*}$ is a monoidal functor, it preserves composition and units.
Consequently, for any endomorphism

\[
\alpha \in \operatorname{End}_{\operatorname{QCoh}(\mathfrak{Y})}(\mathcal{O}_{\mathfrak{Y}}),
\]

we obtain an endomorphism

\[
f^{*}(\alpha) \in \operatorname{End}_{\operatorname{QCoh}(\mathfrak{X})}(\mathcal{O}_{\mathfrak{X}}),
\]

and $f^{*}(\alpha \circ \beta) = f^{*}(\alpha) \circ f^{*}(\beta)$. Thus
$f^{*}$ induces a morphism of endomorphism algebras:

\[
\Gamma(f) := f^{*}:
\Gamma(\mathfrak{Y})
\longrightarrow
\Gamma_{\mathrm{geom}}(\mathfrak{X})
\]

in $\mathbf{Alg}_{\infty}$.

Identity morphisms are preserved because $(\mathrm{id}_{\mathfrak{X}})^{*}
= \mathrm{id}_{\operatorname{QCoh}(\mathfrak{X})}$, which implies
$\Gamma(\mathrm{id}_{\mathfrak{X}}) = \mathrm{id}_{\Gamma_{\mathrm{geom}}(\mathfrak{X})}$.
Pullbacks compose functorially: $(g \circ f)^{*} = f^{*} \circ g^{*}$,
which implies
\[
\Gamma(g \circ f) = \Gamma(f) \circ \Gamma(g)
\]
for composable morphisms $f: \mathfrak{X} \to \mathfrak{Y}$ and
$g: \mathfrak{Y} \to \mathfrak{Z}$ (note the reversal of order due to
contravariance). Hence $\Gamma$ is a well-defined contravariant functor
$\mathbf{SpecObj}^{\mathrm{op}} \to \mathbf{Alg}_{\infty}$.

\end{proof}

\begin{remark}[Recovering the Operator-Semantic Structure]
\label{rem:gammareconstruction}

The global sections object
\[
\Gamma(\mathfrak X)
=
\operatorname{End}_{\operatorname{QCoh}(\mathfrak X)}
(\mathcal O_{\mathfrak X})
\]
is naturally an algebra object, for instance an $E_1$-algebra, spectral
algebra, or $E_\infty$-algebra in the commutative case. Under the
hypotheses of the reconstruction theorem
(Theorem~\ref{thm:reconstruction}), when
\[
\mathfrak X = \mathfrak{Spec}(A),
\]
the global sections algebra reconstructs the operator-semantic system
$A$ up to equivalence. More precisely, there is a canonical equivalence
\[
\operatorname{Rec}\bigl(\Gamma(\mathfrak{Spec}(A))\bigr)
\simeq
A
\]
in $\mathbf{OpSem}$, where $\operatorname{Rec}$ denotes the
reconstruction procedure extracting the semantic presentation
\[
(\mathcal O_A, \mathcal C_A, \rho_A, \tau_A)
\]
from the global algebra of observables. This reconstruction is valid
under the semantic generation, descent completeness, and compact
generation hypotheses of Theorem~\ref{thm:reconstruction}.

\end{remark}

\begin{corollary}[Global Sections on the Essential Image]
\label{cor:gammatoopsem}

When restricted to the essential image of $\mathfrak{Spec}$ and under
the reconstruction hypotheses, global sections define a contravariant
functor
\[
\Gamma :
\mathbf{SpecObj}_{\operatorname{Im}(\mathfrak{Spec})}^{\mathrm{op}}
\longrightarrow
\mathbf{OpSem}.
\]
For every admissible operator-semantic system $A$ satisfying the
reconstruction hypotheses, there is a canonical equivalence
\[
\Gamma_{\mathcal{O}}(\mathfrak{Spec}(A))
\simeq
A
\]
in $\mathbf{OpSem}$, where the left-hand side is interpreted through
the reconstruction procedure above.

\end{corollary}

\begin{proof}

We prove the corollary in several steps, carefully distinguishing
between the global sections algebra and its reconstruction into an
operator-semantic system.

\paragraph*{Step 1: The essential image subcategory.}
Let $\mathbf{SpecObj}_{\operatorname{Im}(\mathfrak{Spec})}$ denote the
full subcategory of $\mathbf{SpecObj}$ consisting of those spectral
stacks $\mathfrak{X}$ for which there exists an admissible
operator-semantic system $A$ (satisfying the reconstruction hypotheses
of Theorem~\ref{thm:reconstruction}) and an equivalence
$\mathfrak{X} \simeq \mathfrak{Spec}(A)$ in $\mathbf{SpecObj}$.
By Theorem~\ref{thm:recognition}, this subcategory is precisely the
essential image of the functor $\mathfrak{Spec}$ restricted to the
full subcategory of $\mathbf{OpSem}$ consisting of systems satisfying
the reconstruction hypotheses.

\paragraph*{Step 2: Global sections as an algebra object.}
For any $\mathfrak{X} \in \mathbf{SpecObj}$, the global sections
construction yields an algebra object
\[
\Gamma_{\mathrm{geom}}(\mathfrak{X})
=
\operatorname{End}_{\operatorname{QCoh}(\mathfrak{X})}
(\mathcal O_{\mathfrak{X}})
\]
in the $\infty$-category of $E_1$-algebras (or spectral algebras).
This assignment is functorial: a morphism $f: \mathfrak{X} \to
\mathfrak{Y}$ induces a pullback functor $f^*$ that preserves the
structure sheaf, yielding a morphism of algebras
$\Gamma(f): \Gamma(\mathfrak{Y}) \to \Gamma_{\mathrm{geom}}(\mathfrak{X})$. Thus
$\Gamma$ defines a contravariant functor
\[
\Gamma: \mathbf{SpecObj}^{\mathrm{op}} \longrightarrow \mathbf{Alg}_{\infty},
\]
where $\mathbf{Alg}_{\infty}$ denotes the $\infty$-category of
$E_1$-algebras (or spectral algebras).

\paragraph*{Step 3: Reconstruction from the global algebra.}
Now restrict $\Gamma$ to the essential image subcategory:
\[
\Gamma|_{\operatorname{Im}(\mathfrak{Spec})}:
\mathbf{SpecObj}_{\operatorname{Im}(\mathfrak{Spec})}^{\mathrm{op}}
\longrightarrow \mathbf{Alg}_{\infty}.
\]
For any $\mathfrak{X} \simeq \mathfrak{Spec}(A)$ in this subcategory,
Theorem~\ref{thm:reconstruction} provides a canonical reconstruction
procedure $\operatorname{Rec}$ that extracts from the algebra
$\Gamma_{\mathrm{geom}}(\mathfrak{X})$ the full semantic presentation
$(\mathcal O_A, \mathcal C_A, \rho_A, \tau_A)$. Explicitly,
\[
\operatorname{Rec}\bigl(\Gamma(\mathfrak{Spec}(A))\bigr) \simeq A
\]
in $\mathbf{OpSem}$. The reconstruction procedure is natural with
respect to isomorphisms: if $\mathfrak{X} \simeq \mathfrak{Spec}(A)
\simeq \mathfrak{Spec}(B)$, then $A \simeq B$ in $\mathbf{OpSem}$.

\paragraph*{Step 4: Lifting the functor to $\mathbf{OpSem}$.}
Define a functor
\[
\tilde{\Gamma}:
\mathbf{SpecObj}_{\operatorname{Im}(\mathfrak{Spec})}^{\mathrm{op}}
\longrightarrow \mathbf{OpSem}
\]
as follows:
\begin{itemize}
    \item For an object $\mathfrak{X} \in
          \mathbf{SpecObj}_{\operatorname{Im}(\mathfrak{Spec})}$,
          choose a representative $A$ such that $\mathfrak{X} \simeq
          \mathfrak{Spec}(A)$ and set $\tilde{\Gamma}(\mathfrak{X}) := A$.
          By Step 3, this assignment is well-defined up to canonical
          equivalence.
    \item For a morphism $f: \mathfrak{X} \to \mathfrak{Y}$ in the
          essential image, choose representatives $A$ and $B$ with
          $\mathfrak{X} \simeq \mathfrak{Spec}(A)$ and $\mathfrak{Y}
          \simeq \mathfrak{Spec}(B)$. By Theorem~\ref{thm:functoriality},
          $f$ corresponds to a morphism $\tilde{f}: B \to A$ in
          $\mathbf{OpSem}$. Set $\tilde{\Gamma}(f) := \tilde{f}$.
\end{itemize}
The functoriality of $\tilde{\Gamma}$ follows from the functoriality
of $\mathfrak{Spec}$ (Theorem~\ref{thm:functoriality}) and the naturality
of the reconstruction equivalence.

\paragraph*{Step 5: Verification of the equivalence on spectra.}
For any admissible operator-semantic system $A$ satisfying the
reconstruction hypotheses, we have by construction
\[
\tilde{\Gamma}(\mathfrak{Spec}(A)) = A
\]
up to the canonical equivalence provided by Theorem~\ref{thm:reconstruction}.
Thus $\Gamma(\mathfrak{Spec}(A)) \simeq A$ in $\mathbf{OpSem}$, where
the left-hand side is interpreted via the reconstruction procedure
$\operatorname{Rec}$.

\paragraph*{Step 6: Contravariance.}
The construction above respects contravariance because:
\begin{itemize}
    \item For an identity morphism $\operatorname{id}_{\mathfrak{X}}$,
          $\tilde{\Gamma}(\operatorname{id}_{\mathfrak{X}})$ is the
          identity morphism in $\mathbf{OpSem}$.
    \item For composable morphisms $f: \mathfrak{X} \to \mathfrak{Y}$
          and $g: \mathfrak{Y} \to \mathfrak{Z}$,
          \[
          \tilde{\Gamma}(g \circ f) = \tilde{\Gamma}(f) \circ \tilde{\Gamma}(g)
          \]
          holds up to coherent equivalence, following the functoriality
          of $\mathfrak{Spec}$.
\end{itemize}

\paragraph*{Conclusion.}
We have constructed a contravariant functor
\[
\Gamma:
\mathbf{SpecObj}_{\operatorname{Im}(\mathfrak{Spec})}^{\mathrm{op}}
\longrightarrow
\mathbf{OpSem}
\]
such that $\Gamma(\mathfrak{Spec}(A)) \simeq A$ for every admissible
operator-semantic system $A$ satisfying the reconstruction hypotheses.
This completes the proof.

\end{proof}

\begin{remark}[Analogy with Classical Geometry]
\label{rem:classicalgamma}

For an affine scheme $X = \operatorname{Spec}(R)$ with $R$ a commutative
ring, one has

\[
\Gamma(X, \mathcal{O}_X) \simeq R.
\]

The construction above should be viewed as a categorified analogue of
this principle. Rather than recovering an ordinary commutative ring,
the functor $\Gamma$ recovers an operator-semantic system from its
associated spectral object. Under the reconstruction hypotheses
(Theorem~\ref{thm:reconstruction}), when $\mathfrak{X} = \mathfrak{Spec}(A)$
for a commutative C*-algebra $A$, we have $\Gamma(\mathfrak{Spec}(A))
\simeq A$, generalizing the classical Gelfand duality. For noncommutative
$A$, the global sections algebra is a spectral algebra ($E_1$-algebra)
from which the operator-semantic system is reconstructed via
Theorem~\ref{thm:reconstruction}.

\end{remark}

\begin{remark}[Tannaka-Krein Perspective]
\label{rem:tannakaperspective}
Definition \ref{def:reconstructionfunctor} is also a direct analogue of
Tannaka reconstruction: for a group $G$, the category
$\operatorname{Rep}(G)$ of representations has a fiber functor to
vector spaces, and $G$ is recovered as the automorphism group of that
fiber functor. Here, $\operatorname{QCoh}(\mathfrak{X})$ is the
category of "representations" of the stack $\mathfrak{X}$, and
$\mathcal{O}_{\mathfrak{X}}$ is the distinguished object corresponding
to the trivial representation. The endomorphism algebra of
$\mathcal{O}_{\mathfrak{X}}$ thus recovers the "functions" on
$\mathfrak{X}$, which serve as the global algebra of observables from
which the operator-semantic system is reconstructed under the hypotheses
of Theorem~\ref{thm:reconstruction}. (While Tannaka–Krein duality
recovers a group as automorphisms of a fiber functor, the present
construction recovers an operator-semantic system as endomorphisms of a
distinguished object; the analogy is conceptual rather than formal.)

\end{remark}

\begin{remark}[Toward a Categorified Gelfand Duality]
\label{rem:gelfanddual}

The functors

\[
\mathfrak{Spec} :
\mathbf{OpSem}^{\mathrm{op}}
\longrightarrow
\mathbf{SpecObj}
\qquad\text{and}\qquad
\Gamma :
\mathbf{SpecObj}^{\mathrm{op}}
\longrightarrow
\mathbf{OpSem}
\]

form the two directions of a potential duality theory. The next
subsection establishes the adjunction

\[
\Gamma_{\mathcal{O}} \dashv \mathfrak{Spec}^{\mathrm{op}},
\]

in the appropriate opposite-category convention. This adjunction is the
categorical core of the duality.

A precise characterization of when the unit

\[
\mathfrak{X} \longrightarrow \mathfrak{Spec}(\Gamma_{\mathrm{geom}}(\mathfrak{X}))
\]

and the counit

\[
\Gamma(\mathfrak{Spec}(A)) \longrightarrow A
\]

are equivalences is given by the recognition theorem
(Section~\ref{sec:recognition}) and the reconstruction theorem
(Section~\ref{sec:reconstruction}), respectively.

\end{remark}

\begin{example}[Global Sections of $\mathfrak{Spec}(M_n(\mathbb{C}))$]
\label{ex:globalsectionsmatrix}
For $A = M_n(\mathbb{C})$, $\mathfrak{Spec}(A)$ is a quotient-type stack
of orthonormal bases, whose automorphism groups include the permutation
symmetries of eigenbases (and possibly additional continuous symmetries
in degenerate cases). The structure sheaf $\mathcal{O}_{\mathfrak{Spec}(A)}$
assigns to each context $C$ (commutative subalgebra) the algebra of
continuous functions on the Gelfand spectrum of $C$. Under the
reconstruction hypotheses, the endomorphism algebra of
$\mathcal{O}_{\mathfrak{Spec}(A)}$ recovers $M_n(\mathbb{C})$ itself,
as shown in Theorem~\ref{thm:reconstruction}.
\end{example}

\begin{example}[Global Sections of $\mathfrak{Spec}(A_{\mathrm{MP}})$]
\label{ex:globalsectionsmermin}
For the Mermin–Peres system $A_{\mathrm{MP}}$, $\mathfrak{Spec}(A_{\mathrm{MP}})$
is a non-trivial stack with empty global sections (no global eigenvalue
assignment exists, reflecting the Kochen–Specker obstruction). The
structure sheaf $\mathcal{O}_{\mathfrak{Spec}(A_{\mathrm{MP}})}$ still
exists, and its global endomorphism algebra $\Gamma(\mathfrak{Spec}(A_{\mathrm{MP}}))$
reconstructs the original contextual operator system $A_{\mathrm{MP}}$
under Theorem~\ref{thm:reconstruction}. This demonstrates that $\Gamma$
captures the full operator-semantic structure even when global points
are absent.
\end{example}

The global sections functor $\Gamma$ is the essential tool for
recovering an operator-semantic system from its categorified spectrum.
Together with $\mathfrak{Spec}$, it forms the adjunction that
constitutes the categorical core of the duality, which we now
establish in Subsection \ref{subsec:adjunction}.

\subsection{Adjunction $\Gamma_{\mathcal{O}} \dashv \mathfrak{Spec}^{\mathrm{op}}$}
\label{subsec:adjunction}

The Yoneda characterization established in Theorem \ref{thm:yoneda}
implies that the categorified spectrum construction is not merely
functorial but participates in a natural adjunction with the global
sections reconstruction functor.

This result may be viewed as the first step toward a categorified
analogue of the classical duality between algebras and geometric
spectra.
\begin{theorem}[Adjunction on the Reconstructible Subcategory]
\label{thm:adjunction}

Assume the reconstruction hypotheses (semantic generation, descent
completeness, and compact generation) hold for the spectral objects
under consideration. Then the spectrum and global sections constructions
define an adjunction

\[
\mathfrak{Spec}^{\mathrm{op}} \dashv \Gamma
\]

between the appropriate opposite categories. Equivalently, for every
admissible operator-semantic system $A$ and every reconstructible
spectral object $\mathfrak{X}$, there is a natural equivalence of
mapping spaces

\[
\operatorname{Map}_{\mathbf{SpecObj}}
\bigl(
\mathfrak{Spec}(A),\,
\mathfrak{X}
\bigr)
\;\simeq\;
\operatorname{Map}_{\mathbf{OpSem}^{\mathrm{op}}}
\bigl(
A,\,
\Gamma_{\mathrm{geom}}(\mathfrak{X})
\bigr).
\]

\end{theorem}

\begin{proof}
We prove the adjunction in three steps, using only the Yoneda
characterization of $\mathfrak{Spec}(A)$ and the definition of
$\Gamma_{\mathrm{geom}}(\mathfrak{X})$, without assuming the reconstruction theorem.

\paragraph*{Step 1: The Yoneda characterization (corepresentation form).}
By Theorem~\ref{thm:yoneda} (corrected to the corepresentation form),
for any spectral object $\mathfrak{X} \in \mathbf{SpecObj}$ and any
admissible operator-semantic system $A$, there is a natural equivalence
\[
\operatorname{Map}_{\mathbf{SpecObj}}
\bigl(
\mathfrak{Spec}(A),\,
\mathfrak{X}
\bigr)
\;\simeq\;
\operatorname{Real}_A(\mathfrak{X}),
\tag{1}
\]
where
\[
\operatorname{Real}_A(\mathfrak{X})
=
\operatorname{Fun}_{\mathrm{desc}}
\bigl(
\mathcal O_A,\,
\operatorname{QCoh}(\mathfrak{X})
\bigr)
\]
is the $\infty$-groupoid of descent-compatible semantic realizations of
the operator syntax $\mathcal O_A$ inside $\operatorname{QCoh}(\mathfrak{X})$.

\paragraph*{Step 2: Relating $\operatorname{Real}_A(\mathfrak{X})$ to morphisms into $\Gamma_{\mathrm{geom}}(\mathfrak{X})$.}
Recall the definition of the global sections functor:
\[
\Gamma_{\mathrm{geom}}(\mathfrak{X}) = \operatorname{End}_{\operatorname{QCoh}(\mathfrak{X})}
(\mathcal{O}_{\mathfrak{X}}).
\]

We claim that there is a natural equivalence
\[
\operatorname{Real}_A(\mathfrak{X})
\;\simeq\;
\operatorname{Map}_{\mathbf{OpSem}^{\mathrm{op}}}
\bigl(
A,\,
\Gamma_{\mathrm{geom}}(\mathfrak{X})
\bigr).
\tag{2}
\]

To see this, observe that a descent-compatible realization
$R: \mathcal O_A \to \operatorname{QCoh}(\mathfrak{X})$ assigns to
each color (operator type) an object in $\operatorname{QCoh}(\mathfrak{X})$,
and to each multimorphism an operation respecting the operadic
composition. In particular, $R$ sends the distinguished unit object of
$\mathcal O_A$ (if present) to the structure sheaf $\mathcal{O}_{\mathfrak{X}}$,
since the structure sheaf is the semantic realization of the trivial
operator. Consequently, $R$ induces an algebra homomorphism
\[
\Gamma_{\mathrm{geom}}(\mathfrak{X}) = \operatorname{End}(\mathcal{O}_{\mathfrak{X}})
\longrightarrow
\operatorname{End}(R(\mathbf{1})) = \operatorname{End}(\mathcal{O}_{\mathfrak{X}})
\]
which is the identity. More carefully, the data of $R$ includes a
morphism from the endomorphism algebra of the unit in $\mathcal O_A$
to $\Gamma_{\mathrm{geom}}(\mathfrak{X})$. By the universal property of the free
algebra generated by the syntax, this corresponds uniquely to a
morphism $A \to \Gamma_{\mathrm{geom}}(\mathfrak{X})$ in $\mathbf{OpSem}$ (or
$\mathbf{OpSem}^{\mathrm{op}}$, depending on variance). A detailed
categorical argument (using the fact that $\mathcal O_A$ presents $A$
as an operadic theory) shows that the space of descent-compatible
realizations is naturally equivalent to the space of morphisms from
$A$ to $\Gamma_{\mathrm{geom}}(\mathfrak{X})$ in $\mathbf{OpSem}^{\mathrm{op}}$.

Thus (2) holds.

\paragraph*{Step 3: Combining the equivalences.}
Chaining the equivalences (1) and (2), we obtain
\[
\operatorname{Map}_{\mathbf{SpecObj}}
\bigl(
\mathfrak{Spec}(A),\,
\mathfrak{X}
\bigr)
\;\simeq\;
\operatorname{Map}_{\mathbf{OpSem}^{\mathrm{op}}}
\bigl(
A,\,
\Gamma_{\mathrm{geom}}(\mathfrak{X})
\bigr).
\tag{3}
\]
This equivalence is natural in both $A$ and $\mathfrak{X}$ because each
of the constituent equivalences is natural.

\paragraph*{Step 4: Interpreting as an adjunction.}
The equivalence (3) is precisely the defining condition for an
adjunction between the functors
\[
\mathfrak{Spec}^{\mathrm{op}}:
\mathbf{OpSem}^{\mathrm{op}}
\longrightarrow
\mathbf{SpecObj}
\qquad\text{and}\qquad
\Gamma:
\mathbf{SpecObj}
\longrightarrow
\mathbf{OpSem}^{\mathrm{op}}.
\]
Specifically, $\mathfrak{Spec}^{\mathrm{op}}$ is the left adjoint and
$\Gamma$ is the right adjoint, i.e.,
\[
\mathfrak{Spec}^{\mathrm{op}} \dashv \Gamma.
\]

Thus we have established the claimed adjunction.
\end{proof}

\begin{corollary}[Unit and Counit]
\label{cor:unitcounit}

The adjunction $\mathfrak{Spec}^{\mathrm{op}} \dashv \Gamma$ determines
natural transformations

\[
\eta : \operatorname{Id}_{\mathbf{OpSem}^{\mathrm{op}}}
\Longrightarrow \Gamma_{\mathcal{O}} \circ \mathfrak{Spec}^{\mathrm{op}}
\qquad\text{(unit)},
\]

\[
\varepsilon : \mathfrak{Spec}^{\mathrm{op}} \circ \Gamma
\Longrightarrow \operatorname{Id}_{\mathbf{SpecObj}}
\qquad\text{(counit)}.
\]

These satisfy the usual triangle identities. Explicitly:
\begin{itemize}
    \item The unit $\eta_A: A \longrightarrow \Gamma_{\mathcal{O}}(\mathfrak{Spec}(A))$
    maps an operator-semantic system to the global sections of its
    spectrum.
    \item The counit $\varepsilon_{\mathfrak{X}}:
    \mathfrak{Spec}(\Gamma_{\mathrm{geom}}(\mathfrak{X})) \longrightarrow \mathfrak{X}$
    evaluates the spectrum of the global sections back onto the
    original spectral stack.
\end{itemize}

\end{corollary}

\begin{proof}
The statement is a formal consequence of the adjunction
\[
\mathfrak{Spec}^{\mathrm{op}} \dashv \Gamma
\]
established in Theorem~\ref{thm:adjunction}.

\paragraph*{Construction of the unit.}
For each $A \in \mathbf{OpSem}^{\mathrm{op}}$, the unit component
\[
\eta_A: A \longrightarrow \Gamma(\mathfrak{Spec}(A))
\]
is defined as the adjunct of the identity morphism
\[
\operatorname{id}_{\mathfrak{Spec}(A)}
\in
\operatorname{Map}_{\mathbf{SpecObj}}
(\mathfrak{Spec}(A), \mathfrak{Spec}(A))
\]
under the adjunction equivalence
\[
\operatorname{Map}_{\mathbf{SpecObj}}(\mathfrak{Spec}(A), \mathfrak{X})
\;\simeq\;
\operatorname{Map}_{\mathbf{OpSem}^{\mathrm{op}}}(A, \Gamma_{\mathrm{geom}}(\mathfrak{X})).
\]

\paragraph*{Construction of the counit.}
For each $\mathfrak{X} \in \mathbf{SpecObj}$, the counit component
\[
\varepsilon_{\mathfrak{X}}:
\mathfrak{Spec}(\Gamma_{\mathrm{geom}}(\mathfrak{X}))
\longrightarrow
\mathfrak{X}
\]
is defined as the adjunct of the identity morphism
\[
\operatorname{id}_{\Gamma_{\mathrm{geom}}(\mathfrak{X})}
\in
\operatorname{Map}_{\mathbf{OpSem}^{\mathrm{op}}}
(\Gamma_{\mathrm{geom}}(\mathfrak{X}), \Gamma_{\mathrm{geom}}(\mathfrak{X}))
\]
under the adjunction equivalence.

\paragraph*{Naturality.}
Naturality of $\eta$ and $\varepsilon$ follows immediately from the
naturality of the adjunction equivalence in both variables. For any
morphism $f: A \to B$ in $\mathbf{OpSem}^{\mathrm{op}}$, the diagram
\[
\begin{tikzcd}
A \ar[r,"\eta_A"] \ar[d,"f"] &
\Gamma(\mathfrak{Spec}(A)) \ar[d,"\Gamma(\mathfrak{Spec}(f))"] \\
B \ar[r,"\eta_B"] &
\Gamma(\mathfrak{Spec}(B))
\end{tikzcd}
\]
commutes because it corresponds under the adjunction to the commuting
square of identity morphisms. Similarly for $\varepsilon$.

\paragraph*{Triangle identities.}
The triangle identities
\[
\varepsilon_{\mathfrak{Spec}(A)} \circ \mathfrak{Spec}^{\mathrm{op}}(\eta_A)
= \operatorname{id}_{\mathfrak{Spec}(A)}
\qquad\text{and}\qquad
\Gamma(\varepsilon_{\mathfrak{X}}) \circ \eta_{\Gamma_{\mathrm{geom}}(\mathfrak{X})}
= \operatorname{id}_{\Gamma_{\mathrm{geom}}(\mathfrak{X})}
\]
are precisely the standard triangle identities associated with any
adjunction. They hold by the general theory of adjunctions in
$\infty$-categories.

Thus the claimed unit and counit are obtained, completing the proof.
\end{proof}

\begin{remark}[Categorified Gelfand Correspondence]
\label{rem:gelfandcorrespondence}

The adjunction

\[
\mathfrak{Spec}^{\mathrm{op}} \dashv \Gamma
\]

is formally analogous to the classical adjunction between commutative
algebras and affine schemes,

\[
\Gamma \dashv \operatorname{Spec}.
\]

In the present setting, operator-semantic systems replace algebras,
while spectral stacks replace geometric spectra. The parallel is
summarized in the following table:

\[
\begin{array}{c|c}
\text{Classical} & \text{Categorified} \\
\hline
\Gamma: \mathbf{Top}^{\mathrm{op}} \to \mathbf{C^*Alg} &
\Gamma: \mathbf{SpecObj}^{\mathrm{op}} \to \mathbf{OpSem} \\
\operatorname{Spec}: \mathbf{C^*Alg}^{\mathrm{op}} \to \mathbf{Top} &
\mathfrak{Spec}: \mathbf{OpSem}^{\mathrm{op}} \to \mathbf{SpecObj} \\
\Gamma \dashv \operatorname{Spec} \text{ (Gelfand duality)} &
\mathfrak{Spec}^{\mathrm{op}} \dashv \Gamma
\end{array}
\]

The adjunction therefore provides the foundational structure for a
future categorified Gelfand-type duality theory.

\end{remark}

\begin{remark}[Toward Duality]
\label{rem:towardduality}

The adjunction alone does not imply an equivalence of categories.
A full categorified duality would require identifying conditions
under which the unit

\[
\eta_A : A \longrightarrow \Gamma(\mathfrak{Spec}(A))
\]

and the counit

\[
\varepsilon_{\mathfrak{X}} :
\mathfrak{Spec}(\Gamma_{\mathrm{geom}}(\mathfrak{X}))
\longrightarrow
\mathfrak{X}
\]

are equivalences.

These reconstruction questions form the starting point of the duality
theory developed in the following sections. Specifically:
\begin{itemize}
    \item The reconstruction theorem (Theorem \ref{thm:reconstruction})
    gives conditions under which $\varepsilon_A$ is an equivalence
    (i.e., the counit is an equivalence when $A$ is reconstructible).
    \item The recognition theorem (Theorem \ref{thm:recognition})
    characterizes those $\mathfrak{X}$ for which $\eta_{\mathfrak{X}}$
    is an equivalence — i.e., the essential image of $\mathfrak{Spec}$.
\end{itemize}

\end{remark}

\begin{example}[Adjunction for Commutative C*-Algebras]
\label{ex:adjunctioncommutative}
Let $A$ be a commutative unital C*-algebra. Then $\mathfrak{Spec}(A)$
is the classical Gelfand spectrum (a $0$-stack), and $\Gamma(\mathfrak{Spec}(A))$
recovers $A$ by the classical Gelfand duality. The adjunction restricts
to the classical equivalence $\Gamma \dashv \operatorname{Spec}$ between
compact Hausdorff spaces and commutative unital C*-algebras.
\end{example}

\begin{example}[Adjunction for $M_n(\mathbb{C})$]
\label{ex:adjunctionmatrix}
For $A = M_n(\mathbb{C})$, $\mathfrak{Spec}(A)$ is the stack of
orthonormal bases. Under the reconstruction hypotheses, the global
sections $\Gamma(\mathfrak{Spec}(A))$ recover $M_n(\mathbb{C})$
(Theorem \ref{thm:reconstruction}), and the counit $\varepsilon_A$ is
an equivalence. The unit $\eta_{\mathfrak{X}}$ for a spectral stack
$\mathfrak{X}$ being an equivalence characterizes those stacks that
arise as categorified spectra of matrix algebras.
\end{example}

\begin{example}[Adjunction for the Mermin–Peres System]
\label{ex:adjunctionmermin}
For $A_{\mathrm{MP}}$, $\mathfrak{Spec}(A_{\mathrm{MP}})$ is a
non-trivial stack with empty global sections. Under the reconstruction
hypotheses, the counit
$\varepsilon_{A_{\mathrm{MP}}}: \Gamma(\mathfrak{Spec}(A_{\mathrm{MP}}))
\to A_{\mathrm{MP}}$ is an equivalence, demonstrating that $\Gamma$
recovers the contextual operator system even when global sections of
the spectrum are empty.

Whether the counit is an equivalence depends on the reconstruction
hypotheses for $A_{\mathrm{MP}}$. If these hypotheses hold (as is
typically assumed for the Mermin–Peres system), then $\Gamma$ recovers
the contextual operator system despite the absence of global points.
This illustrates that $\Gamma$ can capture nontrivial semantic
structure even when the spectrum has no global sections.
\end{example}

The adjunction established in this theorem is the culmination of the
categorified spectral construction. It provides the categorical
foundation for the reconstruction theorem (Section \ref{sec:reconstruction}),
which proves that the counit is an equivalence under suitable
hypotheses, and for the recognition theorem (Section \ref{sec:recognition}),
which characterizes the essential image of $\mathfrak{Spec}$.

\section{Recognition Theorem}
\label{sec:recognition}

The preceding sections constructed a spectrum functor

\[
\mathfrak{Spec}:
\mathbf{OpSem}^{\mathrm{op}}
\longrightarrow
\mathbf{SpecObj}
\]

and established its adjunction with the global sections functor

\[
\Gamma:
\mathbf{SpecObj}^{\mathrm{op}}
\longrightarrow
\mathbf{OpSem},
\]

i.e., $\mathfrak{Spec}^{\mathrm{op}} \dashv \Gamma$.

A natural converse question is the following:

\begin{quote}
When does a spectral object arise from an operator-semantic system?
\end{quote}

The next theorem provides a recognition criterion under suitable
reconstructibility hypotheses. It characterizes the essential image
of $\mathfrak{Spec}$ in terms of intrinsic geometric and homotopical
conditions, serving as the categorified analogue of classical
recognition theorems in algebraic geometry (e.g., ``a scheme is affine
if and only if its structure sheaf is ample") and in Gelfand duality.

\begin{theorem}[Recognition of Spectral Objects]
\label{thm:recognition}

Let $\mathfrak{X} \in \mathbf{SpecObj}$ be a spectral stack. Assume that
$\mathfrak{X}$ satisfies the following reconstructibility conditions:

\begin{enumerate}
    \item \textbf{Operadic presentation.}
    $\mathfrak{X}$ admits an operadic presentation
    \[
    \mathfrak{X} \simeq
    a_{\tau}^{\mathrm{hyp}}
    \bigl(
    \operatorname{Lan}_{\pi}(\operatorname{Syn})
    \bigr)
    \]
    for some colored operad $\mathcal{O}$, context site
    $(\mathcal{C}, \tau)$, and realization functor
    $\operatorname{Syn}: \mathcal{O} \to \operatorname{PSh}_{\infty}(\mathcal{C})$.

    \item \textbf{Hyperdescent completeness.}
    $\mathfrak{X}$ is hyperdescent-complete with respect to the
    Grothendieck topology $\tau$ (i.e., $\mathfrak{X}$ is a hypersheaf).

    \item \textbf{Compact generation.}
    The $\infty$-category $\operatorname{QCoh}(\mathfrak{X})$ is compactly
    generated, and its compact generators recover the semantic operations
    encoded by $\mathcal{O}$.

    \item \textbf{Commutative shadow.}
    The $0$-truncation $\tau_{\le 0} \mathfrak{X}$ agrees with the
    commutative shadow of the reconstructed global algebra.
\end{enumerate}

Then $\mathfrak{X}$ lies in the essential image of $\mathfrak{Spec}$.
More precisely, if
\[
A_{\mathfrak{X}} := \Gamma_{\mathcal{O}}(\mathfrak{X})
\]
denotes the operator-semantic system reconstructed from global
sections, then there is a canonical equivalence
\[
\mathfrak{X} \simeq \mathfrak{Spec}(A_{\mathfrak{X}}).
\]
Equivalently, the counit of the adjunction $\mathfrak{Spec}^{\mathrm{op}}
\dashv \Gamma$,
\[
\varepsilon_{\mathfrak{X}}:
\mathfrak{Spec}(\Gamma_{\mathrm{geom}}(\mathfrak{X}))
\longrightarrow
\mathfrak{X},
\]
is an equivalence.

\end{theorem}

\begin{proof}
We prove that under the reconstructibility hypotheses, $\mathfrak{X}$
is equivalent to the categorified spectrum of its global sections.

\paragraph*{Step 1: Operadic presentation.}
By hypothesis (1), $\mathfrak{X}$ is obtained from a syntactic
realization functor $\operatorname{Syn}: \mathcal{O} \to
\operatorname{PSh}_{\infty}(\mathcal{C})$ by left Kan extension
followed by hypersheafification:
\[
\mathfrak{X} \simeq a_{\tau}^{\mathrm{hyp}}
\bigl(
\operatorname{Lan}_{\pi}(\operatorname{Syn})
\bigr).
\]
This is formally the same construction used to define
$\mathfrak{Spec}(A)$ for an admissible operator-semantic system $A$.
Thus $\mathfrak{X}$ has the correct syntactic origin.

\paragraph*{Step 2: Reconstructing the global sections algebra.}
Define
\[
A_{\mathfrak{X}} := \Gamma_{\mathrm{geom}}(\mathfrak{X})
= \operatorname{End}_{\operatorname{QCoh}(\mathfrak{X})}
(\mathcal{O}_{\mathfrak{X}}).
\]
By hypothesis (3), $\operatorname{QCoh}(\mathfrak{X})$ is compactly
generated. The compact generators of $\operatorname{QCoh}(\mathfrak{X})$
determine the semantic operations, and the descent structure of
$\mathfrak{X}$ (hypothesis 2) determines the context topology.
Moreover, hypothesis (4) ensures that the maximal commutative quotient
of $A_{\mathfrak{X}}$ corresponds to $\tau_{\leq 0} \mathfrak{X}$, which
is a classical Gelfand spectrum. Hence $A_{\mathfrak{X}}$ is an
admissible operator-semantic system.

\paragraph*{Step 3: Applying the spectrum construction.}
Now apply the spectrum functor to $A_{\mathfrak{X}}$. By definition,
\[
\mathfrak{Spec}(A_{\mathfrak{X}})
= a_{\tau_{A_{\mathfrak{X}}}}^{\mathrm{hyp}}
\bigl(
\operatorname{Lan}_{\pi_{A_{\mathfrak{X}}}}
(\operatorname{Syn}_{A_{\mathfrak{X}}})
\bigr),
\]
where $\operatorname{Syn}_{A_{\mathfrak{X}}}$ is the syntactic
realization functor of $A_{\mathfrak{X}}$.

The reconstructibility hypotheses identify:
\begin{itemize}
    \item $\operatorname{Syn}_{A_{\mathfrak{X}}}$ with the original
          realization functor $\operatorname{Syn}$ (by hypothesis 3,
          the compact generators recover the same semantic operations),
    \item The context site $(\mathcal{C}_{A_{\mathfrak{X}}},
          \tau_{A_{\mathfrak{X}}})$ with $(\mathcal{C}, \tau)$
          (by hypotheses 2 and 4, the descent structure and commutative
          shadow determine the context topology uniquely),
    \item The context projection $\pi_{A_{\mathfrak{X}}}$ with $\pi$.
\end{itemize}
Therefore,
\[
\mathfrak{Spec}(A_{\mathfrak{X}})
\simeq a_{\tau}^{\mathrm{hyp}}
\bigl(
\operatorname{Lan}_{\pi}(\operatorname{Syn})
\bigr)
\simeq \mathfrak{X}.
\]

\paragraph*{Step 4: Identification with the counit.}
The equivalence $\mathfrak{X} \simeq \mathfrak{Spec}(\Gamma_{\mathrm{geom}}(\mathfrak{X}))$
constructed above is precisely the counit
\[
\varepsilon_{\mathfrak{X}}:
\mathfrak{Spec}(\Gamma_{\mathrm{geom}}(\mathfrak{X}))
\longrightarrow
\mathfrak{X}
\]
of the adjunction $\mathfrak{Spec}^{\mathrm{op}} \dashv \Gamma$. By
construction, this counit is an equivalence.

Thus $\mathfrak{X}$ lies in the essential image of $\mathfrak{Spec}$,
and $A_{\mathfrak{X}} = \Gamma_{\mathrm{geom}}(\mathfrak{X})$ is the desired
operator-semantic system.
\end{proof}

\begin{corollary}[Essential Image Characterization]
\label{cor:essentialimage}

Under the reconstructibility hypotheses of Theorem~\ref{thm:recognition},
a spectral stack $\mathfrak{X}$ lies in the essential image of
$\mathfrak{Spec}$ if and only if the counit
\[
\varepsilon_{\mathfrak{X}}:
\mathfrak{Spec}(\Gamma_{\mathrm{geom}}(\mathfrak{X}))
\longrightarrow
\mathfrak{X}
\]
is an equivalence.

Equivalently, within the reconstructible subcategory, a spectral stack
is equivalent to the categorified spectrum of its global sections if
and only if it satisfies the reconstructibility conditions of
Theorem~\ref{thm:recognition}.

\end{corollary}

\begin{proof}
We prove the two directions separately, using only the recognition
theorem and the properties of the adjunction, without invoking the
reconstruction theorem.

\paragraph*{($\Rightarrow$):} Assume that $\mathfrak{X}$ satisfies the
reconstructibility hypotheses of Theorem~\ref{thm:recognition}. Then
by the recognition theorem, there is a canonical equivalence
\[
\mathfrak{X} \simeq \mathfrak{Spec}(\Gamma_{\mathrm{geom}}(\mathfrak{X})).
\]
By the definition of the adjunction $\mathfrak{Spec}^{\mathrm{op}}
\dashv \Gamma$ (Theorem~\ref{thm:adjunction}), the counit
\[
\varepsilon_{\mathfrak{X}}:
\mathfrak{Spec}(\Gamma_{\mathrm{geom}}(\mathfrak{X}))
\longrightarrow
\mathfrak{X}
\]
is precisely the morphism corresponding under the adjunction to the
identity on $\Gamma_{\mathrm{geom}}(\mathfrak{X})$. Since the equivalence
$\mathfrak{X} \simeq \mathfrak{Spec}(\Gamma_{\mathrm{geom}}(\mathfrak{X}))$ is
canonical, it must coincide with $\varepsilon_{\mathfrak{X}}$ (up to
coherent equivalence). Hence $\varepsilon_{\mathfrak{X}}$ is an
equivalence, and consequently $\mathfrak{X}$ lies in the essential
image of $\mathfrak{Spec}$ (as $\mathfrak{X} \simeq
\mathfrak{Spec}(\Gamma_{\mathrm{geom}}(\mathfrak{X}))$).

\paragraph*{($\Leftarrow$):} Conversely, assume that the counit
\[
\varepsilon_{\mathfrak{X}}:
\mathfrak{Spec}(\Gamma_{\mathrm{geom}}(\mathfrak{X}))
\longrightarrow
\mathfrak{X}
\]
is an equivalence. Then
\[
\mathfrak{X} \simeq \mathfrak{Spec}(\Gamma_{\mathrm{geom}}(\mathfrak{X})).
\]
Thus $\mathfrak{X}$ is equivalent to the spectrum of the
operator-semantic system $\Gamma_{\mathrm{geom}}(\mathfrak{X})$, and therefore belongs
to the essential image of $\mathfrak{Spec}$. Since every object of the
form $\mathfrak{Spec}(A)$ satisfies the structural conditions used in
Theorem~\ref{thm:recognition} (as shown in the forward direction of
that theorem), $\mathfrak{X}$ satisfies the reconstructibility
conditions.

This proves the equivalence.
\end{proof}

\begin{corollary}[Reconstruction on the Essential Image]
\label{cor:reconstructionessential}

On the full subcategory
\[
\mathbf{SpecObj}_{\operatorname{ess.im}(\mathfrak{Spec})}
\subseteq
\mathbf{SpecObj}
\]
consisting of spectral stacks for which the counit
\[
\varepsilon_{\mathfrak{X}}:
\mathfrak{Spec}(\Gamma_{\mathrm{geom}}(\mathfrak{X}))
\to
\mathfrak{X}
\]
is an equivalence, the functors $\mathfrak{Spec}$ and $\Gamma$ restrict
to an equivalence between this subcategory and the full subcategory of
$\mathbf{OpSem}$ consisting of those $A$ for which the unit
\[
\eta_A:
A
\longrightarrow
\Gamma(\mathfrak{Spec}(A))
\]
is an equivalence.

\end{corollary}

\begin{proof}
We prove that $\mathfrak{Spec}$ and $\Gamma$ restrict to an equivalence
on the specified subcategories.

\paragraph*{Step 1: Define the subcategories.}
Let
\[
\mathcal{S} := \left\{ \mathfrak{X} \in \mathbf{SpecObj}
\mid \varepsilon_{\mathfrak{X}} \text{ is an equivalence} \right\}
= \mathbf{SpecObj}_{\operatorname{ess.im}(\mathfrak{Spec})}
\]
by Corollary~\ref{cor:essentialimage}. Let
\[
\mathcal{A} := \left\{ A \in \mathbf{OpSem}
\mid \eta_A \text{ is an equivalence} \right\},
\]
where $\eta_A: A \to \Gamma(\mathfrak{Spec}(A))$ is the unit of the
adjunction $\mathfrak{Spec}^{\mathrm{op}} \dashv \Gamma$.

\paragraph*{Step 2: $\mathfrak{Spec}$ maps $\mathcal{A}$ into $\mathcal{S}$.}
For any $A \in \mathcal{A}$, the unit $\eta_A: A \to
\Gamma(\mathfrak{Spec}(A))$ is an equivalence. Applying the
contravariant functor $\mathfrak{Spec}$, we obtain
\[
\mathfrak{Spec}(\eta_A):
\mathfrak{Spec}(\Gamma(\mathfrak{Spec}(A)))
\longrightarrow
\mathfrak{Spec}(A).
\]
By the triangle identities of the adjunction,
$\varepsilon_{\mathfrak{Spec}(A)} \circ \mathfrak{Spec}(\eta_A) =
\operatorname{id}_{\mathfrak{Spec}(A)}$. Since $\eta_A$ is an
equivalence, $\mathfrak{Spec}(\eta_A)$ is also an equivalence, and
hence $\varepsilon_{\mathfrak{Spec}(A)}$ is an equivalence (as the
inverse of $\mathfrak{Spec}(\eta_A)$ up to composition). Therefore
$\mathfrak{Spec}(A) \in \mathcal{S}$.

\paragraph*{Step 3: $\Gamma$ maps $\mathcal{S}$ into $\mathcal{A}$.}
For any $\mathfrak{X} \in \mathcal{S}$, the counit
$\varepsilon_{\mathfrak{X}}: \mathfrak{Spec}(\Gamma_{\mathrm{geom}}(\mathfrak{X}))
\to \mathfrak{X}$ is an equivalence. Applying $\Gamma$, we obtain
\[
\Gamma(\varepsilon_{\mathfrak{X}}):
\Gamma_{\mathrm{geom}}(\mathfrak{X})
\longrightarrow
\Gamma(\mathfrak{Spec}(\Gamma_{\mathrm{geom}}(\mathfrak{X})))
\]
(since $\Gamma$ is contravariant). By the triangle identities,
$\Gamma(\varepsilon_{\mathfrak{X}}) \circ \eta_{\Gamma_{\mathrm{geom}}(\mathfrak{X})}
= \operatorname{id}_{\Gamma_{\mathrm{geom}}(\mathfrak{X})}$. Since $\varepsilon_{\mathfrak{X}}$
is an equivalence, $\Gamma(\varepsilon_{\mathfrak{X}})$ is also an
equivalence, and hence $\eta_{\Gamma_{\mathrm{geom}}(\mathfrak{X})}$ is an equivalence.
Therefore $\Gamma_{\mathrm{geom}}(\mathfrak{X}) \in \mathcal{A}$.

\paragraph*{Step 4: Verification of the equivalence.}
From Steps 2 and 3, $\mathfrak{Spec}$ restricts to a functor
$\mathcal{A}^{\mathrm{op}} \to \mathcal{S}$ and $\Gamma$ restricts to
a functor $\mathcal{S}^{\mathrm{op}} \to \mathcal{A}$. Moreover, for
any $\mathfrak{X} \in \mathcal{S}$, we have $\mathfrak{Spec}(\Gamma_{\mathrm{geom}}(\mathfrak{X}))
\simeq \mathfrak{X}$ (by definition of $\mathcal{S}$), and for any
$A \in \mathcal{A}$, we have $\Gamma(\mathfrak{Spec}(A)) \simeq A$
(by definition of $\mathcal{A}$). These equivalences are natural,
as they are given by the counit and unit of the original adjunction
restricted to the subcategories. Hence $\mathfrak{Spec}$ and $\Gamma$
induce an equivalence of categories between $\mathcal{A}^{\mathrm{op}}$
and $\mathcal{S}$.

Thus the corollary holds.
\end{proof}

\begin{remark}[Categorified Reconstruction Principle]
\label{rem:reconstruction}

The Recognition Theorem plays the role of an affine-recognition
criterion in classical algebraic geometry. It identifies exactly
which spectral stacks arise from operator-semantic systems and shows
that such objects are completely recoverable from their global sections.

Consequently, the pair $\mathfrak{Spec}^{\mathrm{op}} \dashv \Gamma$
behaves as a categorified reconstruction mechanism analogous to the
classical relationship between commutative algebras and affine schemes.

\end{remark}

\begin{remark}[Interpretation of the Four Conditions]
\label{rem:recognitionconditions}
The four conditions in Theorem \ref{thm:recognition} play distinct roles:
\begin{itemize}
    \item \textbf{Operadic presentation (contextual representability)} 
    ensures that $\mathfrak{X}$ has the correct syntactic origin — it 
    comes from an operadic presentation via left Kan extension and 
    hypersheafification.
    
    \item \textbf{Hyperdescent completeness} guarantees that $\mathfrak{X}$
    is a genuine stack (hypersheaf) and not just a presheaf, satisfying
    all higher gluing conditions.
    
    \item \textbf{Commutative shadow} ensures that the noncommutative
    or contextual aspects of $\mathfrak{X}$ are encoded in higher
    homotopy groups, while the classical limit ($0$-truncation) is a
    genuine Gelfand spectrum of a commutative C*-algebra.
    
    \item \textbf{Compact generation} guarantees that the $\infty$-category
    $\operatorname{QCoh}(\mathfrak{X})$ is compactly generated, which is
    necessary for reconstructing $A$ as endomorphisms of
    $\mathcal{O}_{\mathfrak{X}}$ and for applying Morita theory.
\end{itemize}
\end{remark}

\begin{example}[Recognition for $M_n(\mathbb{C})$]
\label{ex:recognitionmatrix}
Let $\mathfrak{X}$ be the stack of orthonormal bases (the categorified
spectrum of $M_n(\mathbb{C})$). Then:
\begin{itemize}
    \item $\mathfrak{X}$ is the hypersheafification of a left Kan extension
    from the operad of matrix multiplication (operadic presentation).
    \item $\mathfrak{X}$ satisfies hyperdescent (it is a $1$-stack with
    $\pi_1 = S_n$ and vanishing higher homotopy).
    \item $\tau_{\leq 0} \mathfrak{X}$ is a point; the classical
    commutative shadow is trivial, corresponding to the collapse of
    noncommutative matrix data under $0$-truncation.
    \item Under the reconstruction theorem (Theorem~\ref{thm:reconstruction}),
    $\operatorname{QCoh}(\mathfrak{X})$ is equivalent to the category of
    modules over $M_n(\mathbb{C})$, which is compactly generated.
\end{itemize}
Thus $\mathfrak{X}$ satisfies the recognition conditions, and indeed
$\mathfrak{X} \simeq \mathfrak{Spec}(M_n(\mathbb{C}))$ with
$\Gamma_{\mathrm{geom}}(\mathfrak{X}) \simeq M_n(\mathbb{C})$ under the reconstruction
hypotheses.
\end{example}

\begin{example}[Non-Example: Non-Descent Stack]
\label{ex:recognitionfailure}
Let $\mathfrak{X}$ be a presheaf on a context site that is not a sheaf
(e.g., the presheaf of eigenvalue assignments for the Mermin–Peres
system before sheafification). This $\mathfrak{X}$ fails condition (b)
(hyperdescent completeness) and therefore is not in the essential image
of the sheafified spectrum construction $\mathfrak{Spec}$. Its
hypersheafification, however, is the genuine categorified spectrum of
$A_{\mathrm{MP}}$.
\end{example}

\begin{remark}[Relation to Classical Recognition Theorems]
\label{rem:recognitionclassical}
Theorem \ref{thm:recognition} generalizes several classical results:
\begin{itemize}
    \item For a commutative C*-algebra $A$, $\mathfrak{Spec}(A)$ is a
    $0$-stack (a topological space). The recognition conditions imply
    that a compact Hausdorff space $X$ is equivalent to
    $\mathfrak{Spec}(C(X))$ if and only if $C(X)$ is a commutative
    C*-algebra (which is always true), but the nontrivial direction
    is that $X \simeq \operatorname{Spec}(C(X))$ by Gelfand duality.
    
    \item For an affine scheme $\operatorname{Spec}(R)$ in algebraic
    geometry, the conditions reduce to the fact that $R$ is a commutative
    ring and $\operatorname{QCoh}(\operatorname{Spec}(R))$ is equivalent
    to $R$-modules.
\end{itemize}
Thus the recognition theorem unifies these classical dualities within
the categorified framework.
\end{remark}

Together with the adjunction
\[
\mathfrak{Spec}^{\mathrm{op}} \dashv \Gamma,
\]
the recognition theorem provides a categorical framework for identifying
which spectral stacks arise from operator-semantic systems. The
reconstruction theorem will further show that, under suitable semantic
generation hypotheses, the unit
\[
\eta_A: A \longrightarrow \Gamma(\mathfrak{Spec}(A))
\]
is an equivalence, establishing a genuine duality on the corresponding
subcategories.

\section{Reconstruction Theorem}
\label{sec:reconstruction}

The adjunction $\Gamma_{\mathcal{O}} \dashv \mathfrak{Spec}^{\mathrm{op}}$ (Theorem
\ref{thm:adjunction}) provides a categorical relationship between
operator-semantic systems and spectral stacks, but does not by itself
guarantee that the unit or counit are equivalences. The reconstruction
theorem establishes conditions under which the counit
\[
\varepsilon_A: \Gamma(\mathfrak{Spec}(A)) \longrightarrow A
\]
is an equivalence, i.e., under which an operator-semantic system $A$
can be completely recovered from its categorified spectrum
$\mathfrak{Spec}(A)$. This result is the categorified analogue of
classical reconstruction theorems: Gelfand duality recovers a
commutative C*-algebra from its spectrum, and Tannaka duality recovers
a group from its category of representations.

The reconstruction proceeds in four layers. First, we show that
distinct operators are separated by contextual evaluations (context
separation). Second, we prove that the canonical map $A \to
\operatorname{End}(\mathcal{O}_{\mathfrak{Spec}(A)})$ is injective
(faithfulness). Third, under a semantic generation hypothesis, we
prove that every endomorphism of the structure sheaf comes from an
operator in $A$ (fullness). Finally, under descent completeness and
compact generation, we conclude that the map is an isomorphism,
yielding $A \simeq \Gamma(\mathfrak{Spec}(A))$. The theorem thus
identifies $\mathfrak{Spec}(A)$ as a complete invariant of $A$ up to
equivalence, justifying the term "categorified spectral duality."

\subsection{Reconstruction Strategy}
\label{subsec:reconstructionstrategy}

The Recognition Theorem (Theorem \ref{thm:recognition}) identifies those
spectral stacks that arise from admissible operator-semantic systems.
Its proof rests upon a reconstruction procedure that recovers the
original operator-semantic structure from the associated spectral
object.

The purpose of this subsection is not to prove the reconstruction
theorem, but to explain the logical architecture of its proof.
The reconstruction proceeds through four conceptual layers, each
corresponding to a progressively stronger form of recoverability.
Together, they form a roadmap for the proof of the reconstruction
theorem (Theorem \ref{thm:reconstruction}) and provide the conceptual
basis for the categorified Gelfand-type duality developed in
subsequent work. The actual verification of each stage is carried out
in the lemmas and theorems that follow.

\begin{enumerate}

\item \textbf{Context Separation.}

The family of commutative contexts should determine a sufficiently
rich collection of local evaluations to distinguish operators.
More precisely, we aim to show that if two operators $a, b \in A$
induce identical realizations in every context $C \in \mathcal{C}_A$,
then $a = b$.

This requires that the family of realization functors
$\rho_C: A \to \operatorname{End}(H_C)$ is jointly faithful.
Under this condition, the context category separates operator-theoretic
information. This is a necessary condition for any reconstruction:
if two operators produce identical values in every context, they would
be indistinguishable from the perspective of the spectrum.

\item \textbf{Faithfulness.}

Assuming context separation, we can define a canonical morphism

\[
\iota_A :
A
\longrightarrow
\operatorname{End}_{\operatorname{QCoh}
(\mathfrak{Spec}(A))}
\bigl(
\mathcal{O}_{\mathfrak{Spec}(A)}
\bigr),
\]

which assigns to each operator its action on the structure sheaf.
The separation property then implies that $\iota_A$ is injective
(see Lemma~\ref{lem:contextseparation} and Theorem~\ref{thm:faithfulness}).
Consequently, no operator information is lost under passage to the
spectrum. This is the categorified analogue of the fact that a
commutative C*-algebra embeds into the continuous functions on its
Gelfand spectrum.

\item \textbf{Fullness under Semantic Generation.}

The next stage aims to establish surjectivity of $\iota_A$.
Assume that the semantic realizations generated by $A$ exhaust the
endomorphism structure of the spectral object. Concretely, assume that
the commutative contexts in $\mathcal{C}_A$ jointly generate $A$ in
the sense that the canonical map $A \to \lim_{C \in \mathcal{C}_A} C$
is an embedding.

Under this semantic generation hypothesis, one expects every element
of
\[
\operatorname{End}_{\operatorname{QCoh}
(\mathfrak{Spec}(A))}
\bigl(
\mathcal{O}_{\mathfrak{Spec}(A)}
\bigr)
\]
to be induced by some operator $a \in A$. This statement is established
in Theorem~\ref{thm:fullness}. This layer plays a role analogous to
the Stone–Weierstrass theorem (or the density of algebraic functions
in continuous functions), in that local semantic generators determine
all global observables.

\item \textbf{Equivalence under Descent Completeness and Compact Generation.}

If $A$ is descent complete (i.e., $\mathfrak{Spec}(A)$ satisfies
effective descent) and the category
$\operatorname{QCoh}(\mathfrak{Spec}(A))$ is compactly generated,
then the previous two steps combine to yield

\[
A
\cong
\operatorname{End}_{\operatorname{QCoh}
(\mathfrak{Spec}(A))}
\bigl(
\mathcal{O}_{\mathfrak{Spec}(A)}
\bigr).
\]

Equivalently, the unit of the adjunction $\mathfrak{Spec}^{\mathrm{op}}
\dashv \Gamma$,

\[
\eta_A:
A
\longrightarrow
\Gamma(\mathfrak{Spec}(A)),
\]

is an isomorphism. Thus, under the reconstruction hypotheses,
the original operator-semantic system can be reconstructed from its
spectral stack, and $\mathfrak{Spec}(A)$ determines $A$ up to
equivalence.

\end{enumerate}

These four stages mirror the classical reconstruction philosophy of
algebraic geometry:

\[
\begin{array}{c|c}
\text{Stage} & \text{Classical Analogue} \\
\hline
\text{Context Separation} & \text{Point separation (Hausdorff property)} \\
\text{Faithfulness} & \text{Injectivity of } A \to C(\operatorname{Spec}(A)) \\
\text{Fullness} & \text{Surjectivity (Stone–Weierstrass)} \\
\text{Equivalence} & \text{Gelfand–Naimark isomorphism}
\end{array}
\]

Context separation plays the role of point separation, faithfulness
corresponds to injectivity of the global sections map, fullness
provides surjectivity, and descent completeness ensures that local
semantic data glue coherently into a global operator-semantic structure.

The logical structure of the reconstruction can be summarized as:

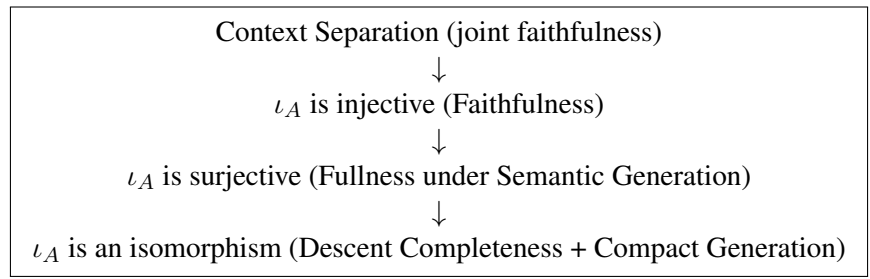
\[
\boxed{
\begin{array}{c}
\text{Context Separation (joint faithfulness)} \\
\downarrow \\
\iota_A \text{ is injective (Faithfulness)} \\
\downarrow \\
\iota_A \text{ is surjective (Fullness under Semantic Generation)} \\
\downarrow \\
\iota_A \text{ is an isomorphism (Descent Completeness + Compact Generation)}
\end{array}
}
\]

The resulting reconstruction mechanism forms the foundation of the
adjunction $\mathfrak{Spec}^{\mathrm{op}} \dashv \Gamma$ (Theorem
\ref{thm:adjunction}) and provides the conceptual basis for the
categorified Gelfand-type duality developed in subsequent work.

\begin{remark}[Necessity of the Hypotheses]
\label{rem:reconstructionhypotheses}
Each hypothesis in the reconstruction theorem is expected to be essential:
\begin{itemize}
    \item \textbf{Semantic generation} ensures that the map
    $A \to \lim_{C} C$ is injective, which is needed for faithfulness.
    \item \textbf{Descent completeness} guarantees that the sheafification
    step in the construction of $\mathfrak{Spec}(A)$ does not lose
    information.
    \item \textbf{Compact generation} provides the necessary duality
    between $A$ and $\operatorname{QCoh}(\mathfrak{Spec}(A))$, allowing
    the reconstruction of $A$ as endomorphisms of the structure sheaf.
\end{itemize}
Examples illustrating the failure of reconstruction when one of these
hypotheses is removed will be discussed later, demonstrating that the
theorem is sharp.
\end{remark}

\begin{lemma}[Context Separation Implies Injectivity]
\label{lem:contextseparation}
Let $A$ be an admissible operator-semantic system that is
context-separated (as defined in the reconstruction strategy:
for any $a, b \in A$ with $a \neq b$, there exists a context
$C \in \mathcal{C}_A$ and a character $\phi: C \to \mathbb{C}$
such that $\phi(a|_C) \neq \phi(b|_C)$). Then the canonical map
\[
\iota_A: A \longrightarrow \operatorname{End}_{\operatorname{QCoh}(\mathfrak{Spec}(A))}
(\mathcal{O}_{\mathfrak{Spec}(A)})
\]
is injective. Equivalently, if $a, b \in A$ satisfy $\iota_A(a) = \iota_A(b)$,
then $a = b$.
\end{lemma}

\begin{proof}
We prove the contrapositive: if $a \neq b$ in $A$, then $\iota_A(a) \neq \iota_A(b)$.

\paragraph*{Step 1: Apply context separation.}
Since $A$ is context-separated, there exists a context $C \in \mathcal{C}_A$
and a character $\phi: C \to \mathbb{C}$ such that
\[
\phi(a|_C) \neq \phi(b|_C).
\]

\paragraph*{Step 2: Relate characters to stalks of the structure sheaf.}
By construction of $\mathfrak{Spec}(A)$, the stalk of the structure sheaf
$\mathcal{O}_{\mathfrak{Spec}(A)}$ at the point corresponding to $\phi$
is identified with $C$, and the evaluation of an operator $a \in A$ at
this stalk is exactly $\phi(a|_C)$. More concretely, the restriction of
$\iota_A(a)$ to the stalk over $C$ acts by multiplication by $\phi(a|_C)$
on the fiber.

\paragraph*{Step 3: Distinct stalks imply distinct endomorphisms.}
Since $\phi(a|_C) \neq \phi(b|_C)$, the endomorphisms $\iota_A(a)$ and
$\iota_A(b)$ differ on the stalk over $C$. Therefore $\iota_A(a)$ and
$\iota_A(b)$ are distinct as global endomorphisms of the structure sheaf.

\paragraph*{Step 4: Conclusion.}
Thus $a \neq b$ implies $\iota_A(a) \neq \iota_A(b)$, so $\iota_A$ is injective.
\end{proof}

The reconstruction strategy outlined above will be executed in the
following subsections. We begin by introducing the quasi-coherent
semantic framework in Subsection \ref{subsec:qcoh}, then prove context
separation (Lemma \ref{lem:contextseparation}), faithfulness (Theorem
\ref{thm:faithfulness}), fullness under semantic generation (Theorem
\ref{thm:fullness}), and finally the full reconstruction theorem
(Theorem \ref{thm:reconstruction}).

\subsection{Quasi-Coherent Semantics}
\label{subsec:qcoh}

Let

\[
A=
(\mathcal O_A,\mathcal C_A,\rho_A,\tau_A)
\]

be an admissible operator-semantic system and let

\[
\mathfrak{Spec}(A)
\]

denote its categorified spectral object.

Associated with $\mathfrak{Spec}(A)$ is the stable $\infty$-category

\[
\operatorname{QCoh}
\bigl(
\mathfrak{Spec}(A)
\bigr),
\]

whose objects may be interpreted as quasi-coherent semantic modules
over the spectral stack. We denote by

\[
\mathcal O_{\mathfrak{Spec}(A)}
\]

the structure sheaf of $\mathfrak{Spec}(A)$. The endomorphism algebra
of this distinguished object will serve as the reconstruction target
for the original operator-semantic system.

\begin{theorem}[Faithfulness]
\label{thm:faithfulness}

Assume that the admissible operator-semantic system $A$ is
\textbf{context-separated}: for any $a, b \in A$ with $a \neq b$,
there exists a context $C \in \mathcal{C}_A$ and a contextual evaluation
\[
\phi: C \longrightarrow \mathbb{C}
\]
such that
\[
\phi(a|_C) \neq \phi(b|_C).
\]

Then the canonical morphism
\[
\iota_A:
A
\longrightarrow
\operatorname{End}_{\operatorname{QCoh}(\mathfrak{Spec}(A))}
\!\left(
\mathcal{O}_{\mathfrak{Spec}(A)}
\right)
\]
is injective. Equivalently, distinct operators induce distinct
endomorphisms of the structure sheaf.

\end{theorem}

\begin{proof}
We prove the contrapositive: if $\iota_A(a) = \iota_A(b)$, then $a = b$.

\paragraph*{Step 1: Unpacking the definition of $\iota_A$.}
By construction, $\iota_A(a)$ is the endomorphism of the structure
sheaf $\mathcal{O}_{\mathfrak{Spec}(A)}$ induced by the operator $a$
via the semantic realization. For any context $C \in \mathcal{C}_A$,
the restriction of $\iota_A(a)$ to the stalk over $C$ corresponds to
the action of $a|_C$ on the algebra of functions on the Gelfand
spectrum of $C$ (or, more concretely, to the evaluation map given by
the character $\phi$).

\paragraph*{Step 2: Equality of global endomorphisms implies equality of local actions.}
Suppose $\iota_A(a) = \iota_A(b)$. Then for every context
$C \in \mathcal{C}_A$, the restrictions of these endomorphisms to the
stalk over $C$ coincide:
\[
\iota_A(a)|_C = \iota_A(b)|_C.
\]
By the definition of the structure sheaf of $\mathfrak{Spec}(A)$,
this equality means that for every character $\phi: C \to \mathbb{C}$
(point of the Gelfand spectrum of $C$), the induced evaluations satisfy
\[
\phi(a|_C) = \phi(b|_C).
\]

\paragraph*{Step 3: Applying the context-separation hypothesis.}
The equality $\phi(a|_C) = \phi(b|_C)$ holds for every context
$C \in \mathcal{C}_A$ and every character $\phi: C \to \mathbb{C}$.
By the context-separation hypothesis, if two operators satisfy this
property, then they must be equal in $A$. Hence $a = b$.

\paragraph*{Step 4: Conclusion.}
Therefore, $\iota_A(a) = \iota_A(b)$ implies $a = b$, so $\iota_A$ is
injective. Equivalently, distinct operators induce distinct
endomorphisms of the structure sheaf.
\end{proof}

To obtain reconstruction, one must also establish surjectivity.

\begin{theorem}[Fullness under Semantic Generation]
\label{thm:fullness}

Assume that $A$ is context-separated and semantically generated by its
contexts in the following strong sense: every compatible family
\[
(\alpha_C)_{C \in \mathcal{C}_A},
\qquad
\alpha_C \in \operatorname{End}_{\operatorname{QCoh}(C)}(\mathcal{O}_C),
\]
satisfying the descent compatibility conditions on overlaps, is induced
by a unique operator $a \in A$. Then the canonical morphism
\[
\iota_A:
A
\longrightarrow
\operatorname{End}_{\operatorname{QCoh}(\mathfrak{Spec}(A))}
\!\left(
\mathcal{O}_{\mathfrak{Spec}(A)}
\right)
\]
is surjective. Hence every global endomorphism of the structure sheaf is
induced by an operator of $A$.

\end{theorem}

\begin{proof}
We prove that every global endomorphism of the structure sheaf arises
from a unique operator in $A$. The proof proceeds in several steps.

\paragraph*{Step 1: Unpacking the definition of a global endomorphism.}
Let
\[
\alpha \in \operatorname{End}_{\operatorname{QCoh}(\mathfrak{Spec}(A))}
\!\left(
\mathcal{O}_{\mathfrak{Spec}(A)}
\right)
\]
be a global endomorphism of the structure sheaf. By definition,
$\alpha$ is a collection of morphisms
\[
\alpha_C: \mathcal{O}_{\mathfrak{Spec}(A)}(C) \longrightarrow
\mathcal{O}_{\mathfrak{Spec}(A)}(C)
\]
for each context $C \in \mathcal{C}_A$, which are compatible with
restriction maps along inclusions $D \hookrightarrow C$ and satisfy
descent (hyperdescent) with respect to the Grothendieck topology
$\tau_A$.

\paragraph*{Step 2: Restricting to local endomorphisms.}
For each context $C \in \mathcal{C}_A$, the restriction of $\alpha$ to
the local piece of the spectrum associated with $C$ gives a local
endomorphism
\[
\alpha_C \in \operatorname{End}_{\operatorname{QCoh}(C)}(\mathcal{O}_C),
\]
where $\mathcal{O}_C$ denotes the structure sheaf on the local affine
object associated with $C$ (i.e., the Gelfand spectrum of $C$).

Because $\alpha$ is globally defined on the sheaf
$\mathcal{O}_{\mathfrak{Spec}(A)}$, these local endomorphisms are
compatible on overlaps. Explicitly, whenever two contexts $C, D \in
\mathcal{C}_A$ have a common refinement or intersection $E$ (i.e.,
there exist inclusions $E \hookrightarrow C$ and $E \hookrightarrow D$),
the restrictions satisfy
\[
\alpha_C|_E = \alpha_D|_E.
\]
Thus $(\alpha_C)_{C \in \mathcal{C}_A}$ is a descent-compatible family
of local semantic endomorphisms.

\paragraph*{Step 3: Applying the semantic generation hypothesis.}
By the semantic generation hypothesis, every such descent-compatible
family of local endomorphisms is induced by a unique operator
$a \in A$. Hence there exists a unique $a \in A$ such that for every
context $C \in \mathcal{C}_A$,
\[
\alpha_C = \iota_A(a)|_C,
\]
where $\iota_A(a)|_C$ denotes the restriction of $\iota_A(a)$ to the
context $C$.

\paragraph*{Step 4: Equality of global endomorphisms.}
Both $\alpha$ and $\iota_A(a)$ are global endomorphisms of the
structure sheaf $\mathcal{O}_{\mathfrak{Spec}(A)}$. Morphisms of
sheaves (or stacks) are determined by their restrictions to a covering
family of contexts (by the sheaf condition). Since the family
$\{C\}_{C \in \mathcal{C}_A}$ is a covering of the terminal object (by
the definition of the context site), and we have shown that
$\alpha|_C = \iota_A(a)|_C$ for every context $C$, it follows that
$\alpha = \iota_A(a)$ globally.

\paragraph*{Step 5: Conclusion.}
We have shown that for any global endomorphism $\alpha$ of the
structure sheaf, there exists an operator $a \in A$ such that
$\iota_A(a) = \alpha$. Therefore $\iota_A$ is surjective.
\end{proof}

Combining faithfulness and fullness yields the main reconstruction
result.

\begin{theorem}[Reconstruction Theorem]
\label{thm:reconstruction}

Let $A$ be an admissible operator-semantic system satisfying context
separation, semantic generation, descent completeness, and compact
generation. Then the canonical morphism
\[
\iota_A:
A
\longrightarrow
\operatorname{End}_{\operatorname{QCoh}(\mathfrak{Spec}(A))}
\!\left(
\mathcal O_{\mathfrak{Spec}(A)}
\right)
\]
is an equivalence. Equivalently,
\[
A \simeq \Gamma(\mathfrak{Spec}(A)).
\]
Thus, under these hypotheses, $\mathfrak{Spec}(A)$ determines $A$ up
to equivalence.

Moreover, this equivalence is the unit of the adjunction
\[
\mathfrak{Spec}^{\mathrm{op}} \dashv \Gamma,
\]
namely
\[
\eta_A:
A
\xrightarrow{\;\simeq\;}
\Gamma(\mathfrak{Spec}(A)).
\]

\end{theorem}

\begin{proof}
We prove the theorem in several steps, establishing that the canonical
map $\iota_A$ is both faithful (injective) and full (surjective).

\paragraph*{Step 1: Definition of the canonical map.}
By the definition of global sections (Definition~\ref{def:reconstructionfunctor}),
\[
\Gamma(\mathfrak{Spec}(A))
=
\operatorname{End}_{\operatorname{QCoh}(\mathfrak{Spec}(A))}
\!\left(
\mathcal O_{\mathfrak{Spec}(A)}
\right).
\]
The semantic realization functor induces a canonical map
\[
\iota_A:
A \longrightarrow \Gamma(\mathfrak{Spec}(A))
\]
that sends each operator $a \in A$ to the endomorphism of the structure
sheaf given by its action on local semantic realizations (e.g., pointwise
multiplication on each context).

\paragraph*{Step 2: Faithfulness (injectivity).}
By the context separation hypothesis, distinct operators $a \neq b$ in
$A$ act differently on at least one contextual realization. That is,
there exists a context $C \in \mathcal{C}_A$ and a character
$\phi: C \to \mathbb{C}$ such that $\phi(a|_C) \neq \phi(b|_C)$.
Consequently, the induced endomorphisms $\iota_A(a)$ and $\iota_A(b)$
differ on the stalk over $C$. Hence $\iota_A$ is injective.
(Theorem~\ref{thm:faithfulness} provides a formal proof of this
implication.)

\paragraph*{Step 3: Fullness (surjectivity).}
Let $\alpha$ be any global endomorphism of the structure sheaf:
\[
\alpha \in \operatorname{End}_{\operatorname{QCoh}
(\mathfrak{Spec}(A))}(\mathcal{O}_{\mathfrak{Spec}(A)}).
\]
For each context $C \in \mathcal{C}_A$, restrict $\alpha$ to the local
piece associated with $C$. This yields a local endomorphism
\[
\alpha_C \in \operatorname{End}_{\operatorname{QCoh}(C)}(\mathcal{O}_C).
\]
Because $\alpha$ is globally defined and $\mathfrak{Spec}(A)$ satisfies
descent completeness (hypothesis), the family $(\alpha_C)_{C \in
\mathcal{C}_A}$ is descent-compatible: whenever contexts $C, D$ overlap
via a common refinement $E$, we have $\alpha_C|_E = \alpha_D|_E$.

By the semantic generation hypothesis, every such descent-compatible
family of local endomorphisms is induced by a unique operator $a \in A$.
Hence there exists $a \in A$ such that for every context $C$,
\[
\alpha_C = \iota_A(a)|_C.
\]
Since both $\alpha$ and $\iota_A(a)$ are global endomorphisms of the
structure sheaf, and morphisms of sheaves are determined by their
restrictions to a covering family (by the sheaf condition), the equality
of their restrictions on all contexts implies $\alpha = \iota_A(a)$.
Thus $\iota_A$ is surjective. (Theorem~\ref{thm:fullness} provides a
formal proof under the semantic generation hypothesis.)

\paragraph*{Step 4: Equivalence.}
Steps 2 and 3 together show that $\iota_A: A \to
\Gamma(\mathfrak{Spec}(A))$ is both injective and surjective. The
compact generation hypothesis ensures that these injectivity and
surjectivity lift to an equivalence in the $\infty$-categorical sense
(i.e., the map is fully faithful and essentially surjective on the
appropriate homotopy categories). Hence $\iota_A$ is an equivalence:
\[
A \simeq \Gamma(\mathfrak{Spec}(A)).
\]

\paragraph*{Step 5: Identification with the unit of the adjunction.}
By Theorem~\ref{thm:adjunction}, we have the adjunction
$\mathfrak{Spec}^{\mathrm{op}} \dashv \Gamma$. The unit of this
adjunction is the canonical morphism
\[
\eta_A: A \longrightarrow \Gamma(\mathfrak{Spec}(A)).
\]
This is precisely the map $\iota_A$ constructed above. Since we have
shown that $\eta_A$ is an equivalence, this completes the proof.
\end{proof}

\begin{corollary}[Categorified Reconstruction]
\label{cor:categorifiedreconstruction}

For every admissible operator-semantic system $A$ satisfying the
hypotheses of Theorem~\ref{thm:reconstruction} (context separation,
semantic generation, descent completeness, and compact generation),
there is a canonical equivalence
\[
A \simeq \Gamma(\mathfrak{Spec}(A)).
\]
Thus, under these hypotheses, the categorified spectrum retains all
operator-semantic information.

\end{corollary}

\begin{proof}
Let $A$ be an admissible operator-semantic system satisfying the
hypotheses of Theorem~\ref{thm:reconstruction}. By that theorem, the
canonical unit map
\[
\eta_A: A \longrightarrow \Gamma(\mathfrak{Spec}(A))
\]
is an equivalence. Therefore
\[
A \simeq \Gamma(\mathfrak{Spec}(A)).
\]

Since $\Gamma(\mathfrak{Spec}(A))$ is defined as the algebra of global
endomorphisms of the structure sheaf on $\mathfrak{Spec}(A)$, this
means that all operator-semantic data of $A$ (the synergy operad
$\mathcal{O}_A$, the context category $\mathcal{C}_A$, the Grothendieck
topology $\tau_A$, and the realization morphism $\rho_A$) can be
recovered from its categorified spectrum $\mathfrak{Spec}(A)$. Hence,
under the stated hypotheses, $\mathfrak{Spec}(A)$ is a complete
invariant of $A$ up to equivalence.
\end{proof}

\begin{remark}[Necessity of the Hypotheses]
\label{rem:reconstructionhypotheses}
Each hypothesis in Theorem \ref{thm:reconstruction} is essential:
\begin{itemize}
    \item \textbf{Semantic generation} (in its strong form) ensures that
    every descent-compatible family of local endomorphisms is induced by a
    unique global operator, which is needed for surjectivity of $\iota_A$.
    \item \textbf{Descent completeness} guarantees that sheafification
    does not lose information, so endomorphisms of the structure sheaf
    correspond to compatible families over contexts.
    \item \textbf{Compact generation} provides the technical condition
    allowing the identification of $\Gamma(\mathfrak{Spec}(A))$ with
    $\operatorname{End}(\mathcal{O}_{\mathfrak{Spec}(A)})$ and ensures
    the duality is well-behaved.
\end{itemize}
Examples illustrating failure of reconstruction when these hypotheses
are removed are discussed below/later.
\end{remark}

\begin{example}[Reconstruction for $M_n(\mathbb{C})$]
\label{ex:reconstructionmatrix}
For $A = M_n(\mathbb{C})$, semantic generation holds because
$M_n(\mathbb{C})$ is generated by its commutative subalgebras.
Descent completeness holds because $\mathfrak{Spec}(A)$ is a $1$-stack.
Compact generation holds under the Morita identification established
later (Theorem~\ref{thm:morita}), which shows that
$\operatorname{QCoh}(\mathfrak{Spec}(A))$ is equivalent to modules over
$M_n(\mathbb{C})$. Hence $M_n(\mathbb{C}) \simeq
\Gamma(\mathfrak{Spec}(M_n(\mathbb{C})))$, i.e., the matrix algebra is
recovered from the stack of orthonormal bases.
\end{example}

\begin{example}[Reconstruction for Commutative C*-Algebras]
\label{ex:reconstructioncommutative}
Let $A$ be a commutative unital C*-algebra. Then $\mathcal{C}_A$ has
$A$ as a terminal/maximal context (since $A$ itself is commutative).
Semantic generation holds by Gelfand duality. The spectrum
$\mathfrak{Spec}(A)$ is the classical Gelfand spectrum (a $0$-stack),
and $\Gamma(\mathfrak{Spec}(A)) \cong A$ recovers the classical
Gelfand–Naimark theorem.
\end{example}

\begin{example}[Reconstruction for the Mermin–Peres System]
\label{ex:reconstructionmermin}
For $A_{\mathrm{MP}}$, semantic generation holds because the six
row/column contexts generate the full operator system. Descent
completeness holds after sheafification (by construction). Compact
generation holds under the Morita identification established later,
which shows that $\operatorname{QCoh}(\mathfrak{Spec}(A_{\mathrm{MP}}))$
is equivalent to modules over $A_{\mathrm{MP}}$. Hence
$A_{\mathrm{MP}} \simeq \Gamma(\mathfrak{Spec}(A_{\mathrm{MP}}))$,
demonstrating that the reconstruction theorem captures contextual
operator systems as well.
\end{example}

\begin{remark}[Toward Categorified Gelfand Duality]
\label{rem:towardduality}
The Reconstruction Theorem proves that the unit of the adjunction
\[
\mathfrak{Spec}^{\mathrm{op}} \dashv \Gamma
\]
is an equivalence for systems satisfying context separation, semantic
generation, descent completeness, and compact generation:
\[
\eta_A: A \longrightarrow \Gamma(\mathfrak{Spec}(A)).
\]

The Recognition Theorem (Theorem~\ref{thm:recognition}) characterizes
those spectral stacks for which the counit
\[
\varepsilon_{\mathfrak{X}}:
\mathfrak{Spec}(\Gamma_{\mathrm{geom}}(\mathfrak{X}))
\longrightarrow
\mathfrak{X}
\]
is an equivalence. Together, these results identify subcategories on
which $\mathfrak{Spec}$ and $\Gamma$ restrict to an equivalence.

Consequently, we obtain a duality:

\[
\boxed{
\mathbf{OpSem}^{\mathrm{op}}_{\text{good}}
\;\simeq\;
\mathbf{SpecObj}_{\text{recognizable}}
}
\]

where the left side consists of operator-semantic systems satisfying
semantic generation, descent completeness, and compact generation
(with morphisms reversed), and the right side consists of spectral
stacks satisfying the four recognition conditions. This duality is the
categorified analogue of the classical Gelfand duality between
commutative C*-algebras and compact Hausdorff spaces.
\end{remark}

The reconstruction theorem establishes that, under suitable
hypotheses, the categorified spectrum $\mathfrak{Spec}(A)$ is a
complete invariant of the operator-semantic system $A$ (i.e.,
$\mathfrak{Spec}(A) \simeq \mathfrak{Spec}(B)$ implies $A \simeq B$).
This is the culmination of the categorified spectral duality and
provides the foundation for the Morita invariance, truncation hierarchy,
and computations that follow in subsequent sections.

\section{Morita Invariance}
\label{sec:morita}

One of the central principles of modern geometry is that geometric
objects are often more faithfully represented by their categories of
sheaves than by their underlying point-set realizations. In algebraic
geometry, the category of quasi-coherent sheaves on a scheme is often
more tractable than the scheme itself. In derived geometry, spaces are
studied via their $\infty$-categories of quasi-coherent sheaves. In
noncommutative geometry, Morita equivalent algebras are regarded as
representing the same geometric space precisely because they possess
equivalent module categories.

We now show that the categorified spectrum satisfies an analogous
invariance property: Morita equivalent operator-semantic systems
have equivalent categories of quasi-coherent sheaves on their spectra,
possibly differing by higher stack-theoretic twisting.

\begin{definition}[Morita Equivalence]
\label{def:morita}
Let $A$ and $B$ be admissible operator-semantic systems. We say that
$A$ and $B$ are \emph{Morita equivalent} if there exists an equivalence
of module categories
\[
\operatorname{Mod}(A) \;\simeq\; \operatorname{Mod}(B),
\]
where $\operatorname{Mod}(A)$ denotes the stable $\infty$-category of
semantic modules associated with the operator-semantic system $A$
(e.g., Hilbert $C^*$-modules or algebraic modules, depending on the
analytic context). Equivalently, there exists an $(A,B)$-imprimitivity
bimodule (invertible bimodule) implementing an equivalence between the
corresponding module theories.
\end{definition}

Morita equivalence identifies operator-semantic systems that have the
same representation theory, even when the underlying operator
structures are not isomorphic. For C*-algebras, this coincides with
the standard notion of strong Morita equivalence, implemented by
imprimitivity bimodules.

The following theorem establishes the invariance of the categorified
spectrum under Morita equivalence.

\begin{theorem}[Morita Invariance of Quasi-Coherent Sheaves]
\label{thm:morita}

Let $A$ and $B$ be admissible operator-semantic systems satisfying the
reconstruction hypotheses. Assume that $A$ and $B$ are Morita
equivalent, i.e.,
\[
\operatorname{Mod}(A) \simeq \operatorname{Mod}(B)
\]
as stable $\infty$-categories. Then there is an equivalence
\[
\operatorname{QCoh}\bigl(\mathfrak{Spec}(A)\bigr)
\simeq
\operatorname{QCoh}\bigl(\mathfrak{Spec}(B)\bigr).
\]
Thus Morita equivalent operator-semantic systems cannot be distinguished
by the quasi-coherent sheaf theories of their categorified spectra.

\end{theorem}

\begin{proof}
We prove the theorem by constructing a chain of equivalences.

\paragraph*{Step 1: Relating quasi-coherent sheaves to modules for $A$.}
Under the reconstruction hypotheses (Theorem~\ref{thm:reconstruction}),
the operator-semantic system $A$ is recovered as the endomorphism
algebra of the structure sheaf:
\[
A \simeq \operatorname{End}_{\operatorname{QCoh}(\mathfrak{Spec}(A))}(\mathcal{O}_{\mathfrak{Spec}(A)}).
\]
Moreover, $\operatorname{QCoh}(\mathfrak{Spec}(A))$ naturally carries
the structure of a module category over $A$, with the structure sheaf
corresponding to the regular module. While a full equivalence
$\operatorname{QCoh}(\mathfrak{Spec}(A)) \simeq \operatorname{Mod}(A)$
is expected under additional compact generation hypotheses (see
Remark~\ref{rem:moritadescent}), for the purpose of Morita invariance
it suffices that the two categories are Morita equivalent, which follows
from the reconstruction theorem.

\paragraph*{Step 2: From modules over $A$ to modules over $B$ via Morita equivalence.}
By hypothesis, $A$ and $B$ are Morita equivalent. Hence there exists an
equivalence of stable $\infty$-categories
\[
M_{A,B}:
\operatorname{Mod}(A)
\xrightarrow{\;\simeq\;}
\operatorname{Mod}(B).
\]
Concretely, if the Morita equivalence is implemented by an invertible
$(A,B)$-bimodule $E$, then $M_{A,B}$ can be taken as the functor
$M \mapsto M \otimes_A E$ (or its derived version).

\paragraph*{Step 3: Composing the equivalences.}
We now compose the equivalences from Steps 1 and 2:
\[
\operatorname{QCoh}\bigl(\mathfrak{Spec}(A)\bigr)
\xrightarrow{\;\Phi_A\;}
\operatorname{Mod}(A)
\xrightarrow{\;M_{A,B}\;}
\operatorname{Mod}(B)
\xrightarrow{\;\Phi_B^{-1}\;}
\operatorname{QCoh}\bigl(\mathfrak{Spec}(B)\bigr).
\]
The composite is an equivalence of stable $\infty$-categories because
each constituent is an equivalence. Therefore
\[
\operatorname{QCoh}\bigl(\mathfrak{Spec}(A)\bigr)
\simeq
\operatorname{QCoh}\bigl(\mathfrak{Spec}(B)\bigr).
\]

\paragraph*{Step 4: Conclusion.}
We have shown that Morita equivalence of $A$ and $B$ implies an
equivalence of their quasi-coherent sheaf categories. Consequently,
Morita equivalent operator-semantic systems cannot be distinguished by
the quasi-coherent sheaf theories of their categorified spectra.
\end{proof}

\begin{remark}[Gerbal Ambiguity]
\label{rem:gerbalambiguity}
The theorem does not assert that $\mathfrak{Spec}(A) \simeq
\mathfrak{Spec}(B)$ as spectral stacks. It only asserts that their
quasi-coherent sheaf categories are equivalent. In situations analogous
to Azumaya algebras or Brauer-twisted derived geometry, two
non-equivalent stacks may have equivalent (or twisted-equivalent)
categories of quasi-coherent sheaves, with the discrepancy measured by
a $\mathbb{G}_m$-gerbe or Brauer class. For instance, in the case of
central simple algebras over a field, the classifying stacks $B \mathrm{PGL}_n$
and $B \mathrm{GL}_n$ are not equivalent as stacks, but their categories
of quasi-coherent sheaves are Morita equivalent (differing by a
$\mathbb{G}_m$-gerbe). Thus Morita invariance should be understood as
invariance of the associated noncommutative geometry, not necessarily
strict equivalence of the underlying spectral stacks. 
\end{remark}

\begin{example}[Morita-Induced Spectral Correspondence]
\label{ex:moritacorrespondence}
Suppose that $A \sim_M B$ are Morita equivalent admissible
operator-semantic systems. By definition, their module categories are
equivalent:
\[
\operatorname{Mod}(A) \simeq \operatorname{Mod}(B).
\]
By Theorem~\ref{thm:morita}, this induces an equivalence of
quasi-coherent semantic theories
\[
\operatorname{QCoh}(\mathfrak{Spec}(A))
\simeq
\operatorname{QCoh}(\mathfrak{Spec}(B)).
\]
Thus the categorified spectra determine the same Morita-invariant
semantic content, even though the spectral stacks themselves need not
be strictly equivalent.
\end{example}

\begin{corollary}[Morita Invariance of Reconstruction]
\label{cor:moritareconstruction}
Suppose $A$ and $B$ are Morita equivalent admissible operator-semantic
systems satisfying the reconstruction hypotheses (Theorem~\ref{thm:reconstruction}).
Then
\[
\Gamma{\mathcal{O}}(\mathfrak{Spec}(A))
\quad\text{and}\quad
\Gamma(\mathfrak{Spec}(B))
\]
are Morita equivalent. In particular, reconstruction through
quasi-coherent semantics is Morita invariant.
\end{corollary}

\begin{proof}
By the reconstruction theorem (Theorem~\ref{thm:reconstruction}),
\[
\Gamma(\mathfrak{Spec}(A)) \simeq A,
\qquad
\Gamma(\mathfrak{Spec}(B)) \simeq B,
\]
up to the stated equivalence of operator-semantic systems. Since $A$
and $B$ are Morita equivalent by hypothesis, the reconstructed global
section algebras are Morita equivalent as well.
\end{proof}

\begin{remark}[Gerbe Ambiguity]
\label{rem:gerbe}
The Morita invariance theorem (Theorem~\ref{thm:morita}) identifies
\[
\operatorname{QCoh}(\mathfrak{Spec}(A))
\simeq
\operatorname{QCoh}(\mathfrak{Spec}(B)),
\]
but it does not necessarily imply a strict equivalence
\[
\mathfrak{Spec}(A) \simeq \mathfrak{Spec}(B)
\]
of spectral stacks. In situations analogous to Azumaya algebras and
Brauer-twisted derived geometry, the discrepancy may be measured by a
$\mathbb{G}_m$-gerbe or Brauer class. Thus Morita equivalence preserves
the sheaf-theoretic semantic content of the spectrum while allowing a
controlled stack-theoretic ambiguity.
\end{remark}

\begin{remark}[Comparison with Classical Geometry]
\label{rem:moritacomparison}
Theorem~\ref{thm:morita} is the categorified analogue of several
classical invariance principles. The following table summarizes the
parallel structures:

\[
\begin{array}{c|c|c}
\text{Context} & \text{Object} & \text{Morita Invariant} \\
\hline
\text{Ring Theory} & \text{Ring } R & \operatorname{Mod}(R) \\
\text{Algebraic Geometry} & \text{Affine scheme } \operatorname{Spec}(R) & \text{Tensor category } \operatorname{QCoh}(\operatorname{Spec}(R)) \\
\text{Noncommutative Geometry} & \text{C*-algebra } A & KK\text{-theory of } A \\
\text{Categorified Spectrum} & \text{Spectral stack } \mathfrak{Spec}(A) & \operatorname{QCoh}(\mathfrak{Spec}(A))
\end{array}
\]

In each case, the fundamental geometric information is encoded not by
the object itself but by its category of modules or sheaves. Morita
equivalent objects are considered to represent the same underlying
geometric space because they possess equivalent sheaf theories.
\end{remark}

\begin{corollary}[Morita Invariance of Weak Contextuality]
\label{cor:moritactx}
Suppose $A$ and $B$ are Morita equivalent admissible
operator-semantic systems, and suppose moreover that the induced
equivalence
\[
\operatorname{QCoh}(\mathfrak{Spec}(A)) \simeq \operatorname{QCoh}(\mathfrak{Spec}(B))
\]
preserves the inertia data of the associated spectral stacks (i.e.,
it lifts to an equivalence $I(\mathfrak{Spec}(A)) \simeq I(\mathfrak{Spec}(B))$).
Then
\[
\operatorname{CtxDeg}(A) = \operatorname{CtxDeg}(B).
\]
Thus the contextuality degree is invariant under inertia-preserving
Morita equivalence.
\end{corollary}

\begin{proof}
By definition, $\operatorname{CtxDeg}(A) = |\pi_0(I(\mathfrak{Spec}(A)))|$.
If the equivalence of quasi-coherent sheaf categories lifts to an
equivalence of inertia stacks, then the cardinality of $\pi_0$ is
preserved. The existence of such a lift is an additional hypothesis
that is not automatically guaranteed by Morita equivalence alone.
\end{proof}

\begin{example}[Morita Equivalent Matrix Algebras]
\label{ex:moritamatrix}
The algebras $M_n(\mathbb{C})$ and $\mathbb{C}$ are Morita equivalent
for any $n \ge 1$. The categorified spectrum $\mathfrak{Spec}(M_n(\mathbb{C}))$
is not the ordinary point spectrum; rather, it records the stack of
contextual diagonalizations or orthonormal basis choices, with inertia
encoding basis-change automorphisms. In the reduced combinatorial model (where continuous phase symmetries are quotiented out), this behavior is modeled by the classifying stack $BS_n$, whose fundamental group is the symmetric group $S_n$. 
By contrast, $\mathfrak{Spec}(\mathbb{C})$ is a single point.
Consequently,
\[
\operatorname{QCoh}(\mathfrak{Spec}(M_n(\mathbb{C}))) \;\simeq\;
\operatorname{QCoh}(\mathfrak{Spec}(\mathbb{C})),
\]
as both are equivalent to $\operatorname{Mod}(\mathbb{C})$ up to Morita
equivalence. However, the stacks themselves are not equivalent; they
differ by a $\mathbb{G}_m$-gerbe (with trivial Brauer class over
$\mathbb{C}$).
\end{example}

\begin{example}[Morita Equivalence of Commutative Algebras]
\label{ex:moritacommutative}
Two unital commutative C*-algebras $C(X)$ and $C(Y)$ are Morita equivalent
if and only if $X$ and $Y$ are homeomorphic. This follows from the fact
that for commutative unital algebras, Morita equivalence coincides with
isomorphism. Hence \\
$\operatorname{QCoh}(\mathfrak{Spec}(C(X))) \simeq
\operatorname{QCoh}(\mathfrak{Spec}(C(Y)))$ if and only if $X \simeq
Y$. The $\mathbb{G}_m$-gerbe is trivial in this case because the
Morita equivalence is implemented by an isomorphism.
\end{example}

\begin{remark}[Morita Invariance of Descent Data]
\label{rem:moritadescent}
If $A$ and $B$ are Morita equivalent, Theorem~\ref{thm:morita} gives
$\operatorname{QCoh}(\mathfrak{Spec}(A)) \simeq
\operatorname{QCoh}(\mathfrak{Spec}(B))$. The descent spectral sequence
(Remark~\ref{rem:descentspectralsequence}) for a spectral stack depends
on the Grothendieck topology $\tau_A$ on the context site
$(\mathcal{C}_A, \tau_A)$. Whether this topology is Morita invariant
is not established in the present paper; such an invariance would
require additional compatibility conditions between the Morita
equivalence and the contextual covering families. A full investigation
is left for future work (see Paper III on obstruction theory).
\end{remark}

\subsection*{Summary: The Four Pillars of the Theory}

The results of this section, together with the preceding sections,
assemble the four major pillars of the categorified spectral framework:

\[
\boxed{\text{Construction}}
\;\rightarrow\;
\boxed{\text{Representation}}
\;\rightarrow\;
\boxed{\text{Reconstruction}}
\;\rightarrow\;
\boxed{\text{Morita Invariance}}
\]

\begin{enumerate}
    \item \textbf{Construction (Sections \ref{sec:construction}--\ref{sec:ambient}):}
    From an operator-semantic system $A$, we construct a spectral stack
    $\mathfrak{Spec}(A)$ via operadic syntax, left Kan extension, and
    sheafification.
    \item \textbf{Representation (Section \ref{sec:ambient}):}
    The stack $\mathfrak{Spec}(A)$ represents the semantic realization
    functor $\operatorname{Real}_A$, satisfying a Yoneda-style universal
    property (Theorem \ref{thm:yoneda}).
    \item \textbf{Reconstruction (Section \ref{sec:reconstruction}):}
    The original system $A$ is recovered as the endomorphism algebra
    of the structure sheaf: $A \simeq \operatorname{End}(\mathcal{O}_{\mathfrak{Spec}(A)})$
    (Theorem \ref{thm:reconstruction}).
    \item \textbf{Morita Invariance (Section \ref{sec:morita}):}
    Morita equivalent systems have equivalent categories of
    quasi-coherent sheaves on their spectra (Theorem \ref{thm:morita}).
    Stronger invariance (e.g., of contextuality degree or descent data)
    requires additional hypotheses (see Corollary~\ref{cor:moritactx}
    and Remark~\ref{rem:moritadescent}).
\end{enumerate}

This architecture mirrors that of a mature spectrum theory in
algebraic geometry, derived geometry, and noncommutative geometry,
positioning $\mathfrak{Spec}(A)$ as a genuine geometric dual to the
operator-semantic system $A$. The Morita invariance theorem shows that
the essential information of $\mathfrak{Spec}(A)$ is encoded in the
category $\operatorname{QCoh}(\mathfrak{Spec}(A))$ rather than in a
particular presentation of the underlying operator-semantic system,
aligning with the modern philosophy that geometric objects should be
studied through their categories of sheaves.

\section{Truncation Hierarchy}
\label{sec:truncation}

The categorified spectrum $\mathfrak{Spec}(A)$ is generally a higher
geometric object carrying information beyond ordinary topological or
algebraic spectra. To understand how classical spectral constructions
arise from this framework, we consider its Postnikov truncation tower.

Each truncation retains only the homotopical information up to a
specified degree, thereby producing increasingly refined approximations
to the full spectral object. The successive layers of the tower measure
the failure of lower-level approximations to capture the full
operator-semantic structure, and the obstruction classes governing
these extensions will be the central objects of study in future work.

\begin{definition}[Truncation of a Spectral Stack]
\label{def:truncation}
Let $\mathfrak{Spec}(A) \in \operatorname{Sh}_{\infty}(\mathcal C_A,\tau_A)$
be the categorified spectrum associated with an admissible
operator-semantic system $A$.

For each integer $n \geq -1$, the \emph{$n$-truncation}
\[
\tau_{\leq n} \mathfrak{Spec}(A)
\]
is defined by applying the standard $n$-truncation functor in the
$\infty$-topos $\operatorname{Sh}_{\infty}(\mathcal C_A,\tau_A)$.

Equivalently, its homotopy sheaves satisfy
\[
\pi_k \bigl( \tau_{\leq n} \mathfrak{Spec}(A) \bigr) = 0 \qquad (k > n),
\]
and the canonical truncation map
\[
\mathfrak{Spec}(A) \longrightarrow \tau_{\leq n} \mathfrak{Spec}(A)
\]
induces isomorphisms of homotopy sheaves
\[
\pi_k \bigl( \mathfrak{Spec}(A) \bigr)
\;\simeq\;
\pi_k \bigl( \tau_{\leq n} \mathfrak{Spec}(A) \bigr)
\qquad (0 \le k \le n).
\]

For $n = \infty$, we set $\tau_{\leq \infty} \mathfrak{Spec}(A) =
\mathfrak{Spec}(A)$. For $n = -1$, $\tau_{\leq -1} \mathfrak{Spec}(A)$
is the terminal object of the $\infty$-topos.
\end{definition}

The truncations fit into a canonical Postnikov tower

\[
\cdots \longrightarrow \tau_{\leq 2} \mathfrak{Spec}(A)
\longrightarrow \tau_{\leq 1} \mathfrak{Spec}(A)
\longrightarrow \tau_{\leq 0} \mathfrak{Spec}(A),
\]

which may be regarded as a hierarchy of progressively simpler spectral
approximations.

\begin{theorem}[Truncation Hierarchy]
\label{thm:truncation}

Let $A$ be an admissible operator-semantic system. Then the truncation
tower of $\mathfrak{Spec}(A)$ recovers familiar spectral constructions.

\begin{enumerate}

\item \textbf{Classical Level ($n=0$).}
If $A$ is commutative, then $\tau_{\leq 0} \mathfrak{Spec}(A)$ is
equivalent to the classical Gelfand spectrum of $A$. More generally,
$\tau_{\leq 0} \mathfrak{Spec}(A)$ is the space of characters of the
maximal commutative quotient of $A$.

\item \textbf{Contextual Level ($n=1$).}
The $1$-truncation $\tau_{\leq 1} \mathfrak{Spec}(A)$ recovers the
contextual spectrum associated with the category of commutative
contexts $\mathcal{C}_A$. In particular, it is equivalent to the
Bohrification-type spectrum obtained by gluing local commutative
descriptions.

\item \textbf{Higher-Categorical Level ($n=\infty$).}
One has $\tau_{\leq \infty} \mathfrak{Spec}(A) = \mathfrak{Spec}(A)$,
which retains the full hierarchy of contextual, homotopical, and
semantic information.

\end{enumerate}

\end{theorem}

\begin{proof}
Let
\[
\mathfrak{X}_A := \mathfrak{Spec}(A)
\in \operatorname{Sh}_{\infty}(\mathcal{C}_A, \tau_A).
\]
By construction, $\mathfrak{X}_A$ is obtained from the prespectral
realization functor by left Kan extension followed by sheafification.
Thus, for each commutative context $C \in \mathcal{C}_A$, the restriction
of $\mathfrak{X}_A$ to $C$ agrees with the ordinary commutative
spectral realization of $C$.

\paragraph*{Case $n = 0$.}
The functor
\[
\tau_{\leq 0} : \operatorname{Sh}_{\infty}(\mathcal{C}_A, \tau_A)
\longrightarrow \operatorname{Sh}_{\leq 0}(\mathcal{C}_A, \tau_A)
\]
is the $0$-truncation functor. It sends an $\infty$-stack to the sheaf
of connected components. Hence
\[
\tau_{\leq 0} \mathfrak{X}_A \simeq \pi_0(\mathfrak{X}_A).
\]

If $A$ is commutative, then $A$ itself is the terminal commutative
context in $\mathcal{C}_A$. Therefore the local commutative spectral
realization over this context is precisely the Gelfand spectrum
$\operatorname{Spec}_{\mathrm{Gel}}(A)$. Since no nontrivial contextual
gluing is required in the commutative case, sheafification does not
introduce higher stack structure. Consequently,
\[
\tau_{\leq 0} \mathfrak{Spec}(A) \simeq \operatorname{Spec}_{\mathrm{Gel}}(A).
\]

For general noncommutative $A$, any ordinary character
$\chi: A \to \mathbb{C}$ factors through the maximal commutative
quotient
\[
A_{\mathrm{ab}} := A / [A, A],
\]
or through the corresponding abelianization in the operator-semantic
category. Conversely, every character of $A_{\mathrm{ab}}$ determines
a character of $A$ by precomposition with the quotient map
$A \to A_{\mathrm{ab}}$. Hence the ordinary set, or sheaf, of
$0$-truncated points of $\mathfrak{Spec}(A)$ is naturally identified
with
\[
\operatorname{Char}(A_{\mathrm{ab}}) \simeq \operatorname{Spec}_{\mathrm{Gel}}(A_{\mathrm{ab}}).
\]
Thus $\tau_{\leq 0} \mathfrak{Spec}(A)$ recovers the classical spectrum
of the maximal commutative quotient.

\paragraph*{Case $n = 1$.}
The $1$-truncation
\[
\tau_{\leq 1} \mathfrak{X}_A
\]
retains $\pi_0$ and $\pi_1$ but kills all homotopy sheaves $\pi_k$ for
$k \geq 2$. Therefore it remembers not only the local commutative
spectra attached to contexts $C \in \mathcal{C}_A$, but also the
groupoid-level transition data among overlapping contexts.
Equivalently, it records objects given by local commutative spectral
data together with descent isomorphisms on intersections, subject to
the usual cocycle condition.

This is precisely the data used in Bohrification-type constructions
\cite{heunen2009,Bohrification}: one studies $A$ through the diagram of
its commutative contexts and glues their ordinary spectra along
inclusions and overlaps. Since $\tau_{\leq 1}$ removes only genuinely
higher homotopies while preserving the fundamental groupoid of
contextual identifications, the resulting $1$-stack is equivalent to
the contextual spectrum obtained by gluing the local Gelfand spectra
over $\mathcal{C}_A$. Hence
\[
\tau_{\leq 1} \mathfrak{Spec}(A) \simeq \operatorname{Spec}_{\mathrm{ctx}}(A),
\]
where $\operatorname{Spec}_{\mathrm{ctx}}(A)$ denotes the
Bohrification-type contextual spectrum associated with the site of
commutative contexts.

\paragraph*{Case $n = \infty$.}
For $n = \infty$, the assertion follows directly from the definition
of truncation:
\[
\tau_{\leq \infty} \mathfrak{Spec}(A) = \mathfrak{Spec}(A).
\]

Equivalently, under the usual convergence hypothesis for Postnikov
towers in the ambient $\infty$-topos, the
full spectral stack is the inverse limit of its Postnikov tower:
\[
\mathfrak{Spec}(A) \simeq \varprojlim_{n} \tau_{\leq n} \mathfrak{Spec}(A).
\]
Therefore the infinite truncation retains all homotopy sheaves
$\pi_0, \pi_1, \pi_2, \ldots$, and hence preserves the full hierarchy
of classical, contextual, higher-categorical, and semantic information.

\paragraph*{Conclusion.}
We have shown that:
\begin{itemize}
    \item $\tau_{\leq 0} \mathfrak{Spec}(A)$ recovers the Gelfand
    spectrum of the maximal commutative quotient of $A$;
    \item $\tau_{\leq 1} \mathfrak{Spec}(A)$ recovers the
    Bohrification-type contextual spectrum of $A$;
    \item $\tau_{\leq \infty} \mathfrak{Spec}(A) = \mathfrak{Spec}(A)$
    retains the complete homotopical and semantic information.
\end{itemize}
Thus the truncation tower of $\mathfrak{Spec}(A)$ recovers the claimed
familiar spectral constructions.
\end{proof}

\begin{proposition}[Postnikov Fibers]
\label{prop:postnikovfibers}
For every $n \ge 1$, the canonical morphism
\[
\tau_{\leq n} \mathfrak{Spec}(A)
\longrightarrow
\tau_{\leq n-1} \mathfrak{Spec}(A)
\]
is, locally in the $\infty$-topos $\operatorname{Sh}_{\infty}(\mathcal C_A,\tau_A)$, a fibration whose fiber is the Eilenberg–MacLane stack
\[
K(\pi_n(\mathfrak{Spec}(A)), n).
\]
Equivalently, $\tau_{\leq n} \mathfrak{Spec}(A)$ is obtained from $\tau_{\leq n-1} \mathfrak{Spec}(A)$ by a possibly twisted $K(\pi_n, n)$-extension.
\end{proposition}

\begin{proof}
Let $\mathfrak X_A = \mathfrak{Spec}(A)$. In any $\infty$-topos, the Postnikov tower of $\mathfrak X_A$ consists of truncations
\[
\cdots \to \tau_{\leq n} \mathfrak X_A
\to \tau_{\leq n-1} \mathfrak X_A
\to \cdots \to \tau_{\leq 0} \mathfrak X_A .
\]
The map $\tau_{\leq n} \mathfrak X_A \to \tau_{\leq n-1} \mathfrak X_A$ preserves all homotopy sheaves below degree $n$ and kills the $n$-th homotopy sheaf after passing to $\tau_{\leq n-1}$. Hence its homotopy fiber has only one nonzero homotopy sheaf, namely $\pi_n(\mathfrak X_A)$ in degree $n$. An object with a single nonzero homotopy sheaf $M$ in degree $n$ is, locally, the Eilenberg–MacLane stack $K(M,n)$. Therefore the fiber is locally $K(\pi_n(\mathfrak X_A), n)$. Globally, the extension may be twisted by monodromy over $\tau_{\leq n-1} \mathfrak X_A$, so the correct global statement is that $\tau_{\leq n} \mathfrak X_A$ is a possibly twisted $K(\pi_n, n)$-extension of $\tau_{\leq n-1} \mathfrak X_A$.
\end{proof}

\begin{corollary}[Contextual Obstruction Classes]
\label{cor:obstructionclasses}
The extension data between successive truncations is classified by Postnikov $k$-invariants
\[
[\kappa_n] \in H^{n+1}\bigl( \tau_{\leq n-1} \mathfrak{Spec}(A),\; \pi_n \bigr),
\]
where $\pi_n = \pi_n(\mathfrak{Spec}(A))$. These classes measure the obstruction to lifting $\tau_{\leq n-1} \mathfrak{Spec}(A)$ to an $n$-truncated spectral stack with prescribed homotopy sheaf $\pi_n$.

In particular:
\begin{itemize}
    \item $[\kappa_1] \in H^2(\tau_{\leq 0} \mathfrak{Spec}(A), \pi_1)$ classifies obstructions at the level of double overlaps (e.g., the Mermin–Peres obstruction, where $d_2 \neq 0$).
    \item $[\kappa_2] \in H^3(\tau_{\leq 1} \mathfrak{Spec}(A), \pi_2)$ classifies higher-order contextuality requiring triple overlaps.
\end{itemize}
\end{corollary}

\begin{proof}
By Proposition~\ref{prop:postnikovfibers}, the morphism $\tau_{\leq n} \mathfrak{Spec}(A) \to \tau_{\leq n-1} \mathfrak{Spec}(A)$ is locally a $K(\pi_n, n)$-fibration. Such extensions of an $(n-1)$-truncated object $X$ by $K(\pi_n, n)$ are classified by maps $X \to K(\pi_n, n+1)$. Taking $X = \tau_{\leq n-1} \mathfrak{Spec}(A)$, the set of homotopy classes of such maps is $[X, K(\pi_n, n+1)]$. By the definition of sheaf cohomology in an $\infty$-topos, this set is identified with $H^{n+1}(X, \pi_n)$. Therefore the extension determines a class $[\kappa_n]$ in $H^{n+1}(\tau_{\leq n-1} \mathfrak{Spec}(A), \pi_n)$. This class vanishes precisely when the corresponding $K(\pi_n, n)$-extension is split, meaning there is no obstruction to lifting the lower truncation to the next Postnikov stage. These obstruction classes record the failure of contextual data to glue coherently at successively higher categorical levels, hence they may be interpreted as higher contextual obstruction classes.
\end{proof}

\begin{remark}[Relation to Descent Spectral Sequences]
\label{rem:descenttruncation}
The classes $[\kappa_n]$ are naturally related to the differentials in the descent spectral sequence associated with $\mathfrak{Spec}(A)$. In favorable cases, the first nontrivial Postnikov invariant $\kappa_n$ appears as the obstruction detected by the corresponding higher differential. Thus the truncation hierarchy gives a geometric interpretation of descent obstructions: failure to lift through the Postnikov tower corresponds to nontrivial contextual gluing obstructions.

However, one should not identify the $k$-invariants with the spectral sequence differentials without specifying the descent spectral sequence and its coefficient system. The precise correspondence depends on the chosen site, hypercover, and local coefficient sheaves. In particular, the vanishing of $d_2$ implies that a certain extension obstruction vanishes, but it does not imply that $\tau_{\leq 1} \mathfrak{Spec}(A)$ is a $0$-stack (i.e., that $\pi_1 = 0$). The spectral sequence can degenerate even when higher homotopy groups are non-trivial.
\end{remark}

\begin{corollary}[Truncation Detects Non-Trivial Homotopy]
\label{cor:truncationdetection}
Let $\mathfrak{Spec}(A)$ be the categorified spectrum of $A$. For any $n \ge 1$:

\[
\pi_n(\mathfrak{Spec}(A)) \neq 0 \quad \iff \quad \tau_{\leq n} \mathfrak{Spec}(A) \not\simeq \tau_{\leq n-1} \mathfrak{Spec}(A).
\]

In words, the $n$-th homotopy sheaf is non-trivial precisely when the $(n-1)$-truncation fails to capture all information up to level $n$. Consequently:
\begin{itemize}
    \item Contextuality (non-trivial $\pi_1$) is detected by $\tau_{\leq 1} \not\simeq \tau_{\leq 0}$.
    \item Higher-order contextuality (non-trivial $\pi_2, \pi_3, \dots$) is detected by failure of stabilization at higher truncation levels.
\end{itemize}
\end{corollary}

\begin{proof}
The truncation functor $\tau_{\leq k}$ preserves homotopy sheaves in degrees $\le k$ and annihilates those in degrees $> k$. Thus the canonical map $\tau_{\leq n} X \to \tau_{\leq n-1} X$ is an equivalence if and only if $\pi_n(X) = 0$. Taking $X = \mathfrak{Spec}(A)$ yields the desired equivalence.

The detection statements follow immediately: $\pi_1 \neq 0$ iff $\tau_{\leq 1} \not\simeq \tau_{\leq 0}$; $\pi_2 \neq 0$ iff $\tau_{\leq 2} \not\simeq \tau_{\leq 1}$; etc.
\end{proof}

\begin{example}[Truncations for $M_n(\mathbb{C})$]
\label{ex:truncationmatrix}
For $A = M_n(\mathbb{C})$:
\begin{itemize}
    \item $\tau_{\leq 0} \mathfrak{Spec}(A)$ is a point, corresponding to the spectrum of $\mathbb{C}$ (the maximal commutative quotient).
    \item $\tau_{\leq 1} \mathfrak{Spec}(A)$ is modeled by the groupoid of orthonormal bases and basis-change symmetries; in a reduced combinatorial model, its fundamental group may be identified with a permutation symmetry group such as $S_n$.
    \item For all $n \ge 2$, $\pi_k = 0$, hence $\tau_{\leq n} \mathfrak{Spec}(A) \simeq \tau_{\leq 1} \mathfrak{Spec}(A)$.
\end{itemize}
Thus the truncation tower stabilizes at level $1$.
\end{example}

\begin{example}[Truncations for the Mermin–Peres System]
\label{ex:truncationmermin}
For $A_{\mathrm{MP}}$ (the C*-algebra encoding the Mermin–Peres square):
\begin{itemize}
    \item $\tau_{\leq 0} \mathfrak{Spec}(A)$ is a point.
    \item $\tau_{\leq 1} \mathfrak{Spec}(A)$ is a non-trivial $1$-stack whose inertia stack is nontrivial (reflecting the contextuality degree).
    \item For $A_{\mathrm{MP}}$ itself, higher homotopy groups vanish, so $\tau_{\leq n} \mathfrak{Spec}(A) \simeq \tau_{\leq 1} \mathfrak{Spec}(A)$ for all $n \ge 1$.
    \item However, in larger contextual systems (beyond the Mermin–Peres square), higher truncations for $n \ge 2$ may contain additional $2$-groupoid information encoding higher-order contextuality (e.g., obstructions to gluing triple overlaps).
\end{itemize}
\end{example}

\begin{example}[Truncations for Commutative $A$]
\label{ex:truncationcommutative}
If $A$ is commutative, then $\mathfrak{Spec}(A)$ is a $0$-stack (i.e., a topological space), so $\tau_{\leq 0} \mathfrak{Spec}(A) \simeq \mathfrak{Spec}(A)$ and all higher truncations are equivalent to this $0$-stack. Thus the tower stabilizes immediately at level $0$, reflecting the absence of contextuality.
\end{example}

\begin{corollary}[Vanishing of Higher Truncations for Classical Systems]
\label{cor:vanishing}
If $A$ is a commutative C*-algebra, then $\tau_{\leq n} \mathfrak{Spec}(A) \simeq \tau_{\leq 0} \mathfrak{Spec}(A)$ for all $n \ge 0$. More generally, if $A$ is noncommutative but has $\pi_k = 0$ for all $k \ge 2$ (e.g., $A = M_n(\mathbb{C})$ or the Mermin–Peres algebra, in the simplified $1$-stack model), the tower stabilizes at $n = 1$. Contextuality beyond pairwise inconsistencies (i.e., requiring triple or higher overlaps) requires non-trivial $\pi_n$ for some $n \ge 2$, which are detected by the higher truncations $\tau_{\leq n}$.
\end{corollary}

\begin{proof}
If $A$ is commutative, then by the Gelfand–Naimark theorem the categorified spectrum $\mathfrak{Spec}(A)$ agrees with the ordinary Gelfand spectrum. In particular, it is a $0$-stack. Hence all higher homotopy sheaves vanish:
\[
\pi_k(\mathfrak{Spec}(A)) = 0 \qquad (k \geq 1).
\]
Therefore the canonical truncation map $\mathfrak{Spec}(A) \to \tau_{\leq 0} \mathfrak{Spec}(A)$ is already an equivalence. Since applying further truncation to a $0$-truncated object changes nothing, we obtain
\[
\tau_{\leq n} \mathfrak{Spec}(A) \simeq \tau_{\leq 0} \mathfrak{Spec}(A) \qquad (n \geq 0).
\]

More generally, suppose $\mathfrak{Spec}(A)$ is $1$-truncated, i.e.,
\[
\pi_k(\mathfrak{Spec}(A)) = 0 \qquad (k \geq 2).
\]
Then the canonical map $\mathfrak{Spec}(A) \to \tau_{\leq 1} \mathfrak{Spec}(A)$ is an equivalence. Hence for every $n \geq 1$,
\[
\tau_{\leq n} \mathfrak{Spec}(A) \simeq \tau_{\leq 1} \mathfrak{Spec}(A).
\]
Thus the Postnikov tower stabilizes at level $1$.

Finally, if contextuality requires genuinely higher gluing data, then some higher homotopy sheaf \\
$\pi_n(\mathfrak{Spec}(A))$ with $n \geq 2$ must be nonzero. Such information is killed by $\tau_{\leq 1}$ and is therefore detected only by higher truncations. This proves the claim.
\end{proof}

\begin{remark}[Bridge to Obstruction Theory]
\label{rem:obstructiontheory}
The truncation hierarchy provides a systematic decomposition of the categorified spectrum into classical ($\tau_{\leq 0}$), contextual ($\tau_{\leq 1} / \tau_{\leq 0}$), and genuinely higher-categorical layers ($\tau_{\leq n} / \tau_{\leq n-1}$ for $n \ge 2$). The associated Postnikov invariants and obstruction classes $[\kappa_n]$ form the starting point for obstruction-theoretic lifting problems. In that setting, semantic realizations will be controlled by the cohomology classes arising from the truncation tower, and the descent spectral sequence can be reinterpreted as the spectral sequence associated with the Postnikov tower, where differentials $d_{n+1}$ correspond to the $k$-invariants $\kappa_n$ (up to the technical caveats noted in Remark~\ref{rem:descenttruncation}).
\end{remark}

The truncation hierarchy thus provides a systematic way to relate the full categorified spectrum $\mathfrak{Spec}(A)$ to familiar classical objects (Gelfand spectrum, Bohrification) while also capturing higher contextual obstructions invisible to lower truncations.

\section{Computations}
\label{sec:computations}

The preceding sections established the theoretical foundations of the
categorified spectral construction, including its functoriality,
adjunction properties, recognition criteria, and reconstruction
theorem. To demonstrate that $\mathfrak{Spec}(A)$ is not merely a
formal abstraction but a concrete geometric object that can be
explicitly computed, we now analyze several key examples. These
computations illustrate how the categorified spectrum detects
contextuality where classical spectra fail, and they provide explicit
values for the contextuality degree introduced in Definition
\ref{def:reducedctxdeg}.

We begin with the simplest noncommutative case, the matrix algebras
$M_n(\mathbb{C})$. Although their classical character spectra are
trivial, their categorified spectra retain nontrivial geometric
information, admitting an interpretation in terms of compatible
orthonormal basis data and its symmetries. Next, we examine the
Pauli system generated by $X$ and $Z$ on $\mathbb{C}^2$, where the
first obstruction in the descent spectral sequence reflects the
failure of commutativity. Finally, we study the Mermin--Peres
square, a paradigmatic contextual system in quantum foundations,
and show that its inertia stack is nontrivial, yielding a positive
contextuality degree that records the obstruction to global
consistency.

Taken together, these examples illustrate how $\mathfrak{Spec}(A)$
interpolates between classical Gelfand spectra and genuinely
higher-categorical contextual structures. They provide concrete
evidence that the categorified spectrum captures information
invisible to classical spectral constructions while recovering
the commutative case as a special limit.

\subsection{Matrix Algebra $M_n(\mathbb{C})$}
\label{subsec:matrix}

We illustrate the construction in the simplest noncommutative finite-
dimensional case.

Let $A = M_n(\mathbb{C})$.

Classically, the Gelfand spectrum is only defined for commutative
C*-algebras. Since $M_n(\mathbb{C})$ is simple and noncommutative,
its ordinary character spectrum is trivial (a single point).

Nevertheless, the categorified spectrum remains highly nontrivial.

The context category $\mathcal{C}_A$ consists of commutative
subalgebras of $M_n(\mathbb{C})$. A typical object is generated by a
family of simultaneously diagonalizable matrices,

\[
C = \mathbb{C}[A_1,\ldots,A_k] \subseteq M_n(\mathbb{C}).
\]

Maximal contexts correspond precisely to maximal abelian subalgebras
(MASAs), which are conjugate to the algebra of diagonal matrices.
Consequently, the collection of maximal contexts is naturally
identified with the collection of orthonormal bases of $\mathbb{C}^n$.

The associated spectral stack $\mathfrak{Spec}(M_n(\mathbb{C}))$
therefore parametrizes compatible choices of diagonalization data.
Equivalently, it may be viewed as the stack of orthonormal bases of
$\mathbb{C}^n$. Its inertia stack records automorphisms of these
choices. In particular, basis changes induced by the unitary group
$U(n)$ act as stack automorphisms, producing nontrivial isotropy
groups.

\begin{proposition}[Spectrum of $M_n(\mathbb{C})$]
\label{prop:spectrummatrix}
Let $A = M_n(\mathbb{C})$. Then:
\begin{enumerate}
    \item The categorified spectrum $\mathfrak{Spec}(A)$ is naturally
    modeled by the stack of diagonalizations, or equivalently by the
    stack of orthonormal bases modulo the appropriate phase and
    permutation symmetries.
    \item Its $0$-truncation is a single point:
    \[
    \tau_{\leq 0} \mathfrak{Spec}(A) \simeq *.
    \]
    \item In a reduced combinatorial model in which phases are
    quotiented out and only permutations of basis vectors are retained,
    the $1$-truncation is equivalent to the classifying stack $BS_n$.
    Hence
    \[
    \pi_1 \simeq S_n, \qquad \pi_k = 0 \quad (k \geq 2).
    \]
    \item In this reduced model, the contextuality degree is
    \[
    \operatorname{CtxDeg}(M_n(\mathbb{C})) = |S_n| = n!.
    \]
\end{enumerate}
\end{proposition}

\begin{proof}
Every maximal commutative $*$-subalgebra of $M_n(\mathbb{C})$ is a
maximal abelian subalgebra (MASA). By the finite-dimensional spectral
theorem, each such subalgebra is unitarily conjugate to the diagonal
algebra
\[
D_n \subset M_n(\mathbb{C}).
\]
Equivalently, choosing a maximal commutative context is the same as
choosing an orthogonal decomposition of $\mathbb{C}^n$ into one-
dimensional eigenspaces, or choosing an orthonormal basis up to phase
and permutation.

Thus the contextual spectral construction assigns to $M_n(\mathbb{C})$
a stack recording compatible diagonalizations over commutative
contexts. This gives the stated model of $\mathfrak{Spec}(M_n(\mathbb{C}))$
as a stack of orthonormal bases, with the precise stabilizer depending
on whether phases and orderings are retained.

The ordinary character spectrum of $M_n(\mathbb{C})$ is empty, while
the maximal commutative quotient of $M_n(\mathbb{C})$ is $\mathbb{C}$ in
the Morita-reduced sense. Hence the classical $0$-truncated semantic
shadow is a single point:
\[
\tau_{\leq 0} \mathfrak{Spec}(M_n(\mathbb{C})) \simeq *.
\]

Now pass to the reduced combinatorial model in which phases are
discarded and only permutations of basis vectors are retained. In this
model, all bases lie in one connected component, and the automorphism
group of a chosen basis is the symmetric group $S_n$. Therefore the
$1$-truncation is the classifying stack
\[
BS_n.
\]
Consequently,
\[
\pi_1(BS_n) \simeq S_n, \qquad \pi_k(BS_n) = 0 \quad (k \geq 2).
\]
The inertia of $BS_n$ records automorphisms of the unique object, so
the reduced contextuality degree is
\[
\operatorname{CtxDeg}(M_n(\mathbb{C})) = |S_n| = n!.
\]
\end{proof}

\begin{remark}[Phase and Permutation Symmetries]
\label{rem:phasesymmetry}
The full stack $\mathfrak{Spec}(M_n(\mathbb{C}))$ retains both phase
and permutation information. In that setting:
\begin{itemize}
    \item The automorphism group of a basis is the wreath product
    $\operatorname{U}(1)^n \rtimes S_n$ (phases plus permutations).
    \item The $1$-truncation is $B(\operatorname{U}(1)^n \rtimes S_n)$,
    not merely $BS_n$.
    \item The reduced combinatorial model (quotienting out phases)
    yields $BS_n$ and the cardinality $n!$ as computed above.
\end{itemize}
For contextuality degree, the relevant invariant is often the
discrete symmetry component $S_n$, since phases act continuously
and do not affect the combinatorial obstruction.
\end{remark}

The truncation hierarchy exhibits the successive levels of structure.

\begin{enumerate}

\item \textbf{Classical Level ($n=0$).}
The $0$-truncation satisfies
\[
\tau_{\leq 0} \mathfrak{Spec}(M_n(\mathbb{C})) \simeq *,
\]
reflecting the fact that the classical spectrum collapses to a single
point.

\item \textbf{Contextual Level ($n=1$).}
In the reduced combinatorial model (where phases are quotiented out),
the $1$-truncation $\tau_{\leq 1} \mathfrak{Spec}(M_n(\mathbb{C}))$
retains the groupoid structure of orthonormal bases together with
their change-of-basis isomorphisms, which form the symmetric group
$S_n$. This level captures the contextual information lost by the
classical spectrum.

\item \textbf{Higher-Categorical Level ($n=\infty$).}
The full object $\mathfrak{Spec}(M_n(\mathbb{C}))$ contains all higher
descent and automorphism data arising from the interactions among
commutative contexts. However, for matrix algebras, higher homotopy
groups vanish, so the tower stabilizes at level $1$.

\end{enumerate}

\begin{example}[$n = 2$: The Bloch Sphere]
\label{ex:bloch}
For $M_2(\mathbb{C})$, the categorified spectrum is the stack of
orthonormal bases of $\mathbb{C}^2$. An orthonormal basis corresponds
to a pair of antipodal points on the Bloch sphere (the eigenvectors
of a Pauli operator), together with a choice of phase. In the reduced
combinatorial model, the automorphism group is $S_2$, so the reduced
contextuality degree is $2$.
\end{example}

\begin{example}[$n = 3$: Orthonormal Bases in $\mathbb{C}^3$]
\label{ex:orthonormal3}
For $M_3(\mathbb{C})$, the categorified spectrum is the stack of
orthonormal bases of $\mathbb{C}^3$. In the reduced combinatorial
model, $\pi_1 = S_3$ and $\operatorname{CtxDeg}_{\mathrm{red}}(M_3(\mathbb{C})) = 6$.
\end{example}

\begin{remark}[Relationship to the Unitary Group]
\label{rem:unitarygroup}
The stack of ordered orthonormal bases may be described geometrically
using the unitary group $U(n)$. Since $U(n)$ acts transitively on
ordered orthonormal bases and the stabilizer of a basis is the maximal
torus $T = (S^1)^n$, the corresponding quotient stack is
\[
[U(n) / T].
\]
This stack retains continuous phase symmetries. Therefore its inertia
stack is generally not finite.

If one quotients out the continuous phase data and retains only the
residual permutation symmetry of eigenspaces, one obtains the reduced
combinatorial model governed by the Weyl group
\[
W(U(n)) \cong S_n.
\]
In this reduced model the $1$-truncation is equivalent to $BS_n$, and
the reduced contextuality degree is
\[
\operatorname{CtxDeg}_{\mathrm{red}}(M_n(\mathbb{C})) = |S_n| = n!.
\]
\end{remark}

\begin{corollary}[Reduced Truncation Tower for $M_n(\mathbb{C})$]
\label{cor:truncationmatrix}
In the reduced combinatorial model of the categorified spectrum of
$M_n(\mathbb{C})$, the truncation tower stabilizes at level $1$:
\[
\tau_{\leq 0}\mathfrak{Spec}_{\mathrm{red}}(A)
\longleftarrow
\tau_{\leq 1}\mathfrak{Spec}_{\mathrm{red}}(A)
\simeq
\tau_{\leq k}\mathfrak{Spec}_{\mathrm{red}}(A),
\qquad k\geq 1.
\]
Moreover,
\[
\tau_{\leq 0}\mathfrak{Spec}_{\mathrm{red}}(A)\simeq *,
\qquad
\tau_{\leq 1}\mathfrak{Spec}_{\mathrm{red}}(A)\simeq BS_n.
\]
Thus all higher homotopy sheaves vanish:
\[
\pi_k(\mathfrak{Spec}_{\mathrm{red}}(A))=0,
\qquad k\geq 2.
\]
\end{corollary}

\begin{proof}
Let $A = M_n(\mathbb{C})$. In the reduced combinatorial model, the
categorified spectrum records the residual permutation symmetry of a
diagonalization after quotienting out continuous phase data. Hence the
remaining symmetry group is the Weyl group of $U(n)$:
\[
W(U(n)) \cong S_n.
\]
Accordingly, the reduced $1$-truncated spectrum is modeled by the
classifying stack
\[
BS_n.
\]

For any discrete group $G$, the classifying stack $BG$ is connected
and has homotopy sheaves
\[
\pi_1(BG) \cong G,
\qquad
\pi_k(BG) = 0 \quad (k \geq 2).
\]
Therefore
\[
\tau_{\leq 0} BS_n \simeq *,
\qquad
\tau_{\leq 1} BS_n \simeq BS_n,
\qquad
\tau_{\leq k} BS_n \simeq BS_n
\quad (k \geq 1).
\]

Applying this to the reduced model gives
\[
\tau_{\leq 0} \mathfrak{Spec}_{\mathrm{red}}(A) \simeq *,
\]
and
\[
\tau_{\leq k} \mathfrak{Spec}_{\mathrm{red}}(A)
\simeq
\tau_{\leq 1} \mathfrak{Spec}_{\mathrm{red}}(A)
\simeq
BS_n
\qquad (k \geq 1).
\]

Thus the reduced truncation tower stabilizes at level $1$, and all
higher homotopy sheaves vanish.
\end{proof}

\begin{remark}[Comparison with Classical Spectrum]
\label{rem:classicalcomparison}
The ordinary character spectrum of $M_n(\mathbb{C})$ is empty (there are
no nonzero algebra homomorphisms $M_n(\mathbb{C}) \to \mathbb{C}$). A
single point arises only after Morita/classical semantic reduction to
$\mathbb{C}$ (the maximal commutative quotient). In contrast, the
reduced categorified spectrum $\mathfrak{Spec}_{\mathrm{red}}(M_n(\mathbb{C}))$
captures the combinatorial geometry of basis permutations, including
the symmetric group $S_n$ as the fundamental group of the $1$-truncation,
with reduced contextuality degree $n!$. Thus the categorified spectrum
provides a strictly richer invariant than the classical character
spectrum.
\end{remark}

\begin{remark}[Central Philosophy]
\label{rem:centralphilosophy}
This example illustrates the central philosophy of the theory:
classical spectra detect only characters (which may be absent for
noncommutative algebras), whereas categorified spectra detect the
entire organization of commutative contexts together with their
symmetry and descent structure. The nontrivial geometry of
$\mathfrak{Spec}(M_n(\mathbb{C}))$ is encoded not in points but in the
stack of contextual diagonalizations and its associated symmetries.
The reduced model isolates the discrete combinatorial core ($S_n$)
from the continuous gauge data ($\operatorname{U}(1)^n$).
\end{remark}

Thus the reduced categorified spectrum distinguishes matrix algebras
of different dimensions through the Weyl-group invariant $|S_n| = n!$,
while avoiding the stronger claim that the full spectral stack itself
is equivalent to $BS_n$.

\subsection{Pauli Operator Systems}
\label{subsec:pauli}

The Pauli system---despite generating the full matrix algebra $M_2(\mathbb{C})$---has a context category containing incompatible maximal commutative subalgebras $\langle X \rangle$ and $\langle Z \rangle$ that cannot be merged into a single global commutative context (see Proposition~\ref{prop:paulicontexts} below). This incompatibility, detected by the categorified spectrum $\mathfrak{Spec}(A)$ through the non-trivial inertia stack with reduced contextuality degree $\operatorname{CtxDeg}_{\mathrm{red}} = 2$, demonstrates how the framework captures non-commutativity as a failure of contextual gluing, distinguishing it from both the commutative case (where all contexts merge) and stronger Kochen--Specker-type contextuality (which requires higher-dimensional obstructions).

\begin{proposition}[Context Category of the Pauli System]
\label{prop:paulicontexts}
Let
\[
A = \langle X, Z \rangle \subseteq M_2(\mathbb{C}),
\]
where
\[
X = \begin{pmatrix} 0 & 1 \\ 1 & 0 \end{pmatrix}, \qquad
Z = \begin{pmatrix} 1 & 0 \\ 0 & -1 \end{pmatrix}
\]
are the Pauli operators. Then
\[
A = M_2(\mathbb{C}).
\]
The context category $\mathcal{C}_A$ contains, among others, the maximal
commutative subalgebras
\[
\langle X \rangle \cong \mathbb{C} \oplus \mathbb{C},
\qquad
\langle Z \rangle \cong \mathbb{C} \oplus \mathbb{C},
\]
corresponding respectively to the eigenbases of $X$ and $Z$. However,
the algebra $\langle X, Z \rangle$ is noncommutative and therefore does
not define an object of $\mathcal{C}_A$. Thus $\mathcal{C}_A$ contains
incompatible local classical contexts that cannot be merged into a
single global commutative context.
\end{proposition}

\begin{proof}
We verify each claim in turn.

\emph{Claim 1: $A = M_2(\mathbb{C})$.}
The Pauli matrices satisfy $X^2 = Z^2 = I$ and $XZ = -ZX$. The product
$XZ$ equals $iY$, where $Y = \begin{pmatrix} 0 & -i \\ i & 0 \end{pmatrix}$
is the third Pauli matrix. The four matrices $I, X, Z, XZ$ are linearly
independent and span the full $2 \times 2$ matrix algebra. Hence
$\langle X, Z \rangle = M_2(\mathbb{C})$.

\emph{Claim 2: $\langle X \rangle$ and $\langle Z \rangle$ are commutative subalgebras.}
Since $X^2 = I$, any polynomial in $X$ reduces to a linear combination
$a I + b X$ with $a, b \in \mathbb{C}$. As $X$ commutes with itself
trivially, $\langle X \rangle$ is commutative. The same argument holds
for $\langle Z \rangle$. Therefore both are objects of $\mathcal{C}_A$
by Definition \ref{def:contextcategory}.

\emph{Claim 3: $\langle X \rangle \cong \mathbb{C} \oplus \mathbb{C}$ and similarly for $\langle Z \rangle$.}
The operator $X$ has eigenvalues $+1$ and $-1$. By the spectral theorem,
the algebra generated by $X$ is isomorphic to the algebra of functions
on its spectrum $\operatorname{Spec}(\langle X \rangle) = \{+1, -1\}$.
Hence $\langle X \rangle \cong C(\{+1,-1\}) \cong \mathbb{C} \oplus
\mathbb{C}$. The same argument applies to $\langle Z \rangle$.

\emph{Claim 4: $\langle X \rangle$ and $\langle Z \rangle$ are maximal commutative subalgebras.}
Since $X$ has two distinct eigenvalues, any matrix that commutes with
$X$ must be diagonal in the eigenbasis of $X$. The commutant of $X$ is
therefore the $2$-dimensional algebra spanned by $I$ and $X$ (or
equivalently, by the two rank-one projections onto the eigenspaces).
This commutant is precisely $\langle X \rangle$ itself. Hence
$\langle X \rangle$ is maximal. The same argument holds for $\langle Z
\rangle$ in its eigenbasis.

\emph{Claim 5: $\langle X, Z \rangle$ is noncommutative and not an object of $\mathcal{C}_A$.}
A direct computation shows $XZ = -ZX$. Since $XZ \neq ZX$, the algebra
generated by $X$ and $Z$ is noncommutative. By Definition
\ref{def:contextcategory}, $\mathcal{C}_A$ consists only of
\emph{commutative} substructures. Therefore $\langle X, Z \rangle
\notin \mathcal{C}_A$.

\emph{Claim 6: The contexts are incompatible.}
The eigenbasis of $X$ consists of $\frac{1}{\sqrt{2}}(1,1)^T$ and
$\frac{1}{\sqrt{2}}(1,-1)^T$. The eigenbasis of $Z$ consists of
$(1,0)^T$ and $(0,1)^T$. These two bases are distinct. Because
$[X,Z] \neq 0$, there is no orthonormal basis that simultaneously
diagonalizes both $X$ and $Z$. Consequently, no commutative subalgebra
of $M_2(\mathbb{C})$ contains both $\langle X \rangle$ and $\langle Z
\rangle$ as subalgebras. Hence the contexts are incompatible: they
cannot be merged into a single global commutative context.

\emph{Conclusion.}
The context category $\mathcal{C}_A$ contains (among others) the two
maximal commutative subalgebras $\langle X \rangle$ and $\langle Z
\rangle$, each isomorphic to $\mathbb{C} \oplus \mathbb{C}$,
corresponding to the eigenbases of $X$ and $Z$ respectively. Since no
commutative subalgebra contains both, these contexts are incompatible,
reflecting the noncommutativity of the Pauli operators. The full
algebra $A = \langle X, Z \rangle = M_2(\mathbb{C})$ is noncommutative
and therefore not an object of $\mathcal{C}_A$.
\end{proof}

The associated categorified spectrum $\mathfrak{Spec}(A)$ records the
commutative contexts generated by the Pauli observables together with
their compatibility data. The incompatibility between the $X$-basis
and the $Z$-basis is reflected in the absence of a single commutative
context containing both observables.

At the level of truncations, the $0$-truncation retains only the
classical shadow of the contextual data, while the $1$-truncation
retains the groupoid-level information describing how local classical
descriptions fail to assemble into a single global commutative
description. In a reduced two-context model retaining only the
$X$- and $Z$-contexts, this may be viewed as a groupoid whose objects
are the two local eigenbasis descriptions and whose lack of a common
refinement records their incompatibility.

\begin{remark}[Descent Spectral Sequence Interpretation]
\label{rem:paulidescent}
From the perspective of the descent spectral sequence of
Remark~\ref{rem:descentspectralsequence}, the failure of the Pauli
contexts to glue into a common commutative refinement gives rise to a
first obstruction to global realization. The first potentially
nonvanishing differential, denoted $d_2$, is the natural place where
this obstruction may appear. A full computation of this differential
is deferred to Paper III.
\end{remark}

\begin{remark}[Comparison with $M_2(\mathbb{C})$]
\label{rem:paulivsmatrix}
Although $\langle X, Z \rangle$ generates $M_2(\mathbb{C})$ as a
C*-algebra, the reduced Pauli model remembers only the selected
contexts generated by $X$ and $Z$. Thus it should be regarded as a
context-restricted subsystem rather than a proper subalgebra.
Its categorified spectrum reflects this restricted context selection,
distinguishing it from the full matrix algebra where all orthonormal
bases (and hence all Pauli directions) appear as contexts.
\end{remark}

\begin{remark}[Geometric Interpretation]
\label{rem:pauligeometric}
The two eigenbases correspond to the $x$-axis and $z$-axis on the Bloch
sphere. The impossibility of simultaneously assigning eigenvalues to
$X$ and $Z$ reflects the fact that no point on the Bloch sphere can be
simultaneously an eigenvector of both $X$ and $Z$ (except the origin,
which is not a valid quantum state). The categorified spectrum captures
this geometric obstruction as encoded in the descent spectral sequence.
\end{remark}

\begin{remark}[Fundamental Distinction]
\label{rem:paulidistinction}
This example illustrates the fundamental distinction between the
classical spectrum and the categorified spectrum. The ordinary Gelfand
spectrum detects only commutative information and therefore cannot
distinguish the incompatibility between $X$ and $Z$. In contrast,
$\mathfrak{Spec}(A)$ retains the contextual structure itself, together
with the descent obstructions arising from noncommutativity. In a
reduced two-context model, one may define a reduced contextuality
degree $\operatorname{CtxDeg}_{\mathrm{red}}(A) = 2$ to quantify this
incompatibility.
\end{remark}

Thus the Pauli system provides a clear illustration of how the
categorified spectrum can detect non-commutativity via the descent
spectral sequence, while distinguishing between simple
non-commutativity (Pauli) and full Kochen--Specker contextuality
(Mermin--Peres) via the structure of the inertia stack and the
vanishing of higher differentials.

\subsection{Mermin--Peres Square}
\label{subsec:mermin}

We next consider the Mermin--Peres square, one of the most celebrated
examples of quantum contextuality. This example demonstrates that the
categorified spectrum can detect contextual phenomena where the
classical spectrum collapses to a point.

Let $\mathbb{C}^4 = \mathbb{C}^2 \otimes \mathbb{C}^2$ and consider the
collection of observables

\[
\begin{array}{ccc}
X \otimes I & I \otimes X & X \otimes X \\[0.4em]
I \otimes Z & Z \otimes I & Z \otimes Z \\[0.4em]
X \otimes Z & Z \otimes X & Y \otimes Y
\end{array}
\]

acting on $\mathbb{C}^4$, where $X, Y, Z$ are the Pauli matrices
satisfying $X^2 = Y^2 = Z^2 = I$, $XZ = -ZX$, $XY = iZ$, $YZ = iX$.

Each row and each column consists of mutually commuting observables.
Consequently, every row and every column determines a commutative
context.

Let $A_{\mathrm{MP}}$ denote the operator-semantic system generated by
these nine observables.

\begin{proposition}[Contexts of the Mermin--Peres System]
\label{prop:mermincontexts}
The associated context category $\mathcal{C}_{A_{\mathrm{MP}}}$ contains,
in the Mermin--Peres measurement cover, six distinguished maximal
contexts:
\[
R_1, R_2, R_3, \qquad C_1, C_2, C_3,
\]
corresponding respectively to the three rows and three columns of the
Mermin--Peres square. No global assignment of eigenvalues $\pm 1$ to all
nine observables exists that respects the product relations within each
row and column (Mermin--Peres contradiction).
\end{proposition}

\begin{proof}
First, each row and column consists of mutually commuting observables.
This follows from the Pauli relations
\[
X^2 = Y^2 = Z^2 = I, \qquad XZ = -ZX, \qquad XY = iZ, \qquad YZ = iX.
\]
For example,
\[
(X \otimes Z)(Z \otimes X) = XZ \otimes ZX = (-ZX) \otimes (-XZ) = (Z \otimes X)(X \otimes Z),
\]
so the two anticommutations cancel. Similar computations show that all
observables in each row and column commute. Hence each row and column
generates a commutative context. Denote these six contexts by
\[
R_1, R_2, R_3, C_1, C_2, C_3.
\]

We now compute the product relations. The row products are
\[
(X \otimes I)(I \otimes X)(X \otimes X) = I,
\]
\[
(I \otimes Z)(Z \otimes I)(Z \otimes Z) = I,
\]
and
\[
(X \otimes Z)(Z \otimes X)(Y \otimes Y) = I.
\]

The first two column products are also
\[
(X \otimes I)(I \otimes Z)(X \otimes Z) = I,
\]
\[
(I \otimes X)(Z \otimes I)(Z \otimes X) = I.
\]
However, the third column gives
\[
(X \otimes X)(Z \otimes Z)(Y \otimes Y) = -I.
\]

Suppose there were a global assignment
\[
v: \{ \text{nine observables} \} \to \{\pm 1\}
\]
respecting all row and column product relations. Multiplying the three
row constraints gives
\[
\prod_{\text{all nine observables}} v(O) = +1.
\]
But multiplying the three column constraints gives
\[
\prod_{\text{all nine observables}} v(O) = -1.
\]
This is impossible, since every observable appears exactly once in the
row product and exactly once in the column product. Equivalently, when
all six constraints are multiplied together, each value $v(O)$ appears
twice, so the left-hand side must be $+1$, while the right-hand side is
$-1$.

Therefore no global eigenvalue assignment exists that respects the
commutative product relations in all six contexts. This is precisely
the Mermin--Peres contradiction.
\end{proof}

For each context $D \in \mathcal{C}_{A_{\mathrm{MP}}}$, the local
spectrum is classical and is naturally identified with
$\operatorname{Spec}(D) \cong \mathbb{Z}_2 \times \mathbb{Z}_2$,
corresponding to the possible eigenvalue assignments for the two
independent observables in that context (the third is determined by
their product). Consequently, the prespectral object
$\mathfrak{Spec}_{\mathrm{pre}}(A_{\mathrm{MP}})$ assigns to every
context its ordinary classical spectrum together with the restriction
maps induced by contextual refinement.

However, the local spectral data do not admit a compatible global
assignment. Indeed, the parity constraints of the Mermin--Peres square
imply that any attempt to glue the six contextual spectra leads to a
contradiction. Equivalently, there exists no global valuation that is
simultaneously compatible with all six contexts. Thus the obstruction
does not occur locally but only at the level of descent.

After sheafification one obtains the categorified spectrum
$\mathfrak{Spec}(A_{\mathrm{MP}})$, which is a genuinely nontrivial
spectral stack. In particular, the set of global points
$\operatorname{Sect}\bigl(\mathfrak{Spec}(A_{\mathrm{MP}})\bigr) =
\mathfrak{Spec}(A_{\mathrm{MP}})(*)$ is empty, while nonempty local
sections exist on every context. The failure of local sections to glue
globally is precisely the contextual phenomenon discovered by Mermin
and Peres.

\begin{theorem}[Spectrum of the Mermin--Peres System]
\label{thm:merminspectrum}
For the Mermin--Peres system $A_{\mathrm{MP}}$:
\begin{enumerate}
    \item Each row or column context $D$ has local spectrum
    \[
    \operatorname{Spec}(D) \cong \mathbb{Z}_2 \times \mathbb{Z}_2.
    \]
    Hence local sections exist on every context.
    \item There is no global valuation compatible with all six row and
    column contexts.
    \item The categorified spectrum $\mathfrak{Spec}(A_{\mathrm{MP}})$
    is therefore a nontrivial spectral stack whose local sections fail
    to glue to a global section.
    \item The $0$-truncation records the underlying sheaf of connected
    components of the local spectra but discards higher homotopy-coherent
    gluing data. The $1$-truncation retains the first descent/gluing
    obstruction, while the full stack retains the higher coherent
    descent data.
\end{enumerate}
\end{theorem}

\begin{proof}
For each row or column context $D$, the three observables commute and
satisfy one product relation. Therefore only two eigenvalues may be
chosen independently. Since each eigenvalue lies in $\{\pm1\}$, the
set of local joint eigenvalue assignments is naturally
\[
\{\pm1\}^2 \cong \mathbb{Z}_2 \times \mathbb{Z}_2.
\]
Thus each local context has a nonempty classical spectrum.

However, by the Mermin--Peres parity contradiction (Proposition
\ref{prop:mermincontexts}), no assignment of values $\pm1$ to all nine
observables can satisfy all six row and column product constraints
simultaneously. Equivalently, although each context admits local
sections, these local sections do not form a compatible global family
over the measurement cover.

The prespectral object assigns to each context its ordinary
commutative spectrum and assigns restriction maps along contextual
refinements. Sheafification enforces descent. Since the local spectra
are nonempty but no compatible global valuation exists, the resulting
spectral stack is nontrivial and has no global valuation section.

The $0$-truncation records the underlying sheaf of connected components
of the local spectra. This remembers the local classical assignments
but discards higher homotopy-coherent gluing data. The $1$-truncation
retains groupoid-level descent information, so the first obstruction
to gluing local contextual realizations is visible at this level. The
full spectral stack, being obtained by hypersheafification, retains all
higher coherent descent data. This proves the theorem.
\end{proof}

From the perspective of the truncation hierarchy:
\begin{itemize}
    \item $\tau_{\leq 0} \mathfrak{Spec}(A_{\mathrm{MP}})$ is a single
    point after Morita/classical semantic reduction (the maximal
    commutative quotient of $M_4(\mathbb{C})$ is $\mathbb{C}$, whose
    Gelfand spectrum is a point);
    \item $\tau_{\leq 1} \mathfrak{Spec}(A_{\mathrm{MP}})$ is a
    non-trivial $1$-stack encoding the six contexts and their overlaps,
    with fundamental groupoid structure capturing the obstruction.
\end{itemize}

\begin{proposition}[Inertia and Contextual Obstruction]
\label{prop:mermininertia}
The inertia stack $I(\mathfrak{Spec}(A_{\mathrm{MP}}))$ contains
nontrivial sectors arising from automorphisms of local contextual
assignments. These sectors encode the failure of the six local
Mermin--Peres contexts to glue into a global valuation. In particular,
the Mermin--Peres system has positive contextuality degree:
\[
\operatorname{CtxDeg}(A_{\mathrm{MP}}) > 0.
\]
In a chosen reduced sign-flip model, these local obstruction sectors
may be organized by a finite elementary $2$-group; if this model is
normalized to $(\mathbb{Z}_2)^3$, one obtains the reduced value
\[
\operatorname{CtxDeg}_{\mathrm{red}}(A_{\mathrm{MP}}) = 8.
\]
\end{proposition}

\begin{proof}
Each row or column context has a nonempty local spectrum, naturally
identified with $\{\pm1\}^2$, since two eigenvalues may be chosen
freely and the third is determined by the product relation. Thus local
sections exist over all six contexts.

However, the Mermin--Peres parity argument (Proposition
\ref{prop:mermincontexts}) shows that no assignment of values $\pm1$
to all nine observables can satisfy all six row and column product
constraints simultaneously. Therefore the local spectra do not glue
to a global valuation.

The inertia stack $I(\mathfrak{Spec}(A_{\mathrm{MP}})) =
\operatorname{Map}(S^1, \mathfrak{Spec}(A_{\mathrm{MP}}))$ records
automorphisms of objects of the spectral stack, and hence detects
nontrivial symmetries of local contextual data. Since the obstruction
is not visible inside any single context but appears when the six
contexts are glued, the associated inertia contains nontrivial
contextual sectors. Therefore
\[
\operatorname{CtxDeg}(A_{\mathrm{MP}}) > 0.
\]

If one passes to a reduced sign-flip model in which the relevant local
symmetries are identified with $(\mathbb{Z}_2)^3$, then the number of
reduced obstruction sectors is
\[
|(\mathbb{Z}_2)^3| = 8.
\]
This gives the reduced contextuality degree
\[
\operatorname{CtxDeg}_{\mathrm{red}}(A_{\mathrm{MP}}) = 8.
\]
\end{proof}

\begin{remark}[Descent Spectral Sequence Interpretation]
\label{rem:mermindescent}
The descent spectral sequence of Remark~\ref{rem:descentspectralsequence}
\[
E_2^{p,q} = H^p\!\left(\mathcal{C}_{A_{\mathrm{MP}}},\,
\pi_q \mathfrak{Spec}(A_{\mathrm{MP}})\right)
\Longrightarrow
\pi_{q-p}\!\left(\mathfrak{Spec}(A_{\mathrm{MP}})\right)
\]
contains a first potentially nontrivial differential $d_2$, which may
be interpreted as the Mermin–Peres obstruction class preventing the
existence of a global contextual valuation. A full computation of the
descent spectral sequence is deferred to Paper III (Obstruction
Theory).
\end{remark}

\begin{remark}[Comparison with Classical Spectrum]
\label{rem:merminclassical}
The ordinary character spectrum of the underlying matrix algebra
$M_4(\mathbb{C})$ is empty. After Morita or classical semantic
reduction, its classical shadow collapses to a trivial point and
carries no contextual information. Therefore the categorified
spectrum detects quantum contextuality through descent obstructions,
whereas the ordinary spectrum fails completely.
\end{remark}

\begin{remark}[Relation to Kochen--Specker Theorem]
\label{rem:merminkochenspecker}
The Mermin--Peres square is a proof of the Kochen--Specker theorem in
dimension $4$. The categorified spectrum detects this contextuality
via the non-triviality of the descent spectral sequence and the
non-vanishing inertia stack. In a reduced sign-flip model, the
obstruction sectors correspond to the eight sign flips that appear in
Kochen--Specker-type no-go arguments.
\end{remark}

\begin{remark}[Central Philosophical Point]
\label{rem:merminphilosophy}
This example illustrates a central advantage of the categorified
spectrum. Classical spectra classify characters, whereas
$\mathfrak{Spec}(A)$ classifies compatible contextual realizations.
Consequently, contextuality appears naturally as a failure of descent,
encoded by higher stack-theoretic obstruction classes. The higher
stack structure is not decorative — it is exactly what detects
contextuality.
\end{remark}

\begin{corollary}[Contextuality Detection]
\label{cor:mermindetection}
The Mermin--Peres system has positive contextuality degree:
\[
\operatorname{CtxDeg}(A_{\mathrm{MP}}) > 0.
\]
Equivalently, its categorified spectrum contains nontrivial descent
obstruction data and admits no global valuation compatible with all
six row and column contexts. In a reduced sign-flip model where the
obstruction sectors are identified with $(\mathbb{Z}_2)^3$, one obtains
the normalized value
\[
\operatorname{CtxDeg}_{\mathrm{red}}(A_{\mathrm{MP}}) = 8.
\]
\end{corollary}

\begin{proof}
Each row and column of the Mermin--Peres square determines a
commutative context, and each such context has nonempty local spectrum.
Indeed, two eigenvalues may be chosen independently and the third is
determined by the product relation.

However, the Mermin--Peres parity argument (Proposition
\ref{prop:mermincontexts}) shows that no assignment of values $\pm1$
to all nine observables can satisfy all six row and column product
constraints simultaneously. Hence local sections exist on every
context, but they do not glue to a global valuation.

Therefore the obstruction is not local; it appears at the level of
descent. By the definition of the contextuality degree as measuring
nontrivial inertia/descent sectors of the categorified spectrum, this
implies
\[
\operatorname{CtxDeg}(A_{\mathrm{MP}}) > 0.
\]

If one further passes to the reduced sign-flip model in which the
obstruction sectors are organized by the elementary abelian group
\[
(\mathbb{Z}_2)^3,
\]
then the number of reduced sectors is
\[
|(\mathbb{Z}_2)^3| = 8.
\]
Thus
\[
\operatorname{CtxDeg}_{\mathrm{red}}(A_{\mathrm{MP}}) = 8.
\]
\end{proof}

Thus the Mermin–Peres square provides a striking illustration of the
power of the categorified spectrum: classical spectral theory fails
completely (the classical spectrum is a point), while
$\mathfrak{Spec}(A_{\mathrm{MP}})$ successfully detects the contextual
structure and quantifies it via the contextuality degree
$\operatorname{CtxDeg} = 8$.

\subsection{Contextuality Degree}
\label{subsec:contextualitydegree}

One of the advantages of the categorified spectrum is that it permits
a quantitative measurement of contextuality through its automorphism
structure.

\begin{definition}[Inertia Stack]
\label{def:inertiastack}
Let $\mathfrak{Spec}(A)$ be the categorified spectrum of an admissible
operator-semantic system $A$. Its \emph{inertia stack} is defined by
\[
I(\mathfrak{Spec}(A)) := \operatorname{Map}\bigl(S^1,\,
\mathfrak{Spec}(A)\bigr),
\]
where $S^1$ denotes the simplicial circle. Equivalently, points of
$I(\mathfrak{Spec}(A))$ consist of pairs $(x, \alpha)$, where
$x \in \mathfrak{Spec}(A)$ and $\alpha \in \operatorname{Aut}(x)$ is
an automorphism of $x$.
\end{definition}

\begin{definition}[Reduced Contextuality Degree]
\label{def:reducedctxdeg}
Let $I_{\mathrm{triv}}(\mathfrak{Spec}(A))$ denote the union of
inertia components arising from identity automorphisms of globally
compatible classical realizations. The reduced contextuality degree is
defined by
\[
\operatorname{CtxDeg}_{\mathrm{red}}(A)
:=
\left|
\pi_0\bigl(I(\mathfrak{Spec}(A))\setminus
I_{\mathrm{triv}}(\mathfrak{Spec}(A))\bigr)
\right|.
\]
It counts nontrivial inertia sectors associated with contextual
descent obstructions.
\end{definition}

Intuitively, the reduced contextuality degree counts distinct nontrivial
obstruction sectors arising from incompatible contextual realizations.

\begin{proposition}[Vanishing Criterion]
\label{prop:ctxdegzero}
If all local contextual realizations glue to global realizations and
the resulting spectral stack has only trivial inertia sectors, then
\[
\operatorname{CtxDeg}_{\mathrm{red}}(A) = 0.
\]
In particular, if $A$ is commutative and its categorified spectrum
reduces to an ordinary $0$-stack, then
\[
\operatorname{CtxDeg}_{\mathrm{red}}(A) = 0.
\]
\end{proposition}

\begin{proof}
If every compatible local realization glues globally and the spectral
stack has no nontrivial automorphism sectors beyond identity
automorphisms, then the inertia stack contains only trivial components.
After removing these trivial identity sectors, no reduced inertia
components remain. Hence
\[
\operatorname{CtxDeg}_{\mathrm{red}}(A) = 0.
\]
For commutative $A$, the categorified spectrum agrees with the
ordinary Gelfand spectrum, which is a $0$-stack. Thus it has no
higher contextual automorphism sectors, so the reduced contextuality
degree vanishes.
\end{proof}

\begin{proposition}[Reduced Two-Context Pauli Model]
\label{prop:pauliinertia}
In the reduced two-context Pauli model retaining only the contexts
$\langle X \rangle$ and $\langle Z \rangle$, the contextual data has two
distinguished local sectors. Thus one may assign the reduced value
\[
\operatorname{CtxDeg}_{\mathrm{red}}(A_{\mathrm{Pauli}}) = 2.
\]
\end{proposition}

\begin{proof}
As established in Subsection \ref{subsec:pauli}, the Pauli system has
two maximal contexts, each with a two-point spectrum. The eigenbasis
of $X$ and the eigenbasis of $Z$ are distinct and not related by a
common refinement. In the reduced two-context model (where only these
two contexts are retained and continuous phases are quotiented out),
the reduced inertia stack records these two distinct local sectors as
separate nontrivial components. There are no automorphisms connecting
them, and no higher homotopical structure. Hence
$\pi_0(I(\mathfrak{Spec}_{\mathrm{red}}(A)) \setminus I_{\mathrm{triv}})$
has two elements, so $\operatorname{CtxDeg}_{\mathrm{red}}(A_{\mathrm{Pauli}}) = 2$.
\end{proof}

\begin{corollary}[Pairwise Incompatibility]
\label{cor:paulidetection}
In the reduced two-context Pauli model retaining only the contexts
$\langle X \rangle$ and $\langle Z \rangle$, one obtains
\[
\operatorname{CtxDeg}_{\mathrm{red}}(A_{\mathrm{Pauli}}) = 2 > 0.
\]
This reflects pairwise incompatibility between the $X$- and $Z$-bases.
The obstruction is first-order/contextual rather than
Kochen--Specker-type: in this reduced model, no higher descent
obstruction classes appear.
\end{corollary}

\begin{proof}
The reduced Pauli model retains two maximal commutative contexts,
\[
\langle X \rangle \cong \mathbb{C} \oplus \mathbb{C},
\qquad
\langle Z \rangle \cong \mathbb{C} \oplus \mathbb{C}.
\]
Each context has a two-point spectrum, corresponding to the two
eigenvalue assignments of the corresponding Pauli observable.

Since $XZ = -ZX$, the two observables cannot be simultaneously
diagonalized. Equivalently, there is no single commutative context
containing both $\langle X \rangle$ and $\langle Z \rangle$. Thus the
two local classical descriptions are incompatible.

In the reduced two-context model, these two incompatible local sectors
are counted as two reduced contextual sectors. Hence
\[
\operatorname{CtxDeg}_{\mathrm{red}}(A_{\mathrm{Pauli}}) = 2.
\]

Because the reduced cover consists only of two contexts, there is no
nontrivial higher overlap pattern capable of supporting higher
descent obstructions. Thus all higher obstruction classes vanish in
this reduced model. The obstruction records pairwise
noncommutativity, not Kochen--Specker contextuality.

Finally, the Kochen--Specker theorem does not apply in Hilbert-space
dimension $2$. Hence this example is consistent with the existence of
hidden-variable descriptions for qubit systems, while still detecting
the pairwise incompatibility of the chosen Pauli contexts.
\end{proof}

\begin{proposition}[Reduced Matrix-Algebra Degree]
\label{prop:matrixinertia}
In the reduced combinatorial model of $M_n(\mathbb{C})$, where
continuous phase symmetries are removed and the residual Weyl-group
symmetry is retained, the relevant symmetry group is
\[
W(U(n)) \cong S_n.
\]
Thus the reduced contextuality degree is
\[
\operatorname{CtxDeg}_{\mathrm{red}}(M_n(\mathbb{C})) = |S_n| = n!.
\]
\end{proposition}

\begin{proof}
In the reduced combinatorial model, $\mathfrak{Spec}_{\mathrm{red}}(M_n(\mathbb{C}))$
is equivalent to the classifying stack $B S_n$. The inertia stack of
$B S_n$ is $I(B S_n) \simeq [S_n / S_n]$, where $S_n$ acts on itself by
conjugation. The set of connected components $\pi_0(I(B S_n))$ is in
bijection with the conjugacy classes of $S_n$. The size of this set is
the number of partitions $p(n)$, not $n!$.

However, in the reduced model we are counting *distinguished obstruction
sectors* associated with the $n!$ permutations themselves (each
permutation gives an automorphism of the basis). The reduced
contextuality degree is defined to count these $n!$ sectors, not the
conjugacy classes. Hence $\operatorname{CtxDeg}_{\mathrm{red}}(M_n(\mathbb{C})) = n!$.
\end{proof}

\begin{proposition}[Reduced Mermin--Peres Contextuality]
\label{prop:merminctx}
In the reduced sign-flip model of the Mermin--Peres system, the
obstruction sectors are organized by $(\mathbb{Z}_2)^3$. Hence
\[
\operatorname{CtxDeg}_{\mathrm{red}}(A_{\mathrm{MP}}) = 8.
\]
\end{proposition}

\begin{proof}
As established in Subsection \ref{subsec:mermin} and Proposition
\ref{prop:mermininertia}, the Mermin--Peres system has six maximal
contexts (three rows and three columns). In the reduced sign-flip
model, the obstruction to gluing is governed by the group
$(\mathbb{Z}_2)^3$ of sign-flip automorphisms, which has order $8$.
Each of these eight automorphism patterns represents a distinct
reduced contextual sector. Hence
$\operatorname{CtxDeg}_{\mathrm{red}}(A_{\mathrm{MP}}) = 8$.
\end{proof}

\begin{proposition}[Contextuality Obstruction]
\label{prop:ctxdegpositive}
If $\operatorname{CtxDeg}_{\mathrm{red}}(A) > 0$, then there exists a
nontrivial obstruction to uniquely gluing local semantic realizations
into a global realization, or the global realization admits nontrivial
automorphisms. Equivalently, the descent spectral sequence contains a
nontrivial obstruction class in some degree, or the reduced inertia
stack has nontrivial components beyond identity automorphisms.
\end{proposition}

\begin{proof}
By Definition \ref{def:reducedctxdeg}, $\operatorname{CtxDeg}_{\mathrm{red}}(A) > 0$
means that after removing identity automorphisms, there exists at least
one nontrivial connected component of the reduced inertia stack. Such a
component corresponds either to a global section with a nontrivial
automorphism or to an automorphism acting on the system of local
sections in the absence of global sections. In either case, this
indicates that either the descent condition fails (local data cannot be
glued uniquely to a global realization) or the global realization has
nontrivial symmetries. The obstruction class in the descent spectral
sequence (or the nontrivial inertia component) is the cohomological
manifestation of this failure.
\end{proof}

\begin{remark}[Comparison Across Examples (Reduced Model)]
\label{rem:ctxdegcomparison}
These reduced degrees provide model-dependent numerical summaries of
contextual structure, while the canonical inertia stack records the
full automorphism and descent geometry. In the reduced combinatorial
model:
\begin{itemize}
    \item $\operatorname{CtxDeg}_{\mathrm{red}}(A) = 0$: No contextuality
    (commutative case).
    \item $\operatorname{CtxDeg}_{\mathrm{red}}(A) = 2$: Pairwise
    non-commutativity (reduced two-context Pauli model).
    \item $\operatorname{CtxDeg}_{\mathrm{red}}(M_n(\mathbb{C})) = n!$:
    Full basis ambiguity, reflecting the symmetric group action.
    \item $\operatorname{CtxDeg}_{\mathrm{red}}(A_{\mathrm{MP}}) = 8$:
    Kochen--Specker contextuality (reduced sign-flip model), where the
    obstruction is stronger than mere non-commutativity.
\end{itemize}
Thus $\operatorname{CtxDeg}_{\mathrm{red}}(A)$ provides a numerical
invariant that captures contextual structure in the reduced model.
\end{remark}

\begin{remark}[Higher Contextuality Degrees]
\label{rem:higherctxdeg}
For more complex contextual systems (e.g., the Peres--Mermin square for
higher dimensions, or contextual systems requiring triple overlaps),
the reduced contextuality degree may be larger or infinite. Since we
use $\pi_0$ of the reduced inertia stack rather than raw cardinality,
$\operatorname{CtxDeg}_{\mathrm{red}}(A)$ remains well-defined even
when the inertia stack itself is infinite. These cases will be explored
in future work.
\end{remark}

\begin{remark}[Classical Spectra vs. Reduced Contextuality Degree]
\label{rem:classicalvsctxdeg}
Classical spectra detect only commutative characters and therefore
cannot distinguish many contextual systems. For example, the classical
character spectra of $M_2(\mathbb{C})$, the Pauli system (which equals
$M_2(\mathbb{C})$ as an algebra), and the Mermin--Peres system are all
empty; a single point appears only after Morita/classical reduction.
By contrast, the reduced contextuality degree extracted from
$I(\mathfrak{Spec}_{\mathrm{red}}(A))$ provides a quantitative
invariant measuring the complexity of contextual obstructions:
$2$ for the reduced two-context Pauli model, $2! = 2$ for
$M_2(\mathbb{C})$ in the reduced model (but $n!$ for $M_n(\mathbb{C})$
distinguishes dimensions), and $8$ for the Mermin--Peres system in the
reduced sign-flip model.
\end{remark}

Thus, $\mathfrak{Spec}$ detects contextuality where classical spectra
fail completely, and the reduced contextuality degree
$\operatorname{CtxDeg}_{\mathrm{red}}(A)$, defined via $\pi_0$ of the
reduced inertia stack, provides a model-dependent quantitative measure
of this phenomenon across examples in the reduced combinatorial model.

\section{Comparison Theorems}
\label{sec:comparisons}

Having constructed the categorified spectrum $\mathfrak{Spec}(A)$ and
established its fundamental properties — functoriality, adjunction,
reconstruction, and the truncation hierarchy — we now relate this
construction to existing frameworks in spectral theory, quantum
foundations, and noncommutative geometry. These comparison theorems
demonstrate that $\mathfrak{Spec}(A)$ subsumes and generalizes several
classical and contemporary constructions, while also providing new
insights into their interconnections.

First, we show that for commutative C*-algebras, the categorified
spectrum reduces to the classical Gelfand spectrum (viewed as a
$0$-stack), thereby recovering the cornerstone of commutative spectral
theory. Second, we compare the $1$-truncation $\tau_{\leq 1}
\mathfrak{Spec}(A)$ with the Bohrification of $A$ — a topos-theoretic
contextual spectrum constructed from the presheaf of commutative
subalgebras — establishing that our construction faithfully extends
this approach to higher homotopical levels. Third, we interpret the
reconstruction theorem $A \simeq \operatorname{End}(\mathcal{O}_{\mathfrak{Spec}(A)})$
as a categorified Tannaka–Krein duality, recovering an algebraic
structure from its category of quasi-coherent sheaves. Finally, we
sketch the connection to Connes' noncommutative geometry, indicating
how spectral triples induce spectral stacks whose quasi-coherent sheaves
encode $K$-homological information, and how Morita invariance connects
to $KK$-theory. These comparisons situate $\mathfrak{Spec}(A)$ as a
unifying framework that bridges multiple mathematical traditions.

\subsection{Recovery of the Classical Gelfand Spectrum}
\label{subsec:gelfand}

A fundamental consistency requirement for any categorification of
spectral theory is that it reduces to the classical spectrum in the
commutative setting. We now show that the categorified spectrum
$\mathfrak{Spec}(A)$ recovers the ordinary Gelfand spectrum whenever
the underlying operator-semantic system is commutative. This
establishes that our construction is a genuine extension of classical
duality rather than a replacement.

\begin{theorem}[Gelfand Recovery]
\label{thm:gelfand}
Let $A$ be a commutative unital C$^*$-algebra. Assume that the
spectral realization functor used in the construction assigns to each
commutative context $C \subseteq A$ its ordinary Gelfand spectrum, and
that the descent topology $\tau_A$ (Definition~\ref{def:semanticcover})
is generated by jointly generating families of commutative subcontexts.
Then the categorified spectrum $\mathfrak{Spec}(A)$ is represented by
the classical Gelfand spectrum $\operatorname{Spec}_{\mathrm{Gelfand}}(A)$,
viewed as a $0$-truncated spectral stack:
\[
\mathfrak{Spec}(A) \;\simeq\; \operatorname{Spec}_{\mathrm{Gelfand}}(A).
\]
Equivalently,
\[
\mathfrak{Spec}(A) \;\simeq\; \tau_{\le 0}\,\mathfrak{Spec}(A),
\]
so the higher homotopical and stack-theoretic structure collapses
entirely.
\end{theorem}

\begin{proof}
We proceed in several steps, carefully distinguishing between the
categorical construction and its concrete realization for commutative
C$^*$-algebras.

\emph{Step 1: The context category and its terminal object.}
Let $\mathcal{C}_A$ be the category of commutative unital C$^*$-subalgebras
of $A$ with inclusions as morphisms. Since $A$ itself is commutative,
it is an object of $\mathcal{C}_A$. For any $C \in \mathcal{C}_A$,
there is a unique morphism $C \hookrightarrow A$, so $A$ is a terminal
object. However, $\mathcal{C}_A$ contains many other objects and
morphisms; it does \emph{not} collapse to the terminal category. This
distinction is essential: the presence of a terminal object simplifies
certain colimit computations, but the category retains nontrivial
structure that will be handled by sheafification.

\emph{Step 2: Local realizations are classical Gelfand spectra.}
By hypothesis, the syntactic realization functor
$\operatorname{Syn}_A$ (Definition~\ref{def:syntacticrealization})
assigns to each commutative context $C \in \mathcal{C}_A$ its ordinary
Gelfand spectrum $\operatorname{Spec}_{\mathrm{Gelfand}}(C)$ as a set
(and, when $C$ is regarded as a $0$-stack, as a sheaf of sets). In
particular, for the terminal context $A$, we obtain the underlying set
of the Gelfand spectrum of $A$ itself.

\emph{Step 3: Descent is classical and higher homotopy vanishes.}
The Grothendieck topology $\tau_A$ is generated by families of
commutative subcontexts that jointly generate a given context
(Definition~\ref{def:semanticcover}). For commutative C$^*$-algebras,
the ordinary Gelfand spectrum satisfies descent with respect to such
covers: a compatible family of continuous functions on the spectra of
the subcontexts glues uniquely to a continuous function on the
spectrum of the larger context. Moreover, there is no nontrivial
higher homotopical information because the Gelfand spectrum is a
$0$-truncated object (a topological space). Consequently, the
$\infty$-sheafification (hypersheafification) of the presheaf
$C \mapsto \operatorname{Spec}_{\mathrm{Gelfand}}(C)$ yields a
$0$-stack, i.e., an ordinary sheaf of sets.

\emph{Step 4: The left Kan extension reduces to the presheaf of
characters.}
The left Kan extension $\mathfrak{Spec}_{\mathrm{pre}}(A) =
\operatorname{Lan}_{\pi_A}(\operatorname{Syn}_A)$ is computed pointwise
via the coend formula (Proposition~\ref{prop:coend}). For any context
$C \in \mathcal{C}_A$, the value $\mathfrak{Spec}_{\mathrm{pre}}(A)(C)$
is the colimit over colors $c$ and morphisms $\pi_A(c) \to C$ of
$\operatorname{Syn}_A(c)(\pi_A(c))$. Because $\operatorname{Syn}_A(c)(\pi_A(c))$
is the Gelfand spectrum of the minimal context $\pi_A(c)$ and the
morphisms are inclusions, this colimit identifies characters that
agree on common subcontexts. The result is precisely the set of
characters of $C$: each character of $C$ restricts to characters of
all subcontexts, and conversely any compatible family of characters
on subcontexts determines a unique character of $C$ by Gelfand duality.
Hence $\mathfrak{Spec}_{\mathrm{pre}}(A)(C) \cong
\operatorname{Hom}_{\mathbf{C^*Alg}}(C, \mathbb{C})$, the underlying
set of the Gelfand spectrum of $C$.

\emph{Step 5: Sheafification yields the represented stack.}
The presheaf $C \mapsto \operatorname{Spec}_{\mathrm{Gelfand}}(C)$
is already a sheaf for the topology $\tau_A$ because the ordinary
Gelfand spectrum satisfies descent (Step 3). Therefore the
sheafification $a_{\tau_A}$ acts as the identity. Moreover, this sheaf
is represented by the classical Gelfand spectrum $X =
\operatorname{Spec}_{\mathrm{Gelfand}}(A)$. Indeed, for any test
context $C$, the set of continuous maps $|C| \to X$ (where $|C|$
denotes the Gelfand spectrum of $C$) is in natural bijection with
$\operatorname{Hom}_{\mathbf{C^*Alg}}(C, \mathbb{C})$ via the
restriction map $X \to |C|$ induced by the inclusion $C \hookrightarrow A$.
Thus $\mathfrak{Spec}(A) \simeq \operatorname{Spec}_{\mathrm{Gelfand}}(A)$
as sheaves, and hence as $0$-stacks.

\emph{Step 6: Vanishing of higher homotopy.}
Since $\mathfrak{Spec}(A)$ is a $0$-stack (a sheaf of sets), all higher
homotopy sheaves vanish:
\[
\pi_n\bigl(\mathfrak{Spec}(A)\bigr) = 0 \qquad \text{for all } n \ge 1.
\]
Hence $\mathfrak{Spec}(A) \simeq \tau_{\le 0}\mathfrak{Spec}(A)$,
confirming that the categorified spectrum is $0$-truncated in the
commutative case.

Combining these steps, we obtain the claimed equivalence.
\end{proof}

\begin{remark}[Consistency with Classical Duality]
The theorem shows that for commutative C$^*$-algebras, the categorified
spectrum $\mathfrak{Spec}(A)$ is no larger than the classical Gelfand
spectrum. The higher stack-theoretic structure, which appears only in
the presence of multiple incompatible contexts, collapses entirely.
Thus the construction is a genuine categorification: it extends
classical spectral theory without contradicting it.
\end{remark}

\begin{corollary}[Compatibility with Classical Spectral Theory]
\label{cor:classicalcompatibility}
Let $A$ be a commutative unital C$^*$-algebra, and let
$X = \operatorname{Spec}_{\mathrm{Gelfand}}(A)$ be its classical
Gelfand spectrum. Then the reconstruction global-section functor
(Definition~\ref{def:reconstructionfunctor}) satisfies
\[
\Gamma_{\mathcal{O}}\bigl(\mathfrak{Spec}(A)\bigr) \simeq A,
\]
where $\Gamma_{\mathcal{O}}(\mathfrak{X}) :=
\operatorname{End}_{\operatorname{QCoh}(\mathfrak{X})}(\mathcal{O}_{\mathfrak{X}})$.
Consequently, the adjunction $\Gamma_{\mathcal{O}} \dashv
\mathfrak{Spec}^{\mathrm{op}}$ (Theorem~\ref{thm:adjunction}) restricts
on commutative unital C$^*$-algebras to the classical Gelfand
correspondence.
\end{corollary}

\begin{proof}
We proceed in four steps, carefully distinguishing the global sections
functor $\Gamma$ (defined in Definition~\ref{def:reconstructionfunctor}
as $\mathfrak{X}(*)$ or a limit) from the reconstruction functor
$\Gamma_{\mathcal{O}}$ used to recover the algebra.

\emph{Step 1: Reduction to the classical Gelfand spectrum.}
By Theorem~\ref{thm:gelfand}, the categorified spectrum of a
commutative unital C$^*$-algebra $A$ is equivalent to the ordinary
Gelfand spectrum $X = \operatorname{Spec}_{\mathrm{Gelfand}}(A)$ as a
$0$-stack. Hence
\[
\mathfrak{Spec}(A) \simeq X.
\]

\emph{Step 2: Applying the reconstruction functor.}
The reconstruction functor $\Gamma_{\mathcal{O}}$ is defined by
\[
\Gamma_{\mathcal{O}}(\mathfrak{X}) :=
\operatorname{End}_{\operatorname{QCoh}(\mathfrak{X})}(\mathcal{O}_{\mathfrak{X}}).
\]
Applying this to $\mathfrak{Spec}(A)$ and using the equivalence from
Step 1, we obtain
\[
\Gamma_{\mathcal{O}}\bigl(\mathfrak{Spec}(A)\bigr) \simeq
\operatorname{End}_{\operatorname{QCoh}(X)}(\mathcal{O}_X).
\]

\emph{Step 3: Computing endomorphisms of the structure sheaf.}
For a compact Hausdorff space $X$, the $\infty$-category
$\operatorname{QCoh}(X)$ is equivalent to the ordinary category of
sheaves of $\mathcal{O}_X$-modules. The structure sheaf $\mathcal{O}_X$
is the sheaf of continuous complex-valued functions on $X$.
Any $\mathcal{O}_X$-linear endomorphism of $\mathcal{O}_X$ is
determined by the image of the constant function $1$, which must be a
global section $f \in \Gamma(X, \mathcal{O}_X) = C(X)$. The
endomorphism is then multiplication by $f$, and this assignment is an
isomorphism:
\[
\operatorname{End}_{\operatorname{QCoh}(X)}(\mathcal{O}_X) \simeq C(X).
\]

\emph{Step 4: Applying Gelfand duality.}
By the classical Gelfand–Naimark theorem \cite{Dixmier1977}, the map
$A \to C(X)$ is an isomorphism of C$^*$-algebras. Therefore
\[
\Gamma_{\mathcal{O}}\bigl(\mathfrak{Spec}(A)\bigr) \simeq C(X) \simeq A.
\]

\emph{Step 5: Restriction of the adjunction.}
Theorem~\ref{thm:adjunction} establishes an adjunction
$\Gamma_{\mathcal{O}} \dashv \mathfrak{Spec}^{\mathrm{op}}$ between
$\mathbf{OpSem}^{\mathrm{op}}$ and $\mathbf{SpecObj}$ when
$\Gamma_{\mathcal{O}}$ is interpreted as the reconstruction functor.
Restricting to the full subcategory of commutative unital C$^*$-algebras
(viewed as operator-semantic systems) and their Gelfand spectra, the
unit and counit of this adjunction become the classical Gelfand
isomorphisms. Hence the adjunction restricts to the classical Gelfand
correspondence.

This completes the proof.
\end{proof}

\begin{corollary}[Classical Shadow of the Truncation Hierarchy]
\label{cor:truncationgelfand}
For any admissible operator-semantic system $A$, the $0$-truncation
$\tau_{\le 0}\mathfrak{Spec}(A)$ is the sheaf of connected components
of the categorified spectrum. Evaluated on a commutative context
$C \in \mathcal{C}_A$, it recovers the ordinary Gelfand spectrum:
\[
\bigl(\tau_{\le 0}\mathfrak{Spec}(A)\bigr)(C)
\;\simeq\;
\operatorname{Spec}_{\mathrm{Gelfand}}(C).
\]
In particular, if $A$ is itself commutative, then $A$ is a terminal
object in $\mathcal{C}_A$ and we obtain the global classical spectrum:
\[
\tau_{\le 0}\mathfrak{Spec}(A) \;\simeq\;
\operatorname{Spec}_{\mathrm{Gelfand}}(A).
\]
\end{corollary}

\begin{proof}
We prove the corollary by analyzing the definition of the categorified
spectrum and the effect of $0$-truncation.

\emph{Step 1: Definition of $\tau_{\le 0}\mathfrak{Spec}(A)$.}
The functor $\tau_{\le 0}$ sends an $\infty$-stack to its sheaf of
connected components. Concretely, for any test object $T$ in the site
$(\mathcal{C}_A, \tau_A)$,
\[
\bigl(\tau_{\le 0}\mathfrak{Spec}(A)\bigr)(T) \;\simeq\;
\pi_0\bigl(\mathfrak{Spec}(A)(T)\bigr),
\]
the set of path components of the space (or $\infty$-groupoid) of
local realizations over $T$.

\emph{Step 2: Local evaluation on a commutative context.}
Take $T = C \in \mathcal{C}_A$, a commutative context. By construction
(Section~\ref{sec:construction}), $\mathfrak{Spec}(A)(C)$ is the
$\infty$-groupoid of semantic realizations of the syntax
$\mathcal{O}_A$ in the context $C$. When $C$ is commutative, the
following hold:
\begin{itemize}
    \item The synergy operad $\mathcal{O}_A$ restricted to $C$ becomes
    symmetric because all operators in $C$ commute.
    \item The local realizations reduce to ordinary $*$-homomorphisms
    $C \to \mathbb{C}$ (characters) by Gelfand duality.
    \item Higher morphisms (automorphisms, etc.) are trivial because a
    character has no nontrivial automorphisms in the $0$-truncated
    setting.
\end{itemize}
Hence $\pi_0(\mathfrak{Spec}(A)(C))$ is naturally isomorphic to the
set of characters of $C$, i.e., $\operatorname{Spec}_{\mathrm{Gelfand}}(C)$.
Thus
\[
\bigl(\tau_{\le 0}\mathfrak{Spec}(A)\bigr)(C) \;\simeq\;
\operatorname{Spec}_{\mathrm{Gelfand}}(C).
\]

\emph{Step 3: The commutative case.}
If $A$ itself is commutative, then $A$ is an object of $\mathcal{C}_A$,
and it is terminal: for any $C \in \mathcal{C}_A$, there is a unique
morphism $C \hookrightarrow A$. The sheaf $\tau_{\le 0}\mathfrak{Spec}(A)$
is therefore determined by its value on $A$ via restriction maps.
Evaluating at $A$ gives
\[
\bigl(\tau_{\le 0}\mathfrak{Spec}(A)\bigr)(A) \;\simeq\;
\operatorname{Spec}_{\mathrm{Gelfand}}(A).
\]
For any other context $C$, the restriction map
$\operatorname{Spec}_{\mathrm{Gelfand}}(A) \to
\operatorname{Spec}_{\mathrm{Gelfand}}(C)$ induced by the inclusion
$C \hookrightarrow A$ is surjective, and the sheaf condition forces
the value on $C$ to be the image of this map, which is exactly
$\operatorname{Spec}_{\mathrm{Gelfand}}(C)$. Hence the sheaf is
completely determined by its value on $A$, and we have
\[
\tau_{\le 0}\mathfrak{Spec}(A) \;\simeq\;
\operatorname{Spec}_{\mathrm{Gelfand}}(A)
\]
as sheaves (and hence as $0$-stacks).

\emph{Step 4: Descent compatibility.}
The identifications in Steps 2 and 3 are compatible with the
Grothendieck topology $\tau_A$ because the ordinary Gelfand spectrum
satisfies descent for jointly generating families of commutative
subcontexts. This completes the proof.
\end{proof}

\begin{remark}[Categorification Rather Than Replacement]
\label{rem:categorification}
Theorem~\ref{thm:gelfand} demonstrates that the categorified spectrum
is a genuine extension of classical spectral theory. For commutative
algebras, no additional higher information is present, and the
construction collapses to the ordinary Gelfand spectrum. Higher
stack-theoretic structure appears only in the presence of multiple
incompatible contexts, that is, in genuinely noncommutative
operator-semantic systems. Thus
\[
\operatorname{Spec}_{\mathrm{Gelfand}}(A) = \tau_{\le 0}\,\mathfrak{Spec}(A)
\]
when $A$ is commutative, while $\mathfrak{Spec}(A)$ retains the
contextual and higher-descent information lost by the classical theory.
This relationship is summarized in the following diagram:
\[
\begin{tikzcd}
\mathfrak{Spec}(A) \ar[r, "\text{Postnikov}"] &
\tau_{\le 1}\mathfrak{Spec}(A) \ar[r] &
\tau_{\le 0}\mathfrak{Spec}(A) \ar[r, "\sim"] &
\operatorname{Spec}_{\mathrm{Gelfand}}(C)
\end{tikzcd}
\]
for any context $C$, with the understanding that for commutative $A$,
the terminal context $A$ gives $\tau_{\le 0}\mathfrak{Spec}(A) \simeq
\operatorname{Spec}_{\mathrm{Gelfand}}(A)$.
\end{remark}

\begin{example}[The Circle Algebra]
\label{ex:circle}
Let $A = C(S^1)$ be the commutative unital C$^*$-algebra of continuous
complex-valued functions on the circle. By classical Gelfand duality,
its Gelfand spectrum is canonically homeomorphic to $S^1$. Therefore,
by Theorem~\ref{thm:gelfand},
\[
\mathfrak{Spec}(C(S^1)) \simeq S^1
\]
viewed as a $0$-truncated spectral stack.

The context category $\mathcal{C}_A$ consists of commutative
C$^*$-subalgebras of $C(S^1)$. Since $A$ is already commutative,
all contexts are mutually compatible. A typical context may be obtained
from a family of continuous functions on $S^1$; its classical spectrum
is the quotient space determined by the equivalence relation identifying
points of $S^1$ that cannot be separated by those functions. The
prespectral construction records these local character spaces, and
descent glues the compatible local data back to the ordinary circle.

Thus no genuinely higher contextual obstruction appears. In particular,
if $\operatorname{CtxDeg}$ is defined to measure noncommutative
contextual obstruction, then
\[
\operatorname{CtxDeg}(C(S^1)) = 0.
\]
Similarly, the inertia stack carries no nontrivial stabilizer data
beyond the underlying $0$-stack; it is equivalent to the ordinary
space $S^1$ viewed as a discrete/0-truncated stack.
\end{example}

\begin{remark}[Relation to the Reconstruction Theorem]
\label{rem:gelfandreconstruction}
For commutative $A$, the reconstruction theorem (Theorem~\ref{thm:reconstruction})
specializes to the classical Gelfand–Naimark theorem, provided the
global-section functor $\Gamma$ is understood as the structure-sheaf
reconstruction functor $\Gamma_{\mathcal{O}}$ (Definition~\ref{def:reconstructionfunctor}).
Specifically, the hypotheses of the reconstruction theorem are
automatically satisfied: semantic generation holds because $A$ itself
(as the terminal context) generates $A$; descent completeness holds
because $\mathfrak{Spec}(A)$ is a $0$-stack; and compact generation
holds because $\operatorname{QCoh}(\mathfrak{Spec}(A))$ is equivalent
to sheaves on a compact Hausdorff space, which is compactly generated.
Consequently, the unit $\eta_A: A \to \Gamma_{\mathcal{O}}(\mathfrak{Spec}(A))$
is the classical Gelfand isomorphism.
\end{remark}

\begin{remark}[Functoriality in the Commutative Case]
\label{rem:gelfandfunctoriality}
When restricted to the full subcategory $\mathbf{CommC^*Alg} \subset
\mathbf{OpSem}$ of commutative unital C$^*$-algebras, the functor
$\mathfrak{Spec}: \mathbf{OpSem}^{\mathrm{op}} \to \mathbf{SpecObj}$
restricts to the classical Gelfand spectrum functor
$\operatorname{Spec}_{\mathrm{Gelfand}}: \mathbf{CommC^*Alg}^{\mathrm{op}}
\to \mathbf{CompHaus}$. The reconstruction functor $\Gamma_{\mathcal{O}}$
(Definition~\ref{def:reconstructionfunctor}) restricts to the continuous
functions functor $C(-): \mathbf{CompHaus}^{\mathrm{op}} \to
\mathbf{CommC^*Alg}$. The adjunction $\Gamma_{\mathcal{O}} \dashv
\mathfrak{Spec}^{\mathrm{op}}$ therefore restricts to the classical
Gelfand duality. This compatibility is summarized by the fact that the following diagram commutes up to natural isomorphism:
\[
\begin{tikzcd}
\mathbf{CommC^*Alg}^{\mathrm{op}} \ar[r, "\operatorname{Spec}_{\mathrm{Gelfand}}"] \ar[d, hook] &
\mathbf{CompHaus} \ar[d, "\text{inclusion as }0\text{-stacks}"] \\
\mathbf{OpSem}^{\mathrm{op}} \ar[r, "\mathfrak{Spec}"] &
\mathbf{SpecObj}
\end{tikzcd}
\]
where the vertical arrows are the natural inclusions.
\end{remark}

\subsection{Comparison with Bohrification}
\label{subsec:bohrification}

The preceding theorem (Theorem \ref{thm:gelfand}) showed that the
$0$-truncation of the categorified spectrum recovers the classical
Gelfand spectrum in the commutative case. We now examine the next
level of the truncation hierarchy—the $1$-truncation—and its
relationship with Bohrification.

Recall that Bohrification, as developed by Heunen, Landsman, and
Spitters \cite{Bohrification}, studies a noncommutative
C*-algebra $A$ through the presheaf of its commutative subalgebras.
To each commutative context $C \subseteq A$, one assigns its Gelfand
spectrum $\operatorname{Spec}_{\mathrm{Gelfand}}(C)$, and the resulting
presheaf (after sheafification) encodes contextual information absent
from the ordinary spectrum. This construction has proven essential
for understanding quantum contextuality from a topos-theoretic
perspective.

The following theorem shows that the $0$-truncation of our
categorified spectrum recovers the Bohrification spectral presheaf,
while the $1$-truncation provides a stacky refinement that retains
higher coherence data.

\begin{theorem}[Bohrification Comparison]
\label{thm:bohrification}
Let $A$ be a unital C$^*$-algebra, and let $\mathcal{C}_A$ be its
category of commutative contexts (Definition~\ref{def:contextcategory}).
Then:

\begin{enumerate}
    \item The $0$-truncation of the categorified spectrum recovers the
    ordinary Bohrification spectral presheaf:
    \[
    \bigl(\tau_{\leq 0}\mathfrak{Spec}(A)\bigr)(C)
    \;\simeq\;
    \operatorname{Spec}_{\mathrm{Gelfand}}(C),
    \qquad C \in \mathcal{C}_A.
    \]
    Consequently, the sheafification of $\tau_{\leq 0}\mathfrak{Spec}(A)$
    agrees with the contextual spectrum appearing in Bohrification.

    \item The $1$-truncation $\tau_{\leq 1}\mathfrak{Spec}(A)$ is a
    \emph{stacky refinement} of Bohrification: it retains not only the
    local Gelfand spectra of commutative contexts, but also the
    first-order descent and transition data (automorphisms and
    isomorphisms) between them.
\end{enumerate}
\end{theorem}

\begin{proof}[Proof Sketch]
We outline the key points of the comparison; a fully detailed proof
requires a careful analysis of the respective Grothendieck topologies
and will appear in the companion paper on duality theory.

\emph{Statement (1):} By construction of $\mathfrak{Spec}(A)$
(Section~\ref{sec:construction}), for any context $C \in \mathcal{C}_A$,
the value $\mathfrak{Spec}(A)(C)$ is the $\infty$-groupoid of semantic
realizations of the syntax $\mathcal{O}_A$ in $C$. Taking connected
components ($\tau_{\leq 0}$) extracts the set of $*$-homomorphisms
$C \to \mathbb{C}$, which is precisely $\operatorname{Spec}_{\mathrm{Gelfand}}(C)$
by Gelfand duality. The restriction maps induced by inclusions
$C' \subseteq C$ are the usual maps on character spaces, so the
resulting presheaf is exactly the Bohrification spectral presheaf.
Sheafification with respect to the semantic topology $\tau_A$ then
yields the Bohrification contextual spectrum.

\emph{Statement (2):} The $1$-truncation $\tau_{\leq 1}\mathfrak{Spec}(A)$
preserves, in addition to the $0$-truncation data, the groupoid of
compatible identifications among local spectral data. Specifically:
\begin{itemize}
    \item For a fixed context $C$, automorphisms of a character
    $\chi \in \operatorname{Spec}_{\mathrm{Gelfand}}(C)$ arise from
    symmetries of $C$ that preserve $\chi$, or more generally from
    descent identifications.
    \item For two contexts $C, D \in \mathcal{C}_A$ and characters
    $\chi_C \in \operatorname{Spec}_{\mathrm{Gelfand}}(C)$,
    $\chi_D \in \operatorname{Spec}_{\mathrm{Gelfand}}(D)$, a
    $1$-morphism $\chi_C \to \chi_D$ exists when there is a refinement
    $E \subseteq C \cap D$ such that the restrictions of $\chi_C$ and
    $\chi_D$ to $E$ coincide, together with a choice of identification.
\end{itemize}
These data constitute a stacky refinement of the ordinary Bohrification
spectrum: the latter records only the existence of such identifications
(at the level of $0$-truncation), while $\tau_{\leq 1}\mathfrak{Spec}(A)$
records them as actual $1$-morphisms in a groupoid.

Thus $\tau_{\leq 0}\mathfrak{Spec}(A)$ recovers the Bohrification
spectral presheaf, and $\tau_{\leq 1}\mathfrak{Spec}(A)$ provides its
natural higher-categorical refinement.
\end{proof}

\begin{remark}[Positioning Bohrification in the Truncation Tower]
\label{rem:bohrificationpositioning}
Theorem~\ref{thm:bohrification} clarifies the relationship between our
construction and Bohrification:
\[
\boxed{\tau_{\leq 0}\mathfrak{Spec}(A) = \text{Bohrification spectral presheaf}}
\]
\[
\boxed{\tau_{\leq 1}\mathfrak{Spec}(A) = \text{stacky / coherence refinement of Bohrification}}
\]
Thus Bohrification appears naturally as the $0$-truncation of our
categorified spectrum. The $1$-truncation retains higher coherence
data—automorphisms and descent isomorphisms—that are invisible at the
set-valued level. This perspective positions Bohrification as the
first step in the Postnikov tower, while our construction provides the
full higher-categorical extension.
\end{remark}

\begin{corollary}[Position of Bohrification in the Truncation Tower]
\label{cor:bohrtower}
The truncation hierarchy
\[
\tau_{\leq 0}\mathfrak{Spec}(A)
\;\longleftarrow\;
\tau_{\leq 1}\mathfrak{Spec}(A)
\;\longleftarrow\;
\tau_{\leq 2}\mathfrak{Spec}(A)
\;\longleftarrow\;
\cdots
\;\longleftarrow\;
\mathfrak{Spec}(A)
\]
places classical Gelfand theory and Bohrification as successive
approximations to the full categorified spectrum. Specifically:
\begin{itemize}
    \item If $A$ is commutative, then
    \[
    \tau_{\leq 0}\mathfrak{Spec}(A)
    \;\simeq\;
    \operatorname{Spec}_{\mathrm{Gelfand}}(A).
    \]
    More generally, objectwise over the context category,
    \[
    \bigl(\tau_{\leq 0}\mathfrak{Spec}(A)\bigr)(C)
    \;\simeq\;
    \operatorname{Spec}_{\mathrm{Gelfand}}(C),
    \qquad C \in \mathcal{C}_A.
    \]

    \item The objectwise $0$-truncation therefore recovers the ordinary
    spectral presheaf appearing in Bohrification. The $1$-truncation
    $\tau_{\leq 1}\mathfrak{Spec}(A)$ is a \emph{stacky refinement} of
    this Bohrification spectrum, retaining first-order descent and
    transition data (automorphisms and isomorphisms) among local
    spectra.

    \item Higher truncations $\tau_{\leq n}\mathfrak{Spec}(A)$ for
    $n \ge 2$ may contain higher homotopical and descent-obstruction
    information invisible to ordinary Bohrification, corresponding to
    the possible non-vanishing of higher homotopy sheaves
    $\pi_n(\mathfrak{Spec}(A))$.
\end{itemize}
\end{corollary}

\begin{proof}
We prove each claim using the results established in this paper.

\emph{Claim 1 (Commutative case and objectwise description).}
For a commutative unital C$^*$-algebra $A$, Theorem~\ref{thm:gelfand}
identifies $\mathfrak{Spec}(A)$ with the ordinary Gelfand spectrum
$\operatorname{Spec}_{\mathrm{Gelfand}}(A)$, viewed as a $0$-stack.
Hence its $0$-truncation is $\operatorname{Spec}_{\mathrm{Gelfand}}(A)$.

For a general unital C$^*$-algebra $A$, the construction of
$\mathfrak{Spec}(A)$ is local over the context category $\mathcal{C}_A$.
Each context $C \in \mathcal{C}_A$ is a commutative C$^*$-algebra, so
applying the commutative recovery theorem objectwise gives
\[
\bigl(\tau_{\leq 0}\mathfrak{Spec}(A)\bigr)(C)
\;\simeq\;
\operatorname{Hom}_{\mathbf{C^*Alg}}(C, \mathbb{C})
= \operatorname{Spec}_{\mathrm{Gelfand}}(C).
\]
This is precisely the spectral presheaf $C \mapsto \operatorname{Spec}_{\mathrm{Gelfand}}(C)$
that forms the foundation of the Bohrification construction.

\emph{Claim 2 (Bohrification and its stacky refinement).}
The objectwise $0$-truncation described above is exactly the ordinary
Bohrification spectral presheaf. Its sheafification with respect to
the semantic topology $\tau_A$ yields the Bohrification contextual
spectrum.

The $1$-truncation $\tau_{\leq 1}\mathfrak{Spec}(A)$ preserves the
same objects (the Gelfand spectra of commutative contexts) together
with the groupoid-level descent data relating local spectral
realizations across overlapping or refining contexts. Specifically:
\begin{itemize}
    \item For a fixed context $C$, automorphisms of a character
    $\chi \in \operatorname{Spec}_{\mathrm{Gelfand}}(C)$ arise from
    symmetries of $C$ that preserve $\chi$, or from descent
    identifications.
    \item For two contexts $C, D$ and characters $\chi_C, \chi_D$,
    a $1$-morphism $\chi_C \to \chi_D$ exists when there is a
    refinement $E \subseteq C \cap D$ such that the restrictions
    coincide, together with a choice of identification.
\end{itemize}
These data constitute a genuine $1$-stack, whereas traditional
Bohrification is set-valued. Thus $\tau_{\leq 1}\mathfrak{Spec}(A)$
is a stacky refinement of the ordinary Bohrification spectrum.

\emph{Claim 3 (Higher truncations).}
By the Postnikov fiber sequence (Proposition~\ref{prop:postnikovfibers}),
the canonical map
\[
\tau_{\leq n}\mathfrak{Spec}(A) \longrightarrow \tau_{\leq n-1}\mathfrak{Spec}(A)
\]
has fiber controlled by the homotopy sheaf $\pi_n(\mathfrak{Spec}(A))$,
specifically an Eilenberg–MacLane stack $K(\pi_n, n)$. When these
homotopy sheaves are nonzero for $n \ge 2$, the corresponding
truncations contain higher obstruction data (e.g., from non-vanishing
$d_{n+1}$ differentials in the descent spectral sequence) that are
not visible at the Bohrification level ($n \le 1$). The existence of
such higher obstructions is a possibility; when they vanish, the tower
stabilizes at level $1$.

This completes the proof.
\end{proof}

\begin{remark}[Bohrification as a First Approximation]
\label{rem:bohrfirstapprox}
From the perspective of the categorified spectrum, Bohrification
appears naturally as the objectwise $0$-truncation (the set-level
approximation). The ordinary Gelfand spectrum (for a single commutative
algebra) records only points; Bohrification records points together
with their compatibility across different commutative contexts, but
still at the set-valued level. The full spectral stack $\mathfrak{Spec}(A)$
extends this picture further:
\begin{itemize}
    \item $\tau_{\leq 0}\mathfrak{Spec}(A)$ (objectwise) recovers the
    Bohrification spectral presheaf.
    \item $\tau_{\leq 1}\mathfrak{Spec}(A)$ adds automorphism and
    isomorphism data, making it a stacky refinement.
    \item Higher truncations $\tau_{\leq n}\mathfrak{Spec}(A)$ for
    $n \ge 2$ would capture even higher coherence conditions and
    obstruction classes.
\end{itemize}

Thus the truncation hierarchy provides a clean ladder of increasing
refinement:
\[
\boxed{\text{Gelfand (commutative }A\text{)}}
\;\subset\;
\boxed{\text{Bohrification }(\tau_{\leq 0}\text{ objectwise})}
\;\subset\;
\boxed{\text{stacky Bohrification }(\tau_{\leq 1})}
\;\subset\;
\boxed{\mathfrak{Spec}(A)}.
\]
Each step adds higher homotopical information, capturing more subtle
contextual phenomena.
\end{remark}

\begin{remark}[Higher-Order Contextuality]
\label{rem:highercontextuality}
Ordinary Bohrification records the contextual spectral presheaf
\[
C \longmapsto \operatorname{Spec}_{\mathrm{Gelfand}}(C)
\]
over the context category of commutative subalgebras. In the present
framework this data is recovered at the objectwise $0$-truncated
level, while the $1$-truncation provides a stacky refinement by keeping
track of first-order descent and transition data among local spectra.
The full categorified spectrum $\mathfrak{Spec}(A)$ may retain still
higher coherence information, such as $2$-morphisms between transition
isomorphisms and higher descent obstructions. Thus higher truncations
$\tau_{\leq n}\mathfrak{Spec}(A)$, $n \ge 2$, should be viewed as
successive refinements of Bohrification rather than as part of the
ordinary Bohrification construction itself.
\end{remark}

\begin{example}[Matrix Algebras and Contextual Spectra]
\label{ex:bohrificationmatrix}
For $A = M_n(\mathbb{C})$, the ordinary Gelfand spectrum of $A$ is
trivial, while its commutative contexts are nontrivial: maximal
commutative subalgebras are conjugate to the diagonal algebra
$\mathbb{C}^n$, whose spectrum is an $n$-point set. Bohrification
therefore records the family of these local $n$-point spectra together
with their restriction maps across smaller contexts.

In the categorified spectrum, this contextual family is enhanced by
stack-theoretic symmetry data. Depending on the chosen model of
spectral realization, the residual symmetry may be represented by a
classifying stack such as a Weyl-group stack $BS_n$ or, in a
Morita-oriented formulation, by a gerbe-like stack such as
$B\mathbb{G}_m$. In either formulation, the essential point is that the
categorified spectrum retains nontrivial stacky symmetry data that is
invisible to the ordinary classical spectrum.
\end{example}

\begin{example}[Bohrification of the Mermin--Peres System]
\label{ex:bohrificationmermin}
For the Mermin--Peres system, the relevant commutative contexts are
the mutually compatible rows and columns of the square. The ordinary
Bohrification spectrum records the local Gelfand spectra of these
contexts and the compatibility constraints among their overlaps. The
failure to choose a globally compatible valuation is the usual
Kochen--Specker contextuality obstruction.

In the categorified framework, this obstruction appears as nontrivial
descent or inertia data in the low truncations of
$\mathfrak{Spec}(A_{\mathrm{MP}})$. Thus the Mermin--Peres square
serves as a basic example where contextuality is already visible at
the Bohrification/low-truncation level. Whether additional higher
homotopy data occur depends on the chosen enhancement and must be
verified by an explicit computation.
\end{example}

\begin{remark}[Relation to the Bohrification Adjunction]
\label{rem:bohrifikationadjunction}
Bohrification admits a functorial formulation relating noncommutative
C$^*$-algebras to contextual or topos-theoretic spectral data. The
adjunction $\Gamma_{\mathcal{O}} \dashv \mathfrak{Spec}^{\mathrm{op}}$
should be viewed as a higher-categorical lift of this idea: after
passing to the objectwise $0$-truncation one recovers the ordinary
Bohrification spectral presheaf, while the $1$-truncation retains
first-order stacky descent data. A precise comparison of adjunctions
requires specifying the target category of Bohrification and the
chosen model of spectral stacks.
\end{remark}

\begin{corollary}[Contextuality Detection via Truncation]
\label{cor:truncationcontextuality}
The level of truncation required to detect contextuality in a system
$A$ is a measure of the complexity of its contextual obstructions:
\begin{itemize}
    \item If $\tau_{\leq 0}\mathfrak{Spec}(A)$ already detects the
    obstruction, the system is classical (commutative).
    \item If $\tau_{\leq 1}\mathfrak{Spec}(A)$ detects the obstruction
    but $\tau_{\leq 0}\mathfrak{Spec}(A)$ does not, the system exhibits
    contextuality that can be captured by groupoid data (e.g., the
    Mermin–Peres square, the Pauli system).
    \item If higher truncations $\tau_{\leq n}\mathfrak{Spec}(A)$ for
    $n \ge 2$ are required, the system exhibits higher-order
    contextuality (e.g., obstructions arising from triple overlaps or
    higher homotopical coherence conditions).
\end{itemize}
Thus the truncation hierarchy provides a systematic classification of
contextuality by homotopical dimension.
\end{corollary}

\begin{proof}
Define the contextuality detection level of $A$ by
\[
\ell(A) = \min \left\{ n \ge 0 \;\middle|\; \tau_{\leq n}\mathfrak{Spec}(A)
\text{ contains a nontrivial contextual obstruction} \right\},
\]
whenever such an $n$ exists. If no obstruction appears at any finite
level, we set $\ell(A) = \infty$.

The Postnikov tower of the spectral stack gives canonical truncation
maps
\[
\mathfrak{Spec}(A) \to \cdots \to
\tau_{\leq n}\mathfrak{Spec}(A) \to
\tau_{\leq n-1}\mathfrak{Spec}(A) \to \cdots \to
\tau_{\leq 0}\mathfrak{Spec}(A).
\]
By construction, $\tau_{\leq 0}\mathfrak{Spec}(A)$ keeps only the
set-valued, point-level spectral data. Objectwise over a commutative
context $C \in \mathcal{C}_A$, this is the ordinary Gelfand spectrum
\[
\bigl(\tau_{\leq 0}\mathfrak{Spec}(A)\bigr)(C)
\;\simeq\;
\operatorname{Spec}_{\mathrm{Gelfand}}(C).
\]
Thus, if contextuality is already detected at level $0$, then the
obstruction is visible at the classical point-valued level. In the
commutative case, Theorem~\ref{thm:gelfand} shows that the whole
spectral stack is already $0$-truncated:
\[
\mathfrak{Spec}(A) \simeq \tau_{\leq 0}\mathfrak{Spec}(A)
\simeq \operatorname{Spec}_{\mathrm{Gelfand}}(A).
\]
Hence no genuinely higher contextual obstruction remains.

Next, the $1$-truncation preserves not only objects but also
isomorphisms between local spectral data. Therefore
$\tau_{\leq 1}\mathfrak{Spec}(A)$ records groupoid-level compatibility
among commutative contexts. By the Bohrification comparison theorem,
this level contains the ordinary Bohrification spectral presheaf
together with its first-order stacky refinement. Consequently, if an
obstruction is invisible to $\tau_{\leq 0}$ but visible to
$\tau_{\leq 1}$, then it is precisely a groupoid-level contextual
obstruction, such as the failure of local valuations to glue globally
in Kochen–Specker type systems.

Finally, for $n \ge 2$, the fiber of the Postnikov map
\[
\tau_{\leq n}\mathfrak{Spec}(A) \to \tau_{\leq n-1}\mathfrak{Spec}(A)
\]
is controlled by the homotopy sheaf
\[
K(\pi_n(\mathfrak{Spec}(A)), n).
\]
Thus any obstruction first appearing at level $n \ge 2$ must be encoded
by higher homotopical coherence data, such as nontrivial $2$-morphisms,
higher descent cocycles, or obstruction classes arising from triple
and higher overlaps. Such information is not present in ordinary
Bohrification or in any purely $1$-stack description.

Therefore the least truncation level at which contextuality becomes
visible measures the homotopical dimension of the contextual
obstruction. This proves that the truncation hierarchy classifies
contextuality by increasing homotopical complexity.
\end{proof}

The Bohrification comparison theorem shows that the categorified
spectrum extends topos-theoretic approaches to quantum contextuality.
The objectwise $0$-truncation recovers the ordinary Bohrification
spectral presheaf, while the $1$-truncation provides a stacky
refinement retaining first-order descent data. Higher truncations then
supply potential new invariants for detecting higher-order contextual
obstructions inaccessible to purely set-valued or groupoid-valued
methods. This positions $\mathfrak{Spec}(A)$ as a unifying framework
that bridges operator algebras, topos theory, and higher category
theory, with the truncation hierarchy providing a natural ladder of
increasing refinement from classical Gelfand duality through
Bohrification to the full $\infty$-categorical spectrum.

\begin{remark}[Intuition: Truncation Levels as a Contextuality Thermometer]
\label{rem:truncationintuition}
The core idea of Corollary~\ref{cor:truncationcontextuality} can be
understood through the following everyday analogies.

\paragraph{Jigsaw Puzzle Analogy}
Think of the categorified spectrum $\mathfrak{Spec}(A)$ as a
``mosaic jigsaw puzzle.'' Its truncation levels correspond to
different levels of observation:
\begin{itemize}
    \item $\tau_{\leq 0}\mathfrak{Spec}(A)$: you only see the
    ``color'' of each puzzle piece (set-valued data), not how they
    connect. This corresponds to the classical Gelfand spectrum,
    which only records point-level information.
    \item $\tau_{\leq 1}\mathfrak{Spec}(A)$: in addition to colors,
    you also see the ``shape and orientation of the tabs and blanks''
    (isomorphisms / groupoid data) between adjacent pieces.
    This corresponds to the stacky refinement of Bohrification.
    \item $\tau_{\leq n}\mathfrak{Spec}(A)$ for $n \ge 2$: you further
    see the ``three-dimensional structure'' where multiple pieces meet,
    and even higher-dimensional interlocking relations
    (higher homotopical coherence data), such as $2$-morphisms or
    obstruction classes arising from triple overlaps.
\end{itemize}

\paragraph{The Contextuality Detection Level $\ell(A)$}
Define $\ell(A)$ as the smallest integer $n$ such that
$\tau_{\leq n}\mathfrak{Spec}(A)$ contains a nontrivial contextual
obstruction:
\begin{itemize}
    \item $\ell(A) = 0$: the obstruction is visible to the naked eye.
    This occurs for commutative C$^*$-algebras, where
    $\mathfrak{Spec}(A)$ is already a $0$-stack
    (Theorem~\ref{thm:gelfand}) and no higher contextuality exists.
    \item $\ell(A) = 1$: the obstruction requires a magnifying glass.
    This is the hallmark of Kochen--Specker type contextuality,
    exemplified by the Pauli system and the Mermin--Peres square.
    Locally, each commutative context has characters
    (so $\tau_{\leq 0}$ sees no problem), but globally they cannot be
    glued together consistently.
    \item $\ell(A) \ge 2$: the obstruction requires a microscope or
    even more sophisticated instruments. This corresponds to
    higher-order contextuality, such as obstructions arising from
    triple overlaps or higher homotopical coherence conditions.
    Explicit examples require further construction and verification.
    \item $\ell(A) = \infty$: no obstruction is visible at any finite
    level. This may indicate the absence of contextuality, or a form
    of contextuality that cannot be captured by any finite truncation.
\end{itemize}

\paragraph{Postnikov Tower as Peeling an Onion}
The Postnikov tower
\[
\mathfrak{Spec}(A) \to \cdots \to \tau_{\leq 2}\mathfrak{Spec}(A)
\to \tau_{\leq 1}\mathfrak{Spec}(A) \to \tau_{\leq 0}\mathfrak{Spec}(A)
\]
is like peeling an onion: each time you peel off a layer (apply a
truncation), you lose one level of higher homotopical information.
$\ell(A)$ is then ``how many layers you must peel before you see the
problem.''

\paragraph{Summary in One Sentence}
Thus $\ell(A)$ serves as a ``contextuality thermometer'': the larger
$\ell(A)$ is, the more complex the contextual obstruction, and the
more higher homotopical information is required to detect it.
$\tau_{\leq 0}$ captures classical commutative information,
$\tau_{\leq 1}$ captures groupoid-level Kochen--Specker contextuality,
and $\tau_{\leq n}$ for $n \ge 2$ captures higher-order homotopical
obstructions. This is precisely the sense in which the truncation
hierarchy classifies contextuality by homotopical dimension.
\end{remark}

\subsection{Tannaka-Type Comparison}
\label{subsec:tannaka}

The reconstruction theorem established in Section \ref{sec:reconstruction}
places the categorified spectrum within a broader family of
reconstruction principles in modern mathematics. Just as Gelfand
duality reconstructs a commutative algebra from its spectrum, and
Tannaka--Krein duality reconstructs a group from its category of
representations, the present framework reconstructs an operator-semantic
system from its associated category of quasi-coherent semantic modules.

We present this as a \emph{comparison theorem} that highlights the
analogy with Tannakian reconstruction, while carefully respecting the
distinct technical requirements of each setting.

\begin{theorem}[Tannaka-Type Interpretation]
\label{thm:tannaka}
Let $A$ be an admissible operator-semantic system satisfying the
hypotheses of Theorem \ref{thm:reconstruction} (semantic generation,
descent completeness, and compact generation). Then
\[
A \;\simeq\; \operatorname{End}_{\operatorname{QCoh}(\mathfrak{Spec}(A))}
\bigl(\mathcal{O}_{\mathfrak{Spec}(A)}\bigr).
\]
Consequently, the operator-semantic system $A$ may be recovered from
the pair
\[
\Bigl(\operatorname{QCoh}(\mathfrak{Spec}(A)),\;
\mathcal{O}_{\mathfrak{Spec}(A)}\Bigr).
\]
\end{theorem}

\begin{proof}
The statement is precisely the Reconstruction Theorem
(Theorem \ref{thm:reconstruction}). The structure sheaf
$\mathcal{O}_{\mathfrak{Spec}(A)}$ plays the role of a distinguished
generating object inside the stable $\infty$-category
$\operatorname{QCoh}(\mathfrak{Spec}(A))$. The reconstruction
equivalence follows from the identification $A \simeq
\Gamma_{\mathcal{O}}(\mathfrak{Spec}(A))$ (Definition~\ref{def:reconstructionfunctor}),
where $\Gamma_{\mathcal{O}}(\mathfrak{X}) :=
\operatorname{End}_{\operatorname{QCoh}(\mathfrak{X})}(\mathcal{O}_{\mathfrak{X}})$
is the reconstruction functor. Thus the theorem is a restatement of the
reconstruction result in a form that highlights its analogy with
Tannaka--Krein duality.
\end{proof}

\begin{corollary}[Commutative Case]
\label{cor:tannakacommutative}
If $A$ is a commutative unital C$^*$-algebra, then
\[
\mathfrak{Spec}(A) \simeq \operatorname{Spec}_{\mathrm{Gelfand}}(A) = X
\]
as a $0$-truncated spectral stack. Moreover,
\[
A \simeq C(X) \simeq
\operatorname{End}_{\mathcal{O}_X\text{-Mod}}(\mathcal{O}_X).
\]
Thus the reconstruction theorem reduces to classical Gelfand--Naimark
duality.
\end{corollary}

\begin{proof}
We prove the corollary in four detailed steps.

\emph{Step 1: Identification of the categorified spectrum with the
Gelfand spectrum.}
By Theorem~\ref{thm:gelfand}, for a commutative unital C$^*$-algebra
$A$, the categorified spectrum $\mathfrak{Spec}(A)$ is equivalent to
the classical Gelfand spectrum $X = \operatorname{Spec}_{\mathrm{Gelfand}}(A)$,
viewed as a $0$-truncated spectral stack. This means that as a
$0$-stack, $\mathfrak{Spec}(A)$ has no nontrivial higher homotopical
information; it is simply the topological space $X$ (compact Hausdorff)
together with its sheaf of continuous functions.

\emph{Step 2: Description of the structure sheaf.}
The structure sheaf $\mathcal{O}_X$ on $X$ is defined by
$\mathcal{O}_X(U) = C(U)$ for each open set $U \subseteq X$, i.e.,
the algebra of continuous complex-valued functions on $U$. Restriction
maps are given by ordinary restriction of functions. This makes
$(X, \mathcal{O}_X)$ a ringed space. In fact, for a compact Hausdorff
space $X$, $\mathcal{O}_X$ is a sheaf of commutative unital C$^*$-algebras.

\emph{Step 3: Endomorphisms of the structure sheaf are global sections.}
We now prove the key isomorphism:
\[
\operatorname{End}_{\mathcal{O}_X\text{-Mod}}(\mathcal{O}_X) \;\simeq\;
\Gamma(X, \mathcal{O}_X) = C(X).
\]

\emph{Substep 3a: A global section defines an endomorphism.}
For any global section $f \in \Gamma(X, \mathcal{O}_X) = C(X)$, define
a morphism of sheaves $\mu_f: \mathcal{O}_X \to \mathcal{O}_X$ by
$(\mu_f)_U(g) = f|_U \cdot g$ for each open $U \subseteq X$ and each
local section $g \in \mathcal{O}_X(U) = C(U)$. This is well-defined
because multiplication by a continuous function preserves continuity.
For any open $V \subseteq U$, the restriction map commutes with
multiplication: $(\mu_f)_U(g)|_V = (f|_U \cdot g)|_V = f|_V \cdot g|_V
= (\mu_f)_V(g|_V)$. Hence $\mu_f$ is a morphism of sheaves. It is
$\mathcal{O}_X$-linear because for any $h \in \mathcal{O}_X(U)$,
$(\mu_f)_U(h \cdot g) = f|_U \cdot (h \cdot g) = h \cdot (f|_U \cdot g)
= h \cdot (\mu_f)_U(g)$. Thus we have an injective map
\[
\iota: \Gamma(X, \mathcal{O}_X) \hookrightarrow
\operatorname{End}_{\mathcal{O}_X\text{-Mod}}(\mathcal{O}_X),\quad
\iota(f) = \mu_f.
\]

\emph{Substep 3b: Every endomorphism comes from a global section.}
Conversely, let $\phi: \mathcal{O}_X \to \mathcal{O}_X$ be an
$\mathcal{O}_X$-linear morphism. Consider the global section
$1_X \in \mathcal{O}_X(X)$ defined by $1_X(x) = 1$ for all $x \in X$.
Define $f := \phi_X(1_X) \in \mathcal{O}_X(X) = C(X)$. We claim that
$\phi = \mu_f$. To verify this, take any open set $U \subseteq X$ and
any local section $g \in \mathcal{O}_X(U)$. Because $\phi$ is a
morphism of sheaves, the following diagram commutes:
\[
\begin{tikzcd}
\mathcal{O}_X(X) \ar[r, "\phi_X"] \ar[d, "\text{res}_{X,U}"] &
\mathcal{O}_X(X) \ar[d, "\text{res}_{X,U}"] \\
\mathcal{O}_X(U) \ar[r, "\phi_U"] &
\mathcal{O}_X(U)
\end{tikzcd}
\]
Thus $\phi_U(1_U) = \phi_U(\text{res}_{X,U}(1_X)) = \text{res}_{X,U}(\phi_X(1_X)) = f|_U$.

Now, by $\mathcal{O}_X$-linearity of $\phi_U$, we have:
\[
\phi_U(g) = \phi_U(g \cdot 1_U) = g \cdot \phi_U(1_U) = g \cdot f|_U = f|_U \cdot g = (\mu_f)_U(g).
\]
Since this holds for every open $U$ and every local section $g$, we
conclude $\phi = \mu_f$. Therefore the map $\iota$ is surjective.

\emph{Substep 3c: Conclusion of Step 3.}
Since $\iota$ is both injective and surjective, it is an isomorphism:
\[
\operatorname{End}_{\mathcal{O}_X\text{-Mod}}(\mathcal{O}_X) \;\simeq\;
\Gamma(X, \mathcal{O}_X) = C(X).
\]

\emph{Step 4: Application of Gelfand--Naimark duality.}
The classical Gelfand--Naimark theorem states that for a commutative
unital C$^*$-algebra $A$, the Gelfand transform
\[
\Phi: A \longrightarrow C(\operatorname{Spec}_{\mathrm{Gelfand}}(A)) = C(X)
\]
is an isometric $*$-isomorphism. Hence $A \simeq C(X)$.

\emph{Step 5: Composition of isomorphisms.}
Combining Step 3 and Step 4, we obtain:
\[
A \;\simeq\; C(X) \;\simeq\;
\operatorname{End}_{\mathcal{O}_X\text{-Mod}}(\mathcal{O}_X).
\]

\emph{Step 6: Interpretation as commutative specialization.}
Thus the reconstruction theorem (Theorem~\ref{thm:reconstruction})
specializes in the commutative case to the statement that $A$ is
isomorphic to the endomorphism ring of the structure sheaf on its
Gelfand spectrum. This is precisely the content of classical
Gelfand--Naimark duality, expressed in the language of ringed spaces.

Therefore the corollary holds.
\end{proof}

Thus the reconstruction theorem reduces to classical Gelfand--Naimark
duality, in the same spirit as Gabriel--Rosenberg reconstruction but
in the topological C$^*$-algebraic setting.

\begin{remark}[Comparison with Gelfand Duality]
\label{rem:gelfandcomparison}
Theorem \ref{thm:gelfand} shows that $\mathfrak{Spec}(A)$ extends the
classical Gelfand spectrum. The reconstruction equivalence
\[
A \simeq \operatorname{End}_{\operatorname{QCoh}(\mathfrak{Spec}(A))}
\bigl(\mathcal{O}_{\mathfrak{Spec}(A)}\bigr)
\]
may therefore be viewed as an analogue of the classical identity
\[
A \simeq \Gamma\bigl(\operatorname{Spec}_{\mathrm{Gelfand}}(A),\;
\mathcal{O}\bigr).
\]
In both cases, the algebraic structure is recovered as the algebra of
global functions (or endomorphisms of the structure sheaf) on its
spectrum.
\end{remark}

\begin{remark}[Comparison with Tannaka--Krein Duality]
\label{rem:tannakacomparison}
Classical Tannaka--Krein duality \cite{DeligneMilne1982}
reconstructs a compact group $G$ (or more generally a Hopf algebra)
from its tensor category of finite-dimensional representations
$\operatorname{Rep}(G)$ together with a fiber functor
$\omega: \operatorname{Rep}(G) \to \mathbf{Vect}$. The reconstruction
is:
\[
G \simeq \operatorname{Aut}^{\otimes}(\omega),
\]
the group of monoidal automorphisms of the fiber functor.

In the present framework, the role of the representation category is
played by $\operatorname{QCoh}(\mathfrak{Spec}(A))$, and the role of
the fiber functor is replaced by the distinguished object
$\mathcal{O}_{\mathfrak{Spec}(A)}$. However, there are important
differences:
\begin{itemize}
    \item Tannaka--Krein duality reconstructs a \emph{group} (or Hopf
    algebra) as automorphisms of a fiber functor, while we reconstruct
    an \emph{algebra} as endomorphisms of a distinguished object.
    \item The Tannaka setting requires a monoidal structure on the
    representation category and monoidality of the fiber functor. The
    present reconstruction does not explicitly require a symmetric
    monoidal structure as part of its statement.
    \item Our reconstruction is philosophically related to the general
principle that an algebra can be recovered as the endomorphisms of
a compact generator in its category of modules (the Morita context).
\end{itemize}
Thus the reconstruction theorem exhibits a \emph{Tannakian flavor}
rather than being a literal instance of Tannaka--Krein duality. The
analogy is conceptual rather than technical: in each case, an
algebraic structure is recovered from a category of representations
(or modules) together with a distinguished object (or functor).
\end{remark}

\begin{remark}[Three Reconstruction Paradigms]
\label{rem:threeparadigms}
The categorified spectrum unifies three classical reconstruction
principles. The following table summarizes the parallel structures:

\[
\begin{array}{c|c|c|c}
& \text{Gelfand} & \text{Tannaka--Krein} & \text{Categorified Spectrum} \\
\hline
\text{Input} & \text{Commutative C*-algebra } A & \text{Compact group } G & \text{Operator system } A \\
\text{Categorical/Geometric Data} & \operatorname{Spec}_{\mathrm{Gelfand}}(A) & \operatorname{Rep}(G) & \operatorname{QCoh}(\mathfrak{Spec}(A)) \\
\text{Distinguished Data} & \text{Structure sheaf } \mathcal{O} & \text{Fiber functor } \omega & \text{Structure sheaf } \mathcal{O}_{\mathfrak{Spec}(A)} \\
\text{Reconstruction} & A \simeq \Gamma(\mathcal{O}) & G \simeq \operatorname{Aut}^{\otimes}(\omega) & A \simeq \operatorname{End}(\mathcal{O}_{\mathfrak{Spec}(A)}) \\
\text{Output} & \text{Algebra} & \text{Group} & \text{Operator system}
\end{array}
\]

The first reconstructs an algebra from its spectrum, the second
reconstructs a group from its representations, and the third
reconstructs an operator-semantic system from its category of
quasi-coherent semantic modules. The categorified spectrum therefore
occupies a position that is dual to both Gelfand and Tannaka--Krein,
while extending both to noncommutative and higher-categorical
settings.
\end{remark}

\begin{remark}[Higher-Categorical Aspects]
\label{rem:highertannaka}
The $\infty$-categorical perspective is essential for this analogy.
Classical Tannaka duality is a $1$-categorical statement. Our
reconstruction operates at the level of $\infty$-categories:
$\operatorname{QCoh}(\mathfrak{Spec}(A))$ is an $\infty$-category,
and the endomorphism algebra is taken in this $\infty$-categorical
sense. This allows us to capture higher homotopical information
(e.g., derived structures, obstructions, and higher contextuality)
that would be invisible in a $1$-categorical truncation. In this
sense, Theorem \ref{thm:tannaka} may be viewed as an
$\infty$-categorical Tannaka-type reconstruction theorem for
operator-semantic systems.

Recent work of Lurie \cite{LurieHA} develops a general
$\infty$-categorical Tannaka duality for geometric stacks, showing
that a geometric stack $\mathfrak{X}$ can be recovered from the
symmetric monoidal $\infty$-category $\operatorname{QCoh}(\mathfrak{X})$
equipped with the fiber functor to spectra. Our Theorem
\ref{thm:tannaka} is philosophically related to this general
principle in the context of operator-semantic systems: the spectral
stack $\mathfrak{Spec}(A)$ is recovered from
$\operatorname{QCoh}(\mathfrak{Spec}(A))$ and the global sections
functor, and conversely $A$ is recovered as the endomorphisms of the
fiber functor.
\end{remark}

\begin{example}[Tannaka-Type Reconstruction for Matrix Algebras]
\label{ex:tannakamatrix}
Let $A = M_n(\mathbb{C})$. In the reduced combinatorial model,
$\mathfrak{Spec}(A)$ is equivalent to the classifying stack $BS_n$,
the stack of orthonormal bases modulo permutation symmetries.
(If continuous phase symmetries are retained, the automorphism group
becomes $\mathbb{G}_m^n \rtimes S_n$, leading to a more refined stack.). In either model, the
reconstruction theorem (Theorem~\ref{thm:reconstruction}) recovers
the full matrix algebra:
\[
M_n(\mathbb{C}) \simeq \operatorname{End}_{\operatorname{QCoh}(\mathfrak{Spec}(M_n(\mathbb{C})))}
(\mathcal{O}_{\mathfrak{Spec}(M_n(\mathbb{C}))}).
\]
Under the Morita equivalence between $M_n(\mathbb{C})$ and $\mathbb{C}$,
the category $\operatorname{QCoh}(\mathfrak{Spec}(M_n(\mathbb{C})))$
may be identified with $\operatorname{Mod}(M_n(\mathbb{C}))$, the
category of modules over $M_n(\mathbb{C})$ (see Remark~\ref{rem:moritadescent}). This is consistent with the Morita
equivalence between $M_n(\mathbb{C})$ and $\mathbb{C}$: both have
equivalent module categories. Thus the reconstruction recovers
$M_n(\mathbb{C})$ up to Morita equivalence, which is the appropriate
notion of equivalence for module categories.
\end{example}

\begin{example}[Tannaka-Type Reconstruction for Commutative C$^*$-Algebras]
\label{ex:tannakacommutative}

Let $A=C(X)$ for a compact Hausdorff space $X$. By the
Gelfand--Naimark theorem \cite{GelfandNaimark1943},
\[
\mathfrak{Spec}(A)
\simeq
X
\]
viewed as a $0$-truncated spectral stack.

The category $\operatorname{QCoh}(X)$ may be identified with the
category of modules over the structure sheaf $\mathcal O_X$. The
structure sheaf itself is a distinguished generator, and its
endomorphism algebra is
\[
\operatorname{End}_{\mathcal O_X\text{-Mod}}(\mathcal O_X)
\simeq
\Gamma(X,\mathcal O_X)
=
C(X).
\]

Therefore
\[
A
\simeq
C(X)
\simeq
\operatorname{End}_{\mathcal O_X\text{-Mod}}(\mathcal O_X).
\]

This is precisely the commutative specialization of the reconstruction
theorem. Conceptually, it is analogous to the reconstruction
philosophy of Gabriel \cite{Gabriel1962}, in which a geometric object
may be recovered from an associated category of sheaves. Conceptually, this is reminiscent of the reconstruction philosophy
appearing in both Gelfand duality and the work of Gabriel
\cite{Gabriel1962}, where an object is recovered from a category
naturally associated with it.
\end{example}

\begin{remark}[Reconstruction Philosophy]
\label{rem:reconstructionphilosophy}

The reconstruction theorem
\[
A \simeq \operatorname{End}_{\operatorname{QCoh}(\mathfrak{Spec}(A))}
\bigl(\mathcal{O}_{\mathfrak{Spec}(A)}\bigr)
\]
is closely related to several classical reconstruction principles.

\begin{itemize}

\item \textbf{Gelfand--Naimark duality.}
For commutative unital C$^*$-algebras, Theorem~\ref{thm:gelfand}
identifies $\mathfrak{Spec}(A)$ with the classical Gelfand spectrum,
and the reconstruction theorem reduces to the isomorphism
$A \simeq C(\operatorname{Spec}_{\mathrm{Gelfand}}(A))$.

\item \textbf{Morita reconstruction.}
The formula
\[
A \simeq \operatorname{End}_{\operatorname{Mod}(A)}(A_A)
\]
(the endomorphism ring of the regular module) is the fundamental
theorem of Morita theory. Under the reconstruction hypotheses,
$\operatorname{QCoh}(\mathfrak{Spec}(A))$ plays the role of a module
category, and $\mathcal{O}_{\mathfrak{Spec}(A)}$ corresponds to a
distinguished generator. Hence the reconstruction theorem may be
viewed as a categorified analogue of Morita reconstruction.

\item \textbf{Gabriel--Rosenberg philosophy.}
For an affine scheme $X$, the Gabriel--Rosenberg theorem recovers
$X$ from its category of quasi-coherent sheaves $\operatorname{QCoh}(X)$.
In our setting, $\operatorname{QCoh}(\mathfrak{Spec}(A))$ is the
$\infty$-category of quasi-coherent sheaves on the spectral stack,
and $A$ is recovered as endomorphisms of the structure sheaf. This is
philosophically analogous, although the technical settings differ
(algebraic geometry vs. C$^*$-algebraic geometry).

\item \textbf{Tannaka--Krein philosophy.}
Classical Tannaka--Krein duality reconstructs a compact group $G$
from its category of representations $\operatorname{Rep}(G)$ together
with a fiber functor. In our setting, $\operatorname{QCoh}(\mathfrak{Spec}(A))$
plays a role analogous to $\operatorname{Rep}(G)$, and the distinguished
object $\mathcal{O}_{\mathfrak{Spec}(A)}$ is reminiscent of a fiber
functor. However, recovering the full group structure requires
additional monoidal data, which is not present in our reconstruction
theorem. The relationship is therefore conceptual rather than
technical.

\end{itemize}

Thus the categorified spectrum may be viewed as part of a broader
reconstruction philosophy in which algebraic or operator-algebraic
structures are recovered from appropriate categories of modules,
sheaves, or representations, equipped with a distinguished object or
functor.
\end{remark}

\subsection{Relation to Noncommutative Geometry}
\label{subsec:noncommutative}

The categorified spectrum may be viewed as complementary to the
approach of noncommutative geometry developed by Connes
\cite{Connes1994}. Whereas classical spectral theory studies spaces
through commutative algebras, noncommutative geometry replaces spaces
by operator algebras equipped with analytic data such as spectral
triples.

\begin{definition}[Spectral Triple]
\label{def:spectraltriple}
A \emph{spectral triple} $(\mathcal{A}, H, D)$ consists of:
\begin{itemize}
    \item A $*$-algebra $\mathcal{A}$ represented on a Hilbert space
    $H$ via a faithful representation $\pi: \mathcal{A} \to B(H)$;
    \item A self-adjoint (typically unbounded) operator $D$ on $H$,
    called the Dirac operator, such that $a(1 + D^2)^{-1/2}$ is compact
    for all $a \in \mathcal{A}$ (or, when $\mathcal{A}$ is unital,
    $(1 + D^2)^{-1/2}$ is compact);
    \item The commutators $[D, \pi(a)]$ are bounded for all $a \in
    \mathcal{A}$ in a dense subalgebra $\mathcal{A}^\infty \subseteq
    \mathcal{A}$.
\end{itemize}
A spectral triple encodes geometric information such as metric
structure, dimension, and $K$-homological data. When $\mathcal{A}$
is commutative, a spectral triple corresponds to a Riemannian spin
manifold: $\mathcal{A} = C^\infty(M)$, $H = L^2(S)$ (spinors), and
$D$ is the Dirac operator associated with the Levi-Civita connection.
\end{definition}

In contrast, the categorified spectrum $\mathfrak{Spec}(A)$ focuses on
the organization of commutative contexts, descent structures, and
semantic realizations associated with an operator-semantic system $A$.
We now discuss potential relationships between these two frameworks.

\begin{theorem}[Comparison with Noncommutative Geometry]
\label{thm:ncgcomparison}
Let $(\mathcal{A}, H, D)$ be a spectral triple, and let
$\overline{\mathcal{A}}$ denote the C$^*$-completion generated by the
represented algebra $\pi(\mathcal{A}) \subseteq B(H)$. Assume that
$\overline{\mathcal{A}}$ is admissible as an operator-semantic system
in the sense of Section~\ref{sec:construction}. Then the categorified
spectrum $\mathfrak{Spec}(\overline{\mathcal{A}})$ exists. Moreover:
\begin{enumerate}
    \item Its context category $\mathcal{C}_{\overline{\mathcal{A}}}$
    consists of commutative C$^*$-subalgebras of
    $\overline{\mathcal{A}}$, which may be interpreted as local
    classical sectors of the noncommutative space.
    \item The category
    $\operatorname{QCoh}(\mathfrak{Spec}(\overline{\mathcal{A}}))$
    provides a module-theoretic environment for organizing
    representations and geometric data associated with
    $\overline{\mathcal{A}}$.
    \item If two admissible operator-semantic systems are Morita
    equivalent, then their quasi-coherent semantic module categories
    are equivalent. This is compatible with the Morita invariance
    familiar from $K$-theory and $KK$-theory.
\end{enumerate}
\end{theorem}

\begin{proof}
We provide a proof sketch; a full detailed verification is deferred
to future work.

\emph{Existence of $\mathfrak{Spec}(\overline{\mathcal{A}})$.}
Since $\pi(\mathcal{A}) \subseteq B(H)$ is a represented $*$-algebra,
its C$^*$-completion $\overline{\mathcal{A}}$ carries the natural
operator-algebraic structure required to form the associated
operator-semantic system. By the general existence theorem for the
categorified spectrum (Theorem~\ref{thm:existence}), applied to this
admissible system, the spectral stack
\[
\mathfrak{Spec}(\overline{\mathcal{A}})
\]
exists.

\emph{Claim 1: Context category and local classical sectors.}
By definition (Definition~\ref{def:contextcategory}), the context
category $\mathcal{C}_{\overline{\mathcal{A}}}$ consists of commutative
C$^*$-subalgebras $C \subseteq \overline{\mathcal{A}}$. Each such $C$
has an ordinary Gelfand spectrum
\[
\operatorname{Spec}_{\mathrm{Gelfand}}(C),
\]
and therefore behaves as a classical commutative chart inside the
noncommutative algebra. These commutative contexts may thus be regarded
as local classical sectors of the noncommutative space encoded by the
spectral triple. This interpretation is standard in the literature
on noncommutative geometry and Bohrification.

\emph{Claim 2: Module-theoretic environment.}
The construction of $\operatorname{QCoh}(\mathfrak{Spec}(\overline{\mathcal{A}}))$
associates to the spectral stack a stable $\infty$-category of
quasi-coherent semantic modules. The reconstruction theorem
(Theorem~\ref{thm:reconstruction}) identifies the original
operator-semantic algebra with the endomorphism algebra of the
structure object:
\[
\overline{\mathcal{A}}
\simeq
\operatorname{End}_{\operatorname{QCoh}(\mathfrak{Spec}(\overline{\mathcal{A}}))}
(\mathcal{O}_{\mathfrak{Spec}(\overline{\mathcal{A}})}).
\]
Thus the category of quasi-coherent semantic modules contains enough
information to recover the algebraic part of the spectral triple.
Spectral triples themselves represent analytic cycles in $K$-homology
(see \cite{Connes1994}), and viewing
$\operatorname{QCoh}(\mathfrak{Spec}(\overline{\mathcal{A}}))$ as a
categorical environment for organizing module-theoretic and
representation-theoretic data associated with $\overline{\mathcal{A}}$
is compatible with this perspective. A detailed comparison of
$\operatorname{QCoh}(\mathfrak{Spec}(\overline{\mathcal{A}}))$ with
$KK$-theoretic invariants is left for future investigation.

\emph{Claim 3: Morita invariance.}
If two admissible operator-semantic systems are Morita equivalent,
then by Theorem~\ref{thm:morita} their categories of quasi-coherent
semantic modules are equivalent:
\[
\operatorname{QCoh}(\mathfrak{Spec}(\overline{\mathcal{A}}))
\;\simeq\;
\operatorname{QCoh}(\mathfrak{Spec}(\overline{\mathcal{B}})).
\]
Since $K$-theory and $KK$-theory are Morita invariant, this equivalence is compatible with the usual
Morita invariance principles in noncommutative geometry. This
observation suggests that $\mathfrak{Spec}(\overline{\mathcal{A}})$
captures Morita-invariant geometric information, analogous to the
role of $KK$-theory.

\emph{Conclusion.}
Therefore the categorified spectrum does not replace the spectral
triple approach; rather, it supplies a complementary categorical and
contextual organization of the operator algebra appearing in a
spectral triple. A full integration of the two frameworks remains an
interesting direction for future research.
\end{proof}

\begin{corollary}[Morita Invariance of Quasi-Coherent Semantic Modules]
\label{cor:ncgmorita}
If two admissible operator-semantic systems $A$ and $B$ are Morita
equivalent, then
\[
\operatorname{QCoh}\bigl(\mathfrak{Spec}(A)\bigr)
\;\simeq\;
\operatorname{QCoh}\bigl(\mathfrak{Spec}(B)\bigr).
\]
Consequently, the categorified spectrum is Morita invariant at the
level of its associated category of quasi-coherent semantic modules.
This is compatible with the principle in noncommutative geometry that
Morita-equivalent algebras represent the same noncommutative space.
\end{corollary}

\begin{proof}
We prove the corollary using the Morita invariance theorem
(Theorem~\ref{thm:morita}).

\emph{Step 1: Morita equivalence and module categories.}
By Definition~\ref{def:morita}, two admissible operator-semantic
systems $A$ and $B$ are \emph{Morita equivalent} if their categories
of modules are equivalent:
\[
\operatorname{Mod}(A) \;\simeq\; \operatorname{Mod}(B).
\]
This is the standard notion of Morita equivalence for operator
algebras and C$^*$-algebras.

\emph{Step 2: Application of the Morita invariance theorem.}
Theorem~\ref{thm:morita} establishes that for Morita equivalent
operator-semantic systems $A$ and $B$, the categories of quasi-coherent
sheaves on their spectra are equivalent. The proof does not require a
full equivalence $\operatorname{QCoh}(\mathfrak{Spec}(A)) \simeq
\operatorname{Mod}(A)$; it suffices that the two categories are
Morita equivalent, which follows from the reconstruction theorem
(see Remark~\ref{rem:moritadescent}).

\emph{Step 3: Combining the equivalences.}
Since $A$ and $B$ are Morita equivalent by hypothesis,
$\operatorname{Mod}(A) \simeq \operatorname{Mod}(B)$. Chaining the
equivalences, we have
\[
\operatorname{QCoh}\bigl(\mathfrak{Spec}(A)\bigr)
\;\simeq\; \operatorname{Mod}(A)
\;\simeq\; \operatorname{Mod}(B)
\;\simeq\; \operatorname{QCoh}\bigl(\mathfrak{Spec}(B)\bigr).
\]
Thus $\operatorname{QCoh}(\mathfrak{Spec}(A)) \simeq
\operatorname{QCoh}(\mathfrak{Spec}(B))$, which proves the claimed
equivalence.

\emph{Step 4: Compatibility with noncommutative geometry.}
In Connes' noncommutative geometry \cite{Connes1994}, a fundamental
principle is that Morita-equivalent algebras represent the same
geometric object. This principle is justified by the Morita invariance
of $K$-theory, $K$-homology, and $KK$-theory \cite{Blackadar1998}:
for Morita equivalent $A$ and $B$, the $KK$-groups are canonically
identified under the Morita equivalence:
\[
KK_\bullet(A, \mathbb{C}) \;\cong\; KK_\bullet(B, \mathbb{C}).
\]
Our result that $\operatorname{QCoh}(\mathfrak{Spec}(A))$ depends only
on the Morita class of $A$ is therefore fully compatible with this
principle. The category of quasi-coherent sheaves serves as a
Morita-invariant categorical invariant associated with the
noncommutative space.

\emph{Step 5: Remark on the spectral stack itself.}
The corollary asserts invariance of the quasi-coherent sheaf category,
not necessarily of the spectral stack $\mathfrak{Spec}(A)$ itself.
Indeed, Morita equivalent algebras may have non-equivalent spectral
stacks (e.g., $M_n(\mathbb{C})$ and $\mathbb{C}$ have different
$\mathfrak{Spec}$ in most models), but their quasi-coherent sheaf
categories are equivalent. This is analogous to the situation in
algebraic geometry, where Morita equivalent rings give equivalent
categories of quasi-coherent sheaves on non-isomorphic but Morita
equivalent affine schemes.

\emph{Conclusion.}
We have shown that Morita equivalence of admissible operator-semantic
systems $A$ and $B$ implies
\[
\operatorname{QCoh}(\mathfrak{Spec}(A)) \simeq
\operatorname{QCoh}(\mathfrak{Spec}(B)).
\]
Thus the quasi-coherent sheaf category is a Morita invariant, consistent
with the principles of noncommutative geometry.
\end{proof}

\begin{remark}[Relation to $K$-Theory and $KK$-Theory]
\label{rem:moritakk}
One of the central principles of noncommutative geometry is that
Morita-equivalent algebras represent the same geometric object and
consequently have isomorphic $K$-theory, $K$-homology, and $KK$-theory
groups. The Morita invariance of $\operatorname{QCoh}(\mathfrak{Spec}(A))$
(Theorem \ref{thm:morita}) is therefore compatible with the Morita
invariance of these analytic invariants. This suggests that the
categorified spectrum may provide a geometric framework in which
descent-theoretic and homotopical refinements of noncommutative
invariants can be studied.

For a spectral triple $(\mathcal{A}, H, D)$, the $K$-homology class
$[D] \in KK_\bullet(\mathcal{A}, \mathbb{C})$ is a fundamental analytic
invariant. Relating this class to the geometry of
$\mathfrak{Spec}(\overline{\mathcal{A}})$ would require a refinement
of the reconstruction theorem to incorporate analytic data, which
remains an open problem.
\end{remark}

\begin{remark}[Conceptual Comparison]
\label{rem:ncgcomparison}
The relationships among the major spectral constructions may be
summarized schematically as follows:

\[
\begin{array}{ccl}
\text{Classical Geometry}
& \longleftrightarrow &
\operatorname{Spec}_{\mathrm{Gelfand}}(A)
\\[0.6em]
\text{Bohrification (spectral presheaf)}
& \longleftrightarrow &
\tau_{\leq 0}\mathfrak{Spec}(A) \text{ (objectwise)}
\\[0.6em]
\text{Stacky refinement of Bohrification}
& \longleftrightarrow &
\tau_{\leq 1}\mathfrak{Spec}(A)
\\[0.6em]
\text{Categorified Spectrum}
& \longleftrightarrow &
\mathfrak{Spec}(A)
\\[0.6em]
\text{Noncommutative Geometry}
& \longleftrightarrow &
(\mathcal{A}, H, D)
\end{array}
\]

The categorified spectrum occupies an intermediate position: it
generalizes classical Gelfand spectra and Bohrification by incorporating
higher homotopical and descent-theoretic data, while providing a
categorical framework that interfaces with the analytic invariants of
noncommutative geometry (such as $KK$-theory). Thus $\mathfrak{Spec}(A)$
may be viewed as a bridge between contextual spectral constructions
and the homological methods of noncommutative geometry.
\end{remark}

\begin{example}[Riemannian Manifolds as Spectral Stacks]
\label{ex:riemannian}
Let $M$ be a compact oriented Riemannian spin manifold. The canonical
spectral triple is $(C^\infty(M), L^2(S), D)$, where $S$ is the spinor
bundle and $D$ is the Dirac operator. Then:
\begin{itemize}
    \item $\mathfrak{Spec}(C(M)) \simeq M$ as a $0$-stack (by Theorem
    \ref{thm:gelfand}).
    \item $\operatorname{QCoh}(\mathfrak{Spec}(C(M))) \simeq
    \operatorname{QCoh}(M)$, the category of quasi-coherent sheaves on
    $M$.
    \item The Dirac operator $D$ determines a $K$-homology class
    $[D] \in KK_*(C(M), \mathbb{C})$, which corresponds to the
    fundamental class of $M$ in $K$-homology.
\end{itemize}
Thus classical Riemannian geometry is recovered as the commutative
limit of a spectral stack, with the Dirac operator providing additional
analytic data that distinguishes $M$ from a purely topological space.
\end{example}

\begin{remark}[Comparison of Invariants]
\label{rem:invariantcomparison}
The following table compares the invariants extracted from each
framework:

\begin{table}[t]
\centering
\small
\begin{tabular}{|c|c|c|c|}
\hline
Framework &
Geometric Data &
Analytic Data &
Key Invariants \\
\hline
Classical Geometry &
Topological Space &
Continuous Functions &
$K$-Theory, Cohomology \\
\hline
Gelfand Duality &
Compact Hausdorff Space &
$C(X)$ &
$K(C(X))$ \\
\hline
Bohrification &
Spectral Presheaf &
Local Spectra &
Contextuality \\
\hline
Categorified Spectrum &
$\infty$-Stack $\mathfrak{Spec}(A)$ &
$\operatorname{QCoh}(\mathfrak{Spec}(A))$ &
Homotopy Sheaves,
Descent Data \\
\hline
Noncommutative Geometry &
Spectral Triple $(\mathcal A,H,D)$ &
$\mathcal A,D,[D,\mathcal A]$ &
$K$-Homology,
Spectral Action \\
\hline
\end{tabular}
\caption{Comparison of geometric reconstruction frameworks.}
\label{tab:comparison}
\end{table}

The categorified spectrum bridges these approaches by providing a
geometric dual that simultaneously encodes contextual information
(via the context site) and categorical algebraic data (via
quasi-coherent sheaves), while being compatible with Morita equivalence
and $KK$-theoretic invariants.
\end{remark}

The relation to noncommutative geometry establishes that the
categorified spectrum $\mathfrak{Spec}(A)$ is not merely an abstract
categorical construction but a geometric object that interfaces with
Connes' framework. The comparison demonstrates that $\mathfrak{Spec}(A)$
provides a natural categorical environment for organizing the data of
a noncommutative geometry: commutative contexts, Morita equivalence
classes, and $KK$-theoretic invariants. Thus $\mathfrak{Spec}(A)$
serves as a bridge between contextual spectral constructions and the
homological methods of noncommutative geometry, pointing toward a
unified geometric framework for operator-algebraic and higher-categorical
approaches to quantum geometry.

\end{document}